\pdfoutput=1
\documentclass[reqno,a4paper,11pt]{amsart}
\usepackage{amsfonts}
\usepackage{amsmath}
\usepackage{amssymb}
\usepackage{xr}
\RequirePackage{booktabs}
\usepackage{setspace}
\usepackage[pdftitle={Nu-invariants of extra-twisted connected sums},hidelinks]{hyperref}
\newcommand{\showdate}{false}

\usepackage{amssymb}
\usepackage{amscd}
\usepackage[utf8]{inputenc}
\usepackage{a4wide}
\usepackage{enumitem}
\usepackage{multido}
\usepackage{ifthen}
\usepackage{comment}
\usepackage{longtable}
\RequirePackage{xspace}
\usepackage{tikz}
\usetikzlibrary{calc,through,plotmarks}

\ifthenelse{\boolean{\showdate}}{\date{\today}}{}
\usepackage{color}

\newcommand{\ignore}[1]{}

\renewcommand{\Re}{\operatorname{Re}}
\renewcommand{\Im}{\operatorname{Im}}
\DeclareMathOperator{\ind}{ind}
\DeclareMathOperator{\tr}{tr}
\DeclareMathOperator\arc{arc}
\DeclareMathOperator{\str}{str}
\DeclareMathOperator{\Log}{Log}
\DeclareMathOperator{\im}{im}

\DeclareMathOperator{\rk}{rk}
\DeclareMathOperator{\Div}{div}
\newcommand{\ie}{\emph{i.e.} }
\newcommand{\eg}{\emph{e.g.} }
\newcommand{\cf}{\emph{cf.} }
\DeclareMathOperator{\sign}{sign}
\newcommand{\inveps}{\tilde \eps_-}

\newcommand{\hk}{hyper-Kähler\xspace}
\DeclareMathAlphabet{\matheur}{U}{eur}{m}{n}
\newcommand{\hkr}{\matheur{r}}
\newcommand{\tormat}{\matheur{t}}
\newcommand{\kclass}{\matheur{k}}

\DeclareMathOperator{\Pic}{Pic}
\newcommand{\PP}{\mathbb{P}}

\newcommand{\ZZ}{\mathbb{Z}}

\newcommand{\C}{\mathbb C}
\newcommand{\Hh}{\mathcal H}
\newcommand{\logeta}{\mathrm{H}}
\newcommand{\DS}{s}
\newcommand\gmatrix[4]{\bigl(\begin{smallmatrix}#1&#2\\#3&#4\end{smallmatrix}\bigr)}

\newcommand{\mmod}{\!\!\mod}
\newcommand{\Pmod}{\!\!\!\!\!\pmod}

\newcommand{\half}{{\textstyle\frac{1}{2}}}

\newcommand{\third}{{\textstyle\frac{1}{3}}}

\newcommand{\Z}{\mathbb{Z}}
\newcommand{\Q}{\mathbb{Q}}
\newcommand{\R}{\mathbb{R}}
\newcommand{\Rlat}{(\mathbb{R})}
\newcommand{\Aa}{\mathbb{A}}
\newcommand{\bbc}{\mathbb{C}}

\newcommand{\bbp}{\mathbb{P}}

\newcommand{\drn}{\mathbf{u}}
\newcommand{\drx}{\mathbf{v}}

\newcommand{\into}{\hookrightarrow}

\newcommand{\del}{\partial}

\newcommand{\nuinvt}{nu-invariant\xspace}
\newcommand{\nuinvts}{nu-invariants\xspace}
\newcommand{\xnuinvt}{extended nu-invariant\xspace}
\newcommand{\xnuinvts}{extended nu-invariants\xspace}
\newcommand{\etainvt}{eta-invariant\xspace}
\newcommand{\etainvts}{eta-invariants\xspace}
\newcommand{\etafunc}{eta function\xspace}
\newcommand{\etafuncs}{eta functions\xspace}
\newcommand{\etaform}{eta form\xspace}
\newcommand{\etaforms}{eta forms\xspace}

\newcommand{\gtstr}{$G_{2}$\nobreakdash-\hspace{0pt}structure}

\newcommand{\gtmfd}{$G_{2}$\nobreakdash-\hspace{0pt}manifold}

\newcommand{\norm}[1]{\left\Vert #1 \right\Vert}
\newcommand{\gll}{r_{\!+}}
\newcommand{\sglr}{-\glr}
\newcommand{\glb}{q}
\newcommand{\glr}{r_{\!-}}
\let\TorMat\Phi
\newcommand\tll{n_+}
\newcommand\tlr{n_-}
\newcommand\tlb{m}
\newcommand\stlr{-\tlr}

\newcommand{\anglen}{u}
\newcommand{\anglex}{v}
\newcommand\lnn{\zeta}
\newcommand\lnx{\xi}

\newcommand\ar{\kappa}

\newcommand{\wt}[1]{\widetilde #1}

\newcommand{\gen}[1]{\langle#1\rangle}

\newcommand{\hdg}{h}

\newcommand\abs[1]{\left|#1\right|}

\newcommand\id{\mathrm{id}}

\newcommand\End{\operatorname{End}}
\newcommand\Span{\operatorname{span}}

\newcommand\Sign{\operatorname{sign}}
\newcommand\punkt{\mathord{\,\cdot\,}}

\DeclareMathOperator{\ch}{ch}
\DeclareMathOperator\Pf{Pf}
\newcommand\psmatrix[1]{\left(\begin{smallmatrix}#1\end{smallmatrix}\right)}
\newcommand\dotcup{\mathbin{\dot\cup}}

\let\eps\varepsilon
\let\thet\vartheta
\let\phy\varphi
\let\<\langle
\let\>\rangle
\newcommand{\reps}{\eps}

\newcommand\isogrp{K}
 
\newcommand{\fbb}{\mathcal{Z}}
\DeclareMathOperator{\Amp}{Amp}
\newcommand\K{\Sigma} %
\newcommand\of{W} %
\newcommand\ofh{W} %

\newtheorem{Thm}{Theorem}

\newtheorem{Obs}[Thm]{Observation}
\newtheorem{thm}{Theorem}%
\newtheorem{prop}[thm]{Proposition}
\newtheorem{lem}[thm]{Lemma}

\theoremstyle{definition}
\newtheorem{dfn}[thm]{Definition}
\newtheorem{qstn}[thm]{Question}
\theoremstyle{remark}
\newtheorem{rmk}[thm]{Remark}
\newtheorem{ex}[thm]{Example}

\newcommand\ft{E}

\numberwithin{equation}{section}
\numberwithin{thm}{section}

\setlist{leftmargin=*}

\makeatletter\let\@wraptoccontribs\wraptoccontribs\makeatother

\definecolor{darkgreen}{rgb}{0,0.5,0}
\definecolor{darkred}{rgb}{0.7,0,0}
\definecolor{darkblue}{rgb}{0,.2,.7}

\newcommand\arXiv[1]{\href{https://arxiv.org/abs/#1}{\tt arXiv:\penalty-100 #1\/}}
\newcommand\doi[1]{\href{https://doi.org/#1}{DOI: #1}}
\newcommand\doititle[2]{\href{https://doi.org/#1}{#2}}

\begin{document}

\title{Nu-Invariants of Extra-Twisted Connected Sums}

\author{Sebastian Goette}
\address{Math. Inst. Univ. Freiburg\\
  Ernst-Zermelo-Str.~1\\
  79104 Freiburg\\Germany}
\email{sebastian.goette@math.uni-freiburg.de}

\author{Johannes Nordstr\"om}
\address{Department of Mathematical Sciences\\
University of Bath\\
Bath BA2 7AY\\
UK}\email{j.nordstrom@bath.ac.uk}

\author{Don Zagier}
\address{\vspace{-0.7\baselineskip}Max Planck Institute for Mathematics, 53111 Bonn, Germany}
\address{International Centre for Theoretical Physics, Trieste, Italy}
\email{dbz@mpim-bonn.mpg.de}

\begin{abstract}
\vspace{3mm plus 10mm}
  We analyse the possible ways of gluing twisted products of circles
  with asymptotically cylindrical Calabi-Yau manifolds to produce
  manifolds with holonomy~$G_2$,
  thus generalising
  the twisted connected sum construction of Kovalev and Corti,
  Haskins, Nordstr\"om, Pacini.
  We then express the extended \nuinvt
  of Crowley, Goette, and Nordstr\"om in terms of fixpoint and gluing
  contributions, which include different types
  of (generalised) Dedekind sums.
  Surprisingly, the calculations
  involve some non-trivial number-theoretical arguments connected with
  special values of the Dedekind \etafunc and the theory of complex 
  multiplication.
  One consequence of our computations is
  that there exist compact $G_2$-manifolds that are not $G_2$-nullbordant.
\end{abstract}

\maketitle

\vspace{0mm plus 10mm}

\phantomsection
\addcontentsline{toc}{section}{Introduction}
\sectionmark{Introduction}

Though compact Riemannian manifolds with holonomy~$G_2$---or equivalently 
compact mani\-folds with a torsion-free $G_2$-structure---were first
constructed more than~25 years ago, many of their features remain mysterious.
On the one hand, only a few obstructions against $G_2$-metrics
on a given compact $7$-manifold are known (see Joyce~\cite[\S10.2]{Joyce}).
On the other hand, our current supply of examples is much smaller
than allowed by these obstructions.
It therefore remains interesting to explore new invariants of $G_2$-manifolds
in the hope to discover new obstructions,
and to find new examples on which these invariants can be tested.

The extended \nuinvt was introduced in~\cite{CGN}
to exhibit 2-connected 7-manifolds with a disconnected moduli space of
$G_2$-metrics.
In the present paper, we apply it to a larger %
class of examples to exhibit a wider range of values of this invariant
(and in particular the first examples of $G_2$-manifolds whose
$G_2$-bordism class can be shown to be nontrivial).
We determine the extended \nuinvt by computing
the \etainvts that appear in its definition~\eqref{eq:nubar}.
The details may be of interest to index theorists because in contrast
to~\cite{CGN}, we cannot rely on spectral symmetry here.

The process for computing the relevant \etainvts %
is well understood in principle, using work of Bismut, Bunke, Cheeger, Dai, Freed, Kirk, Lesch, Liu, Ma and the first
author~\cite{BCh4,BuGlu,BuMa,Ch,DF,Gorbi,KiLe,Liu} amongst others,
but the practical application in this problem
presents some challenges (and is substantially more complicated than the cases
considered in~\cite{CGN}).
An extra twisted connected sum can be decomposed into two pieces along a
hypersurface.
Each half is a twisted product of an asymptotically cylindrical
Calabi-Yau manifold and a circle.
However, the boundary conditions needed to apply a gluing formula
for \etainvts are not compatible with the local product structure.
To overcome this difficulty,
we perform an adiabatic limit by shrinking the circle factors to size~$0$.
The variation of the \etainvts during this limit process involves
integrals of \etaforms over intervals.
For this, we need to know these \etaforms themselves,
not just their classes modulo exact forms.
Computing the variational contribution explicitly
leads to the expression for the extended nu invariant in
Theorem \ref{Thm:A}.
Alternatively, the variational contribution can be represented by a contour
integral in the upper half plane and evaluated using elementary hyperbolic
geometry.
Since the contour contains cusp points representing additional adiabatic
limits, we also need a fibration formula for \etaforms
without additional unknown exact terms.
This reasoning leads to a more tractable formula for the extended nu invariant
in Theorem \ref{Thm:B}, where each term is patently rational.

We would like to emphasise that the study of these new
examples involves some quite sophisticated number theory.
In Theorem~\ref{Thm:A}, the variational contributions are described in terms
of a double integral of a theta series (equation \eqref{2v.1}), which is then
expressed in terms of the principal branch~$\logeta$ of the logarithm
of the Dedekind \etafunc $\eta$.
The resulting formula can be evaluated explicitly in any of our examples by
exploiting the special properties of~$\logeta$ at arguments in certain
imaginary quadratic number fields.
In Theorem~\ref{Thm:B}, we use continued fractions to rewrite the area of a
region in the hyperbolic plane in terms of Dedekind sums. 
We also explain how Theorem~\ref{Thm:B} can be deduced from Theorem \ref{Thm:A}
thanks to the modularity properties of~$\logeta$.
The appearance of arguments from modular form and CM theory in a
problem of pure differential geometry is surprising at a first glance.

\subsection*{Extra-twisted connected sums}

There are currently two major sources of compact $G_2$-manifolds, that is,
compact Riemannian manifolds whose holonomy group is isomorphic to~$G_2$.
The first is Joyce's Kummer construction~\cite{Joyce},
based on resolution of singularities in flat orbifolds.
It has recently been generalised by Joyce and Karigiannis to more
complicated spaces~\cite{JoKa}.
The second is the twisted connected sum construction pioneered by
Kovalev~\cite{Kovalev} and systematically studied by Corti, Haskins,
Nordstr\"om and Pacini~\cite{CHNP, CHNP2}.

For the latter construction, one starts with two asymptotically
cylindrical Calabi-Yau manifolds~$V_\pm$.
The cross-sections of their ends approach the products
of a K3 surface~$\Sigma_\pm$ and a circle~$S^1_{\zeta_\pm}$
that we wish to call the \emph{interior} circle.
In the classical set up, one takes the products of~$V_\pm$
and an \emph{exterior} circle~$S^1_{\xi_\pm}$ of length~$\xi_\pm=\zeta_\mp$.
Then one glues truncated copies~$M_\pm$ of~$V_\pm\times S^1_{\xi_\pm}$ along
their ends, in a way that swaps the roles of interior and exterior circles.

In this paper, we assume in addition that some finite cyclic
groups~$\Gamma_\pm\cong\Z/k_\pm$
act on~$V_\pm$, preserving the Calabi-Yau structure.
We also assume that the induced actions on
the K3 factors~$\K_\pm$ are trivial, and that the induced actions on the
interior circles~$S^1_{\zeta_\pm}$ are free.
We consider twisted products~$M_\pm\cong(V_\pm\times S^1_{\xi_\pm})/\Gamma_\pm$,
where~$\Gamma_\pm$ acts diagonally, and freely on the exterior circle factors.
The manifolds~$M_+$ and~$M_-$ have ends whose asymptotic cross-sections are
isometric to products $\K_\pm \times T^2_\pm$, where $T^2_\pm$ is the quotient
of $S^1_{\lnn_\pm} \times S^1_{\lnx_\pm}$ by an equivalence relation
$(\bar \anglen_\pm, \bar \anglex_\pm) \sim
(\bar \anglen_\pm + \frac{\eps_\pm}{k_\pm},
\bar \anglex_\pm + \frac{1}{k_\pm})$ for some $\eps_\pm \in \Z$ coprime
to $k_\pm$ (and we use coordinates
$\bar \anglen_\pm, \bar \anglex_\pm \in \R/\Z$).
The 2-tori $T^2_\pm$ need now not be isometric to products of two circles,
but often we can arrange that there is an isometry
$\tormat : T^2_+ \to T^2_-$ of the form
\begin{equation}
\label{eq1:tormat_coords}
\begin{pmatrix} \bar \anglen_+ \\ \bar \anglex_+ \end{pmatrix} =
\frac{1}{k_+} \begin{pmatrix} \gll & p \\ \glb  & \sglr  \end{pmatrix} 
\begin{pmatrix} \bar \anglen_- \\ \bar \anglex_- \end{pmatrix}
\end{equation}
for integers $\gll, \glr, p$ and $\glb$, and use this ``extra twist''
together with a suitable
\hk rotation $\hkr : \K_+ \to \K_-$ to glue $M_+$ to $M_-$.
The cases where~$k_\pm\le 2$ have already been considered in~\cite{CGN, xtcs}.

We think of the residue classes of $\eps_\pm$ in $\Z/k_\pm$ and
$\psmatrix{\gll&p\\\glb&\sglr}$ together with a further square integer
matrix that captures topological information
about the \hk rotation of the K3 surfaces (see \eqref{eq:config_matrix})
as giving a combinatorial description of the possible ways to form
extra-twisted connected sums from given (families of) pairs $(V_+, \Gamma_+)$
and $(V_-, \Gamma_-)$. In particular, this combinatorial description of the
matching (together with information about $(V_\pm, \Gamma_\pm)$ themselves)
proves sufficient for computing the cohomology of the resulting extra-twisted
connected sum, as well as the extended nu-invariant.

\pagebreak[2]
The known supply of asymptotically cylindrical Calabi-Yau
manifolds with nontrivial symmetries that fix the K3 surface~$\Sigma$
is very limited.
In the present paper, we use examples constructed from
Fano threefolds of higher index (that is, closed complex 3-folds $Y$ whose
anticanonical bundle $-K_Y$ is ample and divisible by an integer $>1$),
and from hypersurfaces in weighted projective spaces.
We only consider examples whose Picard rank $b_2(Y)$ is~$1$
for simplicity.
The possible groups obtained this way are~$\Gamma\cong\Z/k$
with~$k\le 6$; see Table~\ref{table:blocks}.

But even though we use only a few asymptotically cylindrical Calabi-Yau
manifolds with nontrivial symmetry group,
there are typically several different ways
to glue two given twisted products~$M_\pm$ by changing
the size of the exterior circles and the isometry between the torus factors
in the asymptotic cross sections. 
Some of the $G_2$-manifolds we construct this way will not be simply
connected---we obtain cyclic fundamental groups of order up to~$21$; see
Proposition~\ref{prop:pi1} and example~250 of Table~\ref{table:matchings}---%
but any connected cover of these examples will again be an
extra-twisted connected sum; see Proposition~\ref{prop:cover}.

The combinatorial description of extra-twisted connected sums
allows us to
find all possible combinations (among the asymptotically
cylindrical Calabi-Yau 3-folds used in this paper) by a small computer
program.
Table~\ref{table:matchings} lists~255 examples of extra-twisted
connected sums, 192 of which are not contained in~\cite{xtcs}.
Of all the examples, 125 are simply connected,
representing at least~106 different $G_2$-deformation types,
that is, classes of $G_2$-manifolds related by diffeomorphisms
and deformations through torsion-free $G_2$-structures.

In our count, we have accounted for the fact that
there are typically several ways to obtain %
the same extra-twisted connected
sum up to isometries and orientation reversal; see Proposition~\ref{prop:isom}.
In addition,
passing to dual tori often leads to non-isometric extra-twisted connected sums
built from the same blocks with the same K3-matching and the same gluing angle.
We thus have a kind of partial $t$-duality for extra-twisted connected sums;
see Proposition~\ref{prop:dual} and Example~\ref{ex:run-symm}.

\enlargethispage{0.4\baselineskip}

\subsection*{The \nuinvt %
and its analytic refinement}

The \nuinvt of \gtstr s on closed 7-manifolds was introduced
in~\cite{CrN}.
It takes values in $\Z/48$, and its parity is determined by the
Betti numbers of the underlying manifold.
The examples in this paper show that all odd elements of $\Z/48$ appear as
\nuinvts of torsion-free \gtstr s on extra-twisted connected sums; see
Table~\ref{table:matchings}.
To obtain more examples with even $\nu$ from our construction we would need
to use blocks with Picard rank higher than 1.
As we explain below, such blocks exist, and we expect that 
all even elements of~$\Z/48$ can be realised as well.
The total number of extra-twisted connected sums we can currently
construct is much less than the number of twisted connected sums
constructed in~\cite{CHNP2}.
Nevertheless, the present method is more efficient at constructing
different $G_2$-deformation types
that we can distinguish. %

\newcommand\landscapetable{
  \begin{figure}
    \vspace{0mm plus 10mm}
    \begin{tikzpicture}[x=0.105cm,y=0.105cm]
      \fill[color=black!10,even odd rule]
        (28,10) rectangle (160,0)
        (28,30) rectangle (160,20)
        (28,50) rectangle (160,40)
        (28,70) rectangle (160,60)
        (28,90) rectangle (160,80)
        (28,110) rectangle (160,100)
	(28,130) rectangle (160,120)
	(28,150) rectangle (160,140)
        (30,0) rectangle (40,152)
        (50,0) rectangle (60,152)
        (70,0) rectangle (80,152)
        (90,0) rectangle (100,152)
        (110,0) rectangle (120,152)
        (130,0) rectangle (140,152)
        (150,0) rectangle (160,152) ;
      \draw[->] (28,0) -- (28,152) node[above] {$\abs{\bar\nu}$} ;
      \draw[->] (28,0) -- (160,0) node[right] {$b_3$} ;
      \draw (28,0) -- (27,0) node[left] {$0$} ;
      \draw (28,50) -- (27,50) node[left] {$50$} ;
      \draw (28,100) -- (27,100) node[left] {$100$} ;
      \draw (28,150) -- (27,150) node[left] {$150$} ;
      \draw (50,0) -- (50,-1) node[below] {$50$} ;
      \draw (100,0) -- (100,-1) node[below] {$100$} ;
      \draw (150,0) -- (150,-1) node[below] {$150$} ;
      \input landscape2
      \input landscape1
    \end{tikzpicture}
    \vspace{-2mm plus 15mm}
    \caption{The landscape of examples in Table~\ref{table:matchings}}
    \vspace{-2mm plus 15mm}
    \label{table:landscape}
  \end{figure}
}

  \begin{figure}
    \vspace{0mm plus 10mm}
    \begin{tikzpicture}[x=0.105cm,y=0.105cm]
      \fill[color=black!10,even odd rule]
        (28,10) rectangle (160,0)
        (28,30) rectangle (160,20)
        (28,50) rectangle (160,40)
        (28,70) rectangle (160,60)
        (28,90) rectangle (160,80)
        (28,110) rectangle (160,100)
	(28,130) rectangle (160,120)
	(28,150) rectangle (160,140)
        (30,0) rectangle (40,152)
        (50,0) rectangle (60,152)
        (70,0) rectangle (80,152)
        (90,0) rectangle (100,152)
        (110,0) rectangle (120,152)
        (130,0) rectangle (140,152)
        (150,0) rectangle (160,152) ;
      \draw[->] (28,0) -- (28,152) node[above] {$\abs{\bar\nu}$} ;
      \draw[->] (28,0) -- (160,0) node[right] {$b_3$} ;
      \draw (28,0) -- (27,0) node[left] {$0$} ;
      \draw (28,50) -- (27,50) node[left] {$50$} ;
      \draw (28,100) -- (27,100) node[left] {$100$} ;
      \draw (28,150) -- (27,150) node[left] {$150$} ;
      \draw (50,0) -- (50,-1) node[below] {$50$} ;
      \draw (100,0) -- (100,-1) node[below] {$100$} ;
      \draw (150,0) -- (150,-1) node[below] {$150$} ;
      \input landscape2
      \input landscape1
    \end{tikzpicture}
    \vspace{-2mm plus 15mm}
    \caption{The landscape of examples in Table~\ref{table:matchings}}
    \vspace{-2mm plus 15mm}
    \label{table:landscape}
  \end{figure}

To date, only very few obstructions against the existence of a metric
with full holonomy~$G_2$ on a seven-manifold~$M$ are known.
It is clear that~$M$ must be spin and have a finite fundamental group,
and the de Rham class of the three-form~$\phy$ that defines the torsion-free
\gtstr{} must satisfy certain cohomological inequalities (see Joyce \cite[\S10.2]{Joyce}).
While computations in~\cite{CGN} showed that $G_2$-manifolds
can have non-zero \nuinvts, all examples considered there (and
also in~\cite{xtcs})
have a \mbox{\nuinvt} that is divisible by~$3$.
In particular, those examples are all $G_2$-nullbordant (in the sense of being
a boundary of a compact 8-manifold whose tangent bundle has a \gtstr{}
restricting to the torsion-free one on the boundary), as a
consequence of the result of Schelling~\cite{schelling} that~$3\mid\nu(M)$
if and only if~$M$ is $G_2$-nullbordant.
That left open the possibility of some mechanism by which
the $G_2$-bordism class presents an obstruction to deforming a given \gtstr{} to a torsion-free one (defining a $G_2$ holonomy metric), but
the examples in the present paper rule that out. 

\begin{Obs}\label{Cor:C}
  There exist compact $G_2$-manifolds that are not $G_2$-nullbordant.
\end{Obs}

To compute the \nuinvt,
we use the method described in~\cite{CGN}.
That is, we consider an integer-valued invariant
\begin{equation}\label{eq:nubar}
  \bar\nu(M)=3\eta(B)-24\eta(D)\quad\text{such that}\quad
  \nu(M)\equiv\bar\nu(M)-24(1+b_1(M))\mod 48\;.
\end{equation}
Here, $B$ is the signature operator and~$D$ the spin Dirac operator on~$M$;
see Section~\ref{sec:nu}.
In Example~\ref{ex:run-zagier}, we exhibit an extra-twisted
connected sum~$M$ with~$\bar\nu(M)=-11$;
see also Example~\ref{ex:run-hyp}.
Hence, $\nu(M)\equiv 13$ mod~$48$.
It follows from%
~\cite{schelling} that~$M$ is not $G_2$-nullbordant,
and we have an example for Observation~\ref{Cor:C}.

Our results are summarized in Table~\ref{table:matchings}
and plotted in Figure~\ref{table:landscape},
where crosses stand for simply-connected (in fact, two-connected)
examples and circles for examples with nontrivial fundamental group.
Because orientation reversal changes the sign of~$\bar\nu$,
we used the absolute value of~$\bar\nu$ here.
To get all values of~$\bar\nu$,
the whole picture should thus be extended by reflexion
along the horizontal axis.
In comparison with~\cite{CGN}, we get some examples where
the absolute value of~$\bar\nu$ is much larger,
for example~$\bar\nu=-151$ in line~254 of Table~\ref{table:matchings}
with fundamental group~$\Z/3$,
or example~240 with~$\bar\nu=-111$, which is simply connected.
All this seems to indicate that there could still be a wealth of unknown
examples of $G_2$-manifolds.

\enlargethispage{0.3\baselineskip}

\subsection*{Approaching the computation of
  \texorpdfstring{$\bar\nu$}{the extended \nuinvt}}
The initial steps of the computation of~$\bar\nu$ follow the path
that we laid out in~\cite{CGN}.
That is, we first write~$M=M_+\cup M_-$ and use the gluing formula
of Bunke~\cite{BuGlu} and Kirk, Lesch~\cite{KiLe} to write~$\bar\nu(M,g)$
as a sum of the contributions~$\bar\nu(M_\pm,g)$ from the two pieces
and a gluing term; see Theorem~\ref{Thm2a.1}.

The gluing term consists of two pieces. %
The first is an integer contribution~$3m_\rho(L;N_+,N_-)$
that depends on the relative positions of the images of~$H^2(V_\pm)$
in~$H^2(\K)$, %
where~$\K$ is the K3 surface that appears as a factor of the cross-section
at infinity of~$M_\pm$.
If~$\thet\in(0,\pi)$ denotes the oriented angle between the exterior circles,
which we call the \emph{gluing angle},
then the second contribution is~$-72\frac\rho\pi$, where~$\rho=\pi-2\thet$.
If the gluing angle is an irrational multiple of~$\pi$,
then this term will be irrational, too.

In~\cite{CGN}, restricting attention to the case when both
$\# \Gamma_\pm \leq 2$ ensured that both pieces have spectral symmetry,
so that $\bar\nu(M_\pm,g) = 0$. Moreover, the gluing angle was forced to be an
integer multiple of $\frac\pi{12}$, see also Remark~\ref{rem:SmallK}.
However, if~$M_\pm$ is a twisted product~$(V_\pm\times S^1)/\Gamma_\pm$
and~$\#\Gamma_\pm\ge 3$,
then the spectra of the relevant Dirac operators on~$M_\pm$
with respect to the appropriate boundary conditions are no longer
necessarily symmetric.

To compute the contribution~$\bar\nu(M_\pm,g)$,
we consider an adiabatic limit by scaling the exterior circle factor
to~$0$.
In Theorem~\ref{Thm5.1},
we state an adiabatic limit formula generalising both Dai's theorem
for manifolds with boundary~\cite{Dai} and the formula for adiabatic
limits of Seifert fibrations~\cite{Gorbi}.
We see that the isolated fixpoints of the $\Gamma_\pm$-action on~$V_\pm$
contribute by generalised Dedekind sums~$D_\gamma(V)\in\Q$
introduced in Definition~\ref{Def2x.1}
(note that involutions on Calabi-Yau 3-folds cannot have isolated fixed points,
consistent with the claim that the \etainvts vanish when $\#\Gamma_\pm \leq 2$).

When passing to the adiabatic limit, we also modify the metric on
the boundary of~$M_\pm$.
Hence, we also have to consider the variational formula
for \etainvts on manifolds with boundary due to Cheeger~\cite{Ch},
Bismut and Cheeger~\cite{BCh4}, and Dai and Freed~\cite{DF}.
The relevant boundary contribution can be computed as a line integral
of a universal \etaform~$\tilde\eta(\Aa)$
for families of two-dimensional tori described in Proposition~\ref{Prop2c.2}.
The relevant \etaform integrals can be expressed in terms of the
logarithm~$\logeta$ of the Dedekind \etafunc (the principal branch
defined in \eqref{eq:explicit_principal}) using results from Section~\ref{A2v}.

The action of the group~$\Gamma_\pm$ on~$T^2$ is described
by an element~$\reps_\pm\in\Z$ that is relatively prime to~$k_\pm$. %
In the following and throughout the paper,
we represent its inverse modulo~$k_\pm$ by~$\eps_\pm^*\in\Z$.
The precise choice of~$\eps_\pm^*$ does not matter in the following result.
Also, the ratio of the lengths of the exterior and the interior circle
is denoted~$\ar_\pm=\frac{\xi_\pm}{\zeta_\pm}$,
which is typically the square root of a positive rational number,
see Proposition~\ref{Prop2w.1}~\ref{2w.1c}.

\begin{Thm}\label{Thm:A}
  For all extra-twisted connected sums~$M$,
  the extended \nuinvt is given by
  \begin{subequations}\label{eq:thmA}
    \begin{align}
      \bar\nu(M)
      &=\bar\nu(M_+,g)+\bar\nu(M_-,g)-72\frac\rho\pi+3m_\rho(L;N_+,N_-)\;,
      \label{eq:thmA.1}\\
      \text{where}\qquad
      \bar\nu(M_\pm,g)
      &=D_{\gamma_\pm}(V_\pm)-\frac{288}\pi\,
      \Im\logeta\biggl(\frac{\ar_\pm^{-1}i-\eps^*_\pm}{k_\pm}\biggr)
      -24\,\frac{\eps_\pm^*}{k_\pm}\;.
      \label{eq:thmA.2}
    \end{align}
  \end{subequations}
\end{Thm}

This is proved in Section~\ref{2h} (assuming Proposition~\ref{Prop1}).
The occurrence of the Dedekind \etafunc can also be motivated
by regarding~$\tilde\eta(\Aa)$ as a connection form of the Chern connection
on a holomorphic determinant line bundle; see Remark~\ref{rmk:atiyah}.
From the theory of complex multiplication one knows that
the values of~$\Im\logeta$ in~\eqref{eq:thmA.2} can be expressed
in terms of logarithms of algebraic numbers;
for the values used in this paper, this is done explicitly in
section~\ref{app:2}.
The linear combinations %
that appear
in~\eqref{eq:thmA.1} can be worked out from the functional
equation~\eqref{Ltransf} of~$\logeta$ (see Proposition~\ref{Prop3}),
giving one approach to Theorem~\ref{Thm:B} below.

\subsection*{Evaluating \texorpdfstring{$\bar\nu$}{the extended \nuinvt}
via elementary hyperbolic geometry}

In Section~\ref{sec:hyp}, we follow a different path to rewrite the
right hand side of~\eqref{eq:thmA.1} in terms of Dedekind sums.
The universal \mbox{\etaform} $\tilde\eta(\Aa)$
for bundles of flat tori can be understood as a 1-form on the upper half plane,
whose exterior derivative turns out to be a constant multiple
of the hyperbolic area form.
The relevant path of integration
consists of two sides of some ideal hyperbolic polygon~$P$,
depending on the gluing data.
The remaining sides can be chosen such that~$\tilde\eta(\Aa)$ vanishes
along them.
To apply Stokes' Theorem,
it remains to determine the area and the contribution from the cusps.
For the latter, we need a strict adiabatic limit
formula for \etaforms, see Proposition~\ref{Prop2y.2}.
Such a formula was proved by Bunke, Ma~\cite{BuMa} and Liu~\cite{Liu}
modulo exact forms,
but that is not enough for our purposes.

We determine the remaining  corners of the polygon~$P$ using continued fractions.
Here, we also need the entries of the gluing matrix~$\psmatrix{\gll&p\\\glb&\sglr}$
that encodes the matching of the tori as described in \eqref{eq1:tormat_coords}.
At this point, one can already finish the computation of~$\bar\nu(M)$
for any particular example by hand.
However, one can simplify these computations by observing that
the contributions from the cusps and the hyperbolic area formula add up to a
classical Dedekind sum~$\DS(d,q)$ given for integers~$\glb>0$ and~$d$ by
\begin{equation}\label{eq:DedekindSum}
  \DS(d,\glb)=\sum_{j=1}^{\glb-1}\biggl(\!\!\biggl(\frac j\glb\biggr)\!\!\biggr)
  \,\biggl(\!\!\biggl(\frac{jd}\glb\biggr)\!\!\biggr)\in\frac 1{6\glb}\,\Z\;,
  \quad\text{where }
  (\!(x)\!)=
  \begin{cases}
    0	&\text{for~$x\in\Z$}\\
    x-\lfloor x\rfloor-\frac12	&\text{for~$x\not\in\Z$}
  \end{cases}
\end{equation}
(a sawtooth function). Clearly $s(d,\glb)$ depends only on the residue 
$d \mmod \glb$.
For more background on the Dedekind \etafunc, Dedekind sums,
and their appearance in topology we refer the reader to~\cite{hz}.
We may use Proposition~\ref{prop:isom} to make sure that~$\gll\ge 0$ and~$\glb>0$.

\begin{Thm}\label{Thm:B}
  Assume that~$\glb>0$.
  Then~$A=\frac{\gll-\reps_+^*\glb}{k_+}$ is an integer (whose residue mod
  $\glb$ does not depend on the choice of the mod $k_+$ inverse $\reps^*_+$
  of $\eps_+$) and
  \begin{equation}\label{eq:ThmB}
      \bar\nu(M)
      =D_{\gamma_+}(V_+)+D_{\gamma_-}(V_-)+3\,m_\rho(L;N_+,N_-)
      +24\,\Bigl(12\,\DS(A,\glb)-\frac\glr{k_-\glb}-\frac\gll{k_+\glb}\Bigr)\;.
  \end{equation}
\end{Thm}

The argument %
involving hyperbolic polygons can be found in Section~\ref{2.ey}.
We see that all possibly irrational terms %
in Theorem~\ref{Thm:A} %
have been subsumed into the last parenthesis, which is always rational.
This makes it easier to check that~$\bar\nu(M)$ is an integer;
see also Remark~\ref{rmk:nu24}.

All in all, our way to a tractable formula for the \nuinvt
consists of many small steps, %
but in each of our examples, the sum of the various contributions is an integer.
This---together with the fact that
the two %
different approaches described above give the same expression---%
could be seen as a sanity check for our results presented here.

\enlargethispage{1.4\baselineskip}

\subsection*{Scope}
  To keep this article within a reasonable size,
  we left out the following aspects.
  \begin{enumerate}
  \item We only consider examples built from blocks of Picard rank~$1$.
    These examples automatically have matchings of pure angle in the sense
    of Remark~\ref{rmk:pure}.
    On the other hand, by~\eqref{eq:b3M} most examples obtained this
    way have even~$b_3(M)$, and hence odd~$\bar\nu(M)$.

    There exist building blocks of Picard rank~$2$ with automorphism
    group~$\Z/k$ for each~\mbox{$k=1$, \dots, $6$}.
    Thus we expect that there are examples of extra-twisted connected
    sums that realise all~$24$ even values of~$\nu(M)\in\Z/48$ as well.
  \item All examples constructed in Section~\ref{sec:ex} are either
    two-connected or have two-connected universal cover.
    We distinguish them only by their extended \nuinvts and
    their third Betti number.
    See Figure~\ref{table:landscape} for a plot of all possible pairs
    of these invariants.
    We do not attempt to compute the torsion part of their fourth cohomology
    or the divisibility of their first Pontryagin classes.
    That would be needed in order to apply diffeomorphism classification
    results~\cite{7class} to exhibit examples of 7-manifolds where
    the moduli space of $G_2$-metrics is disconnected, but such examples
    have already been seen in~\cite{CGN}.
  \item %
    We might consider asymptotically cylindrical Calabi-Yau manifolds with
    arbitrary automorphisms that act on the asymptotic cylinder
    $\K \times T^2 \times \R$ as a product but not necessarily fixing
    the $\K$ factor.
    But first of all, it looks more difficult to construct matchings
    in this situation.
    And worse, these examples would never be simply connected.
    Instead, their universal covers would again be extra-twisted
    connected sums of the type considered here.
    This has been explained in~\cite[Remark~1.12]{CGN2}.
  \item
    It is not clear how to define a \nuinvt for non-compact
    or singular $G_2$-spaces.
    Theorem~\ref{Thm2a.1} contains a possible definition for
    $G_2$-manifolds with an asymptotically cylindrical end.
    However, it is also not clear to us how to interpret the resulting numbers
    in~\eqref{eq:thmA.2}.
    From Dai and Freed's point of view in~\cite{DF},
    the invariant~$e^{2\pi i\,\frac{\nu(M_\pm,g)}6}$ should take values
    in a certain determinant line associated to the cross-section at infinity.
    The invariant~$\bar\nu(M_\pm)$ would therefore take values in
    a ``logarithm'' of this determinant line.
  \end{enumerate}

\subsection*{Organisation}

In Section~\ref{sec:xtcs},
we recall the extra-twisted connected sum construction.
We discuss the matching problem for tori in Subsection~\ref{2.w},
and for K3 surfaces in Subsection~\ref{3.c}.

Theorem~\ref{Thm:A} is proved
in Section~\ref{sec:nu}.
The fixpoint contributions are computed in Subsection~\ref{2.b},
the variational formula is discussed in~\ref{2.c},
and a direct computation of the \etaform integrals
can be found in~\ref{2h}.

We discuss the combinatorics of torus matchings
in Section~\ref{sec:torus}. We also discuss the topology of K3 fibrations
(in particular, how some information about the topology of the link of singular
fibres can be read off from the combinatorial gluing data) that are candidates
to be perturbed to coassociative ``Kovalev-Lefschetz fibrations'' by results of
Englebert \cite{Englebert}.

Theorem~\ref{Thm:B} is proved by elementary hyperbolic geometry
in Section~\ref{sec:hyp}.
Adiabatic deformations of tori are identified with
hyperbolic geodesics in Subsection~\ref{2.v},
and the contribution from the cusps is explained in~\ref{2.y}.
In Subsection~\ref{2.ex}, %
we use continued fractions to construct ideal polygons,
and in~\ref{2.ey}, we rewrite the sum of the cusp contributions
as a Dedekind sum.

Section~\ref{sec:ex} contains more details about the construction
of examples, in particular we describe some building blocks with group actions
in Subsection~\ref{sec:blocks}, and possible K3 matchings
in Subsection~\ref{sec:match-ex}.

Section~\ref{sec:proofs} provides the proofs of some technical results
used in Sections~\ref{sec:nu} and~\ref{sec:hyp}.
Section~\ref{A2v} contains the evaluation of \etaform
integrals in terms of the Dedekind \etafunc,
as well as explicit formulas for the values that we use in this paper.

\subsubsection*{Acknowledgements}
We thank Jean-Michel Bismut, Jeff Cheeger, Xianzhe Dai and Matthias Lesch
for inspiring discussions on
adiabatic limits and variational formulas for \etainvts
on manifolds with boundary.
We thank Katrin Wendland, Emanuel Scheidegger, Jesus Martinez Garcia
and Stefan Kebekus for conversations about K3 surfaces and Fano threefolds.
We would like to thank the Simons foundation for its support of their
research under the Simons Collaboration on ``Special Holonomy in Geometry,
Analysis and Physics'' (grants \#488617, Sebastian Goette, and \#488631, Johannes
Nordstr\"om).
We gratefully acknowledge support from the Simons Center for Geometry and Physics, Stony Brook University at which some of the research for this paper was carried out.

\vspace{0mm plus 2mm}

\section{Extra-twisted connected sums}\label{sec:xtcs}

We generalise the twisted connected sum construction of~\cite{Kovalev,CHNP2}
by allowing twisted products of asymptotically cylindrical Calabi-Yau
manifolds with circles.
This approach was already employed in~\cite{CGN},
where we considered products twisted by an involution (on one or both sides).
Here, we allow twists by more general finite cyclic groups.

\subsection{The gluing construction}\label{subsec:glue}
\label{2.a}

Let~$V_\pm$
be asymptotically cylindrical Calabi-Yau manifolds
of complex dimension~$3$, and assume that their ends are asymptotic
to~$\Sigma_\pm\times S^1_{\lnn_\pm}\times(0,\infty)$,
where~$\Sigma_\pm$ are K3 surfaces,
and~$S^1_{\lnn_\pm} = \R/\lnn_\pm \Z$.
The Calabi-Yau structure on $V_\pm$ can be described in terms of a pair
$(\Omega_\pm, \omega_\pm)$, where $\Omega_\pm$ is a complex 3-form (holomorphic with respect to the complex structure) and $\omega_\pm$ is a Kähler form,
normalised so that 
$8\omega_\pm^3 = 6\Omega_\pm \wedge \overline \Omega_\pm$.
Along the cylindrical end $\Sigma_\pm\times S^1_{\zeta_\pm}\times(0,\infty)$,
the asymptotic limits of $\Omega_\pm$ and $\omega_\pm$ are of the form
\[ \Omega_\pm:= (d\anglen_\pm - i dt_\pm) \wedge (\omega^J_\pm + i \omega^K_\pm), \qquad
\omega:= dt_\pm \wedge d\anglen_\pm + \omega^I_\pm \]
respectively, where $t_\pm$ is the coordinate on the $(0,\infty)$ factor,
$\anglen_\pm$ is the coordinate on $S^1_{\lnn_\pm} = \R/\lnn_\pm \Z$, and
the triple $(\omega^I_\pm, \omega^J_\pm, \omega^K_\pm)$
defines a \hk structure on $\K_\pm$.
Such a Calabi-Yau structure induces a metric $g_{V_\pm}$ of holonomy $SU(3)$
whose asymptotic limit is of the form $dt_\pm^2 + d\anglen_\pm^2 + g_{\K_\pm}$,
where $g_{\K_\pm}$ is a metric of holonomy $SU(2)$ induced by the
\hk structure. Note in particular that the circumference of the circle factor
in the asymptotic cylinder is $\lnn_\pm$.

\begin{rmk}
\label{rem:pi1}
The condition on the cylindrical ends forces $V_\pm$ to be simply-connected by
\mbox{\cite[Theorem A]{hhn}}.
\end{rmk}

Letting $S^1_{\lnx_\pm} = \R/\lnx_\pm \Z$ and denoting its coordinate by
$\anglex_\pm$ , we can define a product torsion-free
\gtstr{} on $V_\pm \times S^1_{\lnx_\pm}$ by
\[ \varphi_\pm = \Re \Omega_\pm + d\anglex_\pm \wedge \omega_\pm . \]
Then $\varphi_\pm$ defines the metric $d\anglex_\pm^2 + g_{V_\pm}$, with
holonomy contained in $G_2$. Note that the exterior circle factor has
circumference $\lnx_\pm$. The asymptotic limit of $\varphi_\pm$ has the form
\begin{equation}
\label{eq:g2limits} 
d\anglex_\pm\wedge dt_\pm \wedge d\anglen_\pm + d\anglex_\pm\wedge\omega_\pm^I +
d\anglen_\pm\wedge\omega_\pm^J + dt_\pm \wedge \omega_\pm^K .
\end{equation}

Now assume further that two groups~$\Gamma_\pm\cong\Z/k_\pm$ act %
on~$V_\pm$, preserving the Calabi-Yau structures (and in particular the
metrics), such that the actions on the
end are products of the trivial actions on~${\Sigma_\pm\times(0,\infty)}$
and free actions on~$S^1_{\lnn_\pm}$.
We extend the~$\Gamma_\pm$-action diagonally
to~${\widetilde M_\pm=V_\pm \times S^1_{\xi_\pm}}$,
such that~$\Gamma_\pm$ acts isometrically and freely on~$S^1_{\xi_\pm}$.
Then the torsion-free \gtstr{} $\varphi_\pm$ descends to the quotient
$M_\pm=\widetilde M_\pm/\Gamma_\pm$, which we can thus regard as an
asymptotically cylindrical $G_2$-manifold.
The cross-section of the asymptotic cylinder is the product~$X_\pm$ of $\K_\pm$
with a torus $(S^1_{\zeta_\pm}\times S^1_{\xi_\pm})/\Gamma_\pm$.

We now suppose that we have a suitable isometry between the cross-sections.
This isometry will necessarily be a product of isometries
\[ \tormat : (S^1_{\zeta_+}\times S^1_{\xi_+})/\Gamma_+ \to (S^1_{\zeta_-}\times S^1_{\xi_-})/\Gamma_- \]
and
\[ \hkr : \K_+ \to \K_- . \]
We require that the isometry
\begin{equation}
\label{eq:cylisom}
\begin{aligned}
\Sigma_+ \times (S^1_{\xi_+} \times S^1_{\zeta_+})/\Gamma_+ \times\R
   \; &\to \;
  \Sigma_- \times (S^1_{\xi_-} \times S^1_{\zeta_-})/\Gamma_- \times\R \\
(x, z, t) &\mapsto (\hkr(x), \tormat(z), -t)
\end{aligned}
\end{equation}
identifies the asymptotic limits~\eqref{eq:g2limits}.
In particular, exactly one of $\tormat$ and $\hkr$ is orientation-preserving.
Our convention is to require $\tormat$ to be orientation-reversing.
We will refer to $\tormat$ as a \emph{torus matching} and to $\hkr$
as a \emph{\hk rotation}.

A key feature of the construction is how the torus matching aligns the
external circle directions. As above, we denote by
\[ u_\pm\in\R/\zeta_\pm\Z, \qquad v_\pm\in\R/\xi_\pm\Z \]
the coordinates in
the direction of the interior and exterior circles respectively.
In~\cite[\eqref{eta-2.5}]{CGN}, the \emph{gluing angle}~$\thet$ was introduced
as the directed angle between the exterior circles under $\tormat$, so
\begin{align}
  \begin{split}\label{2.a.0}
    \del_{v_-}&=\cos\thet\,\del_{v_+}+\sin\thet\,\del_{u_+}\;,\\
    \del_{u_-}&=\sin\thet\,\del_{v_+}-\cos\thet\,\del_{u_+}\;.
  \end{split}
\end{align}
\enlargethispage{0.8\baselineskip}
The condition that~\eqref{eq:cylisom} preserves the asymptotic limits of the
\gtstr s is now equivalent to the following condition;
see~\cite[\S1.2]{xtcs}.

\begin{dfn}
Let $\K_\pm$ be K3 surfaces with \hk structures
$(\omega^I_\pm, \omega^J_\pm, \omega^K_\pm)$. Call a diffeomorphism
$\hkr : \K_+ \to \K_-$ a \emph{\hk rotation} with angle $\thet$ if
\begin{equation}
\label{eq:hkr}
\begin{aligned}
\hkr^*\omega^K_- &= - \omega^K_+ \\
\hkr^*(\omega^I_- + i \omega^J_-) &= e^{i\vartheta} (\omega^I_+ - i\omega^J_+) .
\end{aligned}
\end{equation}
\end{dfn}

Extend the cylindrical coordinate $t_\pm$ to a smooth function on all $V_\pm$
(taking negative values away from the cylindrical end), and let~$V_{\pm,\ell}$
be the truncation $\{\,x \in V_\pm\mid t_\pm(x) \leq 2\ell\,\}$.
Let~$\widetilde M_{\pm,\ell}=V_{\pm,\ell}\times S^1_{\xi_\pm}$ and
$M_{\pm,\ell}=\widetilde M_{\pm,\ell}/\Gamma_\pm$ %
and let~$\tilde X_\pm\cong\K_\pm\times S^1_{\zeta_\pm}\times S^1_{\xi_\pm}$
and~$X_\pm=\tilde X_\pm/\Gamma_\pm$ denote their boundaries.
For sufficiently large~$\ell$,
it is possible to obtain a new closed $G_2$-manifold~$M_\ell$ (with an
approximately cylindrical neck region of length $4\ell$)
by gluing~$M_{+.\ell}$ and~$M_{-,\ell}$ along their boundaries using
a diffeomorphism~$X_+\cong X_-$.
This procedure is described in detail in~\cite{xtcs} following
the ideas in~\cite{Kovalev, CHNP2}.
Let us summarise.

\begin{thm}
\label{thm:glue}
Given a pair of ACyl Calabi-Yau 3-folds $V_\pm$
with asymptotic cross-sections $\K_\pm \times S^1_{\zeta_\pm}$
and automorphisms $\Gamma_\pm$,
a torus matching
$\tormat : (S^1_{\zeta_+}\times S^1_{\xi_+})/\Gamma_+ \to
 (S^1_{\zeta_-}\times S^1_{\xi_-})/\Gamma_-$ and a \hk rotation
$\hkr : \K_+ \to \K_-$ with angle $\thet$ equal to the gluing
angle of $\tormat$, the manifold $M_\ell$ above admits torsion-free \gtstr s.
Moreover, the holonomy $G_2$ metrics in $M_\ell$ can be taken to be close
to the metrics on the two halves in the sense explained in \S \ref{ssec:gell}.
\end{thm}

Note that after gluing, $M_+$ and~$M_-$ induce opposite orientations
on the 2-torus.
Note that~$\thet$ is invariant under swapping the roles of~$M_+$ and~$M_-$.
The angle between the interior circles is~$\pi-\thet$;
see Figure~\ref{Fig2g.1}.

For further discussion of how extra-twisted connected sums using different
data can be essentially the same see Proposition~\ref{prop:isom}.

\begin{rmk}\label{rmk:spinstr} %
  The $G_2$-structure on~$M$ defines a unique spin structure on~$M$
  that we need for the analytic description of the \nuinvt.
  Its restriction to~$M_\pm$ is the spin structure induced
  by the $SU(3)$-structure,
  and hence by the Calabi-Yau structure on~$V_\pm$.

  Because~$\Gamma_\pm$ preserves the Calabi-Yau structure on~$V_\pm$,
  it acts canonically on the associated
  complex spinor bundle~$SV_\pm\cong\Lambda^{0,\bullet}T^*V_\pm$.
  The spinor bundle on~$M_\pm$ that is induced by the $G_2$-structure
  then satisfies~$SM_\pm\cong p^*SV_\pm/\Gamma_\pm$.
  On the cylinder~$\Sigma_\pm\times(0,\infty)
  \times(S^1_{\zeta_\pm}\times S^1_{\xi_\pm})/\Gamma_\pm$,
  it is isomorphic to the pullback of the direct sum
  of two copies of~$S\Sigma_\pm$.
\end{rmk}

\subsection{Setting up the matching problem}

Understanding the possible torus matchings
${\tormat\colon (S^1_{\zeta_+}\times S^1_{\xi_+})/\Gamma_+ \to
(S^1_{\zeta_-}\times S^1_{\xi_-})/\Gamma_-}$ for given values of
$k_\pm = \#\Gamma_\pm$---and in particular the possible gluing
angles $\thet$---is essentially a combinatorial problem, which will be
discussed in the next subsection and in Subsection~\ref{subsec:combinatorics}.
Given a torus matching, Theorem~\ref{thm:glue} raises the question of how to
find pairs of ACyl Calabi-Yau 3-folds with automorphisms and a \hk rotation of
the correct angle $\thet$ between the K3 surfaces in the asymptotic
cross-section. We now explain how this question can be reduced to complex
algebraic geometry, as in~\cite[\S6]{CHNP2} and~\cite[\S6]{xtcs}.

\begin{dfn}
\label{def:block}
Let $Z$ be a non-singular algebraic 3-fold and $\Sigma \subset Z$ a
non-singular K3 surface. Let $N$ be the image of $H^2(Z) \to H^2(\K)$.
We call $(Z,\K)$ a \emph{building block} if 
\begin{enumerate}[leftmargin=*]
\item the class in $H^2(Z)$ of the anticanonical line bundle $-K_Z$
is indivisible,
\item
$\K\in |{-}K_Z|$ (\ie $\Sigma$ is an anticanonical divisor),
and there is a  
projective morphism $f\colon Z\to \PP^1$ with $\K=f^\star (\infty)$,
\item The inclusion $N\hookrightarrow H^2(\K)$ is primitive, that is,
  $H^2(\K)/N$ is torsion-free.
\label{N:primitive}
\item The group $H^3(Z)$---and thus also $H^4(Z)$---is torsion-free.
\end{enumerate}
We call $N$, equipped with the restriction of the intersection form on
$H^2(\K)$, the \emph{polarising lattice} of the block.
(Because $H^{2,0}(Z)$
is automatically trivial, $N \subseteq H^{1,1}(\K)$ \cite[Lemma 3.6]{CHNP2},
so that $\K$ is `$N$-polarised'.)

If $\Gamma$ is a group acting faithfully on $Z$ by biholomorphisms that fix
$\K$ pointwise then we call $(Z, \K, \Gamma)$ a \emph{building block with
automorphisms}. ($\Gamma$ is then necessarily cyclic.)
\end{dfn}

Given such a $(Z, \K)$, \cite[Theorem D]{hhn} gives the existence of ACyl
Calabi-Yau structures on $V := Z \setminus \K$, and it is easy to see that
$\Gamma$ restricts to isomorphisms of these structures.

Rather than to look for a \hk rotation for a given pair of ACyl Calabi-Yau
structures, it is easier to first choose the diffeomorphism
$\hkr : \K_+ \to \K_-$ satisfying obvious necessary conditions in terms of
cohomology classes and then find Calabi-Yau structures that make $\hkr$ a \hk
rotation. Recall that the \emph{period} of a complex K3 surface $\K$ is the
positive-definite 2-plane $\Pi \subset H^2(\K; \R)$ spanned by the
real and imaginary parts of elements of $H^{2,0}(\K;\C)$.

\begin{dfn}
Let $(Z_\pm, \K_\pm)$ be a pair of building blocks, and let 
$\Pi_\pm \subset H^2(\K_\pm; \R)$ be the periods. Call a diffeo\-morphism
$\hkr : \K_+ \to \K_-$ a \emph{K3 matching} with angle $\thet$ if
there are Kähler classes $\kclass_\pm \in H^2(Z_\pm; \R)$
such that (with respect to the intersection form)
the angle between $\hkr^*(\kclass_-)$ and $\Pi_+$ is $\thet$,
the angle between $(\hkr^{-1})^*(\kclass_+)$ and $\Pi_-$ is $\thet$
and $\Pi_+ \cap \hkr^*\Pi_-$ is non-trivial.
\end{dfn}

It is easy to see that the ACyl Calabi-Yau structures of~\cite[Theorem D]{hhn}
can be chosen to ensure that a given K3 matching is a \hk rotation;
see~\cite[Theorem 1.1 and Lemma~6.2]{xtcs}.
 
\begin{thm}
\label{thm:match_to_hkr}
Given $\zeta_\pm > 0$, blocks $(Z_\pm, \K_\pm)$ and a K3 matching
$\hkr : \K_+ \to \K_-$ with angle~$\thet$, there exist ACyl Calabi-Yau
structures $(\Omega_\pm, \omega_\pm)$ on
$V_\pm := Z_\pm \setminus \K_\pm$ with asymptotic limit
\[ \left((du_\pm - i dt_\pm) \wedge (\omega^J_\pm + i \omega^K_\pm),
dt_\pm \wedge du_\pm + \omega^I_\pm \right) \]
on $\K_\pm \times S^1_{\zeta_\pm}$ (in particular, the circumference of the
$S^1$ factor with respect to the induced metric is $\zeta_\pm$),
such that $\hkr$ is an angle $\thet$ \hk rotation
of the \hk structures $(\omega^I_\pm, \omega^J_\pm, \omega^K_\pm)$.
\end{thm}

Thus given a K3 matching of two building blocks with automorphism
and a torus matching with the corresponding $k_+$, $k_-$ and gluing angle
$\thet$, we can find ACyl Calabi-Yau structures so that Theorem~\ref{thm:glue}
can be applied to build a \gtmfd.

Given a pair of blocks, there is no reason to expect to be able to find
any K3 matchings at all. However, if we consider the sets of complex
deformations of a pair of blocks, one can in many cases guarantee that there
exist some elements of each of the two sets that admit a K3 matching, and
moreover control the topology of the resulting \gtmfd. This will be discussed
further in \S\ref{3.c}.

\subsection{Isometries of quotients of rectangular tori}\label{2.w}

In this section, we analyse how to find torus
matchings~$\tormat$ in the sense of~\S\ref{2.a}.
We will describe the tori~$T_\pm$ and~$\tormat\colon T_+\to T_-$
by two units~$\eps_\pm\in\Z/k_\pm$
and a \emph{gluing matrix\/}~$G=\psmatrix{\gll&p\\\glb&\sglr}$.
These data satisfy some additional properties that make it easy
to enumerate all possibilities for~$(G,\eps_+,\eps_-)$,
given numbers~$k_+$, $k_-\ge 1$, see Proposition~\ref{Prop2w.1} below.

\newcommand\matchingfigure{
\begin{figure}
    \begin{tikzpicture}
      \fill[color=lightgray] (0,0) -- (3,0) -- (4,1.414) -- (1,1.414) -- cycle ;
      \begin{scope}[color=lightgray]
        \draw (-1.3,0) -- (4.3,0) ;
        \draw (0,1.414) node[left,color=black] {$\gll\!$} -- (4,1.414)
        -- (4,0) node[below,color=black] {$\glb$} ;
        \draw (0,2.828) node[right,color=black] {$p$} -- (-1,2.828)
        -- (-1,0) node[below,color=black] {$\sglr$} ;
      \end{scope}
      \draw (0,0) rectangle (3,4.242) ;
      \draw (0,0) -- (4,1.414) -- (3,4.242) -- (-1,2.828)  -- cycle ;
      \begin{scope}[color=red,line width=1.5pt,->]
        \draw (0,0) -- (0,1) node[right,color=black] {$\del_{v_+}$} ;
        \draw (0,0) -- (0.943,0.333)
	node[above,color=black] {$\del_{v_-}$} ;
      \end{scope}
      \begin{scope}[color=blue,line width=1.5pt,->]
        \draw (0,0) -- (1,0) node[below,color=black] {$\del_{u_+}$} ;
        \draw (0,0) -- (-0.333,0.943)
	node[below left,color=black] {$\del_{u_-}\!\!\!$} ;
      \end{scope}
      \draw[line width=0.8pt,->] (0,0.7) arc (90:19.4:0.7) ;
      \draw (0,0) ++(54:0.4) node {$\thet$} ;
      \fill (-1,2.828) circle(1.5pt) node[left] {$\lambda_-$}
      ++(1,1.414)  circle(1.5pt) node[left] {$\mu_+$} ;
      \fill (0,0)  circle(1.5pt)
      ++(1,1.414)  circle(1.5pt)
      node[above] {$\nu_+$} %
      ++(1,1.414)  circle(1.5pt)
      ++(1,1.414)  circle(1.5pt) ;
      \fill (3,0) circle(1.5pt) node[below] {$\lambda_+$}
      ++(1,1.414)  circle(1.5pt) node[right] {$\mu_-$} ;
    \end{tikzpicture}
  \caption{Fundamental domains of $T$ and~$\widetilde T_\pm$.}\label{Fig2g.1}
\end{figure}
}

We consider~$M_+=\widetilde M_+/\Gamma_+$
with covering space~$\widetilde M_+=V_+\times S^1_{\xi_+}$
and~$\Gamma_+=\Z/k_+$.
The asymptotic cross-section of the covering space is isometric to a product
\begin{equation*}
  \del\widetilde M_+\cong\Sigma_+\times\widetilde T_+\qquad\text{with}\qquad
  \widetilde T_+\cong S^1_{\zeta_+}\times S^1_{\xi_+}\;,
\end{equation*}
where~$\Sigma_+$ is a K3 surface and~$\zeta_+$, $\xi_+$ are the lengths
of the interior and exterior circle, respectively.

By a \emph{torus matching} we refer to the following data:
numbers~$k_\pm\ge 1$,
actions of $\Gamma_\pm = \Z/k_\pm$ on
$\widetilde T_\pm = S^1_{\zeta_\pm}\times S^1_{\xi_\pm}$ that are free on both
factors, and an orientation-reversing
isomorphism~$\tormat\colon\widetilde T_+/\Gamma_+ \to \widetilde T_-/\Gamma_-$
of flat tori,
such that there exist lengths~$\zeta_+$, $\xi_+$, $\zeta_-$, $\xi_- > 0$
for which~$\tormat$ becomes an isometry.
We consider two torus matchings to be equivalent if there exist (linear)
isomorphisms of the respective tori that map exterior circles to exterior
circles, interior circles to interior circles,
and that intertwine the actions of~$\Gamma_\pm$ and~$\tormat$
(we consider other symmetries in Proposition~\ref{prop:isom}).
It is clear that using torus matchings that are equivalent in this sense
in Theorem~\ref{thm:glue} yields $G_2$ metrics related by deformation.

Equip~$\R^2\cong\C$ with the standard Euclidean metric.
We choose~$\zeta_+$, $\xi_+$, $\zeta_-$, $\xi_- > 0$ as above
and represent the torus~$\widetilde T_+$ isometrically as~$\C/\tilde\Lambda_+$,
where~$\tilde\Lambda_+\subset\C$ is the lattice
with orthogonal basis~$(\mu_+,\lambda_+)=(i\xi_+,\zeta_+)$.

We assume that~$\Gamma_+\cong\Z/k_+$ acts on~$\widetilde T_+=\C/\tilde\Lambda_+$
such that the action on both circles is free.
(If the action on the exterior circle was not free,
then the quotient~$M_+$ would be an orbifold. If the action on the interior
circle had a kernel~$\Gamma_{+0}$, then
we could reduce the exterior circle to the quotient~$S^1_{\xi_+}/\Gamma_{+0}$
without changing~$M_+$,
so we do not have to consider this situation.)
We fix a generator that rotates the exterior circle
by the angle~$\frac{2\pi}{k_+}$.
If~$k_+\ge 2$, its action on the interior
circle is given by~$\frac{2\pi\eps_+}{k_+}$ for some $\eps_+ \in \Z$.
For the moment we only care about the residue $\eps_+ \in \Z/k_+$,
which is uniquely defined.
The requirement that the action on the interior circle is
free means $\eps_+$ is coprime to $k_+$, in other words~$\gcd(\eps_+,k_+)=1$.
We represent~$T=\widetilde T_+/\Gamma_+$ by the lattice
\[ \Lambda = \biggl\{ \frac{m \lambda_+ + n \mu_+}{k_+}\biggm|
m \equiv \eps_+ n \!\!\pmod{k_+} \biggr \} . \]
This is sketched in Figure~\ref{Fig2g.1} for~$k_+=3$ and~$\eps_+=1$.
A fundamental domain for~$\Lambda$ is shaded.

\matchingfigure

Represent $T_-=\widetilde T_-/\Gamma_-$ similarly, and define $\eps_- \in\Z/k_-$
with~$\gcd(\eps_-,k_-)=1$ analogously.
\begin{comment}
Sometimes it is also convenient to consider the multiplicative
inverse $\inveps \in (\Z/k_-)^\times$ (if~$k_-=1$, we put~$\inveps=0$);
then there is a generator of~$\Gamma$ that rotates the internal circle
by $\frac{2\pi}{k_-}$,
and the external circle by $\frac{2\inveps\pi}{k_-}$.
\end{comment}
The isometry $\tormat\colon\widetilde T_+/\Gamma_+ \to\widetilde T_-/\Gamma_-$
determines a sublattice~$\tilde\Lambda_-\subset\Lambda$ such
that~$\widetilde T_-\cong\C/\tilde\Lambda_-$.
Let~$(\mu_-,\lambda_-)$ denote a basis of~$\tilde\Lambda_-$,
where~$\lambda_-$ and~$\mu_-$ correspond
to the interior and exterior circle as above.
We represent this basis as
\begin{equation}\label{2.w.6}
  (\mu_-,\lambda_-)=\frac1{k_+}\cdot(\mu_+,\lambda_+)\cdot
  \begin{pmatrix}\gll&p\\\glb&\sglr\end{pmatrix} .
\end{equation}
Then $\psmatrix{\gll&p\\\glb&\sglr}\in M_2(\Z)$ because
$\tilde \Lambda_- \subseteq \frac{1}{k_+} \tilde \Lambda_+$, 
and we call~$G=\psmatrix{\gll &p\\\glb&\sglr}$ the \emph{gluing matrix.}
This is equivalent to the coordinate description of $\tormat : T^2_+ \to T^2_-$ in~\eqref{eq1:tormat_coords} (note that the coordinates used there are
related to those in~\S\ref{2.a} by $\anglen_\pm = \lnn_\pm \bar \anglen_\pm$
and $\anglex_\pm = \lnx_\pm \bar \anglex_\pm$).
In Figure~\ref{Fig2g.1},
we have~${k_-=3}$, $\eps_-=1$, and the gluing matrix
is~$\bigl(\begin{smallmatrix}1&2\\4&-1\end{smallmatrix}\bigr)$;
see entry~209 in Table~\ref{table:matchings}.

In summary, we can associate to a torus matching the following combinatorial
data that is clearly invariant under our notion of equivalence
\begin{itemize}
\item $k_+$ and $k_-$
\item the gluing matrix~$G$
\item the residue classes of~$\eps_+$ in~$\Z/k_+$
  and of~$\eps_-$ in~$\Z/k_-$,
  where~$\gcd(\eps_+,k_+)=\gcd(\eps_-,k_-)=1$.
\end{itemize}
For the construction, we also need the more geometric data
\begin{itemize}
\item the angle $\thet$ between the exterior circle directions
\item the ratios $\frac{\xi_+}{\xi_-}$ and $\ar_\pm=\frac{\xi_\pm}{\zeta_\pm}$
\end{itemize}
\noindent
which are not obviously invariant.
However, among the (selection of) compatibility conditions that we now show,
we see that the angle~$\thet$ is in fact also determined by the equivalence
class of the torus matching, and that if $\thet\notin\frac\pi2\,\Z$,
then the aspect ratios~$\frac{\xi_\pm}{\zeta_\pm}$ and~$\frac{\xi_+}{\xi_-}$
are as well.

\begin{prop}\label{Prop2w.1}
  \begin{enumerate}
  \item\label{it:necc}
    The data of a torus matching satisfies the following relations.
    \begin{equation}
      \label{eq:det}
      \det \begin{pmatrix}\gll&p\\\glb&\sglr\end{pmatrix}= -k_-k_+,
    \end{equation}
    \begin{subequations}
      \label{eq:epses}
      \begin{align}
        \label{eq:epses+}
        \eps_+\gll -\glb &\equiv \eps_+p +\glr \equiv 0 \mod k_+ \\
        \label{eq:epses-}
        \eps_- \glr - \glb &\equiv \eps_- p + \gll \equiv 0 \mod k_-
      \end{align}
    \end{subequations}
    \begin{subequations}\label{eq:gcd}
      \begin{align}
        \label{eq:gcd+}
        \gcd\Bigl(\frac{\glb-\eps_+\gll}{k_+},\gll\Bigr)
        &=\gcd\Bigl(\frac{\glr+\eps_+p}{k_+},p\Bigr) =\gcd(\eps_+,k_+)= 1,\\
        \label{eq:gcd-}
        \gcd\Bigl(\frac{\glb-\eps_-\glr}{k_-},\glr\Bigr) &=
        \gcd\Bigl(\frac{\gll+\eps_-p}{k_-},p\Bigr)
        = \gcd(\eps_-,k_-)=1,
      \end{align}
    \end{subequations}
  \item\label{2w.1c}
    Either~$\gll=\glr=0$,
    or~$p=\glb=0$, or~$\frac{\glb\glr}{p\gll}>0$
    and~$\ar_+=\frac{\xi_+}{\zeta_+}=\sqrt{\frac{\glb\glr}{p\gll}}$.
    In the latter case,
    we also have~$\zeta_-=\sqrt{\frac{\glr k_-}{\gll k_+}}\,\zeta_+$,
    $\xi_-=\sqrt{\frac{\glb k_-}{pk_+}}\,\zeta_+$,
    and~$\ar_-=\frac{\xi_-}{\zeta_-}=\sqrt{\frac{\glb\gll}{p\glr }}$.
  \item\label{2w.1d} 
    The gluing angle~$\thet$ is given as
    \begin{equation*}
      \thet=\arg\bigl(\gll\ar_+ + i\glb\bigr)\in(-\pi,\pi]\;.
    \end{equation*}
    In particular, $\thet\in(0,\pi)$ if and only if~$\glb>0$,
    and~$\cos\thet=\sign(\gll )\sqrt{\frac{\gll \glr }{k_+k_-}}$.
\end{enumerate} 
\end{prop}

It will become apparent in the proof that
the conditions in~\ref{it:necc} are consequences of asking
that~\eqref{eq1:tormat_coords} really defines a
diffeomorphism~$\tormat : T^2_+ \to T^2_-$
of the quotients~$T_\pm=\widetilde T_\pm/\Lambda_\pm$,
while the condition
$p\glb \gll \glr \geq 0$ in~\ref{2w.1c} comes from demanding
that there exist choices~$\lnn_\pm, \lnx_\pm$ that make~$\tormat$ an isometry
(\cf Example \ref{ex:TorusMatch}(iii)).

In the proof of the proposition as well as in many of the later arguments in
the paper it is convenient to fix integers~$\eps_\pm$ in their residue classes.
Given $\eps_+ \in \Z$, we consider the basis
\begin{equation}\label{2.w.5}
  (\nu_+,\lambda_+)=(\mu_+,\lambda_+)
  \cdot\begin{pmatrix}\frac1{k_+}&0\\\frac{\eps_+}{k_+}&1\end{pmatrix}
  =\biggl(\frac{\eps_+\zeta_++i\xi_+}{k_+},\zeta_+\biggr)\;.
\end{equation}
of $\Lambda$. (We have included $\nu_+$ in Figure~\ref{Fig2g.1}, for the
choice $\eps_+ = 1 \in \Z$.)

Given $\eps_-\in\Z$, define $\nu_-$ similarly.
The bases~$(\nu_-,\lambda_-)$ and~$(\nu_+,\lambda_+)$ of~$T=T_+\cong T_-$
induce opposite orientations,
so they are related by a matrix~$\TorMat\in GL(2,\Z)\setminus SL(2,\Z)$
such that~$(\nu_-,\lambda_-)=(\nu_+,\lambda_+)\cdot\TorMat$.
This matrix describes~$\tormat$ in the given bases.
Combining~\eqref{2.w.5} and its analogue for~$T_-$ with~\eqref{2.w.6},
we see that
\begin{equation}\label{2.w.7}
  \TorMat=
  \begin{pmatrix}
    \frac{\gll+p\eps_-}{k_-}&p\\
    \frac{q-\eps_+\gll-\eps_-\glr-\eps_+\eps_-p}{k_-k_+}&
    -\frac{\glr+p\eps_+}{k_+}
  \end{pmatrix}\;.
\end{equation}
We note that the coefficients of~$\TorMat$ depend on our choice
of~$\eps_\pm\in\Z$ (since $\nu_\pm$ do).

\begin{proof}[Proof of Proposition \ref{Prop2w.1}]
  We deduce from~\eqref{2.w.7} that~$\det G=k_-k_+\,\det\TorMat$
  and obtain~\eqref{eq:det} because~$\TorMat$ reverses orientations.
  We represent~$\lambda_-$ and~$\mu_-$
  in the basis~\eqref{2.w.5} of~$\Lambda$.
  By~\eqref{2.w.6}, we get
  \begin{align}
    \begin{split}\label{eq:bases}
    (\mu_-,\lambda_-)
    &=\frac1{k_+}\cdot(\mu_+,\lambda_+)
    \cdot\begin{pmatrix}\frac1{k_+}&0\\\frac{\eps_+}{k_+}&1\end{pmatrix}
    \cdot\begin{pmatrix}k_+&0\\-\eps_+&1\end{pmatrix}
    \cdot\begin{pmatrix}\gll&p\\\glb&\sglr\end{pmatrix}\\
    &=(\nu_+,\lambda_+)
    \cdot\begin{pmatrix}\gll &p\\
    \frac{\glb-\eps_+\gll }{k_+}&\frac{\sglr -\eps_+p}{k_+}\end{pmatrix}\;.
    \end{split}
  \end{align}
  Because~$\tilde\Lambda_-\subset\Lambda$,
  the coefficients are integers, and~\eqref{eq:epses+} follows.

  The group~$\Gamma_-\cong\Lambda/\tilde\Lambda_-$
  again acts freely by rotations
  on both the interior and the exterior circle of~$\widetilde T_-$.
  Equivalently, the elements~$\lambda_-$ and~$\mu_-$
  corresponding to the factors~$S^1_{\zeta_-}$ and~$S^1_{\xi_-}$ are primitive
  in~$\Lambda$, which gives the first two gcd conditions in~\eqref{eq:gcd+}.
  The last condition in~\eqref{eq:gcd+} holds if and only if~$\Gamma_+$
  acts freely on~$\widetilde T_+$.

  If we swap the roles of~$T_+$ and~$T_-$,
  we similarly obtain~\eqref{eq:epses-} and~\eqref{eq:gcd-}.

  The vectors~$\lambda_-$, $\mu_-$ in~\eqref{2.w.6} are perpendicular
  with respect to the standard metric on~$\C\cong\R^2$ if and only if
  \begin{equation*}
     0=\frac{\gll p\xi_+^2-\glb\glr \zeta_+^2}{k_+^2}
    =\bigl(\gll p\ar_+^2-\glb\glr \bigr)\cdot\frac{\zeta_+^2}{k_+^2}\;,
  \end{equation*}
  and the conditions on~$\psmatrix{\gll &p\\\glb&\sglr }$ and~$\ar_+$ follow.
  The remaining claims in~\ref{2w.1c} follow because
  \begin{align*}
    \zeta_-&=\abs{\lambda_-}=\frac{\abs{\sglr +ip\ar_+}\zeta_+}{k_+}
    =\frac{\sqrt{\glr^2+p^2\ar_+^2}}{k_+}\,\zeta_+
    =\sqrt{\frac{\glr(\gll \glr+\glb p)}\gll}\,\frac{\zeta_+}{k_+}
    =\sqrt{\frac{\glr k_-}{\gll k_+}}\,\zeta_+\;,\\
    \xi_-&=\abs{\mu_-}=\frac{\sqrt{\glb^2+\gll^2\ar_+^2}}{k_+}\,\zeta_+
    =\sqrt{\frac{\glb(\glb p+\gll\glr)}p}\,\frac{\zeta_+}{k_+}
    =\sqrt{\frac{\glb k_-}{pk_+}}\,\zeta_+\;.
  \end{align*}

  In~\cite{CGN}, the gluing angle~$\thet\in(-\pi,\pi]$
  has been defined as the directed angle between~$\mu_-$ and~$\mu_+$,
  see also~\eqref{2.a.0}.
  We have~$\thet\in(0,\pi)$ if and only if the scalar
  product~$\<\mu_-,\lambda_+\>=\frac \glb{k_+}\,\abs{\zeta_+}^2$
  is positive.
  Hence, we get~\ref{2w.1d} by
  \begin{equation*}
    \thet=\arg\frac{\mu_+}{\mu_-}
    =\arg\frac{ik_+\xi_+}{\glb\zeta_++i\gll\xi_+}
    =\arg\frac{k_+(\gll\xi_+^2+i\glb\xi_+\zeta_+)}{\glb^2\zeta_+^2+\gll^2\xi_+^2}
    =\arg\bigl(\gll\ar_++i\glb\bigr)\;.\qedhere
  \end{equation*}
\end{proof}

\newcommand{\thefigures}{
\protect \begin{figure}
\protect \begin{minipage}{0.48\textwidth}
\centering
\begin{tikzpicture}[x=0.8cm,y=0.8cm]
  \fill[color=lightgray] (0,0) -- (0,2) -- (1,1) -- (1,-1) -- cycle;
  \multido{\ix=-2+2}{3}{
    \multido{\iy=-2+2}{3}{
      \fill (\ix,\iy) circle (1.5pt) ;
    }
  }
  \multido{\ix=-1+2}{2}{
    \multido{\iy=-1+2}{2}{
      \fill (\ix,\iy) circle (1.5pt) ;
    }
  }
  \draw (0,0) -- (0,2) ;
  \draw (0,0) -- (1,-1) ;
  \draw (0,0) -- (1,1) ;
  \draw (0,0) -- (2,0) ;
  \begin{scope}[->, line width=1pt]
    \draw[color=blue] (0,0) -- (90:1.131cm)
	node[above left, color=black] {$\partial_{u_+}$} ;
    \draw[color=blue] (0,0) -- (-45:1.131cm)
	node[below right, color=black] {$\partial_{u_-}$} ;
    \draw[color=red] (0,0) -- (45:1.131cm)
	node[above right, color=black] {$\partial_{v_-}$} ;
    \draw[color=red] (0,0) -- (0:1.131cm)
	node[below right, color=black] {$\partial_{v_+}$} ;
  \end{scope}
  \fill (0,0) circle (3pt) ;
  \draw[->] (0,0) ++(0:0.6cm) arc (0:45:0.6cm) ;
  \node at (22.5:0.85cm) {$\vartheta$} ;
\end{tikzpicture}
\caption{\texorpdfstring{$G=\protect\gmatrix111{-1},\vartheta = \tfrac{\pi}{4}$}{angle pi/4}}
\label{fig:1/4}
\protect \end{minipage}\hfill
\protect \begin{minipage}{0.48\textwidth}
\centering
\begin{tikzpicture}[x=1cm,y=0.577cm]
  \fill[color=lightgray] (0,0) -- (0,2) -- (1,-1) -- (1,-3) -- cycle;
  \multido{\ix=-2+2}{3}{
    \multido{\iy=-2+2}{3}{
      \fill (\ix,\iy) circle (1.5pt) ;
    }
  }
  \multido{\ix=-1+2}{2}{
    \multido{\iy=-3+2}{4}{
      \fill (\ix,\iy) circle (1.5pt) ;
    }
  }
  \draw (0,0) -- (0,2) ;
  \draw (0,0) -- (1,-3) ;
  \draw (0,0) -- (1,1) ;
  \draw (0,0) -- (2,0) ;
  \begin{scope}[->, line width=1pt]
    \draw[color=blue] (0,0) -- (90:1.154cm)
	node[above, color=black] {$\partial_{u_+}$} ;
    \draw[color=blue] (0,0) -- (-60:1.154cm)
	node[right, color=black] {$\partial_{u_-}$} ;
    \draw[color=red] (0,0) -- (30:1.154cm)
	node[above right, color=black] {$\partial_{v_-}$} ;
    \draw[color=red] (0,0) -- (0:1.154cm)
	node[below right, color=black] {$\partial_{v_+}$} ;
  \end{scope}
  \fill (0,0) circle (3pt) ;
  \draw[->] (0,0) ++(0:0.75cm) arc (0:30:0.75cm) ;
  \node at (15:1cm) {$\vartheta$} ;
\end{tikzpicture}
\caption{\texorpdfstring{$G=\protect\gmatrix111{-3},\vartheta = \tfrac{\pi}{6}$}{angle pi/6}}
\label{fig:1/6}
\protect \end{minipage}
\protect \end{figure}
}

\newcommand{\morefigures}{
\protect \begin{figure}
\protect \begin{minipage}{0.48\textwidth}
\centering
\begin{tikzpicture}[x=1cm,y=0.577cm]
  \fill[color=lightgray] (0,0) -- (0,2) -- (1,1) -- (1,-1) -- cycle;
  \multido{\ix=-2+2}{3}{
    \multido{\iy=-2+2}{3}{
      \fill (\ix,\iy) circle (1.5pt) ;
    }
  }
  \multido{\ix=-1+2}{2}{
    \multido{\iy=-3+2}{4}{
      \fill (\ix,\iy) circle (1.5pt) ;
    }
  }
  \draw (0,0) -- (0,2) ;
  \draw (0,0) -- (1,-1) ;
  \draw (0,0) -- (1,3) ;
  \draw (0,0) -- (2,0) ;
  \begin{scope}[->, line width=1pt]
    \draw[color=blue] (0,0) -- (90:1.154cm)
	node[above, color=black] {$\partial_{u_+}$} ;
    \draw[color=blue] (0,0) -- (-30:1.154cm)
	node[below right, color=black] {$\partial_{u_-}$} ;
    \draw[color=red] (0,0) -- (60:1.154cm)
	node[right, color=black] {$\partial_{v_-}$} ;
    \draw[color=red] (0,0) -- (0:1.154cm)
	node[below right, color=black] {$\partial_{v_+}$} ;
  \end{scope}
  \fill (0,0) circle (3pt) ;
  \draw[->] (0,0) ++(0:0.6cm) arc (0:60:0.6cm) ;
  \node at (30:0.85cm) {$\vartheta$} ;
\end{tikzpicture}
\caption{\texorpdfstring{$G=\protect\gmatrix113{-1},\vartheta = \tfrac{\pi}{3}$}{angle pi/3}}
\label{fig:1/3}
\protect \end{minipage}\hfill
\protect \begin{minipage}{0.48\textwidth}
\centering
\begin{tikzpicture}[x=0.7cm,y=0.577cm]
  \multido{\ix=-3+3}{3}{
    \multido{\iy=-3+3}{3}{
      \fill (\ix,\iy) circle (1.5pt) ;
    }
  }
  \fill (-1.5,-1.5) circle (1.5pt)
  (-1.5,1.5) circle (1.5pt)
  (1.5,-1.5) circle (1.5pt)
  (1.5,1.5) circle (1.5pt) ;
  \draw (-1.5,-1.5) -- (-1.5,1.5) ;
  \draw (-1.5,-1.5) -- (1.5,-1.5) ;
  \begin{scope}[->, line width=1pt]
    \draw[color=blue] (-1.5,-1.5) -- ++(90:1cm)
	node[above left, color=black] {$\partial_{u_+}=\partial_{v_-}$} ;
    \draw[color=blue] (-1.5,-1.5) -- ++(0:1cm)
	node[below right, color=black] {$\partial_{u_-}=\partial_{v_+}$} ;
  \end{scope}
  \fill (-1.5,-1.5) circle (3pt) ;
  \draw[->] (-1.5,-1.5) ++(0:0.6cm) arc (0:90:0.6cm) ;
  \node at (-0.5,-0.5) {$\vartheta$} ;
\end{tikzpicture}
\caption{\texorpdfstring{$G=\protect\gmatrix0220,\vartheta = \frac\pi2$}{angle pi/2}}
\label{fig:trivmatch}
\protect \end{minipage}
\protect \end{figure}
}

\thefigures
\morefigures

\newcommand\evenmorefigures{
  \protect \begin{figure}
\protect \begin{minipage}{0.48\textwidth}
\centering
\begin{tikzpicture}[x=1cm,y=0.707cm]
  \fill[color=lightgray] (0,0) -- (0,3) -- (1,2) -- (1,-1) -- cycle;
  \multido{\ix=0+3}{2}{
    \multido{\iy=0+3}{2}{
      \fill (\ix,\iy) circle (1.5pt) ;
    }
  }
    \multido{\iy=-1+3}{2}{
      \fill (1,\iy) circle (1.5pt) ;
    }
  \multido{\ix=-1+3}{2}{
    \multido{\iy=-2+3}{2}{
      \fill (\ix,\iy) circle (1.5pt) ;
    }
  }
  \draw (0,0) -- (0,3) ;
  \draw (0,0) -- (1,-1) ;
  \draw (0,0) -- (1,2) ;
  \draw (0,0) -- (3,0) ;
  \begin{scope}[->, line width=1pt]
    \draw[color=blue] (0,0) -- (90:1cm)
	node[above left, color=black] {$\partial_{u_+}$} ;
    \draw[color=blue] (0,0) -- (-35.3:1cm)
	node[below, color=black] {$\partial_{u_-}$} ;
    \draw[color=red] (0,0) -- (54.7:1cm)
	node[right, color=black] {$\partial_{v_-}$} ;
    \draw[color=red] (0,0) -- (0:1cm)
	node[below right, color=black] {$\partial_{v_+}$} ;
  \end{scope}
  \fill (0,0) circle (3pt) ;
  \draw[->] (0,0) ++(0:0.6cm) arc (0:54.7:0.6cm) ;
  \node at (27.3:0.85cm) {$\vartheta$} ;
\end{tikzpicture}
\caption{\texorpdfstring{$G=\protect\gmatrix112{-1},\vartheta = \arc\cos \tfrac1{\sqrt 3}$}{angle pi/6}}
\label{fig:A4}
\protect \end{minipage}
\protect \begin{minipage}{0.48\textwidth}
\centering
\begin{tikzpicture}[x=1.591cm,y=0.225cm]
  \fill[color=lightgray] (0,0) -- (0,3) -- (1,-2) -- (1,-5) -- cycle;
  \multido{\ix=-1+3}{2}{
    \multido{\iy=-4+3}{5}{
      \fill (\ix,\iy) circle (1.5pt) ;
    }
  }
  \multido{\ix=0+3}{2}{
    \multido{\iy=-3+3}{5}{
      \fill (\ix,\iy) circle (1.5pt) ;
    }
  }
  \multido{\iy=-5+3}{6}{
    \fill (1,\iy) circle (1.5pt) ;
  }
  \draw (3,0) -- (0,0) -- (0,3) ;
  \draw (1,10) -- (0,0) -- (1,-5) ;
  \begin{scope}[->, line width=1pt]
    \draw[color=blue] (0,0) -- (90:0.67cm)
	node[above left, color=black] {$\partial_{u_+}$} ;
    \draw[color=blue] (0,0) -- (-35.3:0.67cm)
	node[below, color=black] {$\partial_{u_-}$} ;
    \draw[color=red] (0,0) -- (54.7:0.67cm)
	node[right, color=black] {$\partial_{v_-}$} ;
    \draw[color=red] (0,0) -- (0:0.67cm)
	node[below right, color=black] {$\partial_{v_+}$} ;
  \end{scope}
  \fill (0,0) circle (3pt) ;
  \draw[->] (0,0) ++(0:0.35cm) arc (0:54.7:0.35cm) ;
  \node at (27.3:0.55cm) {$\vartheta$} ;
  \end{tikzpicture}
  \caption{\texorpdfstring{$G=\protect\gmatrix11{10}{-5},
      \vartheta = \arc\cos \tfrac1{\sqrt 3}$}{running example}}
  \label{fig:running}
  \protect\end{minipage}
\protect\end{figure}
}

\begin{rmk}\label{rem:SmallK}
When $k_\pm \leq 2$, the only possibilities (up to swapping~$M_+$ and~$M_-$)
with \mbox{$r_\pm$, $\glb \ge 0$}
and~$p=1$ are the ones already studied in~\cite{xtcs,CGN},
illustrated in Figures~\ref{fig:1/4}--\ref{fig:1/3}.
If we allow~$p\ge 1$,
there are two more with gluing matrices~$\gmatrix131{-1}$
and~$\gmatrix0220$; the latter is depicted in Figure~\ref{fig:trivmatch}.
Notice that~$\thet=\frac\pi 2$ in this example, so the radii~$\xi_+=\zeta_-$
and~$\xi_-=\zeta_+$ can be chosen independently.

Once we allow $k_+$ or $k_-$ to be larger than 2, there are many more
possibilities. Figure~\ref{fig:A4} illustrates a torus matching with
$k_+ = 1$ and $k_- = 3$, where $\ar_+ = \sqrt{2}$ and $\ar_- = \frac{1}{\sqrt{2}}$
(so the tori have proportions of A4 paper).
We consider this further in Section~\ref{sec:torus}.
Let us for now give a single more complicated example that we will refer
to in the course of our calculations.
\end{rmk}
\evenmorefigures

\begin{ex}\label{ex:run-gluing}
For $k_+ = 3$ and $k_- = 5$, one valid gluing matrix is
\[ \begin{pmatrix} 1 & 1 \\ 10 & -5 \end{pmatrix} \]
with $\eps_+ = 1$ and $\eps_- = -1$. The torus matching is illustrated
in Figure~\ref{fig:running}.
The aspect ratios are~$\ar_+=5\sqrt 2$ and~$\ar_-=\sqrt 2$,
and the gluing angle is~$\thet=\arg(1+\sqrt 2\,i)=\arccos\frac1{\sqrt 3}$.
One example with this gluing matrix may be found
in Table~\ref{table:matchings}, no.~228.
\end{ex}

Generalising the computations from~\cite[\S 1.3]{xtcs}, the gluing matrix
also determines the fundamental group of the extra-twisted connected sum.

\begin{prop}\label{prop:pi1}
  An extra-twisted connected sum $M$ with gluing matrix~$G=\psmatrix{\gll&p\\\glb&\sglr}$
  has fundamental group isomorphic to~$\Z/p$.
\end{prop}

\begin{proof}
  Let~$\iota_\pm\colon T^2\to M_\pm$ denote the inclusion map
  and note that~$\pi_1(T^2)\cong\pi_1(X)\cong\Z^2$.
  Since $\pi_1 V_\pm = 1$ by Remark~\ref{rem:pi1},
  we also have $\pi_1(M_\pm) \cong \Z$, and
  the interior circle~$S^1_{\zeta_\pm}$ is null-homotopic in~$M_\pm$, and
  we have a short exact sequence
  \begin{equation*}
    0\longrightarrow\pi_1(S^1_{\zeta_\pm})\longrightarrow
    \pi_1(T^2)\stackrel{\iota_{\pm*}}\longrightarrow\pi_1(M_\pm)
      \longrightarrow 0\;.
  \end{equation*}
  Because~$\iota_{\pm*}$ is surjective,
  it follows from the Seifert-van Kampen theorem %
  that
  \begin{equation*}
    \pi_1(M)\cong\pi_1(T^2)/\bigl(\ker(\iota_{+*})+\ker(\iota_{-*})\bigr)\;.
  \end{equation*}
  
  As basis of~$\Lambda=\pi_1(T^2)$,
  we choose the vectors~$\nu_+=\frac{\mu_++\eps_+\lambda_+}{k_+}$
  and~$\lambda_+$ as in~\eqref{2.w.5}.
  Dividing out~$\pi_1(S^1_{\zeta_+})=\ker(\iota_+)$,
  we are left with a cyclic group generated
  by~$\nu_+$.
  Modulo~$\pi_1(S^1_{\zeta_+})$, the group~$\pi_1(S^1_{\zeta_-})=\ker\iota_-$
  is generated by~$p\nu_+$,
  so~$\pi_1(M)\cong\Z/p$.
\end{proof}

We will discuss covering spaces in Proposition~\ref{prop:cover}.
Some examples of non-simply connected extra-twisted connected sums
will be given in Examples~\ref{ex:TorusMatch}~\ref{exnonex.1}
and~\ref{ex:run-symm}.

\begin{rmk}\label{rmk:eps=r}
  Consider a torus matching that leads to simply connected extra-twisted sums,
  so~$p=1$. %
  Then~$\eps_+=-r_-$ is uniquely determined modulo~$k_+$ by
  the second congruence in~\eqref{eq:epses+}.
  Similarly, $\eps_-=-r_+$ modulo~$k_-$ by~\eqref{eq:epses-}.
  If we choose the integers $\reps_+ = -r_-$ and $\reps_- = -r_+$ to represent
  these residue classes, then
  the matrix~$\TorMat$ becomes~$\psmatrix{0&1\\1&0}$; in other words
  $\nu_+ = \lambda_-$ and $\nu_- = \lambda_+$.
  In particular, the lattice~$\Lambda$ in~\eqref{2.w.7}
  is generated by the two vectors~$\lambda_+$ and~$\lambda_-$
  that generate the interior circles of~$T_\pm$.
  We have shaded the corresponding fundamental domains in
  Figures~\ref{fig:1/4}, \ref{fig:1/6}, \ref{fig:1/3}, \ref{fig:A4}
  and~\ref{fig:running}, all of which have~$p=1$.
\end{rmk}

\subsection{Matchings and polarising lattices}\label{3.c}
\label{subsec:match}

In Theorem~\ref{thm:glue} we set up our gluing construction, using
a torus matching $\tormat$ and a \hk rotation $\hkr$. We studied the torus
matchings in~\S\ref{2.w}, while Theorem~\ref{thm:match_to_hkr} reduced the
problem of finding \hk rotations to the less metric problem of finding K3
matchings. For the final piece of the machine, we review from~\cite[\S 6]{xtcs}
how to find K3 matchings between building blocks.

The properties of the \gtmfd s produced by Theorem~\ref{thm:glue} can clearly
depend not just on the choices of building blocks and torus matching but
also on the choice of \hk rotation. However, the topological properties that we care about depend on the \hk rotation only via what we term its associated
``configuration'' of the polarising lattices of the building blocks.

Recall from Definition~\ref{def:block} that the polarising lattice of
a block $(Z,\K)$ refers to the image $N$ of $H^2(Z)$ in $H^2(\K)$, equipped
with the intersection form. We use $L$ to denote a fixed even unimodular
lattice with signature $(3,19)$, so that $H^2(\K)$ is isometric to $L$
for any K3 surface~$\K$.

\begin{dfn}[{\cite[Definition 6.3]{xtcs}}]
\label{def:config}
A \emph{configuration} of polarising lattices $N_+$, $N_-$ is a pair of
primitive embeddings $N_\pm \into L$. Two configurations are equivalent
if they are related by the action of the isometry group~$O(L)$.
\end{dfn}

Clearly we can associated a configuration to any \hk rotation $\hkr$.
Given the claim that the topology depends mainly on the blocks and
the configuration, it is natural to phrase the matching problem as follows.

\begin{qstn}
\label{qn:config}
Given $\thet \in \R$ and a pair
of sets of building blocks $(Z_\pm, \K_\pm)$ (each family with fixed
topology, and in particular with fixed polarising lattice $N_\pm$),
which configurations of $N_+$ and $N_-$ are realised by a
$\thet$-\hk rotation of some elements of the families?
\end{qstn}

For any lattice $\Lambda$, let $\Lambda\Rlat := \Lambda \otimes_\Z \R$.
Given a configuration, let $\pi_\pm : L\Rlat \to N_\pm\Rlat$ denote the
orthogonal projection, 
and let $N_\pm\Rlat^\mu$ denote the $(\cos \mu)^2$-eigenspace
of the self-adjoint endomorphism $\pi_\pm \pi_\mp : N_\pm\Rlat \to N_\pm\Rlat$,
and let $N_\pm\Rlat^{\not=\mu}$ denote the orthogonal complement to
$N_\pm\Rlat^\mu$ in~$N_\pm\Rlat$.
\begin{rmk}
\label{rmk:cos2rational}
By Proposition~\ref{Prop2w.1}\ref{2w.1d},
the gluing angle $\thet$ of any torus matching has $(\cos \thet)^2$ rational.
Therefore 
$N_\pm\Rlat^\thet$ and $N_\pm\Rlat^{\not=\thet}$ are both spanned by their
respective subsets $N_\pm^\thet$ and~$N_\pm^{\not=\thet}$ of integral points.
\end{rmk}

The condition~\eqref{eq:hkr} implies that if there
exists a $\thet$-\hk rotation compatible with a given configuration
then there are positive classes $[\omega^I_\pm], [\omega^J_\pm] \in N_\pm\Rlat$
and $[\omega^K_\pm] \in N_\pm^\perp\Rlat$ such that
\begin{align*}
[\omega^K_-] &= - [\omega^K_+] \\
([\omega^I_-] + i [\omega^J_-]) &= e^{i\vartheta} ([\omega^I_+] - i[\omega^J_+]) .
\end{align*}
From this we can deduce the following necessary conditions for realising
a given configuration (see~\cite[\S 6.3]{xtcs} for explanation)
by a $\thet$-\hk rotation of some $(Z_+, \K_-)$ and $(Z_-, \K_-)$.
\begin{enumerate}
\item $N_+ + N_-$ is non-degenerate of signature $(2, \rk - 2)$.
\item\label{it:kaehler} $N_\pm^\thet$ contains the restriction of some Kähler class from $Z_\pm$;
in particular $N_\pm^\thet$ is non-trivial.
\item \label{it:lambda}
The Picard lattice $\Pic \K_\pm := H^2(\K_\pm;\Z) \cap H^{1,1}(\K_\pm; \C)$
contains both $N_\pm$ and $N_\mp^{\not=\thet}$. 
\end{enumerate}
Let $\Lambda_\pm$ be the set of integral points in 
$N_\pm\Rlat + N_\mp\Rlat^{\not=\thet}$.
Then $\Lambda_\pm$ is primitive in $L$, in the sense
that $L/\Lambda_\pm$ is torsion-free, and could also be described as the
``primitive hull'' of the sublattice $N_\pm + N_\mp^{\not=\thet} \subset L$
\ie its smallest overlattice that is primitive.
It is a non-degenerate lattice of signature $(1, \rk \Lambda_\pm {-} 1)$,
and (iii) means that $\Pic \K_\pm$ contains $\Lambda_\pm$,
so that~$\K_\pm$ is~\mbox{\emph{$\Lambda_\pm$-polarised}}.

On the other hand, it turns out to be possible to express a sufficient
condition for being able to match some elements from a pair of families
in terms of those families containing suitably generic $\Lambda_\pm$-polarised
K3 surfaces.
For completeness, let us describe the notion of genericity that we need,
even though we will not use any of the technical details.
Recall that marked K3 surfaces whose Picard lattice contains a fixed primitive
lattice $\Lambda \subset L$ of signature $(1, \rk \Lambda - 1)$
can be parametrised by their periods, which belong
to the Griffiths domain
 \[  D_\Lambda =
\{\textrm{oriented positive-definite planes }\Pi \subset \Lambda^\perp (\R) \}
\cong \{ \Pi \in \bbp(\Lambda^\perp(\bbc)) : \Pi^2 = 0,\; \Pi \, \overline\Pi > 0\} , \]
where the second description gives rise to a complex analytic structure.

\begin{dfn}[{\cite[Definition 2.27]{xtcs}}]
\label{def:generic}
Let $N \subset L$ be a primitive sublattice, $\Lambda \subset L$ a primitive
overlattice of $N$, and $\Amp_\fbb$ an open subcone
of the positive cone in $N\Rlat$. We say that a set of building blocks
$\fbb$ with polarising lattice $N$ is \emph{$(\Lambda, \Amp_\fbb)$-generic} if
there is a subset $U_\fbb$ of the Griffiths domain $D_\Lambda$
with complement a countable union of complex analytic submanifolds of
positive codimension with the property that:
for any $\Pi \in U_\fbb$ and $\kclass \in \Amp_\fbb$ there is a building block
$(Z,\K) \in \fbb$ and a marking $\hdg : L \to H^2(\K; \ZZ)$ such that
$\hdg(\Pi) = H^{2,0}(\K)$, and $\hdg(\kclass)$ is the image of the restriction
to $\K$ of a Kähler class on $Z$.
\end{dfn}

All that matters for the purposes of this paper is that the conditions in the
definition make the following proposition work.

\begin{prop}[{\cite[Theorem 6.10]{xtcs}}]
\label{prop:matching}
Let $\fbb_\pm$ be a pair of sets of building blocks with polarising
lattices $N_\pm$, and $\thet \in \R$ such that $(\cos \thet)^2$ is rational
\footnote{The hypothesis on $\cos \thet$ is missing from
the statement \cite[Theorem 6.10]{xtcs}, but the proof there implicitly assumes
that $N_\pm\Rlat + N_\mp\Rlat^{\not=\thet}$ is spanned by $\Lambda_\pm$.
As per Remark \ref{rmk:cos2rational},
this always holds in the context where we apply the result.}  
.
Let $N_\pm \into L$ be a configuration of the polarising lattices, and let
$\Lambda_\pm \subset L$ be the lattice of integral points in
$N_\pm\Rlat + N_\mp\Rlat^{\not=\thet}$. 
Suppose that the set $\fbb_\pm$ is $(\Lambda_\pm, \Amp_{\fbb_\pm})$-generic.
If
\begin{equation}
\label{eq:amps}
\cos \thet \not= 0 \textrm{ and } (\sign \cos \thet) \pi_-(N_+\Rlat^\thet \cap \Amp_{\fbb_+}) \cap \Amp_{\fbb_-}
\neq \emptyset .
\end{equation}
or
\begin{equation}
\label{eq:amps2}
\cos \thet = 0
 \textrm{ and } N_+\Rlat^\frac{\pi}{2} \cap \Amp_{\fbb_+} \neq \emptyset 
 \textrm{ and } N_-\Rlat^\frac{\pi}{2} \cap \Amp_{\fbb_-} \neq \emptyset 
\end{equation}
then there exist
$(Z_\pm, \K_\pm) \in \fbb_\pm$ with an angle $\thet$ K3 matching
$\hkr : \K_+ \to \K_-$ with the prescribed configuration. 
\end{prop}

In~\cite{CGN} we found that the following invariants of a configuration
play a key role in the calculation of $\nu$ (see Theorem~\ref{Thm2a.1}).

\begin{dfn}\label{def:angles}
Given a configuration $N_+, N_- \subset L$, let $A_\pm : L\Rlat \to L\Rlat$
denote the reflection of $L\Rlat := L \otimes \R$ in $N_\pm$ (with respect to the
intersection form of $L\Rlat$; this is well-defined since $N_\pm$ is non-degenerate).
Suppose $A_+ \circ A_-$ preserves some decomposition $L\Rlat = L^+ \oplus L^-$
as a sum of positive and negative-definite subspaces.
Then the \emph{configuration angles} are the arguments
$\alpha^+_1, \alpha^+_2, \alpha^+_3$ and
$\alpha^-_1, \ldots, \alpha^-_{19}$ of the eigenvalues of the restrictions
$A_+ \circ A_- : L^+ \to L^+$ and $A_+ \circ A_- : L^- \to L^-$ respectively.
\end{dfn}

\begin{rmk}
\label{rmk:pure}
  Since $\Lambda_\pm$ is always at least as big as $N_\pm$, the
  genericity results required to apply Proposition~\ref{prop:matching}
  are the weakest possible when $\Lambda_\pm = N_\pm$.
  This happens in particular if~$N^\thet_\pm=N_\pm$, that is,
  if~$\pi_\pm\circ\pi_\mp|_{N_\pm}=(\cos \thet)^2\,\id_{N_\pm}$.
  In that case we will say that~$N_+$ and~$N_-$ meet
  \emph{at pure angle}~$\thet$.

  Unless~$\thet=\pm\frac\pi2$, meeting at pure angle $\thet$ implies
  that~$\rk N_+=\rk N_- = $ multiplicity of $\pm 2\thet$ as configuration
  angles, while the remaining configuration angles are all $0$.
\end{rmk}

Even with Proposition~\ref{prop:matching} in hand, it is still hard in general
to completely answer Question~\ref{qn:config} concerning which configurations
can be realised by matching; in some situations it is hard to prove any
genericity result of the type required (roughly, this becomes harder the
larger $\Lambda$ is), and even if one does, it may be hard to
know whether one has found the ``best possible'' choice of $\Amp$ (or whether
some blocks in the family have a bigger Kähler cone than the generic members).

However, all examples we know of building blocks do in fact have the property
that they come in families that are $(N, \Amp)$ generic, for $N$ the polarising
lattice, and $\Amp$ some open cone in~$N\Rlat$ (in particular,
Proposition~\ref{prop:generic} asserts this for the examples in this paper).
Finding all matchings of a pair of blocks where the configurations
are at pure angle $\thet$ and the Kähler classes are required to be in a
particular cone $\Amp$ is only a lattice-arithmetic problem.
That can certainly be solved by a brute force algorithm, 
though not very easily by hand if the ranks of the polarising lattices are
greater than 1.

In this paper, we will restrict attention to blocks where the polarising
lattices have rank~1, which makes it possible to answer Question
\ref{qn:config} decisively.
Condition~\ref{it:kaehler} then automatically requires
the configurations to be of pure angle, and there is no ambiguity in the
choice of~$\Amp$. If the generators of the polarising
lattices have square-norms $n_+$ and $n_-$, then the bilinear form
on $N_+ + N_-$ imposed by the configuration will be defined by a matrix
\begin{equation}
\label{eq:config_matrix}
\begin{pmatrix} n_+ & h \\ h & n_- \end{pmatrix} ,
\end{equation}
and the gluing angle is determined by
\begin{equation}\label{3.c.3}
(\cos \thet)^2 = \frac{h^2}{n_+ n_-} .
\end{equation} 
Thus there exists a matching of the blocks with gluing angle $\thet$ %
only if $\cos \thet \sqrt{n_+n_-}$ is an integer.
By Nikulin \cite[Theorem 1.12.4]{nikulin}, this is also sufficient.

We give here one example that we will refer to while developing the
calculations in Sections~\ref{sec:nu} and~\ref{sec:hyp}; see Subsection
\ref{sec:match-ex} for further examples of matchings. 

\begin{ex}\label{ex:run-matching}
  Consider two building blocks~$Z_+$, $Z_-$ of rank~$1$ with
  polarising lattices~$(6)$ and~$(2)$, respectively.
  We consider the configuration such that the restriction of the intersection
form to $N_+ + N_-$ is defined by
  \begin{equation*}
    \begin{pmatrix}
      6&2\\2&2
    \end{pmatrix},
  \end{equation*}
  which has pure angle~$\thet=\arccos\frac1{\sqrt 3}$.
  We will combine this configuration with the gluing data
  of Example~\ref{ex:run-gluing},
  using a $\Z/3$-block from Example~\ref{ex:quadric} as~$Z_+$
  and a $\Z/5$-block from Example~\ref{ex:five} as~$Z_-$,
  see Table~\ref{table:matchings}, no.~228.
  The configuration angles are
  \begin{equation*}
    \alpha^+_1 = -\alpha^+_2 = 2\arccos\frac1{\sqrt 3}
    \qquad\text{and}\qquad\alpha^+_3=\alpha^-_1=\cdots=\alpha^-_{19}=0\;.
  \end{equation*}
\end{ex}

\section{Computing the extended \nuinvt} %
\label{sec:nu}

We prove Theorem~\ref{Thm:A}, following the path outlined
in the introduction.
We recall the adapted Dirac operator (Section~\ref{sec:moddirac})
and the gluing formula (Section~\ref{2.g}) from~\cite{CGN}.
The contributions from both halves consist of an adiabatic limit
(Section~\ref{2.b}) and a variational term (Section~\ref{2.c}).
To complete the computation,
we rewrite the variational term in Section~\ref{2h},
using Proposition~\ref{Prop1}. %
In Section~\ref{sec:hyp}, we will present an alternative approach to
the computation of the variational terms that leads to Theorem~\ref{Thm:B}.

\subsection{A modification of the spin Dirac operator}\label{sec:moddirac}
The extended \nuinvt of a $G_2$-manifold is defined in~\eqref{eq:nubar}
using the \etainvt of the signature operator~$B$
and the spin Dirac operator~$D$.
For computations, it is much more comfortable to work with
a Riemannian metric that is of product type in the gluing region
and sufficiently close to some $G_2$-metric.
However, the \etainvt of the spin Dirac operator
of such a gluing metric typically differs from the one in the $G_2$-case
both by a small local contribution and by a $\Z$-valued spectral flow.
To avoid the latter, we modified the spin Dirac operator
in~\cite{CGN}.
Because all our following considerations rely on the gluing metric
and the modified Dirac operator,
we take the time to introduce them now.

\newcommand{\thegluingfig}{
\begin{figure}
  \begin{tikzpicture}
    \draw[line width=1pt] (7,0.5)  .. controls (7,1.3) and (6.6,1.62) ..
    (6,1.5) .. controls (5,1.3) ..
    (4,1.2) .. controls (3.5,1.15) and (3.5,1) ..
    (3,1) --
    (-3,1) .. controls (-3.5,1) and (-3.5,1.15) ..
    (-4,1.2) .. controls (-5,1.3) ..
    (-6,1.5) .. controls (-6.6,1.62) and (-7,1.3) ..
    (-7,0.5) .. controls (-7,-0.3) and (-6.6,-0.62) ..
    (-6,-0.5) .. controls (-5,-0.3) ..
    (-4,-0.2) .. controls (-3.5,-0.15) and (-3.5,0) ..
    (-3,0) -- node[below] {$(M_\ell,g_\ell)$}
    (3,0)  .. controls (3.5,0) and (3.5,-0.15) ..
    (4,-0.2) .. controls (5,-0.3) ..
    (6,-0.5) .. controls (6.6,-0.62) and (7,-0.3) ..
    (7,0.5) ;
    \draw (6,1.5) arc (90:270:0.1 and 1) ;
    \draw[dashed] (6,1.5) arc (90:-90:0.1 and 1) ;
    \draw (4,1.2) arc (90:270:0.07 and 0.7) ;
    \draw[dashed] (4,1.2) arc (90:-90:0.07 and 0.7) ;
    \draw (3,1) arc (90:270:0.05 and 0.5) ;
    \draw[dashed] (3,1) arc (90:-90:0.05 and 0.5) ;
    \draw (0,1) arc (90:270:0.05 and 0.5) ;
    \draw[dashed] (0,1) arc (90:-90:0.05 and 0.5) ;
    \draw (-3,1) arc (90:270:0.05 and 0.5) ;
    \draw[dashed] (-3,1) arc (90:-90:0.05 and 0.5) ;
    \draw (-4,1.2) arc (90:270:0.07 and 0.7) ;
    \draw[dashed] (-4,1.2) arc (90:-90:0.07 and 0.7) ;
    \draw (-6,1.5) arc (90:270:0.1 and 1) ;
    \draw[dashed] (-6,1.5) arc (90:-90:0.1 and 1) ;
    \begin{scope}[line width=0.1pt]
      \draw[->] (-6.3,0.5) -- (6.5,0.5) node[right] {$\scriptstyle t$} ;
      \draw (-6,0.4) -- (-6,0.6) ;
      \node[right] at (-6,0.3) {$\scriptstyle-2\ell$} ;
      \draw (-4,0.4) -- (-4,0.6) ;
      \node[right] at (-4,0.3) {$\scriptstyle-\ell-1$} ;
      \draw (-3,0.4) -- (-3,0.6) ;
      \node[right] at (-3,0.3) {$\scriptstyle-\ell$} ;
      \draw (0,0.4) -- (0,0.6) ;
      \node[right] at (0,0.3) {$\scriptstyle 0$} ;
      \draw (6,0.4) -- (6,0.6) ;
      \node[right] at (6,0.3) {$\scriptstyle2\ell$} ;
      \draw (4,0.4) -- (4,0.6) ;
      \node[right] at (4,0.3) {$\scriptstyle\ell+1$} ;
      \draw (3,0.4) -- (3,0.6) ;
      \node[right] at (3,0.3) {$\scriptstyle\ell$} ;
      \draw[->] (-6.3,-2.5) -- (3.5,-2.5) node[right] {$\scriptstyle t_-$} ;
      \draw (-6,-2.4) -- (-6,-2.6) ;
      \node[right] at (-6,-2.7) {$\scriptstyle 0$} ;
      \draw (-4,-2.4) -- (-4,-2.6) ;
      \node[right] at (-4,-2.7) {$\scriptstyle\ell-1$} ;
      \draw (0,-2.4) -- (0,-2.6) ;
      \node[right] at (0,-2.7) {$\scriptstyle2\ell$} ;
      \draw (3,-2.4) -- (3,-2.6) ;
      \node[right] at (3,-2.7) {$\scriptstyle3\ell$} ;
      \draw[->] (-3.5,-4.5) -- (3.5,-4.5) node[right] {$\scriptstyle t$} ;
      \draw (-3,-4.4) -- (-3,-4.6) ;
      \node[right] at (-3,-4.7) {$\scriptstyle-\ell$} ;
      \draw (0,-4.4) -- (0,-4.6) ;
      \node[right] at (0,-4.7) {$\scriptstyle0$} ;
      \draw (3,-4.4) -- (3,-4.6) ;
      \node[right] at (3,-4.7) {$\scriptstyle\ell$} ;
      \draw[->] (6.3,-6.5) -- (-3.5,-6.5) node[left] {$\scriptstyle t_+$} ;
      \draw (6,-6.4) -- (6,-6.6) ;
      \node[right] at (6,-6.7) {$\scriptstyle 0$} ;
      \draw (4,-6.4) -- (4,-6.6) ;
      \node[right] at (4,-6.7) {$\scriptstyle\ell-1$} ;
      \draw (0,-6.4) -- (0,-6.6) ;
      \node[right] at (0,-6.7) {$\scriptstyle2\ell$} ;
      \draw (-3,-6.4) -- (-3,-6.6) ;
      \node[right] at (-3,-6.7) {$\scriptstyle3\ell$} ;
    \end{scope}
    \draw[line width=1pt]
    (-4,-1.8) .. controls (-5,-1.7) ..
    (-6,-1.5) .. controls (-6.6,-1.38) and (-7,-1.7) ..
    (-7,-2.5) .. controls (-7,-3.3) and (-6.6,-3.62) ..
    (-6,-3.5) .. controls (-5,-3.3) ..
    (-4,-3.2) ;
    \draw[line width=1pt,dashed]
    (-4,-1.8) .. controls (-2,-1.93) ..
    (0,-1.95) .. controls (2,-1.96) ..
    (3.5,-1.97)
    (-4,-3.2) .. controls (-2,-3.07) ..
    (0,-3.05) .. controls (2,-3.04) ..
    (3.5,-3.03) ;
    \draw (-6,-1.5) arc (90:270:0.1 and 1) ;
    \draw[dashed] (-6,-1.5) arc (90:-90:0.1 and 1) ;
    \draw (-4,-1.8) arc (90:270:0.07 and 0.7) ;
    \draw[dashed] (-4,-1.8) arc (90:-90:0.07 and 0.7) ;
    \draw[dashed] (-6,-2) arc (90:270:0.05 and 0.5) ;
    \draw[dashed] (-6,-2) arc (90:-90:0.05 and 0.5) ;
    \draw[dashed] (-6,-2) -- (3.5,-2) (-6,-3) -- (3.5,-3) ;
    \node[right] at (4.3,-2.5) {$(V_-\times S^1_{\lnx_-})/\Gamma_-$} ;
    \draw[line width=1pt] (-3,-4) -- (3,-4) (3,-5) -- (-3,-5) ;
    \draw[line width=1pt,dashed] (3,-4) -- (3.5,-4) (3,-5) -- (3.5,-5)
    (-3,-4) -- (-3.5,-4) (-3,-5) -- (-3.5,-5) ;
    \draw (-3,-4) arc (90:270:0.05 and 0.5) ;
    \draw[dashed] (-3,-4) arc (90:-90:0.05 and 0.5) ;
    \draw (0,-4) arc (90:270:0.05 and 0.5) ;
    \draw[dashed] (0,-4) arc (90:-90:0.05 and 0.5) ;
    \draw (3,-4) arc (90:270:0.05 and 0.5) ;
    \draw[dashed] (3,-4) arc (90:-90:0.05 and 0.5) ;
    \node[right] at (4.3,-4.5) {$X\times\R$} ;
    \draw[line width=1pt]
    (4,-5.8) .. controls (5,-5.7) ..
    (6,-5.5) .. controls (6.6,-5.38) and (7,-5.7) ..
    (7,-6.5) .. controls (7,-7.3) and (6.6,-7.62) ..
    (6,-7.5) .. controls (5,-7.3) ..
    (4,-7.2) ;
    \draw[line width=1pt,dashed] (4,-5.8) .. controls (2,-5.93) ..
    (0,-5.95) .. controls (-2,-5.96) ..
    (-3.5,-5.97)
    (4,-7.2) .. controls (2,-7.07) ..
    (0,-7.05) .. controls (-2,-7.04) ..
    (-3.5,-7.03) ;
    \draw (6,-5.5) arc (90:270:0.1 and 1) ;
    \draw[dashed] (6,-5.5) arc (90:-90:0.1 and 1) ;
    \draw (4,-5.8) arc (90:270:0.07 and 0.7) ;
    \draw[dashed] (4,-5.8) arc (90:-90:0.07 and 0.7) ;
    \draw[dashed] (6,-6) arc (90:270:0.05 and 0.5) ;
    \draw[dashed] (6,-6) arc (90:-90:0.05 and 0.5) ;
    \draw[dashed] (6,-6) -- (-3.5,-6) (6,-7) -- (-3.5,-7) ;
    \node[left] at (-4.3,-6.5) {$(V_+\times S^1_{\lnx_+})/\Gamma_+$} ;
  \end{tikzpicture}
  \caption{The gluing metric \texorpdfstring{$g_\ell$ on~$M_\ell$}{gl on Ml}}
  \label{fig:gluingmetric}
\end{figure}
}

\subsubsection{}\label{ssec:gell}%
Let~$(M_\ell,\bar g_\ell)$ denote the $G_2$-manifold produced by
Theorem \ref{thm:glue}.
For~$\ell\gg 1$, it is close to a Riemannian manifold~$(M_\ell,g_\ell)$ produced
by naive gluing, in a sense we want to make precise.

Recall from Section~\ref{subsec:glue} that~$(V_\pm,g^{V_\pm})$
are Calabi-Yau manifolds with one end each that
is asymptotic to a cylinder~$\Sigma_\pm\times S^1_{\lnn_\pm}\times(0,\infty)$.
We extend~$t_\pm$ to smooth functions on~$X_\pm$ that are nonpositive
outside the cylindrical region.
Then we first choose new metrics~$g^{V_\pm}_\ell$ that agree with
the original Calabi-Yau metrics~$g^{V_\pm}$
on~$\bigl\{\,x\in V_\pm\bigm|t_\pm(x)\le\ell-1\,\bigr\}$
and with the cylindrical metrics
on~$\bigl\{\,x\in V_\pm\bigm|t_\pm(x)\ge\ell\,\bigr\}$.
This can be done
such that~$\bigl\|g^{V_\pm}_\ell-g^{V_\pm}\bigr\|_{C^k}=O(e^{-c\ell})$
for a fixed~$c>0$ and all~$k$.

We consider the twisted
products~$M_\pm=(V_\pm\times S^1_{\lnx_\pm})/\Gamma_\pm$
of~$(V_\pm,g^{V_\pm}_\ell)$ and a circle~$S^1_{\lnx_\pm}$ of length~$\lnx_\pm$.
Let us regard~$t_\pm$ as functions on~$M_\pm$.
The cross sections~$t_\pm^{-1}(s)$ for~$s\ge\ell$
are isometric to~$X=\Sigma\times T^2$
with~$T^2\cong(S^1_{\lnn_\pm}\times S^1_{\lnx_\pm})/\Gamma_\pm$
by the construction in Section~\ref{sec:xtcs},
but with different orientations.
Hence,
there is an orientation preserving isometry between the cylindrical
regions~$\bigl\{\,x\in M_\pm\bigm|t_\pm(x)\in[\ell,3\ell]\,\bigr\}$
that identifies~$t_+$ with~$4\ell-t_-$.
This allows us to glue~$M_+$ to~$M_-$ after chopping off the
ends~$\bigl\{\,x\in M_\pm\bigm|t_\pm(x)>3\ell\,\bigr\}$,
see Figure~\ref{fig:gluingmetric}.
The resulting manifold will be denoted~$(M_\ell,g_\ell)$,
and we refer to~$g_\ell$ as the \emph{gluing metric.\/}
Let~$t$ be the function on~$M$ that agrees with~$t_--2\ell$
and~$2\ell-t_+$ wherever those are defined.

By~\cite[Rem~\ref{eta-Rem2.2}]{CGN},
the resulting metric~$g_\ell$ has the following properties.

\begin{enumerate}
\item\label{2.g.1}
  The restriction of~$g_\ell$ to $\{\,x \in M_\ell \mid \pm t(x) \geq - \ell\,\}$
  is isometric to the region $t_\pm \leq 3\ell$ in the twisted product
  $(V_\pm \times S^1_{\xi_\pm})/\Gamma_\pm$.
  of~$(V_\pm ,g^{V_\pm}_\ell)$ and a
  circle~$S^1_{\xi_\pm}$ of length~$\xi_\pm$.
\item\label{2.g.2}
  The restriction of the metric~$g^{V_\pm}_\ell$ %
  to~$\{\, x \in V_\pm \mid t_\pm(x) \leq \ell - 1\,\}$
  agrees with the original asymptotically cylindrical Calabi-Yau
  metric~$g^{V_\pm}$.
\item\label{2.g.3}
  The manifold~$t^{-1}((-\ell,\ell)) \subset M_\ell$ is the Riemannian product
  of the K3 surface~$\Sigma$, the torus~$T^2$ and the interval~$(-\ell,\ell)$
  of length~$2\ell$.
\item\label{2.g.4}
  We
  have~$\norm{g_\ell|_{X\times(\pm[\ell,\ell+1])}-g^X\oplus dt^2}_{C^k}=O(e^{-c\ell})$
  for all~$k$.
\item\label{2.g.5}
  There exists~$c>0$ such that for all~$k$,
  we have %
  \begin{equation*}
    \norm{g_\ell-\bar g_\ell}_{C^k}=O\bigl(e^{-c\ell}\bigr)\;.
  \end{equation*}
\end{enumerate}
It follows from~\ref{2.g.2} and~\ref{2.g.3} that the metric~$g_\ell$
has local holonomy contained in~$G_2$
except over the set~$X\times([-\ell-1,-\ell]\cup[\ell,\ell+1])$,
where it is controlled by~\ref{2.g.1} and~\ref{2.g.4}.
The gluing region contains
pieces~$X\times\pm[\ell+1,2\ell]$ that have the geometry of a product
of a circle and an asymptotically cylindrical Calabi-Yau manifold.
If they are long enough, the $G_2$-structure on~$(M_\ell,g_\ell)$
has sufficiently small torsion.
The piece~$X\times[-\ell,\ell]$ is a straight cylinder.
If it is long enough,
we can control the kernel of the Dirac operator~$D_{M,\ell}$,
see section~\ref{ssec:dirac} below.
In~\cite[Section~5]{CGN}, we have seen that the lengths of
both pieces can be chosen on the same scale.

\thegluingfig

We now consider the two halves %
separately.
Let~$V_{\pm,\ell}=V_\pm\setminus((2\ell,\infty)\times S^1_{\lnn_\pm}\times\K)$
as before.
For~$a>0$, put
\begin{equation}\label{2.b.0}
  \widetilde M_{\pm,a}=V_{\pm,\ell}\times S^1_{a\zeta_\pm}
  \qquad\text{and}\qquad
  M_{\pm,a}=\widetilde M_{\pm,a}/\Gamma_\pm\;,
\end{equation}
where~$S^1_{a\zeta_\pm}$ denotes an exterior circle of length~$\xi_\pm=a\zeta_\pm$
and~$V_\pm$ carries the metric~$g^{V_\pm}_\ell$
introduced in~\ref{2.g.1} above.
Then the new metric~$g_{\pm,a}$ on~$M_{\pm,a}$ satisfies properties
analogous to~\ref{2.g.1}--\ref{2.g.4} above.
For~$a=\ar_\pm$ %
as in Proposition~\ref{Prop2w.1}~\ref{2w.1c},
we recover the restriction of the metric~$g_\ell$.
We consider the odd signature operator~$B_{M_{\pm,a}}$ for the new metric.

\subsubsection{}\label{ssec:spinor}
By Remark~\ref{rmk:spinstr} and property~\ref{2.g.1}
of the metrics~$g_{\pm,a}$,
we may describe the spinor bundle on~$M_{\pm,a}$ with Hermitian metric
and Clifford connection~$\nabla^{SM}=\nabla^{SM_{\pm,a}}$ as
\begin{equation}\label{eq:spinbdl}
  SM_{\pm,a}=S\widetilde M_{\pm,a}/\Gamma_\pm
  \qquad\text{with}\qquad
  S\widetilde M_{\pm,a}=p^*SV_{\pm,\ell}\;,
\end{equation}
where~$p\colon V_{\pm,\ell}\times S^1_{a\zeta_\pm}\to V_{\pm,\ell}$ is the
projection.
Let~$\del_{v_\pm}$ denote the unit tangent vector
to the exterior circle factor in the twisted Riemannian product~$M_{\pm,a}$
for all~$a>0$.
As in \mbox{\cite[Section~\ref{eta-Abs3.2}]{CGN}},
we construct a unit spinor~$s$ on~$M_{\pm,a}$
($s_{\ell,1}$ in the notation of~\cite{CGN}) such that
\begin{enumerate}
\item the spinor~$s$
  is pulled back from a $\Gamma_\pm$-invariant unit spinor on~$V_{\pm,\ell}$
  independent of~$a$,
\item its derivative~$\nabla^{SM}s$ is supported
  on~$X\times(\pm[\ell,\ell+1])$,
\item there exists~$c>0$
  such that~$\norm{\nabla^{SM}s}=O(e^{-c\ell})$,
\item we have~$\nabla^{SM}_{\del_{\scriptstyle v_\pm}}s=0$.
\end{enumerate}
In~\cite{CGN}, we also identify~$SM_{\pm,\ell}$ with the %
the spinor bundle for the metric~$\bar g_\ell|_{M_{\pm,\ell}}$
in such a way that~$s$ corresponds to the restriction of the parallel spinor
on the $G_2$-manifold~$(M,\bar g_\ell)$.

\subsubsection{}\label{ssec:dirac}
Let~$D'_{M_{\pm,a}}$ denote the geometric spin Dirac
operator of~$M_{\pm,a}$,
and let~$c_{v_\pm}$ denote Clifford multiplication by~$\del_{v_\pm}$.
Decomposing~$D'_{M_{\pm,a}}s$ using~\eqref{eq:spinbdl}
and the properties of~$g_{\pm,a}$ and~$s$ above,
we find functions~$f_\pm$, $h_\pm$ on~$V_\pm$
and a spinor~$r_\pm\in\Gamma(SM)$ that is pulled back from a
$\Gamma_\pm$-invariant spinor on~$V_{\pm,\ell}$,
all independent of~$a$, such that
\begin{equation}\label{Seq}
  D'_{M_{\pm,a}}s=f_\pm\cdot s+h_\pm\cdot c_{v_\pm}s+r_\pm\;,
\end{equation}
and such that~$r_\pm$
is perpendicular to~$s$ and~$c_{v_\pm}s$ everywhere.
As in~\cite[\eqref{eta-ModifiedDirac1} \& \eqref{eta-ModifiedDirac2}]{CGN}, put
\begin{equation}\label{DmodEq}
\begin{aligned}
  D_{M_{\pm,a}}
  =D'_{M_{\pm,a}}
  &-\<\punkt,s\>\bigl(f_\pm s+h_\pm c_{v_\pm}s+r_\pm\bigr)
  -\<\punkt,r_\pm\>\,s\\
  &-\<\punkt,c_{v_\pm}s\>\bigl(h_\pm s-f_\pm c_{v_\pm}s-c_{v_\pm}r_\pm\bigr)
  +\<\punkt,c_{v_\pm}r_\pm\>\,c_{v_\pm}s\;.
\end{aligned}
\end{equation}
Then~$\ker D_{M_{\pm,a}}$ contains the parallel unit spinors~$s$ and~$c_{v_\pm}s$
for all~$\ell$ and all~$a>0$.

In the special case~$a=\ar_\pm$,
the operators above combine to an operator~$D_{M,\ell}$ on~$M_\ell$.
\begin{enumerate}%
\item\label{2.g.9}
  On~$M_\ell\setminus\bigl(X\times([-\ell-1,-\ell]\cup[\ell,\ell+1])\bigr)$,
  the operator~$D_{M,\ell}$ agrees with the geometric spin Dirac operator
  of the gluing metric~$g_\ell$ described above.
\item\label{2.g.7}
  By~\cite[Prop~\ref{eta-prop:smallev}]{CGN},
  the kernel of~$D_{M,\ell}$ is spanned by the nowhere vanishing section~$s$.
\item\label{2.g.10}
  We have~$D_{M,\ell}|_{M_{\pm}}(c_{\anglex_\pm}s)=0$ by~\eqref{Seq}
  and~\eqref{DmodEq}, see~\cite[\eqref{eta-3.22}]{CGN}.
\item\label{2.g.8}
  By~\cite[Prop~\ref{eta-prop3.13}]{CGN}, there is a constant~$c>0$
  such that for~$\ell\gg 1$, we have
  \begin{equation*}
    \eta(D_{(M,\bar g_\ell)})=\eta(D_{M,\ell})+O\bigl(e^{-c\ell}\bigr)\;,
  \end{equation*}
  where~$D_{(M,\bar g_\ell)}$ is the geometric Dirac operator
  of the $G_2$-manifold~$(M,\bar g_\ell)$.
\end{enumerate}
In particular,
there is no spectral flow if we deform~$D_{(M,\bar g_\ell)}$ into~$D_{M,\ell}$.
From~\ref{ssec:gell}~\ref{2.g.5} and~\ref{ssec:dirac}~\ref{2.g.8},
we conclude that
\begin{equation}\label{2g.1}
  \bar\nu(M)
  =\lim_{\ell\to\infty}\bigl(3\eta(B_{M,\ell})-24\eta(D_{M,\ell})\bigr)\;.
\end{equation}
Note that the linear combination of \etainvts on the right hand
in principle differs from the extended \nuinvt of~$(M,\bar g_\ell)$
by a Mathai-Quillen
term and two differences of \etainvts.
These terms tend to~$0$ as~$\ell\to\infty$ by our construction
(In fact, it turns out that
we already get~$\bar\nu(M)=3\eta(B_{M,\ell})-24\eta(D_{M,\ell})$
for any sufficiently large~$\ell$, \cf Remark \ref{rmk:ell_dep}).

\subsection{The gluing formula}\label{2.g}
In the following, we cut~$(M,g_\ell)$ into two halves~$M_{\pm,\ar_\pm}$
with common boundary~$\Sigma\times T^2\times\{0\}$.
We identify~$\Lambda^\bullet T^*(\Sigma\times T^2)$
with the restriction of~$\Lambda^{\mathrm{ev}}T^*M$,
let the boundary operator~$A$ of the odd signature operator~$B$
act on~$\Omega^\bullet(X)$ as in~\cite[\eqref{eta-2.10}]{CGN}, and put
\begin{equation}\label{LBeq}
  L_{B_\pm}=\im\bigl(H^\bullet(M_\pm)\to H^\bullet(X)\bigr)
  \subset H^\bullet(X)\;,
\end{equation}
then~$L_{B_\pm}$ are Lagrangian subspaces of~$H^\bullet(X)\cong\ker A$.
As in~\cite[sec~\ref{eta-Sect4.1a}]{CGN},
let~$\eta(B_{M_{\pm,a}};L_{B_\pm})$ denote the \etainvt
of~$B_{M_{\pm,a}}$
with respect to APS boundary conditions modified by~$L_{B_\pm}$;
in particular, the forms in the domain of~$B_{M_{\pm,a}}$ orthogonally project to~$0$
on the Lagrangian in~$\Omega^\bullet(X)$ given as the direct sum
of~$L_{B_\pm}$ with the sum of all eigenspaces of~$A$
of eigenvalues of sign~$\pm$.

For the operator~$D_{M_{\pm,a}}$, we define similar boundary conditions
as in~\cite[\eqref{eta-3.23}]{CGN}.
We identify the spinor bundle of~$\Sigma\times T^2$ with the restriction
of the spinor bundle of~$M$ and write
\begin{equation*}
  D_{M_{\pm,a}}|_{X\times(-\ell,\ell)}
  =c_t\,\biggl(\frac\del{\del t}+A_{\pm,a}\biggr)\;,
\end{equation*}
where~$c_t$ denotes Clifford multiplication with~$\frac\del{\del t}$
and~$A_{\pm,a}$ now denotes the boundary operator of the modified
spin Dirac operator~$D_{M_{\pm,a}}$.
Then~$\ker A_{\pm,a}\cong H^{0,\bullet}(X)\cong\C^4$ is independent of~$a$.

Together with the $L^2$-metric on spinors, $c_t$ introduces a
symplectic structure on~$\ker A_{\pm,a}$.
Let~$s$ span~$\ker(D_{M,\ell})$ as in~\ref{ssec:dirac}~\ref{2.g.7} above,
then by~\ref{ssec:dirac}~\ref{2.g.10},
\begin{equation}\label{LDeq}
    L_{D_-}=\Span\{s,c_{v_-}s\} %
    \qquad\text{and}\qquad
    L_{D_+}=\Span\{s,c_{v_+}s\} %
\end{equation}
are subspaces of~$\ker A_{\pm,a}$.
As explained in~\cite[Section~\ref{eta-SpecSymm}]{CGN},
they are exactly the Lagrangian subspaces of $A_{\pm,a}$-harmonic spinors on~$X$
that extend to $D_{M_{\pm,a}}$-harmonic spinors on~$M_{\pm,a}$
for each~$a$.
Define~$\eta(D_{M_{\pm,a}};L_{D_\pm})$ as above.
In particular, the spinors in the domain of~$D_{M_{\pm,a}}$ project
orthogonally to~$0$ on the Lagrangian in~$\Gamma(SX)$ given as the direct sum
of~$L_{D_\pm}$ with the sum of all eigenspaces of~$A_{\pm,a}$ of eigenvalues
of sign~$\pm$.

We define
\begin{equation}
\label{eq:nu_half}
    \bar\nu(M_{\pm,a})=\lim_{\ell\to\infty}\bigl(3\eta(B_{M_{\pm,a}};L_{B_\pm})
    -24\eta(D_{M_{\pm,a}};L_{D_\pm})\bigr) .
\end{equation}

\begin{rmk}
\label{rmk:ell_dep}
While $\eta(B_{M_{\pm,a}};L_{B_\pm})$ and $\eta(D_{M_{\pm,a}};L_{D_\pm})$ can
depend on $\ell$, it turns out in Theorems \ref{Thm2x.1} and \ref{Thm2h.0} that
$3\eta(B_{M_{\pm,a}};L_{B_\pm}) -24\eta(D_{M_{\pm,a}};L_{D_\pm})$ does not.
\end{rmk}

We recall the gluing angle~$\thet$ from~\eqref{2.a.0}.
In Definition~\ref{def:angles}, we have introduced the
configuration angles~$\alpha_1^+$, $\alpha_2^+$, $\alpha_3^+$,
$\alpha_1^-$, \dots, $\alpha_{19}^-\in(-\pi,\pi]$.
For $\rho \in \R$, define
\begin{equation}
\label{eq:maslov}
\begin{aligned}
m_\rho(L;N_+,N_-)&=
    \Sign\rho\,\bigl(\#\bigl\{\,j
    \bigm|\alpha_j^-\in\{\pi-\abs\rho,\pi\}\,\bigr\}-1\bigr) \\
    &\qquad
    +2\Sign\rho\,\#\bigl\{\,j
    \bigm|\alpha_j^-\in(\pi-\abs\rho,\pi)\,\bigr\}\qquad\in\Z\;,
\end{aligned}
\end{equation}

\begin{thm}[{\cite[Theorem~\ref{eta-GluingThm}]{CGN}}]\label{Thm2a.1}
  Let $M$ be an extra-twisted connected sum with~$\thet\notin\pi\Z$, and
  put~$\rho=\pi-2\thet$.
  Then the extended \nuinvt of~$M$
  is given by
  \begin{equation}\label{2.a.2}
    \bar\nu(M,g)
    =\bar\nu(M_{+,\ar_+})+\bar\nu(M_{-,\ar_-})
    -72\frac\rho\pi+3m_\rho(L;N_+,N_-)\;.%
  \end{equation}
\end{thm}

Note that in the examples in~\cite{CGN},
we had~$\Gamma_\pm=\Z/k_\pm$ with~$k_\pm\in\{1,2\}$.
In these cases, one could find orientation reversing isometries of~$M_{\pm}$
that anticommute with~$B_{M_{\pm}}$,
$D_{M_{\pm}}$ and preserve the boundary conditions,
leading to~$\eta(B_{M_{\pm}};L_{B_\pm})=\eta(D_{M_{\pm}};L_{D_\pm})=0$.
Here, we want to deal with examples where this is no longer the case.
We have examples where~$\thet\notin\Q\pi$, so that~$\frac\rho\pi\notin\Q$.
In these cases at least one of the invariants~$\bar\nu(M_{\pm,\ar_\pm})$
above must be irrational.
In particular, they can no longer both vanish.

\begin{ex}\label{ex:run-angles}
  With the configuration angles of Example~\ref{ex:run-matching},
  we get~$\rho=\pi-2\arccos\frac1{\sqrt 3}>0$ and hence
  \begin{equation*}
    -72\frac\rho\pi+3m_\rho(L;N_+,N_-)
    =-72+\frac{144}{\pi}\,\arccos\frac1{\sqrt 3}-3\;.
  \end{equation*}
\end{ex}

\subsection{The adiabatic limit of \etainvts}
\label{2.b}
To compute~$\bar\nu(M_{\pm,\ar_\pm})$ in Theorem~\ref{Thm2a.1} above,
we consider the limits~$\bar\nu(M_{\pm,a})$ for~$a\to 0$ in
this subsection.
These limits differ from~$\bar\nu(M_{\pm,\ar_\pm})$,
and the difference is described as an integral over a variational term
in the next subsection.
A direct computation of this contribution is given in Section~\ref{2h},
completing the proof of Theorem~\ref{Thm:A}.
Two technical intermediate results have been postponed
to Section~\ref{sec:proofs} for better readability.

We still work on the manifolds~$M_{\pm,a}$,
which are twisted Riemannian products by property~\ref{ssec:gell}~\ref{2.g.1}
above.
We also still consider the modification~$D_{M_{\pm,a}}$
of the spin Dirac operator considered in~\eqref{DmodEq}.

We write
\begin{equation}\label{2.b.1}
  \bar\nu\bigl(M_\pm\bigr)
  =\bar\nu\bigl(M_{\pm,\ar_\pm}\bigr)
  =\lim_{a\to 0}\bar\nu\bigl(M_{\pm,a}\bigr)
  +\int_0^{\ar_\pm}\frac d{da}\,\bar\nu\bigl(M_{\pm,a}\bigr)\,da\;.
\end{equation}
We consider~$\of_\pm=V_{\pm,\ell}/\Gamma_\pm$ as an orbifold with boundary,
where the boundary itself is a manifold by assumption.
Let~$\Lambda\of_\pm$ denote its inertia orbifold.
The orbifold $\hat A$-form on~$\Lambda\of_\pm$
is defined in~\cite[(1.6)]{Gorbi}.
We will also need the orbifold $\hat L$-form; see~\cite[Cor~1.10]{Gorbi}.
Let~$\Aa$ denote the Bismut superconnection of the fibrewise spin Dirac
operator for the map~$p\colon M_{\pm,a}\to W_\pm$
with respect to the fibrewise trivial spin structure;
see Remark~\ref{rmk:spinstr}.
Let~$\eta_{\Lambda\of_\pm}(\Aa)\in\Omega^\bullet(\Lambda\of_\pm)$
denote the orbifold \etaform as in~\cite[Def~1.7]{Gorbi}.

To compute the adiabatic limit, we need to combine Dai's result
for manifolds with boundary in~\cite{Dai} with the version for Seifert
fibrations without boundary in~\cite{Gorbi}.
This in done in Theorem~\ref{Thm5.1},
which we have postponed
because its proof is technical and independent from the problem
at hand.
It implies that
\begin{equation}\label{2.b.2}
  \lim_{a\to 0}\bar\nu\bigl(M_{\pm,a}\bigr)
  =\int_{\Lambda\of_\pm\setminus\of_\pm}
  \Bigl(3\hat L_{\Lambda\of_\pm}\bigl(T\ofh_\pm,\nabla^{T\ofh_\pm}\bigr)
  -24\hat A_{\Lambda\of_\pm}\bigl(T\ofh_\pm,\nabla^{T\ofh_\pm}\bigr)\Bigr)
  \,2\eta_{\Lambda\of_\pm}(\Aa)\;.
\end{equation}
Because~$\of_\pm$ is even-dimensional,
there is no contribution from \etainvts on~$\of_\pm$.
Moreover, there are no very small eigenvalues in our situation.
We remark that the circle orbibundle~$M_{\pm,a}\to\of_\pm$
is flat by construction,
so the integral above localises at the orbifold singularities of~$\of_\pm$,
and there is no contribution from the principal stratum.
We are in a local product situation, so the orbifold \etaforms all
reduce to equivariant \etainvts.

The action of~$\Gamma=\Z/k\Z$ on~$V$ is faithful
because it is free on~$\del V$.
At each fixpoint~$p\in V^\gamma$ of~$\gamma\in\Gamma$,
the tangent space~$T_p V$ splits as a sum of complex eigenspaces
of the differential of~$\gamma$ with eigenvalues~$e^{i\alpha_\ell}$,
with~$\alpha_1$, $\alpha_2$, $\alpha_3\in\frac{2\pi}k\,\Z$.
Because~$\gamma$ preserves the holomorphic volume form,
the angles~$\alpha_\ell$ add up to a multiple of~$2\pi$.
Hence, the complex codimension of the fixpoint set has to be at least~$2$.
If the fixpoints are not isolated, then~$ V^\gamma\subset V$
is totally geodesic, and the eigenspaces locally form
bundles over~$V^\gamma$.
The tangent bundle~$T V^\gamma$ corresponds to~$\alpha_\ell=0$.
Let~$\nu_\gamma\to V^\gamma$ denote the normal bundle.

We assume that the coordinate~$v\in\R/\xi\Z$ on the exterior circle has
been chosen such that inserting~$\del_v$ into the $G_2$-form~$\phy$
gives the K\"ahler form on~$ V$; see~\cite[\eqref{eta-2.3}]{CGN}.
Then let~${\gamma\in\Gamma}$ be the generator that acts on the exterior circle
by sending~$v$ to~$v+\frac{\xi}k$.
We start by defining a  {\em generalised Dedekind sum\/}
as in~\cite{Gorbi}.
Note that it depends on the particular choice of generator~$\gamma$.

\begin{dfn}\label{Def2x.1}
  Let~$\gamma\in\Gamma\cong\Z/k\Z$ be a generator.
  For~$0<j<k$,
  let~$ V^{0,j}$ denote the set of isolated fixpoints
  of~$\gamma^j$,
  and for each~$p\in V^{0,j}$, let~$\alpha_{j,1}(p)$, $\alpha_{j,2}(p)$,
  $\alpha_{j,3}(p)$ denote the angles of the action of~$\gamma^j$
  on the 3-dimensional complex vector space~$T_p V$,
  chosen such that~$\alpha_{j,1}(p)+\alpha_{j,2}(p)+\alpha_{j,3}(p)\in4\pi\Z$.
  Then define
  \begin{equation*}
    D_\gamma( V)=
    \frac3{k}\sum_{j=1}^{k-1}\cot\frac{\pi j}{k}
    \sum_{p\in V^{0,j}}
    \frac{\cos\frac{\alpha_{j,1}(p)}2\cos\frac{\alpha_{j,2}(p)}2\cos\frac{\alpha_{j,3}(p)}2-1}
       {\sin\frac{\alpha_{j,1}(p)}2\sin\frac{\alpha_{j,2}(p)}2\sin\frac{\alpha_{j,3}(p)}2}\;,
  \end{equation*}
\end{dfn}

\begin{thm}\label{Thm2x.1}
  Let~$\gamma_\pm\in\Gamma_\pm\cong\Z/k_\pm\Z$ be the generator that
  acts on the exterior circle~$\R/\xi_\pm\Z$
  by sending~$\anglex_\pm$ to~$\anglex_\pm+\frac{\xi_\pm}k$.
  Define~$D_{\gamma_\pm}( V_\pm)$ as above, then
  \begin{equation*}
    \lim_{a\to0}\bar\nu(M_{\pm,a})=D_{\gamma_\pm}( V_\pm)\;.
  \end{equation*}
\end{thm}

In particular, non-isolated fixpoints do not contribute to
the adiabatic limit of the extended \nuinvt.
We have shown in~\cite{CGN} that~$\bar\nu(M_\pm)=0$ if~$k_\pm=1$ or~$k_\pm=2$.
This is consistent
because involutions of odd-dimensional Calabi-Yau manifolds cannot have isolated
fixpoints.

\begin{proof}
  Being a Calabi-Yau manifold, $V$ has a %
  spin structure with spinor bundle~$\Lambda^{0,\bullet}T^* V$.
The K\"ahler metric identifies~$\Lambda^{0,1}T^* V$ with~$T V$
with its natural complex structure.
Let~$\gamma\in\Gamma$.
Because~$ V^\gamma$ is at most one-dimensional,
we can split~$T_p V$ into one-dimensional eigenspaces
that are also invariant under the
curvature tensor~$F\in\Lambda^{1,1}\End(T_p V)$.
This allows us to decompose the action of~$\gamma e^{-\frac F{2\pi i}}$
on the spinor space~$\Lambda^{0,\bullet}T^* V |_{ V^\gamma}$ as
\begin{equation*}
  \gamma e^{-\frac F{2\pi i}}|_{\Lambda^{0,\bullet}T^*_p V}\cong
  \bigotimes_{j=1}^3
  \begin{pmatrix}
    1\\&e^{i\alpha_\ell}(1+\beta_\ell)
  \end{pmatrix}\;,
\end{equation*}
where~$\beta_\ell\in\Lambda^{1,1}T^* V^\gamma$
are real differential forms that represent the Chern roots
of the subbundle of~$T V|_{ V^\gamma}$
corresponding to the eigenvalue~$e^{i\alpha_\ell}$.
We assume that~$\alpha_1+\alpha_2+\alpha_3=0$.
Because~$ V$ is Ricci-flat,
we know that~$\beta_1+\beta_2+\beta_3=0$.
This allows us to twist each tensor factor above with a line~$L_\ell$
on which~$\gamma e^{-\frac F{2\pi i}}$
acts as~$e^{-i\frac{\alpha_\ell}2}\bigl(1-\frac{\beta_\ell}2\bigr)$,
and we get
\begin{equation}\label{2.x.2}
  \gamma e^{-\frac F{2\pi i}}|_{\Lambda^{0,\bullet}T^*_p V}\cong
  \bigotimes_{\ell=1}^3
  \begin{pmatrix}
    e^{-\frac{i\alpha_\ell}2}\Bigl(1-\frac{\beta_\ell}2\Bigr)\\
    &e^{\frac{i\alpha_\ell}2}\Bigl(1+\frac{\beta_\ell}2\Bigr)
  \end{pmatrix}\;.
\end{equation}

Finally, we note that the Seifert fibration~$M_\pm\to V_\pm/\Gamma_\pm$
is locally of product geometry.
Therefore, the equivariant \etaform~$\tilde\eta_{\gamma^j}(\Aa)$ reduces
to half the equivariant \etainvt.
If~$\gamma\in\Gamma$ denotes the preferred generator, then
\begin{equation}\label{2.x.3}
  \eta_{\gamma^j}(D_{S^1})=\eta_{\gamma^j}(B_{S^1})
  =-i\cot\frac{\pi j}k
  \quad\in\Omega^0\bigl( V\bigr)
\end{equation}
with respect to the preferred orientations.

\subsubsection*{Complex one-dimensional fixpoint sets}
Assume that~$C\subset V^{\gamma^j}$
is a connected component of the fixpoint set of~$\gamma^j$
with~$\dim_\C C=1$,
and with normal bundle~$\nu_C\to C$ in~$ V$.
Along~$C$, we have~$\alpha_2=-\alpha_1$ and~$\alpha_3=0$.
To compute the orbifold $\hat A$-class following~\cite[(1.6) \& (1.7)]{Gorbi}
and~\cite[sect.~6.4]{BGV},
we need
\begin{align*}
  \ch\bigl(\gamma^j,\Lambda^{0,\mathrm{ev}}\nu_C^*
  -\Lambda^{0,\mathrm{odd}}\nu_C^*\bigr)|_{(C,\gamma^j)}
  &=\prod_{\ell=1}^2\str
  \begin{pmatrix}
    e^{-\frac{i\alpha_\ell}2}\Bigl(1-\frac{\beta_\ell}2\Bigr)\\
    &e^{\frac{i\alpha_\ell}2}\Bigl(1+\frac{\beta_\ell}2\Bigr)
  \end{pmatrix}\\
  &=4\sin^2\frac{\alpha_1}2
  -2i(\beta_1-\beta_2)\sin\frac{\alpha_1}2\cos\frac{\alpha_1}2\;,
\end{align*}
which follows from~\eqref{2.x.2}.
For dimension reasons, $\hat A(TC)=1$, so~$\beta_3$
cannot contribute.
Then by~\eqref{2.x.3},
the whole contribution of~$(C,\gamma^j)\in\Lambda V$
to the untwisted \etainvt is
\begin{multline}\label{2.x.5}
  \frac{(-1)^{\rk_\C\nu_C}\hat A(TC)}
       {k\ch(\gamma^j,\Lambda^{0,\mathrm{ev}}\nu_C^*
         -\Lambda^{0,\mathrm{odd}}\nu_C^*)}[C]
       \cdot\eta_{\gamma^j}(D_{S^1})\\
       \begin{aligned}
    &=\frac1
    {4\sin^2\frac{\alpha_1}2
      -2i(\beta_1-\beta_2)\sin\frac{\alpha_1}2\cos\frac{\alpha_1}2}[C]
       \cdot\frac{\eta_{\gamma^j}(D_{S^1})}k\\
    &=\frac{i\cos\frac{\alpha_1}2}
         {8k\sin^3\frac{\alpha_1}2}\cdot(\beta_1-\beta_2)[C]
      \cdot\eta_{\gamma^j}(D_{S^1})\;.
       \end{aligned}
\end{multline}

For the signature \etainvt,
we have to compute the equivariant twist Chern character
following~\cite[Def~6.15]{BGV}.
The spinor bundle~$\Lambda^{0,\bullet}T_p^*C$
of~$T_pC\subset T_p V$ contributes only by its rank.
By~\eqref{2.x.2}, we have
\begin{align*}
  \ch\bigl(\gamma^j,\Lambda^{0,\bullet}T^*_p V\bigr)|_{(C,\gamma^j)}
  &=2\prod_{\ell=1}^2\tr
  \begin{pmatrix}
    e^{-\frac{i\alpha_\ell}2}\Bigl(1-\frac{\beta_\ell}2\Bigr)\\
    &e^{\frac{i\alpha_\ell}2}\Bigl(1+\frac{\beta_\ell}2\Bigr)
  \end{pmatrix}\\
  &=8\cos^2\frac{\alpha_1}2
  +4i(\beta_1-\beta_2)\cos\frac{\alpha_1}2\sin\frac{\alpha_1}2\;.
\end{align*}
By~\eqref{2.x.3} and the above,
the whole contribution of~$ V^{\gamma^j}$
to the signature \etainvt is
\begin{multline}\label{2.x.6}
  \frac{(-1)^{\rk_\C\nu_C}\hat A(TC)\ch\bigl(\gamma^j,\Lambda^{0,\bullet}T^*_p V\bigr)}
       {k\ch(\gamma^j,\Lambda^{0,\mathrm{ev}}\nu_C^*
         -\Lambda^{0,\mathrm{odd}}\nu_C^*)}[C]
       \cdot\eta_{\gamma^j}(B_{S^1})\\
  \begin{aligned}
    &=
    \frac{8\cos^2\frac{\alpha_1}2+4i(\beta_1-\beta_2)\cos\frac{\alpha_1}2\sin\frac{\alpha_1}2}
         {4\sin^2\frac{\alpha_1}2-2i(\beta_1-\beta_2)\cos\frac{\alpha_1}2\sin\frac{\alpha_1}2}[C]
         \cdot\eta_{\gamma^j}(D_{S^1})\\
   &=\frac{i\cos\frac{\alpha_1}2}
         {k\sin^3\frac{\alpha_1}2}\cdot(\beta_1-\beta_2)[C]\cdot\eta_{\gamma^j}(D_{S^1})\;.
  \end{aligned}
\end{multline}
From~\eqref{2.b.2}, \eqref{2.x.5} and~\eqref{2.x.6},
we see that~$( V^{\gamma^j},\gamma^j)$ does not contribute
to~$\lim_{a\to0}\bar\nu(M_a)$.

\subsubsection*{Isolated fixpoints}
At an isolated fixpoint~$p$ of~$\gamma=\gamma^j$,
we have~$\nu_p=T_p V$.
The action of~$\gamma$ is determined by
three nonzero angles~$\alpha_\ell=\alpha_{j,\ell}(p)$
for~$\ell=1$, $2$, $3$ that add up to~$0$.
If necessary, we add a multiple of~$2\pi$ to one of the angles.

The contribution to the orbifold $\hat A$-form is the number
\begin{equation*}
  \ch\bigl(\gamma,\Lambda^{0,\mathrm{ev}}T^*_p V
  -\Lambda^{0,\mathrm{odd}}T^*_p V\bigr)|_{(p,\gamma)}
  =\prod_{\ell=1}^3\str
  \begin{pmatrix}
    e^{-\frac{i\alpha_\ell}2}\\&e^{\frac{i\alpha_\ell}2}
  \end{pmatrix}
  =8i\,\sin\frac{\alpha_1}2\sin\frac{\alpha_2}2\sin\frac{\alpha_3}2\;.
\end{equation*}
By~\eqref{2.x.3}, the contribution to the untwisted \etainvt is
\begin{equation}\label{2.x.8}
  \frac{(-1)^{\rk_\C T_p V}}
       {k\ch(\gamma,\Lambda^{0,\mathrm{ev}}T^*_p V
         -\Lambda^{0,\mathrm{odd}}T^*_p V)}[p]
       \cdot\eta_{\gamma^j}(D_{S^1})
  =\frac{\cot\frac{\pi j}k}
       {8k\sin\frac{\alpha_1}2\sin\frac{\alpha_2}2\sin\frac{\alpha_3}2}\;.
\end{equation}

For the signature \etainvt,
we multiply the above with the equivariant Chern character
\begin{equation*}
  \ch\bigl(\gamma,\Lambda^{0,\bullet}T^*_p V\bigr)
  =\prod_{\ell=1}^3\tr
  \begin{pmatrix}
    e^{-\frac{i\alpha_\ell}2}\\&e^{\frac{i\alpha_\ell}2}
  \end{pmatrix}
  =8\cos\frac{\alpha_1}2\cos\frac{\alpha_2}2\cos\frac{\alpha_3}2
\end{equation*}
and obtain the contribution to the signature \etainvt
\begin{equation}\label{2.x.9}
  \frac{(-1)^{\rk_\C T_p V}
       \ch\bigl(\gamma,\Lambda^{0,\bullet}T^*_p V\bigr)}
       {k\ch(\gamma,\Lambda^{0,\mathrm{ev}}T^*_p V
         -\Lambda^{0,\mathrm{odd}}T^*_p V)}[p]
       \cdot\eta_{\gamma^j}(B_{S^1})
  =\frac1k\,\cot\frac{\alpha_1}2\cot\frac{\alpha_2}2\cot\frac{\alpha_3}2
       \cot\frac{\pi j}k\;.
\end{equation}
From~\eqref{2.b.2}, \eqref{2.x.8} and~\eqref{2.x.9},
we obtain the Theorem.
\end{proof}

\begin{rmk}\label{Rem2x.1}
  When considering examples,
  it is more convenient to fix the generator~$\tau$
  that acts on the interior circle~$S^1_\zeta\cong\R/\zeta\Z$
  by~$u\mapsto u+\frac\zeta k$.
  If a unit~$\eps\in\Z/k$ is chosen as in Section~\ref{2.w}
  then the generator chosen above satisfies~$\gamma=\tau^\eps$,
  and the contribution of the isolated fixpoints to the extended
  \nuinvt is given by~$D_{\tau^\eps}(V)$.

  Now assume that~$Z$ is a building block with a $\Gamma$-action
  that fixes an anticanonical divisor~$\K$ of~$Z$ pointwise,
  and that~$V\cong Z\setminus\K$.
  The orientation convention of~\cite[equation~(3.2)]{CHNP2}
  says that the complex structure rotates the outward cylindrical
  direction into the positive direction of the interior circle.
  If we identify the asymptotic cylinder with the normal bundle to~$\K$ in~$Z$,
  then the outward cylindrical direction becomes the inward normal direction.
  This means that the interior circle through~$v\in\nu_\K$
  is oriented by~$-iv\in T_v\nu_\K$.
  Hence~$\tau\in\Gamma$ should act on~$\nu_\K$ by~$e^{-\frac{2\pi i}k}$.
\end{rmk}

\begin{ex}\label{ex:run-fixpoint}
  In Example~\ref{ex:five},
  we describe a building block with $\Z/5$-symmetry.
  It has one isolated fixpoint~$p$.
  Let~$\zeta_5=e^{\frac{2\pi i}5}$,
  then there is a generator~$\tau$ of~$\Z/5$
  that acts on~$T_pZ$
  as~$\operatorname{diag}(\zeta_5,\zeta_5,\zeta_5^{-2})$
  and on~$\nu_\K$ by~$\zeta_5^{-1}=e^{-\frac{2\pi i}5}$.
  Hence, the generator~$\gamma$ in Theorem~\ref{Thm2x.1}
  corresponds to~$\tau^\eps$,
  for~$\eps\not\equiv 0$ mod~$5$.
  Then~$\gamma$ acts on~$T_pZ$
  as~$\operatorname{diag}(\zeta^\eps_5,\zeta^\eps_5,\zeta_5^{-2\eps})$, so
  \begin{equation*}
    \alpha_{j,1}(p)=\alpha_{j,2}(p)=\frac{2\eps\pi}5\qquad\text{and}\qquad
    \alpha_{j,3}(p)=-\frac{4\eps\pi}5\;.
  \end{equation*}
  If we represent~$\eps\in\Z/5\setminus\{0\}$
  by an element of~$\{-2,-1,1,2\}$, we get
  \begin{equation*}
    D_\gamma(V)=D_{\tau^\eps}(V)=
    \lim_{r\to0}\bar\nu\bigl((V\times S^1_r)/\Gamma\bigr)
    =\frac{24}{5\eps}\;.
  \end{equation*}
  We can use this block as~$Z_-$ in Example~\ref{ex:run-gluing},
  with~$\eps_-=-1$ and hence~$D_{\gamma_-}(V^-)=-\frac{24}5$.
\end{ex}

\subsection{The variation of \etainvts}
\label{2.c}
In this section,
we apply a variation formula for \etainvts on manifolds with boundary
by Dai and Freed~\cite[Theorem~1.9]{DF}.
Similar formulas in the case where the boundary operator is invertible
have been established
by Cheeger \cite[Section~8]{Ch} and
Bismut and Cheeger~\cite[Theorem~6.36]{BCh4}.
Dai and Freed actually interpret the reduced \etainvt in~$\R/\Z$
with respect to a certain class of possible boundary conditions as a
section of the dual of the determinant line bundle
of the fibrewise boundary operators,
equipped with the Quillen metric and the Bismut-Freed connection.
Because we have fixed the boundary condition in Section~\ref{2.g},
we recover the reduced \etainvt as an $\R/\Z$-valued function.

We now consider the family~$\mathcal M_\pm=M_\pm\times(0,\infty)\to(0,\infty)$
with fibre~$M_{\pm,a}$ over~$a\in(0,\infty)$.
Recall that the $M_{\pm, a}$ implicitly depends on a choice of
parameter $\ell$; this parameter is fixed throughout this subsection.
We choose the trivial connection~$T^H\mathcal M_\pm\subset T\mathcal M_\pm$,
and the fibrewise metric is induced
from the metric~$g^{V_\pm}_\ell\oplus a^2\lnn_\pm^2g^{S^1}$ on~$\widetilde M_\pm$.
Using these data, Bismut and Freed~\cite[(1.7)]{BF} construct
a connection~$\tilde\nabla^u$
on the infinite-dimensional vector
bundle~$\Omega^\bullet(\mathcal M_\pm/(0,\infty))\to(0,\infty)$
of fibrewise exterior differential forms
that is unitary with respect to the fibrewise $L^2$-metrics.

We similarly obtain a unitary connection
on~$\Omega^\bullet(\del\mathcal M_\pm/(0,\infty))\to(0,\infty)$,
where~$\del\mathcal M_\pm$ denotes the collection of fibrewise boundaries.
In our situation, it is not hard to see that the subbundle
of fibrewise harmonic forms and its subbundle representing
the subspaces~$L_{B_\pm}$ of~\eqref{LBeq}
in each fibre are parallel with respect to~$\tilde\nabla^u$.
Dai and Freed regard~$L_{B_\pm}$ as graphs of isometries~$H^+(X)\to H^-(X)$
whose determinants define a section of unit length of the determinant line
bundle~$\det H^\bullet(X)=\Lambda^{\max}H^+(X)^*\otimes\Lambda^{\max}H^-(X)$.
This section is again parallel with respect to the connection
induced by~$\tilde\nabla^u$ on~$\det H^\bullet(X)$.

The variational formula is typically phrased in terms of the Quillen metric and
the Bismut-Freed connection on the determinant line bundle
over~$(0,\infty)$ (which preserves the Quillen metric).
However, if the kernels of the boundary operators form a bundle
over the base, as in the case at hand,
it is easier to work with the $L^2$-metric above.
A simple fibrewise rescaling of the determinant line bundle transforms
one metric into the other, as in~\cite[Prop.~2.15]{DF}.
It is shown in~\cite[(3.8)]{DF} that the Bismut-Freed connection becomes
a unitary connection with respect to the $L^2$-metric.
It is given by
\begin{equation}\label{eq:EtaForm}
  \tilde\nabla^u-2\pi i\,\tilde\eta(\mathbb B)\;,\qquad\text{where}\qquad
  \tilde\eta(\mathbb B)
  =-\frac1{4\pi i}\int_0^\infty\str\Bigl(A_{X}\,\bigl[\tilde\nabla^u,A_{X}\bigr]
  \,e^{-tA_{X}^2}\Bigr)\,dt
\end{equation}
is the \etaform of the family of boundary operators~$A_X$,
and~$\mathbb B_t=\sqrt t\,A_X+\tilde\nabla^u$ is the corresponding
Bismut superconnection.
Note that we do not need to specify the degree~$1$ component
of~$\tilde\eta(\mathbb B)$ explicitly because our base space
is one-dimensional here and the degree~$0$ component vanishes.

The situation for the spin Dirac operators is completely analogous.
However, we need to check that there is no spectral flow for~$a\in(0,\ar_\pm)$
under our boundary conditions.

\begin{prop}\label{Prop:SpecFlow}
  If~$\ell$ is sufficiently large and~$a>0$, then
  the operator~$D_{M_{\pm,a}}$ has trivial kernel
  under APS-boundary conditions modified by the Lagrangian subspace
  $L_{D_\pm}\subset\ker A_{\pm,a}$.
\end{prop}

\begin{proof}
  Let~$s$ denote the $D_{M_{\pm,a}}$-harmonic spinor
  constructed in~\ref{ssec:spinor}.
  Recall that~$\ker(A_{\pm,a})$ is spanned by the restrictions
  of~$s$, $c_{v_\pm}s$, $c_ts$, and~$c_tc_{v_\pm}s$.
  If~$u\in\ker D_{M_{\pm,a}}$,
  the divergence theorem implies that
  \begin{equation*}
    \<u,c_ts\>_{L^2(\del M_{\pm,a})}
    =\<u,D'_{M_{\pm,a}}s\>_{L^2(M_{\pm,a})}
    -\<D'_{M_{\pm,a}}u,s\>_{L^2(M_{\pm,a})}\;,
  \end{equation*}
  where~$D'_{M_{\pm,a}}$ denotes the geometric spin Dirac
  operator on~$M_{\pm,a}$.
  Because~$D'_{M_{\pm,a}}-D_{M_{\pm,a}}$ is self-adjoint
  by~\eqref{DmodEq}, the right hand side vanishes,
  and the restriction of~$u$ to the boundary is $L^2$-perpendicular to~$c_ts$.
  For the same reason, using~\ref{ssec:dirac}~\ref{2.g.10},
  it is also $L^2$-perpendicular to~$c_tc_{v_\pm}s$.
  
  If~$u$ moreover satisfies the modified APS boundary of Section~\ref{2.g},
  then~$u$ is also $L^2$-perpendicular to~$s$ and~$c_{v_\pm}s$ at the boundary,
  and hence to the whole~$\ker(A_{\pm,a})$.
  The claim now follows
  from~\cite[Proposition~\ref{eta-prop:nonres}]{CGN}.
\end{proof}

Hence, fixing APS boundary conditions modified by
the Lagrangian of~\eqref{LBeq}, \eqref{LDeq} as before,
the variational formulas in the version of~\cite[Theorem~3.3]{DF}
in our situation read
\begin{equation}\label{2.c.1}
  \begin{aligned}
    d\eta(B_{M_{\pm,a}})
    &=\int_{\mathcal M_\pm/(0,\infty)}2\hat L\bigl(\nabla^{T(\mathcal M_\pm/(0,\infty))}\bigr)
    -2\tilde\eta(\mathbb B)&&\in\Omega^1((0,\infty))\;,\\
    d\eta(D_{M_{\pm,a}})
    &=\int_{\mathcal M_\pm/(0,\infty)}2\hat A\bigl(\nabla^{T(\mathcal M_\pm/(0,\infty))}\bigr)
    -2\tilde\eta(\mathbb D)&&\in\Omega^1((0,\infty))\;,\\
  \end{aligned}
\end{equation}
where~$\int_{\mathcal M_\pm/(0,\infty)}$ denotes integration along the fibres.
The first term is the usual local variation formula for \etainvts
on closed manifolds.
The second term is the boundary contribution.
In the second line,
$\mathbb D$ is the superconnection on the bundle of fibrewise spinors
associated to the
boundary operators~$C_{X_\pm,a}$ corresponding to~$D_{M_{\pm,a}}$.

\begin{prop}\label{Prop2c.1}
  The local variation terms vanish, that is
  \begin{equation*}
    \int_{\mathcal M_\pm/(0,\infty)}\hat L\bigl(\nabla^{T(\mathcal M_\pm/(0,\infty))}\bigr)
    =\int_{\mathcal M_\pm/(0,\infty)}\hat A\bigl(\nabla^{T(\mathcal M_\pm/(0,\infty))}\bigr)
    =0\;.
  \end{equation*}
\end{prop}

\begin{proof}
  We split the vertical tangent bundle
  \begin{equation*}
    T(\mathcal M_\pm/(0,\infty))\cong p_{V_\pm}^*TV_\pm\oplus\underline{\R}\;,
  \end{equation*}
  where~$p_{V_\pm}\colon \mathcal M_\pm\to V_\pm$ denotes obvious projection.
  This splitting is parallel with respect to the Bismut connection
  on the vertical tangent bundle.
  Because the metric on~$V_\pm$ is unchanged,
  the Bismut connection on~$p_{V_\pm}^*TV_\pm$ is pulled back from~$V_\pm$.
  The connection on~$\underline{\R}$ is Euclidean and therefore flat.
  We conclude that
  \begin{equation*}
    \hat A\bigl(\nabla^{T\mathcal M_\pm/(0,\infty)}\bigr)
    =\hat A\bigl(\nabla^{p_{V_\pm}^*TV_\pm}\bigr)
    \cdot\hat A\bigl(\nabla^{\underline{\R}}\bigr)
    =p_{V_\pm}^*\hat A\bigl(\nabla^{TV_\pm}\bigr)\;.
  \end{equation*}
  Because this expression is of horizontal degree~$0$ and the fibres are
  odd-dimensional, the integral in the proposition vanishes.
  The same holds for the $\hat L$-form integral above.
\end{proof}

We now consider the \etaforms~$\eta(\mathbb B)$ and~$\eta(\mathbb D)$
in~\eqref{2.c.1}.
We write the family~$\del\mathcal M_\pm$ as a product~$\Sigma_\pm\times E_\pm$,
where~$E_\pm\to(0,\infty)$ denotes the family of
tori~$E_{\pm,a}=(S^1_{\lnn_\pm}\times S^1_{a\lnn_\pm})/\Gamma_\pm$
for~$a\in(0,\infty)$.
Let~$\Aa$ denote the superconnection associated to the fibrewise spin
Dirac operator for this family,
equipped with the trivial spin structure.

\begin{prop}\label{Prop2c.2}
  The variation of~$\bar\nu(M_{\pm,a})$ is given by
  \begin{equation}\label{2.c.3}
    d\bar\nu(\mathcal M_\pm/(0,\infty))
    =288\tilde\eta(\Aa)\;.
  \end{equation}
\end{prop}

\begin{proof}
  By the definition of~$\bar\nu(M_{\pm,a})$ in Theorem~\ref{Thm2a.1},
  equation~\eqref{2.c.1} and Proposition~\ref{Prop2c.1},
  we are left with
  \begin{equation*}
    d\bar\nu(\mathcal M_\pm/(0,\infty))
    =48\tilde\eta(\mathbb D)-6\tilde\eta(\mathbb B)\;.
  \end{equation*}
  Let~$B_\Sigma$ and~$D_\Sigma$ denote the signature operator and the untwisted
  Dirac operator on the K3 surface~$\Sigma$.
  Then~$\ind(B_\Sigma)=-16$ and~$\ind(D_\Sigma)=2$.
  The proposition follows because
  \begin{subequations}
    \label{eq:EtaProd}
    \begin{align}
      \tilde\eta(\mathbb B)&=2\ind(B_\Sigma)\,\tilde\eta(\Aa)\;,
      \label{eq:EtaProdB}\\
      \tilde\eta(\mathbb D)&=\ind(D_\Sigma)\,\tilde\eta(\Aa)\;.
      \label{eq:EtaProdD}
    \end{align}
  \end{subequations}
  These equations will be proved in Section~\ref{A3}.
\end{proof}

\subsection{A direct computation of the \etaform
  integral}\label{2h}
We rewrite the \etaform integral directly in terms of the eigenvalues
of the Dirac operator on the family of flat tori over~$\Hh$.
Bismut and Cheeger did similar computations in~\cite{BChTorus}.
In Section~\ref{A2v}, we exhibit another way to compute the
contribution from the variational formula to the \nuinvt
in terms of logarithms of Dedekind \etafuncs.

Because \etaforms are invariant under rescaling,
we may assume that~$\lnn=1$ and
consider a family of tori~$(S^1\times S^1_a)/\Gamma$ for~$a\in(0,\infty)$
as in Proposition~\ref{Prop2c.2}.
With~$\eps$ relatively prime to~$k$ as in Section~\ref{2.w},
we have~$(S^1\times S^1_a)/\Gamma\cong\R^2/\Lambda_a$,
where
\begin{equation*}
  \Lambda_a=\binom10\Z\oplus\binom\eps a\,\frac1k\Z\quad\subset\quad\R^2\;.
\end{equation*}
Let us denote the total space of this family by~$\ft$ and the fibres by~$Z$.
Consider a flat connection on~$\underline\R^2\to(0,\infty)$ given by
\begin{equation*}
  \nabla=d-\begin{pmatrix}0&0\\0&1\end{pmatrix}\,\frac{da}a\;,
\end{equation*}
then~$\Lambda$ is parallel with respect to~$\nabla$.
This connection induces a splitting~$T\ft=TZ\oplus T^H\ft$.
A horizontal lift of~$V=\frac\del{\del a}$
at a point~$(x,y,a)\in\R^2\times(0,\infty)$ is given as
\begin{equation*}
  \bar V_{(x,y,a)}=\frac ya\,\frac\del{\del y}+\frac\del{\del a}\;.
\end{equation*}

We equip~$\underline\R^2\to(0,\infty)$ and~$\ft$
with the fibrewise metric~$g^{TZ}$ induced from the standard metric on~$\R^2$.
The Levi-Civita connection on~$\ft$ induces a Euclidean connection~$\nabla^{TZ}$
on~$TZ\cong\pi^*\underline{\R^2}\to \ft$ that coincides with the pullback
of the trivial connection~$d$.
The mean curvature of the fibres is given as
\begin{equation*}
  h=-\frac12\,\tr\bigl((g^{TZ})^{-1}\,\logeta_{\bar V}g^{TZ}\bigr)\,da
  =-\frac{da}a\;.
\end{equation*}

We consider the fibrewise product spin structure,
so~$S^+\cong S^-\cong\underline{\C}\to \ft$.
Let~$c_1$, $c_2$ denote
Clifford multiplication with the standard orthonormal basis vectors~$e_1$,
$e_2$ on~$S^+\oplus S^-$,
then the complex Clifford volume element~$ic_1c_2$ acts by~$\pm 1$
on~$S^\pm$.
The Levi-Civita connection induces the trivial connection on~$S^\pm$.
For the Bismut superconnection,
we have to consider a connection on~$\pi_*S^\pm$ of the form
\begin{equation*}
  \tilde\nabla^u_{\frac\del{\del a}}f
  =\biggl(\nabla^{S^\pm}-\frac h2\biggr)_{\bar V}f
  =\biggl(\frac ya\,\frac\del{\del y}+\frac\del{\del a}+\frac1{2a}\biggr)f
\end{equation*}
under the natural
identification~$\Gamma(\pi_*S^\pm)\cong\Gamma(S^\pm)$.
Then the Bismut superconnection is given by
\begin{equation}\label{eq:bismutsc}
  \Aa_t
  =\sqrt t\,A_{T^2}+\tilde\nabla^u
  =\sqrt t\,\biggl(c_1\,\frac\del{\del x}+c_2\,\frac\del{\del y}\biggr)
  +da\,\biggl(\frac ya\,\frac\del{\del y}+\frac\del{\del a}+\frac1{2a}\biggr)\;.
\end{equation}
Starting from~\eqref{eq:EtaForm},
we compute
\begin{align}
  \begin{split}\label{2h.2}
    \tilde\eta(\Aa)
    &=-\frac1{4\pi}\,\frac{da}a\,\int_0^\infty
    \tr_{\pi_*S}\biggl(\frac{\del^2}{\del x\,\del y}
    \,e^{t\bigl(\frac{\del^2}{\del x^2}+\frac{\del^2}{\del y^2}\bigr)}\biggr)\,dt\;.
  \end{split}
\end{align}
We have used the definition~$\str(\punkt)=\tr(ic_1c_2\punkt)$
of the supertrace.
Also $\str(1)=\tr(ic_1c_2)=0$.

With respect to the standard Euclidean metric, the
lattice dual to~$\Lambda_a$ is given by
\begin{equation*}
  \Lambda_a^*=\bigl\{\,\mu\in\C\bigm|\<\lambda,\mu\>\in\Z\text{ for all }
  \lambda\in\Lambda\,\bigr\}
  =\biggl\{\,\binom n{m/a}\biggm|\eps n+m\equiv 0\Pmod k\,\Bigr\}\;.
\end{equation*}
For~$m$, $n$ as above,
we consider sections
\begin{equation*}
  \phy^\pm_{m,n}(x,y,a)=\frac1{\sqrt a}\,e^{2\pi i(nx+my/a)}\in\Gamma(S^\pm)
  \cong C^\infty(\ft;\C)
\end{equation*}
of $L^2$-norm~$1$.
They are parallel under~$\tilde\nabla^u$, and they are
eigensections of the fibrewise Laplacian
for the eigenvalue~$4\pi^2\,\bigl(n^2+\frac{m^2}{a^2}\bigr)$.
Note that each admissible pair~$(m,n)$ appears twice (once for~$S^+$
and once for~$S^-$), hence~\eqref{2h.2} becomes
\begin{align*}
  \tilde\eta(\Aa)
  &=-\frac1{2\pi}\,\frac{da}a\,\int_0^\infty\sum_{m+\eps n\equiv 0\Pmod k}
  \biggl(-4\pi^2\,\frac{mn}a\biggr)\,
  e^{-4\pi^2 t\,\bigl(n^2+\tfrac{m^2}{a^2}\bigr)}\,dt\\
  &=\frac{da}{2\pi}\,\int_0^\infty\sum_{m+\eps n\equiv 0\Pmod k}
  mn\,
  e^{-t\,(m^2+a^2n^2)}\,dt\;.
\end{align*}

In the definition below, we substitute~$-m$ for~$m$.

\begin{dfn}\label{Def2h.0}
  For each~$\eps$ relatively prime to~$k$,
  we define a function~$F_{k,\eps}\colon(0,\infty)\to\R$ by
  \begin{equation}\label{2v.1}
    F_{k,\eps}(s)
    =\int_0^\infty\int_0^s\sum_{m\equiv\eps n \Pmod k}mn\,e^{-t(m^2+n^2a^2)}\,da\,dt\;.
  \end{equation}
\end{dfn}
  
\begin{prop}\label{Prop2v.2}
  Consider the family~$\ft\to(0,\infty)$ above.
  Then
  \begin{equation*}
    \int_{[0,s]}\tilde\eta(\Aa)
    =-\frac1{2\pi}\,F_{k,\eps}(s)\;.
    \hfill \qed
  \end{equation*}
\end{prop}

\begin{thm}\label{Thm2h.0}
  The variation of~$\bar\nu(M_{\pm,a})$ is given by
  \begin{equation*}
    \bar\nu\bigl(M_{\pm,\ar_\pm}\bigr)-\lim_{a\to 0}\bar\nu\bigl(M_{\pm,a}\bigr)
    =-\frac{144}\pi\,F_{k_\pm,\eps_\pm}(\ar_\pm)\;.
  \end{equation*}
\end{thm}

\begin{proof}
  This follows from Propositions~\ref{Prop2c.2} and~\ref{Prop2v.2}.
\end{proof}

We can now give a formula for the extended \nuinvt.

\begin{thm}\label{Thm2h.1}
  The extended \nuinvt of an extra-twisted connected sum is given as
  \begin{multline}\label{2h.3}
      \bar\nu(M)
      =D_{\gamma_+}(V_+)+D_{\gamma_-}(V_-)\\
      -\frac{144}\pi\,\bigl(F_{k_+,\eps_+}(\ar_+)+F_{k_-,\eps_-}(\ar_-)\bigr)
      -72\frac\rho\pi+3m_\rho(L;N_+,N_-)\;.
  \end{multline}
\end{thm}

\begin{proof}
  This follows from Theorems~\ref{Thm2a.1}, \ref{Thm2x.1} and~\ref{Thm2h.0}.
\end{proof}

\begin{proof}[Proof of Theorem~\ref{Thm:A}]
  Combine Theorem~\ref{Thm2h.1} with Proposition~\ref{Prop1}.
\end{proof}

\begin{rmk}\label{rmk:atiyah}
  Using ideas and results of Atiyah~\cite{Atiyah}, Bismut and Freed~\cite{BF},
  and Ray and Singer~\cite{RS},
  we can motivate the appearance of the Dedekind \etafunc.
  We consider the universal family~$p\colon E\to\Hh$ of flat tori over the upper
  half plane that we will describe in more detail in Section~\ref{sec:hyp}.
  There exists a K\"ahler structure on~$E$
  whose restriction to each fibre~$p^{-1}(\tau)$ induces the flat
  Riemannian metric of volume~$1$ with the conformal structure
  induced by~$\tau\in\Hh$.
  The fibrewise canonical bundle of~$p$ is holomorphically trivial,
  so we may regard the bundle of fibrewise antiholomorphic forms
  as a model for the fibrewise spinor bundle on~$E$.

  Following~\cite{BF},
  the \etaform~$\tilde\eta(\Aa)$ describes a natural
  connection on the determinant line bundle of the fibrewise Dirac operator.
  Atiyah explains that this connection agrees with the Chern connection
  on the determinant line bundle with respect to the Quillen metric.
  Using results of Ray and Singer~\cite[Theorem~4.1]{RS},
  he shows that the determinant line bundle admits a holomorphic section
  whose norm can be written in terms of the Dedekind \etafunc,
  see the discussion before~\cite[(5.19)]{Atiyah}.
  This implies that the \etaform itself can be described
  by the logarithmic derivative of the Dedekind \etafunc,
  which is just a multiple of the not quite modular Eisenstein series of
  weight 2 on $SL(2,\Z)$.
\end{rmk}

\begin{ex}\label{ex:run-zagier}
  We consider the gluing data from Examples~\ref{ex:run-gluing},
  then~$\ar_-=\sqrt 2$, $\ar_+=5\sqrt 2$ and~$\thet=\arccos\frac1{\sqrt 3}$.
  From Theorem~\ref{Thm:A} and Examples~\ref{ex:run-angles},
  \ref{ex:run-fixpoint}, we get
  \begin{multline*}
    \bar\nu(M)
    =-\frac{24}5+\frac{144}\pi\biggl(\arccos\frac1{\sqrt 3}-\frac12\biggr)-3\\
    -\frac{144}\pi\Biggl(2\Im\logeta\biggl(\frac{\sqrt{-2}-10}{30}\biggr)
	+\frac{\pi}{18}+2\Im\logeta\biggl(\frac{\sqrt{-2}+2}{10}\biggr)
	-\frac{\pi}{30}\Biggr)\;.
  \end{multline*}
  The functional equation~\eqref{Ltransf} for~$\logeta$
  allows us to conclude that
  \begin{multline*}
    2\Im\logeta\biggl(\frac{\sqrt{-2}-10}{30}\biggr)
	+\frac{\pi}{18}+2\Im\logeta\biggl(\frac{\sqrt{-2}+2}{10}\biggr)
	-\frac{\pi}{30}
        +\frac12-\arccos\frac1{\sqrt 3}\\
    =\frac\pi 6\,\biggl(\frac1{30}+\frac 1{10}-12\,\DS(3,10)\biggr)\;,
  \end{multline*}
  see Proposition~\ref{Prop3}.
  Because~$3^2\equiv-1$ mod~$10$, we have~$\DS(3,10)=0$, and hence
  we confirm entry~228 of Table~\ref{table:matchings} by computing %
  \begin{equation*}
    \bar\nu(M)
    =-\frac{24}5-3-24\biggl(\frac1{30}+\frac 1{10}\biggr)=-11\;.
  \end{equation*}
  Theorem~\ref{Thm:B} would of course give the same result.
  In Example~\ref{ex:run-hyp},
  we explain how to regroup the terms in formula~\eqref{eq:ThmB}
  to obtain a sum of integers.
\end{ex}

\section{Torus matchings}
\label{sec:torus}

We now pick up the thread from \S\ref{2.w}
and discuss the combinatorics of gluing matrices and torus isometries
more systematically.
While this is slightly tangential to the main narrative of the paper,
it will help us in Section~\ref{sec:ex} to give an accurate count
of the extra twisted connected sums that we can construct from our
supply of building blocks.
Moreover, we find an additional divisibility property \eqref{eq:invmodn}
of the gluing data
that we use when proving Theorem~\ref{Thm:B} in Section~\ref{sec:hyp}.

\subsection{Combinatorics of torus isometries}
\label{subsec:combinatorics}

We analyse the conditions necessary to reconstruct
a torus matching from given gluing data.

\newcommand\matchfigure{
\begin{figure}
  \begin{minipage}{0.3\textwidth}
    \centering
\begin{tikzpicture}[scale=1.5, x=1.2cm, y=1cm]
      \draw (0,0) rectangle (1,1) ;
      \begin{scope}[->,line width=1.5pt]
        \draw[color=blue] (0,0) --
               (1,0) node[right, color=black] {$\scriptstyle\lambda_+=\mu_-$} ;
        \draw[color=red] (0,0) --
               (0,1) node[left, color=black] {$\scriptstyle\mu_+=\lambda_-$} ;
      \end{scope}
      \fill (0,0) circle(1pt) (1,0) circle (1pt) (0,1) circle (1pt)
      (1,1) circle (1pt) (0.2,0.2) circle (1pt) (0.4,0.4) circle (1pt)
      (0.6,0.6) circle (1pt) (0.8,0.8) circle (1pt) ;
    \end{tikzpicture}
  \end{minipage}
  \begin{minipage}{0.25\textwidth}
    \centering
\begin{tikzpicture}[scale=1.5]
      \draw (0,0) rectangle (0.577,1) ;
      \draw (0,0) -- (-0.289,0.5) -- (0.577,1) -- (0.866,0.5) -- cycle ;
      \begin{scope}[->,line width=1.5pt]
        \draw[color=blue] (0,0) --
        (0.577,0) node[right, color=black] {$\scriptstyle\lambda_+$} ;
        \draw[color=red] (0,0) --
        (0,1) node[left, color=black] {$\scriptstyle\mu_+$} ;
        \draw[color=blue] (0,0) --
        (-0.289,0.5) node[left, color=black] {$\scriptstyle\lambda_-$} ;
        \draw[color=red] (0,0) --
        (0.866,0.5) node[right, color=black] {$\scriptstyle\mu_-$} ;
      \end{scope}
      \fill (0,0) circle(1pt) (0.577,0) circle(1pt) (0,1) circle(1pt)
      (0.577,1) circle(1pt) (-0.289,0.5) circle(1pt) (0.866,0.5) circle(1pt)
      (0.072,0.125) circle(1pt) (0.144,0.25) circle(1pt)
      (0.216,0.375) circle(1pt) (0.289,0.5) circle(1pt)
      (0.361,0.625) circle(1pt) (0.433,0.75) circle(1pt)
      (0.505,0.875) circle(1pt) (0.65,0.125) circle(1pt)
      (0.722,0.25) circle(1pt) (0.794,0.375) circle(1pt)
      (-0.216,0.625) circle(1pt) (-0.144,0.75) circle(1pt)
      (-0.072,0.875) circle(1pt) ;
    \end{tikzpicture}
  \end{minipage}
  \begin{minipage}{0.35\textwidth}
    \centering
\begin{tikzpicture}[scale=1.5]
      \draw (0,0) -- (1,0) -- (2,1) -- (1,1) -- cycle ;
      \draw (0,0) -- (1,0) -- (0,1) -- (-1,1) -- cycle ;
      \begin{scope}[->,line width=1.5]
        \draw[color=blue] (0,0) --
        (1,0) node[right, color=black] {$\scriptstyle\lambda_+=\mu_-$} ;
        \draw[color=red] (0,0) --
        (1,1) node[left, color=black] {$\scriptstyle\mu_+$} ;
        \draw[color=blue] (0,0) --
        (-1,1) node[left, color=black] {$\scriptstyle\lambda_-$} ;
      \end{scope}
      \fill (0,0) circle(1pt) (1,0) circle(1pt) (-1,1) circle(1pt)
      (0,1) circle(1pt) (1,1) circle(1pt) (2,1) circle(1pt)
      (0.5,0.25) circle(1pt) (1,0.5) circle(1pt) (1.5,0.75) circle(1pt)
      (0,0.5) circle(1pt) (0.5,0.75) circle(1pt) (-0.5,0.75) circle(1pt) ;
    \end{tikzpicture}
  \end{minipage}\\[2\bigskipamount]
  \begin{minipage}{0.3\textwidth}
    \centering
\begin{tikzpicture}[scale=1.5, x=1.2cm, y=1cm]
      \draw (0,0) rectangle (1,1) ;
      \begin{scope}[->,line width=1.5pt]
        \draw[color=blue] (0,0) --
               (1,0) node[right, color=black] {$\scriptstyle\lambda_+=\mu_-$} ;
        \draw[color=red] (0,0) --
               (0,1) node[left, color=black] {$\scriptstyle\mu_+=\lambda_-$} ;
        \end{scope}
        \fill (0,0) circle(1pt) (1,0) circle (1pt) (0,1) circle (1pt)
        (1,1) circle (1pt) (0.2,0.4) circle (1pt) (0.4,0.8) circle (1pt)
        (0.6,0.2) circle (1pt) (0.8,0.6) circle (1pt) ;
      \end{tikzpicture}
    \caption{} %
    \label{fig:glueex.4}
  \end{minipage}
  \begin{minipage}{0.25\textwidth}
    \centering
\begin{tikzpicture}[scale=1.5]
      \draw (0,0) rectangle (0.577,1) ;
      \draw (0,0) -- (-0.289,0.5) -- (0.577,1) -- (0.866,0.5) -- cycle ;
      \begin{scope}[->,line width=1.5pt]
        \draw[color=blue] (0,0) --
        (0.577,0) node[right, color=black] {$\scriptstyle\lambda_+$} ;
        \draw[color=red] (0,0) --
        (0,1) node[left, color=black] {$\scriptstyle\mu_+$} ;
        \draw[color=blue] (0,0) --
        (-0.289,0.5) node[left, color=black] {$\scriptstyle\lambda_-$} ;
        \draw[color=red] (0,0) --
        (0.866,0.5) node[right, color=black] {$\scriptstyle\mu_-$} ;
      \end{scope}
      \fill (0,0) circle(1pt) (0.577,0) circle(1pt) (0,1) circle(1pt)
      (0.577,1) circle(1pt) (-0.289,0.5) circle(1pt) (0.866,0.5) circle(1pt)
      (0.072,0.625) circle(1pt) (0.144,0.25) circle(1pt)
      (0.216,0.875) circle(1pt) (0.289,0.5) circle(1pt)
      (0.361,0.125) circle(1pt) (0.433,0.75) circle(1pt)
      (0.505,0.375) circle(1pt) (0.65,0.625) circle(1pt)
      (0.722,0.25) circle(1pt) (-0.072,0.375) circle(1pt)
      (-0.144,0.75) circle(1pt) ;
    \end{tikzpicture}
    \caption{} %
    \label{fig:glueex.5}
  \end{minipage}
  \begin{minipage}{0.35\textwidth}
    \centering
\begin{tikzpicture}[scale=1.5]
      \draw (0,0) -- (1,0) -- (2,1) -- (1,1) -- cycle ;
      \draw (0,0) -- (1,0) -- (0,1) -- (-1,1) -- cycle ;
      \begin{scope}[->,line width=1.5]
        \draw[color=blue] (0,0) --
             (1,0) node[right, color=black] {$\scriptstyle\lambda_+=\mu_-$} ;
             \draw[color=red] (0,0) --
             (1,1) node[left, color=black] {$\scriptstyle\mu_+$} ;
             \draw[color=blue] (0,0) --
             (-1,1) node[left, color=black] {$\scriptstyle\lambda_-$} ;
      \end{scope}
      \fill (0,0) circle(1pt) (1,0) circle(1pt) (-1,1) circle(1pt)
      (0,1) circle(1pt) (1,1) circle(1pt) (2,1) circle(1pt)
      (0,0.25) circle(1pt) (0,0.5) circle(1pt) (0,0.75) circle(1pt)
      (1,0.25) circle(1pt) (1,0.5) circle(1pt) (1,0.75) circle(1pt) ;
    \end{tikzpicture}
    \caption{} %
    \label{fig:glueex.6}
  \end{minipage}
\end{figure}
}

\pagebreak[2]
\begin{ex}\label{ex:TorusMatch}
  We start with some examples and non-examples of torus matchings.
  \begin{enumerate}
  \item\label{exnonex.1} Let~$M$ be a twisted connected sum as
    in~\cite{CHNP2, Kovalev}.
    Assume that the group~$\Gamma\cong\Z/k$ acts by isomorphisms on the two
    ACyl Calabi-Yau manifolds~$V_\pm$ used in the construction of~$M$
    such that the induced action on the cross-section acts trivially
    on the K3 factor and freely on the interior circle~$S^1_{\zeta_\pm}$.
    Then~$(\Z/k)^2$ acts on~$M$, where each factor~$\Z/k$ acts on the ACyl
    Calabi-Yau manifold~$Y_\pm$ on one side and on the exterior circle
    on the other.
    For each~$\eps_+\in\Z/k$ with~$\gcd(\eps_+,k)=1$, we obtain
    a free $\Z/k$-action on~$M$ where a generator acts
    as~$(1,\eps_+)\in(\Z/k)^2$,
    see Figure~%
    \ref{fig:glueex.4} for~$k=5$
    and~$\eps_+=1$, $3$.
    The points of the lattice~$\Lambda$ are indicated by dots.

    The corresponding torus matching has~$k_+=k_-=k$,
    gluing matrix~$\psmatrix{0&k\\k&0}$, and ${\eps_+\eps_-\equiv 1}$ mod~$k$.
    The gluing angle is~$\thet=\pm\frac\pi 2$,
    and we have~$\xi_+=\zeta_-$ and~$\zeta_+=\xi_-$,
    but the ratio~$\ar_+=\ar_-^{-1}$ can be chosen arbitrarily.
  \item\label{exnonex.2} There are also examples
    with gluing angle~$\thet\notin\frac\pi 2\Z$
    where the gluing matrix together with the numbers~$k_+$, $k_-$
    does not determine the torus matching completely.
    As an example, consider~$k_+=k_-=8$
    and the gluing matrix~$\psmatrix{4&4\\12&-4}$.
    This determines~$\ar_+=\ar_-=\sqrt 3$ and~$\thet=\arctan\sqrt 3$.
    We can either pick~$\eps_+=\eps_-=1$ or~$\eps_+=\eps_-=-3$,
    see Figure~%
    \ref{fig:glueex.5}.
  \item\label{exnonex.3} If we want to construct a torus matching from a gluing
    matrix~$\psmatrix{\gll&p\\\glb&\sglr}$, numbers~$k_\pm$ and~$\eps_\pm$,
    it is not quite enough to satisfy only the conditions listed
    in Proposition~\ref{Prop2w.1}~\ref{it:necc}.
    If we set~$k_+=k_-=4$ and pick the gluing matrix~$\psmatrix{0&4\\4&-8}$,
    we may choose~$\eps_+=-\eps_-=\pm 1$.
    Then all conclusions in Proposition~\ref{Prop2w.1}~\ref{it:necc}
    hold, but equation~\eqref{eq:neg} is violated,
    which is part of the conclusions
    in Proposition~\ref{Prop2w.1}~\ref{2w.1c}.
    Hence we cannot have a matching of quotients of rectangular tori,
    see Figure~%
    \ref{fig:glueex.6}.
  \end{enumerate}
\end{ex}

\matchfigure

To search systematically for torus matchings with fixed $k_+$ and $k_-$,
it is helpful to first note the following formal consequences of
\eqref{eq:epses} and Proposition~\ref{Prop2w.1}~\ref{2w.1c}.
\begin{subequations}
\label{eq:gcds}
\begin{align}
\gcd(\gll,\glb) = \gcd(\gll, k_+) = \gcd(\glb, k_+) & \textrm{ and }
\gcd(p,\glr) = \gcd(p, k_+) = \gcd(\glr, k_+) \\
\gcd(\gll,p) = \gcd(\gll, k_-) = \gcd(p, k_-) & \textrm{ and }
\gcd(\glb,\glr) = \gcd(\glb, k_-) = \gcd(\glr, k_-)
\end{align}
\end{subequations}
\begin{equation}
\label{eq:neg}
\glb p\cdot \gll\glr \ge 0\;,
\qquad\text{and if~$\glb p\cdot \gll\glr=0$, then either~$\glb=p=0$ or }\gll=\glr=0\;.
\end{equation}

\begin{prop}\label{prop:cond}
  Let~$k_+>0$, let $\psmatrix{\gll&p\\\glb&\sglr}$
  be a gluing matrix with $\det\psmatrix{\gll&p\\\glb&\sglr}$ negative and
  divisible by $k_+$, and let~$\eps_+ \in (\Z/k_+)^*$.
  Suppose that~\eqref{eq:epses+}, \eqref{eq:gcd+} and~\eqref{eq:neg}
  are satisfied.
  Then there exists a unique torus matching with these data.
  If one chooses~$a$, $b\in\Z$ such that
  \begin{subequations}\label{eq:ab}
    \begin{align}
      1&=bp+a\frac{\glr+\eps_+p}{k_+}\;,\label{eq:abchoice}\\
      \text{then}\qquad
      \eps_-&\equiv a\frac{\glb-\eps_+\gll}{k_+}-b\gll\mod k_-\;.\label{eq:abeps-}
    \end{align}
  \end{subequations}
\end{prop}

\begin{proof}
  Let~$k_+$, $\psmatrix{\gll&p\\\glb&\sglr}$ and~$\eps_+$ be given as above,
  and
  let~$\tilde\Lambda_+\subset\C$ be the lattice
  with basis~$(\mu_+,\lambda_+)=(i\xi_+,\zeta_+)$ as in Section~\ref{2.w}.
  Then we can construct a sublattice~$\Lambda\subset\frac1{k_+}\tilde\Lambda_+$
  of index~$k_+$ with basis~$(\nu_+,\lambda_+)$ given by~\eqref{2.w.5},
  and~$\tilde\Lambda_+\subset\Lambda$ is also a sublattice of index~$k_+$.
  By assumption~\eqref{eq:epses+},
  the gluing matrix then determines a sublattice~$\tilde\Lambda_-\subset\Lambda$
  of index~$k_-=-\det\psmatrix{\gll&p\\\glb&\sglr}/k_+$ with
  basis~$(\mu_-,\lambda_-)$
  determined by~\eqref{eq:bases}.

  We conclude that~$\Lambda\subset\frac1{k_-}\tilde\Lambda_-$ is
  again a sublattice of index~$k_-$.
  For~$c$, $d\in\{0,\dots,k_- {-}1\}$,
  let~$\lambda(c,d)=\frac c{k_-}\mu_-+\frac d{k_-}\lambda_-
  \in\frac 1{k_-}\tilde\Lambda_-$.
  Suppose that~$\lambda(c,d)\in\Lambda$.
  The vectors~$\lambda_-$ and~$\mu_-$ are primitive in~$\Lambda$
  by~\eqref{eq:gcd+} and~\eqref{eq:bases},
  so~$c=0$ if and only if~$d=0$.
  Similarly, for each~$c$ there can at most be one~$d$ in the given range
  such that~$\lambda(c,d)\in\Lambda$ and vice versa.
  Because there are exactly~$k_-$ elements of~$\Lambda$
  with coordinates~$c$, $d$ in the given range,
  for each~$c$ there is exactly one~$d$
  such that~$\lambda(c,d)\in\Lambda$, and vice versa.
  Specifying~$c=1$, we hence get a unique~$\eps_-=d\in\Z/k_-$.
  Moreover, $\gcd(\eps_-,k_-)=1$,
  and~$\Lambda$ is an extension of~$\tilde\Lambda_-$
  by a cyclic group~$\Gamma_-\cong\Z/k_-$.
  This proves existence and uniqueness of the gluing data.

  We set~$k_-=-\frac1{k_+}\,\det\psmatrix{\gll&p\\\glb&\sglr}$.
  To compute~$\eps_-$ modulo~$k_-$,
  we note that the matrix~$\TorMat=\psmatrix{\tll&p\\\tlb&\stlr}$
  from~\eqref{2.w.7} depends on our choices of~$\eps_\pm\in\Z$.
  Because~$1=-\det\TorMat=p\tlb+\tll\tlr$,
  we may pick~$a=\tll$ and~$b=\tlb$ in~\eqref{eq:abchoice}.
  Using~\eqref{2.w.7} and~\eqref{eq:det}, we compute
  \begin{align*}
    \tll\frac{\glb-\eps_+\gll}{k_+}-\tlb\gll
    &=\frac{(\gll+p\eps_-)(\glb-\eps_+\gll)
      -(\glb-\eps_+\gll-\eps_-\glr-\eps_+\eps_-p)\gll}{k_-k_+}\\
    &=\frac{p\glb+\gll\glr}{k_-k_+}\,\eps_-=\eps_-\;.\qedhere
  \end{align*}
\end{proof}

Given positive integers~$k_-$ and~$k_+$,
there are only finitely many matrices~$\psmatrix{\gll&p\\\glb&\sglr}\in M_2(\Z)$
that satisfy conditions~\eqref{eq:det} and~\eqref{eq:neg}. For 
\begin{equation*}
  \glb p+\gll\glr=k_-k_+>0
\end{equation*}
and~\eqref{eq:neg} imply that
\begin{equation}\label{eq:np.mq.sign}
  \glb p, \, \gll\glr \in\{0,\dots,k_-k_+\}
\end{equation}

By Remark~\ref{rmk:geosym} below,
we may assume in addition~$\thet\in(0,\pi)$.
Then~$\glb>0$ by Proposition~\ref{Prop2w.1}~\ref{2w.1d},
and therefore also~$p>0$.
In other words, $\psmatrix{\gll&p\\\glb&\sglr}$ can be chosen to be either
off-diagonal with non-negative entries,
or with exactly three positive entries and one negative entry,
which can only be~$\gll$ or~$\sglr$.

For small $k_+$ and $k_-$ it is now easy even by hand to enumerate all
$\psmatrix{\gll&p\\\glb&\sglr}$ that satisfy~\eqref{eq:gcds} in addition.
Most of those have  $\gcd \psmatrix{\gll&p\\\glb&\sglr} = 1$
and thus give rise to a unique torus matching,
while for the few remaining ones it is easy to enumerate 
any $\eps_+$ that satisfy~\eqref{eq:epses+} and~\eqref{eq:gcd+}.

\begin{rmk}\label{rmk:latticeindex}
  Assume that we are given gluing data~$k_\pm\in\Z$, $\eps_\pm\in\Z/k_\pm$
  and~$\psmatrix{\gll&p\\\glb&\sglr}\in M_2(\Z)$.
  Then the lattice~$\Lambda$
  is a sublattice of the lattice~$\tilde\Lambda_++\tilde\Lambda_-$
  spanned by~$\tilde\Lambda_+$ and~$\tilde\Lambda_-$.
  Because we have assumed the groups~$\Gamma_\pm=\Z/k_\pm$ act
  freely on the interior and on the exterior circles,
  we can compute the index of~$\Lambda$ in~$\tilde\Lambda_++\tilde\Lambda_-$
  in four different
  ways:
  \begin{equation*}
    \bigl[\tilde\Lambda_++\tilde\Lambda_-:\Lambda\bigr]
    =\gcd(\gll,p,k_+)=\gcd(\glb,\glr,k_+)=\gcd(\gll,\glb,k_-)=\gcd(p,\glr,k_-)\;.
  \end{equation*}
  For the first equation, we choose the vectors~$\frac{\lambda_+}{k_+}$
  and~$\frac{\mu_+}{k_+}$ as a basis for~$\R^2$; see Figure~\ref{Fig2g.1}.
  Then the smallest positive second coordinate of an element of~$\Lambda$
  is~$1$, because~$\Gamma_+\cong\Z/k_+$ acts freely on the exterior
  circle of~$\widetilde M_+$.
  On the other hand, the smallest positive second coordinate of an element
  of~$\tilde\Lambda_++\tilde\Lambda_-$ is~$\gcd(\gll,p,k_+)$.
  The other equations follow similarly.

  We note that we have not used the numbers~$\eps_\pm$ in the argument above.
  If ${\tilde\Lambda_++\tilde\Lambda_-=\Lambda}$,
  then~$\Lambda$ is uniquely determined by~$k_\pm$ and the gluing matrix,
  and so are~$\eps_\pm$.
  This is the case in most examples with~$\thet\ne\frac\pi2$
  in Table~\ref{table:matchings}.
  If~$\thet=\frac\pi 2$ and~$k_+=k_-\ge 3$, then there are at least
  two possible choices for~$\eps_+$.
  On the other hand, in examples~237, 238, 254 and~255
  in Table~\ref{table:matchings},
  we have~$[\tilde\Lambda_++\tilde\Lambda_-:\Lambda]>1$,
  but nevertheless, $\Lambda$ and~$\eps_\pm$ are uniquely determined.
\end{rmk}

\subsection{New extra-twisted connected sums from old}\label{subsec:newmatch}
Having discussed how to enumerate torus matchings, we now move on to discuss
relations between them.
We will find several ways to describe isometric extra-twisted connected sums,
but we also discuss covering spaces
and a kind of ``$t$-duality''.

\begin{prop}\label{prop:isom}
  Let~$M$ be an extra-twisted connected sum constructed from
  asymptotically cylindrical Calabi-Yau manifolds~$V_\pm$
  with gluing data~$k_\pm\in\Z$, $\eps_\pm\in\Z/k_\pm$
  and~$\psmatrix{\gll&p\\\glb&\sglr}\in M_2(\Z)$, and with gluing angle~$\thet$.
  Then the following gluing data describe an isometric extra-twisted
  connected sum~$M'$, possibly with opposite orientation:
  \begin{align}
    \begin{pmatrix}\gll'&p'\\\glb'&\sglr'\end{pmatrix}
    &=\begin{pmatrix}\glr&p\\\glb&-\gll\end{pmatrix}\;,\quad&
    \begin{matrix}
      k'_+=k_-\;,&\quad\eps'_+=\phantom{-}\eps_-\;,\\[\smallskipamount]
      k'_-=k_+\;,&\quad\eps'_-=\phantom{-}\eps_+\;,
    \end{matrix}
    \qquad\text{and}\qquad\thet'&=\thet\;;\label{symgrp.1}\displaybreak[0]\\
    \begin{pmatrix}\gll'&p'\\\glb'&\sglr'\end{pmatrix}
    &=\begin{pmatrix}\gll&-p\\-\glb&\sglr\end{pmatrix}\;,\quad&
    \begin{matrix}
      k'_+=k_+\;,&\quad\eps'_+=-\eps_-\;,\\[\smallskipamount]
      k'_-=k_-\;,&\quad\eps'_-=-\eps_+\;,
    \end{matrix}
    \qquad\text{and}\qquad\thet'&=-\thet\;;\label{symgrp.3}\displaybreak[0]\\
    \begin{pmatrix}\gll'&p'\\\glb'&\sglr'\end{pmatrix}
    &=\begin{pmatrix}-\gll&-p\\-\glb&\glr\end{pmatrix}\;,\quad&
    \begin{matrix}
      k'_+=k_+\;,&\quad\eps'_+=\phantom{-}\eps_-\;,\\[\smallskipamount]
      k'_-=k_-\;,&\quad\eps'_-=\phantom{-}\eps_+\;,
    \end{matrix}
    \qquad\text{and}\qquad\thet'&=\thet\pm\pi\;.\label{symgrp.5}
  \end{align}
  In~\eqref{symgrp.1}, we have to swap the roles of~$V_+$ and~$V_-$.
  In~\eqref{symgrp.3}, we change the orientation of~$M$.
  In~\eqref{symgrp.5}, we pass to the opposite Calabi-Yau structure
  on one side.
  
  The three elements above generate a group~$\isogrp\cong(\Z/2)^3$
  that acts on the set of gluing data describing a given deformation
  family up to isomorphism.
\end{prop}

\begin{proof}
  We obtain~\eqref{symgrp.1} by exchanging the roles of~$V_+$ and~$V_-$.
  We get the new gluing matrix by inverting the base change~\eqref{2.w.6}.
  Because the definition of the gluing angle is symmetric,
  it does not change.

  For~\eqref{symgrp.3}, we change the orientation of~$M$ by swapping
  the orientations of the two exterior circles.
  This changes the sign of the gluing angle and of~$\eps_\pm$.
  The new gluing matrix arises by conjugating with~$\psmatrix{-1&0\\0&1}$.

  In~\eqref{symgrp.5}, we rotate one of the two sides, say~$M_+$,
  by~$\pi$, which leads to the new gluing angle.
  This has the effect of changing the orientations of both~$V_+$
  and the exterior circle,
  and hence the gluing matrix is multiplied by~$-\psmatrix{1&0\\0&1}$.
  If~$\omega_+$ and~$\Omega_+$ describe the old Calabi-Yau structure
  on~$V_+$, the new one carries the opposite complex structure and is
  given by~$-\omega_+$ and~$-\bar\Omega_+$.
\end{proof}

\begin{rmk}\label{rmk:geosym}
  The subgroup spanned by~\eqref{symgrp.3} and~\eqref{symgrp.5}
  is rich enough to make sure that we
  can always assume~$\gll$, $\glb$, $p$, $\glr \ge 0$.
  Moreover, if we use the same building block
  and the same finite group~$\Gamma\cong\Z/k$ for~$M_+$ and~$M_-$,
  we may apply~\eqref{symgrp.1} to %
  get $\glr \ge \gll$.
  In Table~\ref{table:matchings}, we only list gluing data
  satisfying these conventions.
\end{rmk}

Recall that~$\K\times T\cong\K_\pm\times T_\pm$ denote the isometric
cross-sections at infinity
of the asymptotically cylindrical $G_2$-manifolds~$M_\pm$
used in Theorem~\ref{thm:glue} to construct~$M$.
If~$\Lambda^*$ denotes the dual of the lattice~$\Lambda\in\C\cong\R^2$
with respect to the standard metric,
we write~$T^*=\R^2/\Lambda^*$ for the dual of the torus~$T=\R^2/\Lambda$.

\begin{prop}\label{prop:dual}
  Let~$M$ be an extra-twisted connected sum
  glued along~$\K\times T$
  with gluing data~$k_\pm\in\Z$, $\eps_\pm\in\Z/k_\pm$
  and~$\psmatrix{\gll&p\\\glb&\sglr}\in M_2(\Z)$ and with gluing angle~$\thet$.
  Then there exists a deformation family of extra-twisted connected sums~$M'$
  with isomorphic asymptotically flat Calabi-Yau manifolds~$V_\pm$,
  glued along~$\K\times T^*$ with gluing data
  \begin{align}
    \begin{pmatrix}\gll'&p'\\\glb'&\sglr'\end{pmatrix}
    &=\begin{pmatrix}\glr&\glb\\p&-\gll\end{pmatrix}\;,\quad&
    \begin{matrix}
      k'_+=k_+\;,&\quad\eps'_+=-\eps_+^*\;,\\[\smallskipamount]
      k'_-=k_-\;,&\quad\eps'_-=-\eps_-^*\;,
    \end{matrix}
    \qquad\text{and}\qquad\thet'&=\thet\;.\label{eq:geosym.3}
  \end{align}
  Together with the group~$\isogrp$ in Proposition~\ref{prop:isom},
  this transformation generates
  a group isomorphic to~$(\Z/2)^4$ that acts on the possible gluing data
  compatible with a given K3 matching.
\end{prop}

Thus, the new extra-twisted connected sum is in a certain sense
``$t$-dual'' to the original one,
but can in general not be deformed into the original one.
We remark that any torsion free $G_2$-structure close to an extra
twisted connected sum is again an extra twisted connected sum.
Hence the duality described here can be used to identify small open subsets
of the moduli space of $G_2$-metrics.

\begin{proof}
  We recall the generators for the involved lattices in~$\C$
  from Section~\ref{2.w}:
  \begin{align*}
    \tilde\Lambda_+&=\bigl\<\zeta_+,i\xi_+\bigr\>\;,\\
    \Lambda&=\biggl\<\zeta_+,\frac{\eps_+\zeta_++i\xi_+}{k_+}\biggr\>\;,\\
    \text{and}\qquad
    \tilde\Lambda_-&=\biggl\<\frac{\sglr\zeta_++ip\xi_+}{k_+},
    \frac{\glb\zeta_++i\gll\xi_+}{k_+}\biggr\>\;.
  \end{align*}
  Because~$T$ is a $k_\pm$-fold quotient of~$T_\pm$,
  the dual torus~$T^*$ is a $k_\pm$-fold covering of~$\widetilde T_\pm^*$,
  or equivalently, a $k_\pm$-fold quotient of~$k_\pm\widetilde T_\pm^*$.
  With respect to the standard Euclidean metric on~${\C\cong\R^2}$,
  we have generators for the dual lattices,
  where we take~$\eps_+^*$ in~$\Z/k_+$ as above:
  \begin{align*}
    k_+\tilde\Lambda_+^*
    &=\biggl\<\frac{k_+}{\zeta_+}, \, \frac{ik_+}{\xi_+}\biggr\>\;,\\
    \Lambda^*&=\biggl\<\frac{k_+}{\zeta_+}, \,
    -\frac{\eps_+^*}{\zeta_+}+\frac i{\xi_+}\biggr\>\;,\\
    \text{and}\qquad
    k_-\tilde\Lambda_-^*
    &=\biggl\<-\frac \gll{\zeta_+}+\frac{i\glb}{\xi_+}, \,
    \frac p{\zeta_+}+\frac{i\glr}{\xi_+}\biggr\>\;.
  \end{align*}
  The gluing data~\eqref{eq:geosym.3} can be read off from this description.
  And because the gluing angle does not change,
  the K3 matching used to construct~$M$ also works for~$M'$.
\end{proof}

Alternatively,
one can rotate both tori~$\widetilde T_\pm$ by a right angle,
thus swapping the role of exterior and interior circles.
This leads to the same gluing data as above.

\begin{rmk}\label{rmk:divisible}
  We recall the matrix~$\TorMat=\psmatrix{\tll&p\\\tlb&\stlr}\in GL(2,\Z)
  \setminus SL(2,\Z)$ from \eqref{2.w.7}, with entries named as in the proof 
  of Proposition \ref{prop:cond}.
  Let us set up an analogous matrix~$\TorMat'$ for the gluing data
  in~\eqref{eq:geosym.3}, then
  \begin{equation*}
    \TorMat'=
    \begin{pmatrix}
      \tll'&p'\\\tlb'&\stlr'
    \end{pmatrix}=
    \begin{pmatrix}
      \frac{\glr-\glb\eps^*_-}{k_-}&\glb\\
      \frac{p+\eps^*_+\glr+\eps^*_-\gll-\eps^*_+\eps^*_-\glb}{k_-k_+}&
      -\frac{\gll-\glb\eps^*_+}{k_+}
    \end{pmatrix}\;.
  \end{equation*}
  Because again~$\det\TorMat'=-1$, we conclude that
  \begin{equation}\label{eq:invmodn}
    \frac{\gll-\eps_+^*\glb}{k_+}\cdot\frac{\glr-\eps_-^*\glb}{k_-}
    =-\det\TorMat'-\tlb'\glb
    \equiv 1\mod \glb\;.
  \end{equation}
\end{rmk}

\pagebreak[2]
\begin{rmk}\label{rmk:rectsym}
  For the sake of completeness,
  let us add the following observations.
  \begin{enumerate}
  \item Given gluing data~$k_\pm$, $\eps_\pm$ and~$\psmatrix{\gll&p\\\glb&\sglr}$
    as above with gluing angle~$\thet\notin\frac\pi2+\pi\Z$.
    Then we can write valid gluing data of the form
  \begin{align}\kern2em
    \begin{pmatrix}\gll'&p'\\\glb'&\sglr'\end{pmatrix}
    &=\begin{pmatrix}p&\gll\\\glr&-\glb\end{pmatrix}\;,\quad&
    \begin{matrix}
      k'_+=k_+\;,&\quad\eps'_+=-\eps_+\;,\\[\smallskipamount]
      k'_-=k_-\;,&\quad\eps'_-=\phantom{-}\eps_-^*\;,
    \end{matrix}
    \quad\text{and}\quad\thet'&=\pm\frac\pi2-\thet\;.\label{symgrp.2}
  \end{align}
    One can check that this transformation together with those
    of Propositions~\ref{prop:isom} and~\ref{prop:dual}
    generates a group isomorphic to~$D_4\ltimes(\Z/2)^2$
    that acts on the set of valid gluing data.

    However, to construct an extra-twisted connected sum,
    we need a K3 matching that is compatible with the gluing angle.
    If both blocks have rank~$1$,
    the compatibility condition is given by~\eqref{3.c.3}.
    The new gluing data above will in general not be compatible
    with the K3 matching used for the original extra-twisted connected sum~$M$.
    And it is not hard to find examples in Table~\ref{table:matchings}
    where the new gluing angle
    is not compatible with any possible K3 matching of rank~$1$ blocks.
  \item 
    Let us now consider matchings with gluing angle~$\thet=\pm\frac\pi2$
    as in Example~\ref{ex:TorusMatch}~\ref{exnonex.1}.
    We %
    recall that~$k_+=k_-=k$, $\eps_-=\eps_+^*\in\Z/k$,
    and the gluing matrix takes the form~$\psmatrix{0&k\\k&0}$
    by Example~\ref{ex:TorusMatch}~\ref{exnonex.1}.
    This implies that the transformation~\eqref{eq:geosym.3}
    of Proposition~\ref{prop:dual}
    acts exactly as the composition of the three
    elements~\eqref{symgrp.1}--\eqref{symgrp.5} of~$\isogrp $.
    Hence, we are reduced to an action of the group~$\isogrp $ in this special case.
    This is not surprising because~\eqref{eq:geosym.3} mainly affects
    the ratio of circle lengths, which is not specified by the gluing
    data if~$\thet=\frac\pi2$,
    see the discussion before Proposition~\ref{Prop2w.1}.
    Exploiting the action of~$\isogrp $, we may assume in Table~\ref{table:matchings}
    that~$p=\glb>0$ and that~$\eps_+>0$.

    One may note that in all our examples of this type
    in Table~\ref{table:matchings}
    we have~$\eps_+^*=\pm\eps_+$, so in none of these examples the group~$\isogrp $
    acts effectively.
  \end{enumerate}
\end{rmk}

\subsection{Covering spaces}
By Proposition~\ref{prop:pi1},
an extra-twisted connected sum with gluing matrix~$\psmatrix{\gll&p\\\glb&\sglr}$
is simply-connected if and only if~$p=1$.
Let us enumerate its connected covering spaces if~$p>1$.

\begin{prop}\label{prop:cover}
  Assume that~$M$ is an extra-twisted connected sum with gluing
  data given by~$k_\pm$, $\eps_\pm\in\Z/k_\pm$ and gluing
  matrix~$\psmatrix{\gll&p\\\glb&\sglr}$.
  Assume that~$p>1$ and that~$\ell\mid p$.
  Then there exists a unique connected $\ell$-fold covering space~$\widetilde M$.
  It is an extra-twisted connected sum constructed from the same building
  blocks as~$M$ with the same gluing angle~$\thet$, and with gluing data
  \begin{subequations}\label{eq:covers}
    \begin{align}
      \tilde k_\pm&=\frac{k_\pm}{\gcd(\ell,k_\pm)}\;,
      \qquad\tilde\eps_\pm=\frac{\ell}{\gcd(\ell,k_\pm)}\eps_\pm
      \;,\label{eq:cover-k-eps}\\
      \begin{pmatrix}
        \tilde r_+ &\tilde p\\\tilde \glb&-\tilde r_-
      \end{pmatrix}
      &=
      \begin{pmatrix}
        \frac \gll{\gcd(\ell,k_-)}&\frac p\ell\\[\medskipamount]
        \frac{\glb\ell}{\gcd(\ell,k_+)\gcd(\ell,k_-)}&\frac \sglr{\gcd(\ell,k_+)}
      \end{pmatrix}\in M_2(\Z)\;.\label{eq:cover-mat}
    \end{align}
  \end{subequations}
\end{prop}

If~$\tilde k_\pm=k_\pm$,
then the covering is constructed using the same finite groups~$\Gamma_\pm$.
Otherwise, %
one has to pass to a proper subgroup of one or both of these groups.

\begin{proof}
  Because~$\pi_1(M)$ is cyclic by Proposition~\ref{prop:pi1},
  there is a unique connected $\ell$-fold covering~$\pi\colon\widetilde M\to M$.
  Let~$\widetilde M_\pm\to M_\pm$ denote its restriction
  to the two halves of~$M$.
  Because~$\pi_*\colon\pi_1(M_\pm)\to\pi_1(M)$ is surjective,
  we see that~$\widetilde M_\pm\to M_\pm$ are also connected $\ell$-fold
  coverings,
  which are uniquely determined by~$\ell$ up to isomorphism
  since~$\pi_1(M_\pm)\cong\Z$.
  It follows that~$\widetilde M$ is an extra-twisted sum,
  glued from~$\widetilde M_+$ and~$\widetilde M_-$.

  Let~$\tilde X\to X$ denote the restriction of~$\pi$ to~$X=M_+\cap M_-$.
  The corresponding sublattice $\tilde\Lambda\subset\Lambda=\pi_1(X)$
  is spanned by the vectors~$\lambda_\pm$ corresponding to the
  interior circles, and by
  \begin{equation*}
    \tilde\nu_\pm=\ell\nu_\pm
    =\frac \ell{k_\pm}(\mu_\pm+\eps_\pm\lambda_\pm)\;.
  \end{equation*}
  The smallest multiples of~$\mu_\pm$ inside~$\tilde\Lambda$ are
  \begin{equation*}
    \tilde\mu_\pm
    =\frac{k_+}{\gcd(\ell,k_\pm)}\cdot\tilde\nu_\pm
    -\frac{\ell\eps_\pm}{\gcd(\ell,k_\pm)}\,\lambda_\pm
    =\frac\ell{\gcd(\ell,k_\pm)}\mu_\pm\;,
  \end{equation*}
  and~$\tilde\Lambda$ is an extension of the lattice spanned
  by~$\lambda_\pm$ and~$\tilde\mu_\pm$ by a finite cyclic
  group~$\tilde\Gamma_\pm\cong\Z/\tilde k_\pm$.
  With respect to the new bases~$(\tilde\mu_\pm,\lambda_\pm)$,
  $\tilde k_\pm$ and~$\tilde\eps_\pm$
  are given by~\eqref{eq:cover-k-eps},
  and the new gluing matrix takes the form of~\eqref{eq:cover-mat},
  which has determinant~$-\tilde k_+\tilde k_-$.
\end{proof}

Let us reverse the construction above.
In fact, the next result is dual to the one above
in the sense of Proposition~\ref{prop:dual}.

\begin{prop}\label{prop:quotients}
  Assume that~$M$ is an extra-twisted connected sum with gluing
  data given by~$k_\pm$, $\eps_\pm\in\Z/k_\pm$ and gluing
  matrix~$\psmatrix{\gll&p\\\glb&\sglr}$.
  Consider the lattice~$\Lambda'\subset\C$ of elements that act
  on the internal circle~$S^1_{\zeta_+}$ by an element of~$\Z/k_+$
  and on~$S^1_{\zeta_-}$ by an element of~$\Z/k_-$
  Then the group~$\Lambda'/\Lambda$ is isomorphic to~$\Z/\glb$
  and acts effectively on~$M$.
  If~$\ell\mid \glb$ then the subgroup of order~$\ell$
  acts freely on~$M$ if and only if~$\gcd(\ell,k_+)=\gcd(\ell,k_-)=1$.
\end{prop}

Note that there can be more automorphisms of~$M$,
for example some that do not project to the identity
of the K3 surface in the neck.

\begin{proof}
  From its description above,
  it is clear that elements of~$\Lambda'$ act on~$V_\pm$
  by elements of~$\Gamma_\pm$,
  and on the external circles~$S^1_{\xi_\pm}$ by rotation.
  Elements of~$\Lambda$ act trivially.
  Hence we get an effective action of~$\Lambda'/\Lambda$ on~$M$ that preserves
  the closed $G_2$-structure one obtains by gluing,
  see Section~\ref{subsec:glue}.
  From the proof of~\cite[Theorem~G2]{Joyce},
  it is clear that it also preserves the resulting torsion free $G_2$-structure
  on~$M$.

  Now let~$\alpha\in\Lambda'$.
  Because~$\alpha$ projects to an element of~$\Gamma_-$
  and the projection map~$\Lambda\to\Gamma_-$ is surjective,
  we can add an element of~$\Lambda$ to make sure that~$\alpha$
  projects to the identity of the internal circle~$\zeta_-\R\subset\C$.
  Hence, $\alpha$ is a multiple of~$\xi_-$.
  The projection to~$\zeta_+\R\subset\C$ will be a multiple
  of~$\frac1{k_+}\zeta_+$.
  From the definition of the gluing matrix in~\eqref{2.w.6},
  we infer that~$\alpha=\frac a\glb\xi_-$,
  and~$\frac a\glb\xi_-\in\Lambda$ if and only if~$\glb\mid a$
  because~$\xi_-$ is a primitive element of~$\Lambda$,
  see also Figure~\ref{Fig2g.1}.
  Hence we conclude that~$\Lambda'/\Lambda$ is a cyclic group of order~$\glb$.

  An element~$\alpha\in\Lambda'$ will act with fixpoints on~$M_-$
  if and only if it projects to the identity
  of the external circle~$S^1_{\xi_-}$,
  that is, if and only if it fixes~$V_-$ as a set.
  If we represent~$\alpha$ by~$\frac a\glb\xi_-$,
  we can add an element~$\lambda\in\Lambda$
  such that~$\alpha+\lambda\in\zeta_-\R$
  if and only if~$\frac a\glb$ can be written as~$\frac b{k_-}$.
  Such~$a\not\equiv 0$ mod~$\glb$ exist if and only if~$\gcd(\glb,k_-)>1$.
  If we restrict to a subgroup of order~$\ell$,
  this can happen if and only if~$\gcd(\ell,k_-)>1$.
  Similarly, there will be fixpoints in~$M_+$ if and only if~$\gcd(\ell,k_+)>1$.
\end{proof}

\begin{rmk}\label{rmk:quotients}
  If the group~$\Lambda'/\Lambda$ acts freely on~$M$,
  the quotient will again be an extra twisted connected sum,
  and we have reversed the covering space construction
  of Proposition~\ref{prop:cover}
  in the special case where~$\tilde k_\pm=k_\pm$.
  If there are fixpoints,
  we obtain an orbifold as a quotient.
  The fixpoint sets will be a union of products of compact subsets of~$V_\pm$
  and~$S^1_{\xi_\pm}$.
  Suppose that we find a crepant resolution of the corresponding
    orbifold quotients of~$V_\pm$.
  Then we can form an extra twisted connected sum using the resolved
  asymptotically cylindrical manifolds~$\tilde V_\pm$.
  The gluing data will now be given by
  \begin{subequations}\label{eq:quotients}
    \begin{align}
      \tilde k_\pm&=\frac{k_\pm}{\gcd(\ell,k_\pm)}\;,
      \qquad\tilde\eps_\pm=\frac{\gcd(\ell,k_\pm)}\ell\eps_\pm
      \;,\label{eq:quot-k-eps}\\
      \begin{pmatrix}
        \tilde r_+ &\tilde p\\\tilde \glb&-\tilde r_-
      \end{pmatrix}
      &=
      \begin{pmatrix}
        \frac \gll{\gcd(\ell,k_+)}&
        \frac{p\ell}{\gcd(\ell,k_+)\gcd(\ell,k_-)}\\[\medskipamount]
        \frac \glb\ell&\frac \sglr{\gcd(\ell,k_-)}
      \end{pmatrix}\in M_2(\Z)\;.\label{eq:quot-mat}
    \end{align}
  \end{subequations}
  This construction is dual to~\eqref{eq:covers} in the sense
  of Proposition~\ref{prop:dual}.

  To invert the covering space construction of Proposition~\ref{prop:cover},
  we need to allow multiples~$k'_\pm$ of~$k_\pm$
  such that the $\Z/k_\pm$-actions on~$V_\pm$ can be extended
  to~$\Z/k'_\pm$-actions.
  Then similar constructions as above are possible.
  We leave the details to the reader.
\end{rmk}

  Table~\ref{table:matchings} describes~255 deformation families of
  extra-twisted connected sums, 125 of which are simply connected.
  Among the remaining examples, there are~64 where taking the universal
  cover implies passing to subgroups of~$\Gamma_+$ or~$\Gamma_-$.

  Among the examples in Table~\ref{table:matchings},
  the one with largest fundamental group~$\pi_1(M)\cong\Z/21$
  is entry~250,
  which has~$k_+=4$ and~$k_-=6$.
  The universal cover has~$\tilde k_+=4$ and~$\tilde k_-=2$.
  It can be found as entry~174 with~$M_+$ and~$M_-$ swapped.
  Entries~175 and~248 are the two intermediate covering spaces.

\begin{ex}\label{ex:run-symm}
  Consider the  Example~\ref{ex:run-gluing} (228),
  where numbers in parentheses refer to Table~\ref{table:matchings},
  possibly up to the isometry~\eqref{symgrp.1}.
  Applying the duality~\eqref{eq:geosym.3}, we get the gluing matrix
  \begin{align*}
    \begin{pmatrix}
      5&10\\1&-1
    \end{pmatrix}
    \qquad&\text{with}\qquad
      \begin{matrix}k_+=3\;,&\quad&\eps_+=-1\;,\\
        k_-=5\;,&\quad&\eps_-=\phantom{-}1\;,%
      \end{matrix}
    &\bar\nu&=-43&(231)\;.
  \end{align*}
  By Proposition~\ref{prop:pi1},
  the corresponding extra-twisted connected sum~$M'$ is not simply connected.
  By Proposition~\ref{prop:cover},
  its universal covering and the intermediate covering spaces have
  gluing data and \xnuinvt
  \begin{align*}
    \begin{pmatrix}1&1\\2&-1\end{pmatrix}
      \qquad&\text{with}\qquad
      \begin{matrix}k_+=3\;,&\quad&\eps_-=-1\;,\\
        k_-=1\;,&
      \end{matrix}
      &\bar\nu&=-19&(21)\;,\displaybreak[0]\\
    \begin{pmatrix}1&2\\1&-1\end{pmatrix}
      \qquad&\text{with}\qquad
      \begin{matrix}k_+=3\;,&\quad&\eps_-=\phantom{-}1\;,\\
        k_-=1\;,
      \end{matrix}
      &\bar\nu&=-35&(23)\;,\displaybreak[0]\\
    \begin{pmatrix}5&5\\2&-1\end{pmatrix}
      \qquad&\text{with}\qquad
      \begin{matrix}k_+=3\;,&\quad&\eps_+=\phantom{-}1\;,\\
        k_-=5\;,&\quad&\eps_-=\phantom{-}2\;,
      \end{matrix}
      &\bar\nu&=-23&(230)\;.
  \end{align*}
  Note that the universal covering is different from
  the original manifold~$M$ from Example~\ref{ex:run-gluing} (228).
  In particular, we have forgotten the $\Z/5$-action on block~12
  of Table~\ref{table:blocks}, thus obtaining block~9.
  The first two lines above are again related by~\eqref{eq:geosym.3}.
  Applying~\eqref{eq:geosym.3} to the last line
  gives an extra-twisted connected sum with fundamental group~$\Z/2$,
  whose universal cover is  the original manifold~$M$,
  and which is described by
  \begin{align*}
    \begin{pmatrix}
      1&2\\5&-5
    \end{pmatrix}
    \qquad&\text{with}\qquad
      \begin{matrix}k_+=3\;,&\quad&\eps_+=-1\;,\\
        k_-=5\;,&\quad&\eps_-=\phantom{-}2\;,%
      \end{matrix}
    &\bar\nu&=-7&(229)\;.
  \end{align*}
\end{ex}

\subsection{Kovalev-Lefschetz fibrations}\label{sec:kl}

It is interesting to study fibrations $M \to B$ of \gtmfd s by coassociative
K3 surfaces with at worst ordinary double point singularities.
Donaldson \cite{Donaldson} initiated a programme to study adiabatic limits
of such fibrations where the size of the fibres goes to zero.
Closed \gtmfd s do not admit smooth fibrations by 4-manifolds by
Baraglia \cite{Baraglia}, so there
is always a non-empty 1-dimensional submanifold $L \subset B$ of singular
fibres, which plays an important role.

Extra-twisted connected sums have obvious approximately coassociative K3
fibrations.  Recent work of Englebert \cite{Englebert} can be applied to
perturb them to a genuine coassociative fibration \emph{provided} that the
singular fibres have only an ordinary double point singularity each.
We shall leave aside the interesting question of how to verify that
(Englebert did so in a particular example of a rectangular twisted
connected sum), and instead focus on describing the link of
singular fibres.

A building block $(Z,\K)$ according to Definition \ref{def:block} comes with
a holomorphic fibration $f : Z \to \PP^1$. Hence the ACyl Calabi-Yau 3-folds
$V := Z \setminus \K$ that result by application of \cite[Theorem D]{hhn}
(as used in Theorem \ref{thm:match_to_hkr}) fibre over
$\PP^1 \setminus \{\infty\} \cong \C$, and
the generic fibres are smooth K3 surfaces.

An automorphism group $\Gamma$ as in Definition \ref{def:block} must preserve
the fibration, descending to an action on $\PP^1$ of rotations with fixed
points $\infty$ and (without loss of generality) 0.
Thus there is also a well-defined induced fibration
$(V \times S^1)/\Gamma \to (\C \times S^1)/\Gamma$, where the fibres are
now coassociative (possibly singular) submanifolds of the ACyl \gtmfd{}
$(V \times S^1)/\Gamma$.
In terms of the conventions from \S\ref{2.w}, a generator of
$\Gamma \cong \Z/k$ acts on
$\C = \PP^1 \setminus \{\infty\}$ by $e^\frac{2\pi \eps i}{k}$
and on the external $S^1$ factor by $e^\frac{2\pi i}{k}$, for an $\eps \in \Z$
coprime to $k$ (and only $\eps \mmod k$ affects for the construction itself).

In an extra-twisted connected sum $M$, the truncations $M_{\pm, \ell}$ of
$(V_\pm \times S^1)/\Gamma_\pm$ that are glued together can
thus be viewed as a coassociative fibration over a solid torus
$(\Delta \times S^1)/\Gamma_\pm$ (where $\Delta \subset \C$ is a closed disc).
Along the cylindrical end $\K \times (T^2 \times $ interval), both fibrations
are simply the projection to the second factor,
so they can be patched together to a fibration $f : M \to B$.
Each singular fibre in the fibrations of the building blocks gives rise to
a corresponding $S^1 \subset B$ of singular fibres of $f$.

First we identify the base $B$. It is the result of gluing the two solid tori
$(\Delta \times S^1)/\Gamma_\pm$ by exactly the torus matching $\tormat$ used
in the construction.
After identifying $(\Delta \times S^1)/\Gamma_\pm \cong \Delta \times S^1$ by
$(w,z) \mapsto (wz^{\reps_\pm}, z^k)$,
that amounts to gluing two copies of $\Delta \times S^1$ by the automorphism of
$S^1 \times S^1$ defined by the matrix~$\TorMat$,
which we introduced in~\eqref{2.w.7} as the
change-of-basis matrix from
$(\nu_+. \lambda_+)$ to $(\nu_-, \lambda_-)$.
As per Saveliev \cite[\S1.5]{Saveliev}, $B$~is therefore the lens
space $L(p, \frac{\eps_+p+\glr}{k_+})$.

The fact that $\pi_1 B \cong \Z/p$ could of course also be 
computed with the argument from Proposition \ref{prop:pi1}.
That shows also that $f_* : \pi_1 M \to \pi_1 B$ is an isomorphism.
By passing to the universal cover of $M$ (another extra-twisted connected sum as
described above), we can thus assume that $B \cong S^3$.

When $B \cong S^3$ let us also consider the links corresponding to selected
fibres in $V_+$ and $V_-$ (in particular, the conclusions apply to the
link $L \subset B$ of singular fibres). Because $p = 1$ we may, as described in
Remark~\ref{rmk:eps=r}, choose %
$\reps_\pm=-r_\mp$
so that $\TorMat=\psmatrix{0&1\\1&0}$; then
the gluing of the solid tori corresponds to the Hopf link.

Each $V_\pm$ has a distinguished fibre (above $0 \in \Delta$) that is fixed by
the action of $\Gamma_\pm$. For any non-zero $w \in \Delta$,
the image in $S^3$ of the corresponding circle
$\{w\} \times S^1 \subset (\Delta \times S^1)/\Gamma_\pm$ is (since we took
$\reps_\pm = - r_\mp$) by definition an~$(-r_\mp, k_\pm)$ torus knot.
A set of non-zero $w_1, \ldots, w_m \in \Delta$ yields an
$(-mr_\mp, mk_\pm)$ torus link; in particular, any two of the strands
have linking number $-r_\mp k_\pm$.
Meanwhile, a strand from one side has linking number $k_+k_-$ with each strand
from the other side.

\section{Hyperbolic geometry and \etaforms}
\label{sec:hyp}
In Section~\ref{sec:nu},
we have proved Theorem~\ref{Thm:A}.
The special values of the function~$\logeta$ appearing there
are hard to describe, see Subsection~\ref{app:2},
however, the linear combinations needed in Theorem~\ref{Thm:A}
have a much easier description by Proposition~\ref{Prop3},
leading to the closed formula for~$\bar\nu(M)$ in Theorem~\ref{Thm:B}.
Here, we will pursue a more geometric approach to derive this formula.
We will consider the \etaform appearing in Proposition~\ref{Prop2c.2}
as the \etaform of the tautological family of flat tori
over the upper half plane~$\Hh\subset\C$.
It is a primitive of the hyperbolic area form.
We will then use elementary hyperbolic geometry
and an adiabatic limit formula for \etaforms as in~\cite{BuMa,Liu}
to complete the computation.

\subsection{A universal family of flat tori}\label{2.d}
We extend the \etaform~$\tilde\eta(\Aa)$ that we
introduced in Section~\ref{2.c} to \mbox{$\Hh\times(0,\infty)$},
which we regard as the moduli space of flat tori.
We then apply the so-called transgression formula
\begin{equation}\label{eq:dEta}
  d\tilde\eta(\Aa)=\int_{E/\Hh}\hat A(T(E/\Hh))-\ch(\ind(A_{T^2}))\;,
\end{equation}
see~\cite[Theorem~10.32]{BGV},
where~$\Aa$ extends the Bismut superconnection~\eqref{eq:bismutsc}.
We will see that the integral of the fibrewise $\hat A$-form over the fibres
vanishes.
The index bundle of the fibrewise operators~$A_{T^2}$
consists of fibrewise parallel sections of the spinor bundle.
This will allow us to give a simple formula for its Chern character form,
and hence for~$d\tilde\eta(\Aa)$.
We follow Bismut and Cheeger~\cite{BChTorus}, but consider tori of varying
area.

We consider a universal family~$p\colon\ft\to\Hh\times(0,\infty)$
of flat tori by setting %
\begin{equation*}
  \ft_{(\tau,r)}=\C\bigm/\sqrt{r/y}\Span_\Z\{\tau,1\}
  \qquad\text{for~$\tau=x+iy\in\Hh$ and }r\in(0,\infty)\;,
\end{equation*}
such that~$\ft_{(\tau,r)}$ has area~$r$ with respect to the standard
Euclidean metric on~$\C\cong\R^2$.
The group~$SL(2,\Z)$ acts on~$\Hh$ by M\"obius transformations,
and we lift this action to an action on~$\ft$ by fibrewise isometries
for~$\psmatrix{a&b\\c&d}\in SL(2,\Z)$ by
\begin{equation*}
  \ft_{(\tau,r)}\longrightarrow\ft_{\left(\frac{a\tau+b}{c\tau+d},r\right)}
  \qquad\text{with}\qquad
  [z]\longmapsto\biggl[\frac{\abs{c\tau+d}}{c\tau+d}\cdot z\biggr]
\end{equation*}
We can extend this action to~$GL(2,\Z)$
by~$\tau\mapsto\frac{a\bar\tau+b}{c\bar\tau+d}$ and
\begin{equation*}
  \ft_{(\tau,r)}\longrightarrow\ft_{\left(\frac{a\bar\tau+b}{c\bar\tau+d},r\right)}
  \qquad\text{with}\qquad
  [z]\longmapsto\biggl[\frac{\abs{c\bar\tau+d}}{c\bar\tau+d}\cdot\bar z\biggr]
\end{equation*}
if~$\det\psmatrix{a&b\\c&d}=-1$.
This action corresponds to multiplication
with the inverse transpose
of~$\psmatrix{a&b\\c&d}$ on~$\R^2/\Z^2$
if we trivialise~$\ft$ by
identifying~$\Hh\times(0,\infty)\times\R^2/\Z^2$ with~$\ft_{(\tau,r)}$
via
\begin{equation*}
  \bigl(\tau,r;[(u,v)]\bigr)\longmapsto\bigl[\sqrt{r/y}(u\tau+v)\bigr]\;.
\end{equation*}

Let~$W=\Hh\times(0,\infty)\times\R^2\to\Hh\times(0,\infty)$ be the fibrewise
universal covering of~$\ft$,
regarded as a trivial oriented Euclidean vector bundle with the standard metric.
There is a unique connection~$\nabla^W$ on~$W$
such that the sections~$\sqrt{r/y}\,\psmatrix{1\\0}$
and~$\sqrt{r/y}\,\psmatrix{x\\y}$ are parallel,
and~$\nabla^W$ is clearly $GL(2,\Z)$-invariant.
Let~$T^HW\subset TW$ denote the induced horizontal subbundle of~$TW$,
and let~$T^H\ft$ denote the induced horizontal subbundle of~$T\ft$.
It induces a metric connection
\begin{equation*}
  {}^0\nabla
  =\frac12\,\bigl(\nabla^W+(\nabla^W)^*\bigr)
\end{equation*}
on~$W$,
where~$(\nabla^W)^*$ denotes the adjoint of~$\nabla^W$
with respect to the standard metric.
These connections are given by
\begin{subequations}\label{eq:NablasW}
  \begin{align}
    \nabla^W
    &=d-
    \begin{pmatrix}
      0&1\\0&0
    \end{pmatrix}\,\frac{dx}y+
    \begin{pmatrix}
      1&0\\0&-1
    \end{pmatrix}\,\frac{dy}{2y}-
    \begin{pmatrix}
      1&0\\0&1
    \end{pmatrix}\,\frac{dr}{2r}\;,\label{eq:NablaW}\\
    (\nabla^W)^*
    &=d+
    \begin{pmatrix}
      0&0\\1&0
    \end{pmatrix}\,\frac{dx}y-
    \begin{pmatrix}
      1&0\\0&-1
    \end{pmatrix}\,\frac{dy}{2y}+
    \begin{pmatrix}
      1&0\\0&1
    \end{pmatrix}\,\frac{dr}{2r}\;,\label{eq:NablaWx}\\
    \text{and}\qquad{}^0\nabla
    &=d+
    \begin{pmatrix}
      0&-1\\1&0
    \end{pmatrix}\,\frac{dx}{2y}\;.\label{eq:Nabla0}
  \end{align}
\end{subequations}
As in~\cite[Prop~2.1]{BChTorus},
the vertical tangent bundle~$T(E/\Hh)$ of~$E\to\Hh$ together
with its natural connection is isomorphic to the pullback of~$(W,{}^0\nabla)$.
In particular, the integral of the form~$\hat A(T(E/\Hh))$ over the fibres
in~\eqref{eq:dEta} vanishes.

Because~$W$ is trivial,
there is an associated spinor bundle~$S=S^+\oplus S^-\to\Hh\times(0,\infty)$.
Let~$\widetilde{GL}(2,\R)$ denote the double cover of~$GL(2,\R)$
that is nontrivial over both connected components,
and let~$\widetilde{GL}(2,\Z)\subset\widetilde{GL}(2,\R)$
denote the preimage of~$GL(2,\Z)$.
Then the induced action on~$\ft$ and~$W$ lifts to~$S$ in a way that
is compatible with Clifford multiplication.
Let~$\widetilde{SL}(2,\Z)$ denote the preimage of~$SL(2,\Z)$
in~$\widetilde{GL}(2,\Z)$, then elements
of~$\widetilde{GL}(2,\Z)\setminus\widetilde{SL}(2,\Z)$ swap the bundles~$S^+$
and~$S^-$.
Because the vertical tangent bundle is isomorphic to~$p^*W$,
the bundle~$p^*S$ becomes a fibrewise spinor bundle on~$\ft$.
Moreover,
the kernel of the fibrewise Dirac operator consists of fibrewise parallel
spinors.
Therefore, the index bundle~$\ind(A_{T^2})$ in~\eqref{eq:dEta}
is isomorphic to~$S$,
and the $L^2$-unitary connection~$\tilde\nabla^u$
and the connection~${}^0\nabla$ induce the same connection on the index bundle.

From equation~\eqref{eq:Nabla0} for~${}^0\nabla$,
we can now compute the curvature~${}^0\nabla^2$ and the Euler class
\begin{equation}
  \Pf\biggl(\frac{{}^0\nabla^2}{2\pi}\biggr)
    =\frac{dx\,dy}{4\pi y^2}=\frac1{4\pi}\,dA_{\mathrm{hyp}}\;,\label{eq:Nabla02}
\end{equation}
where~$A_{\mathrm{hyp}}$ denotes the hyperbolic area form.
Here, we have used that the hyperbolic metric on~$\Hh$
is given by~$g^{\mathrm{hyp}}=\frac1{y^2}\,g^{\mathrm{Eucl}}$.

\subsection{The \etaform} %
\label{2.u}
We collect some properties of
the \etaform~$\eta(\Aa)\in\Omega^\bullet(\Hh\times(0,\infty))$.

The data considered above suffice to define the Bismut superconnection~$\Aa$
for the spinor bundle~$S$ on~$E\to\Hh\times(0,\infty)$,
which extends the superconnection~$\Aa$ introduced
in Proposition~\ref{Prop2c.2}; see also~\eqref{eq:bismutsc}.
It follows that~$\Aa$ is $\widetilde{GL}(2,\Z)$-invariant.

\begin{prop}\label{prop:EtaNoR}
  The \etaform~$\tilde\eta(\Aa)\in\Omega^\bullet\bigl(\Hh\times(0,\infty)\bigr)$
  equals the pullback of its restriction to~$\Hh\times\{1\}$ along the
  product projection.
\end{prop}

\begin{proof}
  We identify sections of~$p_*S\to\Hh\times(0,\infty)$
  with sections of~$\ft$.
  Then the unitary connection~$\tilde\nabla^u$ on~$p_*S$
  equals~${}^0\nabla+\frac{dr}{2r}$.
  Let~$D_{(\tau,r)}$ denote the fibrewise Dirac operator acting
  on sections of~$S|_{E_{(\tau,r)}}$.
  Then the Bismut superconnection~$\Aa$ is given by %
  \begin{align*}
    \Aa_t&=\sqrt t\,D_{(\tau,r)}+{}^0\nabla+\frac{dr}{2r}
    =\sqrt{t/r}\,D_{(\tau,1)}+{}^0\nabla+\frac{dr}{2r}\;.%
  \end{align*}
  For~$r=1$, this is explained in~\cite{BChTorus} after Definition~2.14,
  in particular, there is no term of horizontal degree~$2$.
  For the general definition of~$\Aa_t$, see~\cite[Proposition~10.15]{BGV}.
  
  Introducing~$t\in(0,\infty)$ as an additional parameter,
  we may define a superconnection~$\overline\Aa$ on the pullback
  of~$p_*S$ to~$\Hh\times(0,\infty)^2$ by
  \begin{equation*}
    \overline\Aa
    =\sqrt{t/r}\,D_{(\tau,1)}+{}^0\overline\nabla+\frac{dr}{2r}-\frac{dt}{2t}\;,
  \end{equation*}
  where~${}^0\overline\nabla={}^0\nabla+\frac\del{\del t}\,dt$.
  Let~$(2\pi i)^{\frac{1-N}2}$ denote the operator that multiplies a $k$-form
  by~$(2\pi i)^{\frac{1-k}2}$.
  Then the \etaform on~$\Hh\times(0,\infty)$ can be defined
  as
  \begin{equation}
    \begin{aligned}
      \tilde\eta(\Aa)
      &=-(2\pi i)^{\frac{1-N}2}\int_0^\infty
      \str\biggl(\frac{\del\Aa_t}{\del t}\,e^{-\Aa_t^2}\biggr)\,dt\\
      &=-\int_{\Hh\times(0,\infty)^2/\Hh\times(0,\infty)}(2\pi i)^{-\frac N2}
      \str\Bigl(e^{-\overline\Aa^2}\Bigr)\;.\label{eq:TotalEta}
    \end{aligned}
  \end{equation}
  The component of degree~$1$ on~$\Hh\times(0,\infty)$
  is described by~\eqref{eq:EtaForm}.
  Note that~$\frac{dr}{2r}-\frac{dt}{2t}$ squares to~$0$
  and supercommutes with the rest of~$\overline\Aa$, and hence 
  contributes neither to~$\overline\Aa^2$ nor to~$\tilde\eta(\Aa)$.

  We observe that~$\overline\Aa$
  can be pulled back from~$\Hh\times(0,\infty)$
  by the map~$(\tau,r,t)\mapsto(\tau,u)$ with~$u=t/r$.
  It follows from~\eqref{eq:TotalEta} that~$\tilde\eta(\Aa)$ does not
  involve the exterior variable~$dr$.
  By substituting~$u$ for~$\frac tr$ in the integral,
  one obtains a formula that no longer depends on~$r$.
\end{proof}

In other words, $\tilde\eta(\Aa)$ is independent of~$r\in(0,\infty)$
and does not contain the exterior variable~$dr$.
From now on, we regard~$\tilde\eta(\Aa)$ as a form on~$\Hh$.
Following Bismut and Cheeger, we get an explicit expression for the right
hand side of~\eqref{eq:dEta} even if the area of the fibres is not constant.

\begin{thm}[{\cite[Theorem~2.22]{BChTorus}}]\label{Thm2d.1}
  The \etaform~$\tilde\eta(\Aa)$ on~$\Hh$
  has the exterior derivative
  \begin{equation*}
    d\tilde\eta(\Aa)
    =-(-1)^{\tfrac{\rk W}2}\,\Pf\biggl(\frac{{}^0\nabla^2}{2\pi}\biggr)
    \,\hat A^{-1}\biggl(\frac{{}^0\nabla^2}{2\pi}\biggr)\;.
  \end{equation*}
\end{thm}

Hence, in our setting, by~\eqref{eq:Nabla02}, we have
\begin{equation}\label{2.d.5}
    d\tilde\eta(\Aa)=\frac1{4\pi}\,dA_{\mathrm{hyp}}\;.
\end{equation}

\begin{rmk}
  The \etaform is not exact on~$\Hh$.
  This does not contradict Proposition~\ref{Prop2c.2}.
  If we were to leave the path in~$\Hh$ given by the adiabatic limit
  construction in section~\ref{2.c},
  the vertical tangent bundle of the family~$\ft_\pm$ would no longer
  split as in the proof of Proposition~\ref{Prop2c.1},
  so the local variation terms in~\eqref{2.c.1} would no longer vanish
  and contribute to~$d\bar\nu$ as well.
\end{rmk}

\begin{lem}\label{Lem2d.1}
  The spinorial \etaform~$\tilde\eta(\Aa)$ is $PGL(2,\Z)$-equivariant,
  more precisely, for~$g\in PGL(2,\Z)$ acting on~$\Hh$ by M\"obius
  transformations, we have
  \begin{equation*}
    g^*\tilde\eta(\Aa)=\det g\cdot\tilde\eta(\Aa)\;.
  \end{equation*}
\end{lem}

\begin{proof}
  The \etaform is invariant under orientation preserving
  spin isometries.
  We know that each~$g\in PGL(2,\Z)$ has two possible lifts to~$GL(2,\Z)$
  that act fibrewise on~$\ft$ over the given action on~$\Hh$.
  Each of these lifts has two lifts to~$\widetilde{GL}(2,\Z)$ that also
  act on the spinor bundle~$S\to\Hh$ and therefore also on the fibrewise
  spinor bundle~$p^*S$ over~$\ft$.

  If~$g\in PSL(2,\Z)$, all four lifts preserve the superconnection~$\Aa$
  and the subbundles~$S^\pm\subset S\to\Hh$.
  Therefore~$g^*\tilde\eta(\Aa)=\tilde\eta(\Aa)$ in this case.

  If~$g\in PGL(2,\Z)\setminus PSL(2,\Z)$,
  then all four lifts of~$g$ to~$\widetilde{GL}(2,\Z)$
  preserve the superconnection~$\Aa$,
  but swap the bundles~$S^+$ and~$S^-$.
  This reverses the sign of the supertrace in the definition of
  the \etaform, so we now have~$g^*\tilde\eta(\Aa)=-\tilde\eta(\Aa)$.
\end{proof}

\subsection{Adiabatic limits and hyperbolic geodesics}\label{2.v}
We represent the two isometric tori \mbox{$T_\pm=\widetilde T_\pm/\Gamma_\pm$}
by points in the upper halfplane~$\Hh$.
By Proposition~\ref{prop:EtaNoR},
we may rescale all tori
to area~$1$ without changing the contribution to the \nuinvt.
When we consider adiabatic limits of~$M_\pm$,
the points corresponding to~$X_{\pm,a}$ trace out geodesic arcs in~$\Hh$.
These arcs will be used to compute the sum of integrals
of the \etaform~$\tilde\eta(\Aa)$ of Proposition~\ref{Prop2c.2}.

We represent~$T=\widetilde T_+/\Gamma_+$ by the basis~\eqref{2.w.5}.
In equation~\eqref{2.b.0} we have considered families of metrics on~$M_\pm$.
These induce two families of metrics on~$T_+$.
We write~$X_{+,a}=\del M_{+,a}=\Sigma_+\times T_{+,a}$.
For the second one, we consider the isomorphism~$T\cong\widetilde T_-/\Gamma_-$
and write~$X_{-,a}=\del M_{-,a}=\Sigma_-\times T_{-,a}$.

We let~$\overline{\Hh}=\Hh\cup\R\cup\{\infty\}$
denote the closure of~$\Hh$ in~$\C P^1$
and write~$\del_\infty\Hh=\R\cup\infty$.
We extend the action of~$GL(2,\Z)$ by M\"obius transformations
to~$\overline\Hh$.
For~$(z,w)\in\C^2$
we put
\begin{equation*}
  \tau=
  \begin{cases}
    \infty	&\text{if~$w=0$,}\\
    z/w		&\text{if~$\Im(z/w)\ge 0$, and}\\
    \bar z/\bar w&\text{if~$\Im(z/w)<0$.}
  \end{cases}
\end{equation*}
If~$(z,w)$ span an integral lattice~$\Lambda$ in~$\C$,
then~$\C/\Lambda$ is isometric to~$\ft_{(\tau,r)}$
for~$r=\abs{\Im(\bar zw)}$.

We recall the gluing matrix~$G=\psmatrix{\gll&p\\\glb&\sglr}$
from~\eqref{2.w.6}.
For the rest of this section,
we will assume that~$p$, $\glb>0$ and~$r_\pm\ge 0$.
This is no loss of generality by Remark~\ref{rmk:geosym}.
As a consequence, we always have a gluing angle~$\thet\in(0,\frac\pi 2]$
and~$\rho=\pi-2\thet\ge 0$.

\begin{lem}\label{Lem2v.1}
  Consider the families of flat tori~$T_{+,a}$ and~$T_{-,a}$ as above,
  assuming~$\gll$, $\glb\ge 0$.
  \begin{enumerate}
  \item\label{2y.1a} The family~$T_{+,a}$ is represented in~$\Hh$
    by a vertical line~$\gamma_+$ with real part~$\frac{\eps_+}{k_+}$.
    The adiabatic limit~$a\to 0$ corresponds
    to~$\frac{\eps_+}{k_+}\in\del_\infty\Hh$.
  \item\label{2y.1b} The family~$T_{-,a}$ is represented in~$\Hh$
    by a hyperbolic geodesic~$\gamma_-$
    between~$\frac{\eps_+}{k_+}-\frac \glb{k_+\gll}$
    and~$\frac{\eps_+}{k_+}+\frac \glr{k_+p}\in\del_\infty\Hh$.
    The adiabatic limit~$a\to 0$ corresponds
    to~$\frac{\eps_+}{k_+}-\frac \glb{k_+\gll}$.
  \item\label{2y.1c} The geodesics~$\gamma_+$ and~$\gamma_-$
    intersect at~$\frac{\eps_++i\ar_+}{k_+}$
    with unoriented angle
    \begin{equation*}
      \abs{\measuredangle_{\frac{\eps_++i\ar_+}{k_+}}\biggl(\frac{\eps_+}{k_+}\;,\;
      \frac{\eps_+}{k_+}-\frac \glb{k_+\gll}\biggr)}=
      2\thet\;.
    \end{equation*}
  \end{enumerate}
\end{lem}

For classical twisted connected sums
and in the case of Example~\ref{ex:TorusMatch}~\ref{exnonex.1},
we have~$2\thet=\pi$,
so both~$\gamma_-$ and~$\gamma_+$ agree with the vertical line
with real part~$\frac{\eps_+}{k_+}$.

\begin{proof}
  We write~$T_{+,a}=\C/\Lambda_{+,a}$, where the lattice~$\Lambda_{+,a}$
  is generated by~$\frac{\eps_++ia}{k_+}\,\zeta_+$ and~$\zeta_+\in\C$,
  so~$T_{+,a}$ %
  is represented by the point~$\frac{\eps_++ia}{k_+}\in\Hh$
  on the hyperbolic geodesic from~$\frac{\eps_+}{k_+}$ to~$\infty$.

  Analogously, the torus~$T_{-,a}=\C/\Lambda_{-,a}$ can be represented
  by~$\frac{\eps_-+ia}{k_-}$
  on the hyperbolic geodesic from~$\frac{\eps_-}{k_-}$ to~$\infty$.
  We now consider the inverse transpose of the matrix~$\TorMat$
  introduced in~\eqref{2.w.7}. It is given by
  \begin{equation*}
    \begin{pmatrix}
      e&f\\g&h
    \end{pmatrix}=
    \begin{pmatrix}
      \frac{\reps_+p+\glr}{k_+} &\frac{\glb-\reps_+\gll-\reps_-\glr-\reps_-\reps_+p}{k_-k_+}\\
      p&-\frac{\reps_-p+\gll}{k_-}
    \end{pmatrix}
  \end{equation*}
  and has determinant~$-1$.
  Assertion~\ref{2y.1b} follows because the antiholomorphic M\"obius
  transformation~$\tau\mapsto\frac{e\bar\tau+f}{g\bar\tau+h}$ maps
  \begin{equation*}
    \frac{\eps_-}{k_-}\longmapsto\frac{\eps_+}{k_+}-\frac \glb{k_+\gll}\;,\qquad
    \infty\longmapsto\frac{\eps_+}{k_+}+\frac \glr{k_+p}\;,\qquad\text{and}\qquad
    \frac{\eps_-+i\ar_-}{k_-}\longmapsto\frac{\eps_++i\ar_+}{k_+}\;.
  \end{equation*}
  In the last step, we have used the formulas for~$\ar_\pm$
  in Proposition~\ref{Prop2w.1}~\ref{2w.1c}.

  We compute the angle in~\ref{2y.1c} using Figure~\ref{Fig2w.1}.
  We note that the hyperbolic upper half plane is conformal
  to the Euclidean half plane, so we may compute the angle
  using Euclidean geometry.
  Let~$c$ denote the Euclidean center of the circle through
  the points~$\frac{\eps_+}{k_+}-\frac \glb{k_+\gll}$,
  $\frac{\eps_+}{k_+}+\frac \glr{k_+p}$ and~$\frac{\eps_++is}{k_+}$.
  The angle between the two hyperbolic geodesic arcs from~$\frac{\eps_+}{k_+}$
  and~$\frac{\eps_+}{k_+}-\frac \glb{k_+\gll}$ to~$\frac{\eps_++is}{k_+}$
  equals the central angle subtending the arc
  from~$\frac{\eps_+}{k_+}-\frac \glb{k_+\gll}$ to~$\frac{\eps_++is}{k_+}$.
  It is therefore twice the inscribed angle
  at~$\frac{\eps_+}{k_+}+\frac \glr{k_+p}$,
  which we recognise as the gluing angle~$\thet\in\bigl(0,\frac\pi2\bigr]$
  from Proposition~\ref{Prop2w.1}~\ref{2w.1d}.
  This is easiest to see by considering the line from~$\frac{\eps_+\gll-\glb}{k_+\gll}$
  to~$\frac{\eps_++i\ar_+}{k_+}$ with direction~$\glb+i\gll\ar_+$,
  see also Figure~\ref{Fig2g.1}.
  Here we have used our assumption that~$\gll$, $\glb\ge 0$.
\end{proof}

  \begin{figure}
\begin{equation*}
  \begin{tikzpicture}
    \draw (2,0) node[below] {$\frac{\eps_+p+\glr}{k_+p}$} arc (0:60:2) ;
    \begin{scope}[color=green, line width=1.5pt]
      \draw (-2,0)
      arc (180:60:2) node[pos=0.5,above,color=black] {$\gamma_-$} -- (1,0) ;
    \end{scope}
    \node at (1.3,0.7) {$\gamma_+$} ;
    \begin{scope}[line width=2pt]
      \draw (-2.5,0) -- (2.5,0) ;
      \draw (2,0) -- ++(120:2) -- (-2,0) ;
      \begin{scope}[->,color=red]
        \draw (-2,0) -- ++(30:1) node[above, color=black] {$\del_{v_-}$} ;
        \draw (-2,0) node[below, color=black] {$\frac{\eps_+\gll-\glb}{k_+\gll}$}
        -- (-2,1) node[above, color=black] {$\del_{v_+}$} ;
      \end{scope}
      \begin{scope}[->,color=blue]
        \draw (-2,0) -- (-1,0) node[below, color=black] {$\quad\del_{u_+}$} ;
        \draw (-2,0) -- ++(120:1) node[above left, color=black] {$\del_{u_-}$} ;
      \end{scope}
    \end{scope}
    \draw (0,0) node[below] {$c$} -- ++(60:2)
    ++(150:0.85) -- ++(-30:1.7) ++(150:0.15) arc(-30:90:0.7)
    ++(0,0.15) -- ++(0,-0.85) ;
    \draw (0,0) ++(60:2) ++(30:0.35) node {$2\thet$} ;
    \draw (0,0) ++(60:1.6) arc (240:150:0.4) ;
    \fill (0,0) ++(60:2) ++(188:0.22) circle (1pt) ;
    \draw (-0.7,0) arc (180:60:0.7) ;
    \draw (1.3,0) arc (180:120:0.7) ;
    \draw (-2,0.7) arc (90:30:0.7) ;
    \draw (0,0) ++(120:0.35) node {$2\thet$} ;
    \draw (-2,0) ++(60:0.45) node {$\thet$} ;
    \draw (2,0) ++(150:0.45) node {$\thet$} ;
    \draw (0.6,0) arc (180:90:0.4) ;
    \fill (1,0) node[below] {$\frac{\eps_+}{k_+}$} ++(135:0.22) circle (1pt) ;
  \end{tikzpicture}
\end{equation*}
\caption{The hyperbolic angle between~$\gamma_+$ and~$\gamma_-$.}
\label{Fig2w.1}
  \end{figure}

\subsection{Axes of hyperbolic reflections}\label{2.z}
The isometry group of the upper half plane with its hyperbolic metric
is isomorphic to~$PGL(2,\R)$, acting by M\"obius transformations.
Each orientation reversing isometry of~$\Hh$
is a hyperbolic glide reflection along a hyperbolic geodesic.
By Lemma~\ref{Lem2d.1},
the \etaform~$\tilde\eta(\Aa)$ changes sign under pullback by elements~$g$
of~$GL(2,\Z)\setminus SL(2,\Z)$.
If~$g$ is a reflection, then the restriction of~$\tilde\eta(\Aa)$
to its axis vanishes.
If~$g$ describes a reflection about~$\gamma_g$,
then~$hgh^{-1}$ is a reflection about~$h(\gamma_g)$.

\begin{rmk}\label{Rem2y.1}
  Let~$\bigl(\begin{smallmatrix}a&b\\c&d\end{smallmatrix}\bigr)\in
  GL(2,\Z)\setminus SL(2,\Z)$
  represent the map~$\tau\mapsto\frac{a\bar\tau+b}{c\bar\tau+d}$.
  It describes a reflection if and only if the trace~$a+d$ vanishes.
  The line of reflection is vertical if and only if ${c=0}$.
  In this case,
  the corresponding reflections in~$PSL(2,\Z)$ are of the
  form~$\psmatrix{-1&k\\0&1}$,
  and the fixed line has real part~$\frac k2$.
  This is implies that all possible axes of reflections
  in~$GL(2,\Z)\setminus SL(2,\Z)$
  are of the form~$h(\gamma)$ for~$h\in SL(2,\Z)$ and~$\gamma$
  a vertical geodesic with real part in~$\frac12\Z$.
\end{rmk}

\begin{lem}\label{Lem:fixedgeo}
  Let~$\frac ab$ and~$\frac cd\in\del_\infty\Hh=\Q\cup\{\infty\}$
  be represented by reduced fractions.
  If \nolinebreak[3] ${ad-bc=\pm 1}$, then there exists a hyperbolic reflection,
  represented by an element~$g\in GL(2,\Z)\setminus SL(2,\Z)$,
  that fixes the hyperbolic geodesic~$\gamma$ from~$\frac ab$ to~$\frac cd$.
\end{lem}

\begin{proof}
  Assume~$ad-bc=1$.
  The element~$h=\psmatrix{a&c\\b&d}\in SL(2,\Z)$ maps~$\infty=\frac 10$
  to~$\frac ab$ and~$0=\frac 01$ to~$\frac cd$.
  Hence, it maps the $y$-axis to~$\gamma$.
  This implies
  that~$g=h\,\psmatrix{-1&0\\0&1}\,h^{-1}\in GL(2,\Z)\setminus SL(2,\Z)$
  is a reflection that fixes~$\gamma$.
\end{proof}

\subsection{Continued Fractions and Hyperbolic Polygons}\label{2.ex}
Let~$M_\pm=\widetilde M_\pm/\Gamma_\pm$ be $\Z/k_\pm$-blocks with
boundary~$X_\pm\cong\Sigma_\pm\times T_\pm$ and~$T_\pm=\widetilde T_\pm/\Gamma_\pm$
as before.
Define~$\eps_+$, $\ar_\pm$ as in Section~\ref{2.w},
and let~$\gamma_\pm$ be the hyperbolic geodesics of Lemma~\ref{Lem2v.1}.
Recall that we assumed~${\glb>0}$ and~$\gll\ge 0$.
To compute the difference of~$\bar\nu(M_{+,\ar_+})$ and the adiabatic
limit~$\lim_{a\to 0}\bar\nu(M_{+,a})$ by formula~\eqref{2.c.3} for~$M_+$,
we integrate~$\tilde\eta(\Aa)$ along the vertical line~$\gamma_+$
from~$\frac{\eps_+}{k_+}$ to~$\frac{\eps_++i\ar_+}{k_+}$.
For~$M_-$, we note that the orientation of the boundary was reversed during
gluing.
Hence, we integrate~$\tilde \eta(\Aa)$ along the geodesic ray~$\gamma_-$
from~$\frac{\eps_++is}{k_+}$ to~$\frac{\eps_+}{k_+}-\frac \glb{k_+\gll}\in\del_\infty\Hh$.
Note that this last point is~$\infty$ in case~$\gll=0$ and~$\thet=\frac\pi2$.
It follows from Theorem~\ref{Thm2x.1} and Proposition~\ref{Prop2c.2}
that the integrals exist.

We will now complete the two geodesic rays above to an ideal
hyperbolic polygon~$P$ with one finite corner~$\frac{\eps_++i\ar_+}{k_+}$
and further corners represented by reduced
fractions~$\frac{\eps_+}{k_+}-\frac \glb{k_+\gll}=\frac{a_0}{b_0}$, \dots,
$\frac{a_\ell}{b_\ell}=\frac{\eps_+}{k_+}\in\Q\cup\{\infty\}\subset\del_\infty\Hh$
such that~$a_jb_{j-1}-b_ja_{j-1}=1$ for~$j=1$, \dots, $\ell$.
Then~$\tilde\eta(\Aa)$ vanishes along the geodesics
joining~$\frac{a_{j-1}}{b_{j-1}}$ and~$\frac{a_j}{b_j}$
by Lemma~\ref{Lem:fixedgeo}.

It turns out to be easier
to construct the image~$P'$ of~$P$ under~$C$,
where~$C=\psmatrix{\eps_+^*&-n\\-k_+&\eps_+}\in SL(2,\Z)$
act as a M\"obius transformation,
where the integer~$n$ is defined by~$\eps_+\eps_+^*=k_+n+1$.
We note that
\begin{equation}\label{eq:ab0abl}
  C\,\Bigl(\frac{\eps_+}{k_+}\Bigr)=\infty\;,\quad
  C\,\Bigl(\frac{\eps_+\gll-\glb}{k_+\gll}\Bigr)=\frac{\gll-\eps_+^*\glb}{k_+\glb}\;,
  \quad\text{and}\quad
  C\,\Bigl(\frac{\eps_+p+\glr}{k_+p}\Bigr)=-\frac{p+\eps_+^*\glr}{k_+\glr}\;.
\end{equation}

We let~$a'_0=\frac{\gll-\eps_+^*\glb}{k_+}$, $b'_0=\glb\in\Z$
(see also Remark~\ref{rmk:divisible}) and
represent~$\frac{a'_0}{b'_0}$ as a continued fraction with minus signs,
\begin{equation*}
  \frac{a'_0}{b'_0}=\frac{\gll-\eps_+^*\glb}{k_+\glb}
  =c_1
  -\frac1{c_2-\punkt_{\textstyle\punkt_{\textstyle\punkt_{\tfrac 1{c_\ell}}}}}\;.
\end{equation*}
This way, we obtain a sequence of integers~$c_1$, \dots, $c_\ell$
with~$c_2$, \dots, $c_\ell\ge 2$ and reduced fractions
\begin{equation*}
  \frac{a'_0}{b'_0}\quad<\quad
  \frac{a'_1}{b'_1}
  =c_1-\frac1{c_2-\punkt_{\textstyle\punkt_{\textstyle\punkt_{\tfrac 1{c_{\ell-1}}}}}}
  \quad<\quad\cdots\quad<\quad
  \frac{a'_{\ell-1}}{b'_{\ell-1}}=\frac{c_1}1\quad\text{and}\quad
  \frac{a'_\ell}{b'_\ell}=\frac10\;.
\end{equation*}
As explained in~\cite[\S V]{Zagier},
the numbers~$a'_j$, $b'_j$ and~$c_j$ are related by the formula
\begin{equation}\label{eq:ElMat}
  \begin{pmatrix}
    a'_j&-a'_{j+1}\\
    b'_j&-b'_{j+1}
  \end{pmatrix}=
  \begin{pmatrix}
    c_1&-1\\1&0
  \end{pmatrix}\cdots
  \begin{pmatrix}
    c_{\ell-j}&-1\\1&0
  \end{pmatrix}\in SL(2,\Z)
\end{equation}
for~$j=0$, \dots, $\ell-1$,
which also shows that~$a'_{j+1}b'_j-a'_jb'_{j+1}=1$.

\begin{figure}
  \begin{tikzpicture}[scale=0.6]
    \draw (-11.5,0) -- (-4.5,0) (-0.5,0) -- (10.5,0) ;
    \draw (-5,0) node[below] {$\tfrac{\eps_+p+\glr}{k_+p}$}
	arc(0:70.5:3) node[above right] {$\tfrac{\eps_++i\ar_+}{k_+}$}
	(-7,2.828) -- (-7,4) ;
    \draw[blue] (-7,0) arc(0:180:2) ;
    \draw[line width=1.5pt, join=round] (-11,0)
	node[below] {$\tfrac{a_0}{b_0}$}
	arc(180:0:0.5) arc(180:0:0.833) arc(180:0:0.417)
        arc(180:0:0.25) node[below] {$\tfrac{a_\ell}{b_\ell}$};
    \draw[red,line width=1.5pt] (-7,0) -- (-7,2.828) ;
    \draw[green,line width=1.5pt] (-11,0) arc(180:70.5:3) ;
    \node at (-8.3,1.9) {$P$} ;
    \draw[line width=1pt,->] (-3.5,1.5) -- node[above]{$C$} (-1.5,1.5) ;
    \draw (0,0) node[below] {$-\tfrac{p+\eps_+^*\glr}{k_+\glr}$}
	arc(180:70.5:1.5)
	node[above left] {$\tfrac{is_+^{-1}-\eps_+^*}{k_+}$} ;
    \draw (2,1.414) -- (2,0) node[below] {$\scriptstyle-\frac{\eps_+^*}{k_+}$} ;
    \draw[blue] (3,0) -- (3,4) ;
    \draw[line width=1.5pt, join=round] (3,0)
	node[below] {$\tfrac{a'_0}{b'_0}$}
        arc(180:0:0.167) arc(180:0:0.833)
        arc(180:0:2.5) node[below] {$\tfrac{a'_{\ell-1}}{b'_{\ell-1}}$}
        -- (10,4) ;
    \draw[red,line width=1.5pt]  (2,1.414) -- (2,4) ;
    \draw[green,line width=1.5pt] (3,0) arc(0:70.5:1.5) ;
    \node at (5,2.5) {$P'$} ;
  \end{tikzpicture}
  \caption{The hyperbolic polygons~$P$ and~$P'$}\label{fig:HypPoly}
\end{figure}

\begin{rmk}\label{rem:b1x}
  For later use, we note that
  \begin{enumerate}
  \item\label{b1x.1} because~$a'_1b'_0-a'_0b'_1=1$,
    the number~$-b'_1$ is inverse to~$a'_0=\frac{\gll-\eps_+^*\glb}{k_+}$
    modulo~$b'_0=\glb$,
  \item\label{b1x.2} because~$\frac{\glr-\eps_-^*\glb}{k_-}$ is also inverse
    to~$a'_0$ modulo $\glb$ by equation~\eqref{eq:invmodn}, we have
    \begin{equation*}
      b'_1\equiv\frac{\sglr+\eps_-^*\glb}{k_-}\mod \glb\;.
    \end{equation*}
  \end{enumerate}
\end{rmk}

Let~$\frac{a_j}{b_j}$ denote the preimage of~$\frac{a'_j}{b'_j}$
under the Möbius transformation~$C$
for~$0<j<\ell$, the cases~$j=0$, $\ell$ being settled by~\eqref{eq:ab0abl}.
Then~$a_jb_{j-1}-a_{j-1}b_j=1$.
The finite corner at~$\frac{\eps_++i\ar_+}{k_+}$
gets mapped to~$\frac{i\ar_+^{-1}-\eps_+^*}{k_+}$.
Thus we have completed the construction of~$P$ and~$P'$;
see Figure~\ref{fig:HypPoly}.

\subsection{The contribution from the cusps}\label{2.y}
Thanks to Theorem~\ref{Thm2d.1},
we can in principle evaluate the integral of~$\eta(\Aa)$ over
the arcs~$\gamma_+$ and~$\gamma_-$ from Lemma~\ref{Lem2v.1}
by computing the area of the polygon~$P$ above.
However, the polygon~$P$ has cusps,
that is,
rational points in the boundary~$\del_\infty\Hh=\R\cup\{\infty\}$ at infinity.
In this section, we compute the contribution from a cusp as a limit
of the integral of the \etaform
over certain horocyclic arcs that escape to infinity.

Adiabatic limit formulas for \etaforms have been proved
by Bunke, Ma~\cite{BuMa} and Liu~\cite{Liu},
but only %
modulo exact forms.
Here, we have to integrate the \etaform over an interval,
so we need an adiabatic limit formula that holds ``on the nose''.
We will prove such a formula in Section~\ref{A2},
but only for the simple special case at hand.

We want to define a distance between two hyperbolic geodesics
ending in a cusp point~$\frac ef$.
To move~$\frac ef$ to~$\infty=\frac10$, assume that the fraction~$\frac ef$
is reduced and find~$a$ and~$b\in\Z$ such that
\begin{equation}\label{2.y.2}
  ae+bf=1\;.
\end{equation}
Then~$\bigl(\begin{smallmatrix}a&b\\-f&e\end{smallmatrix}\bigr)\in SL(2,\Z)$,
and the M\"obius transformation~$z\mapsto\frac{az+b}{-fz+e}$
rotates the cusp~$\frac ef$ into~$\infty$.
Now consider hyperbolic geodesics starting at~$x$, $y\in\R\cup\{\infty\}$
and ending in~$\frac ef$.
They get rotated by the matrix above to vertical lines
with real parts~$\frac{ax+b}{-fx+e}$ and~$\frac{ay+b}{-fy+e}$.
Because of~\eqref{2.y.2}, the difference is
\begin{equation*}
  \frac{ax+b}{-fx+e}-\frac{ay+b}{-fy+e}
  =\frac{x-y}{(fy-e)(fx-e)}\;.
\end{equation*}

\begin{dfn}\label{Def2y.1}
  The {\em cusp angle\/} between two hyperbolic geodesics
  going from~$x$, $y\in\R\cup\{\infty\}=\del_\infty\Hh$ to
  a cusp point represented by a reduced fraction~$\frac ef\in\Q\cup\{\infty\}$,
  with~$x\ne\frac ef\ne y$,
  is defined as
  \begin{equation*}
    \measuredangle_{\frac ef}(x,y)=\frac{x-y}{(fx-e)(fy-e)}\in\R
  \end{equation*}
  if~$x$, $y\in\R$, and by the obvious extension of this formula
  if one of the points is~$\infty$.
\end{dfn}

Note that~$\measuredangle_{\frac ef}(x,y)$ is a geometric
notion for the covering map~$\Hh\to\Hh/SL(2,\Z)$.
If measures how often a line joining the two geodesics above in the universal
covering space~$\Hh$ winds around the cusp in~$\Hh/SL(2,\Z)$.
The sign is chosen such that oriented ideal triangles
have positive cusp angles.
In particular, cusp angles are $SL(2,\Z)$-invariant.

From~\eqref{eq:ElMat},
we see that~$\psmatrix{a'_j\\b'_j}=C_j\psmatrix{c_{\ell-j}\\1}$,
$\psmatrix{a'_{j+1}\\b'_{j+1}}=C_j\psmatrix{1\\0}$
and~$\psmatrix{a'_{j+2}\\b'_{j+2}}=C_j\psmatrix{0\\-1}$,
with~$C_j=\psmatrix{a'_{j+1}&-a'_{j+2}\\b'_{j+1}&-b'_{j+2}}\in SL(2,\Z)\;.$
Because cusp angles are $SL(2,\Z)$-invariant,
we can now compute the cusp angles of~$P$.
For~$j=0$, \dots, $\ell-2$, we obtain
\begin{subequations}\label{eq:CuspAngles}
  \begin{align}
    \measuredangle_{\tfrac{a_{j+1}}{b_{j+1}}}
    \biggl(\frac{a_j}{b_j},\frac{a_{j+2}}{b_{j+2}}\biggr)
    &=\measuredangle_{\tfrac{a'_{j+1}}{b'_{j+1}}}
    \biggl(\frac{a'_j}{b'_j},\frac{a'_{j+2}}{b'_{j+2}}\biggr)
    =\measuredangle_\infty(c_{\ell-j},0)=c_{\ell-j}\;,\label{eq:CAj}\\
    \measuredangle_{\tfrac{a_0}{b_0}}
    \biggl(\frac{\eps_++i\ar_+}{k_+},\frac{a_1}{b_1}\biggr)
    &=\measuredangle_{\tfrac{a'_0}{b'_0}}
    \biggl(\frac{p+\eps_+^*\glr}{k_+p},\infty\biggr)
    +\measuredangle_{\tfrac{a'_0}{b'_0}}\biggl(\infty,\frac{a'_1}{b'_1}\biggr)
    =\frac{\glr}{k_-\glb}+\frac{b'_1}{b'_0}\;,\label{eq:CA0}\\
    \measuredangle_{\tfrac{a_\ell}{b_\ell}}
    \biggl(\frac{a_{\ell-1}}{b_{\ell-1}},\frac{\eps_++i\ar_+}{k_+}\biggr)
    &=\measuredangle_\infty\biggl(c_1,-\frac{\eps_+^*}{k_+}\biggr)
    =c_1+\frac{\eps_+^*}{k_+}\;.\label{eq:CAl}
  \end{align}
\end{subequations}
For~\eqref{eq:CA0},
we have used Lemma~\ref{Lem2v.1}~\ref{2y.1b} and equation~\eqref{eq:ab0abl}.

\begin{prop}\label{Prop2y.1}
  Let~$\frac ef\in\Q\cup\{\infty\}\subset\del_\infty\Hh$
  be a cusp point, and let~$x$, $y\in\del_\infty\Hh$.
  Assume that~$\alpha_r$ is a family of horocyclic arcs centered at~$\frac ef$
  from the geodesic from~$\frac ef$ to~$x$ to the geodesic from~$\frac ef$
  to~$y$ that converges to~$\frac ef$ as~$r\to\infty$.
  Then
  \begin{equation*}
    \lim_{r\to\infty}\int_{\alpha_r}\tilde\eta(\Aa)
    =\frac{\measuredangle_{e/f}(x,y)}{12}\;.
  \end{equation*}
\end{prop}

\begin{proof}
  By the considerations above, we can rotate the cusp point~$\frac ef$
  to~$\infty$.
  Hence we consider the horocycle~$x\mapsto x+iy$ for large fixed~$y$
  and prove that
  \begin{equation}\label{2.d.12}
    \lim_{y\to\infty}\tilde\eta(\Aa)|_{\R+iy}
    =-\frac{dx}{12}\;.
  \end{equation}

  We pull the universal family~$\ft\to\Hh\times(0,\infty)$
  from Section~\ref{2.d} back to~$\R$ by~$x\mapsto(x+iy,y)$.
  Let us write~$\ft_{x+iy}$ for~$\ft_{(x+iy,y)}=\C/\Span_\Z\{x+iy,1\}$
  and~$F_x=\R/y\Z$,
  and consider the map~$\ft_{x+iy}\to F_x$ given by the imaginary part.
  Then we obtain a family of fibred tori,
\begin{equation}\label{2.d.13}
  \begin{CD}
    S^1_1@>>>\ft_{x+iy}@>>>\ft\\
    &&@V\Im VV@VVV\\
    &&F_x@>>>F\\
    &&&&@VVV\\
    &&&&\R\rlap{$\mathord{}\owns x\;.$}&\qquad
  \end{CD}
\end{equation}
Here, the interior circles form the fibres of~$\ft_{x+iy}$ of length~$1$,
and the base~$F_y$ can be identified with the exterior circles
in~$\widetilde T_+$.
In particular, the orientation of~$\ft_{x+iy}$ agrees with the one
in~\cite{BuMa,Liu}.
Let~$T^H\ft$ denote the horizontal subbundle
for the fibration~$\ft\to\R$ induced
by the pullback of the connection~$\nabla^W$ from~\eqref{eq:NablaW}.
The Euclidean metric on~$\ft_{x+iy}$ defines a connection
for the fibration~$\ft\to F$ with holonomy~$-x\in\R/\Z$.
In Figure~\ref{Fig2y.1}, 
the dashed line is horizontal with respect to this connection.

\begin{figure}
  \begin{equation*}
    \begin{tikzpicture}
      \begin{scope}[line width=1.5pt]
        \draw (0,0) node[below] {$0$} -- (1,0) node[below] {$1$}
        -- node[right]{$\ft_{x+iy}$} (1.3,3)
        -- (0.3,3) node[above] {$x+iy$} -- cycle ;
        \draw (-3,0) node[left] {$0$}
        -- node[left] {$F_x$} (-3,3) node[left] {$y$} ;
        \draw[->] (-0.9,1.5) -- (-2,1.5) ;
      \end{scope}
      \draw[dash pattern=on 3pt off 4pt] (0.5,0) -- (0.5,3) ;
    \end{tikzpicture}
  \end{equation*}
  \caption{The torus $\ft_{x+iy}$ as a bundle.}\label{Fig2y.1}
\end{figure}

Let~$L\to F$ denote a Hermitian line bundle with a fibrewise flat
connection that contains~$\ft\to F$ as a circle bundle.
Then over~$F_x$, the bundle~$L$ has holonomy~$e^{-2\pi ix}$.
Let~$\phy$ be a coordinate on~$F_x=\R/y\Z$.
The bundle~$L$ carries a Hermitian connection~$\nabla^L$
that we can describe as
\begin{equation*}
  \nabla^L=d+\frac{2\pi ix}y\,d\phy
\end{equation*}
in a trivialisation over a neighbourhood of~$F_x$ in~$F$.
We compute curvature and first Chern form of~$(L,\nabla^L)$ on~$F$ as
\begin{equation*}
  (\nabla^L)^2=\frac{2\pi i}y\,dx\,d\phy\qquad\text{and}\qquad
  c_1(\nabla^L)=-\frac1y\,dx\,d\phy\;.
\end{equation*}

By Proposition~\ref{Prop2y.2} below, we have
\begin{equation}\label{2.c.2}
  \lim_{y\to\infty}\tilde\eta(\Aa)|_{\R+iy}
  =\int_{F/\R}\tilde\eta(\Aa')\;,
\end{equation}
where now~$\Aa'$ is the superconnection of the fibrewise spinor bundle
over the circle bundle ${\ft\to F}$.
The corresponding \etaform of~$\ft\to F$ was computed by Zhang
in \cite[Theorem 1.7]{ZhRokhlin}. %
Here, it only has a component of degree~$2$, given by
\begin{equation*}
  \tilde\eta(\Aa')=\frac1{12}\,c_1\bigl(\nabla^L\bigr)
  =-\frac1{12y}\,dx\,d\phy\;.
\end{equation*}
Integration over the fibre of~$F\to\R$ of length~$y$ proves~\eqref{2.d.12}.
\end{proof}

\begin{figure}
  \begin{equation*}
    \begin{tikzpicture}
      \begin{scope}[dash pattern=on 3pt off 4pt]
        \draw (-1,3) -- (3,3) ;
        \draw (0,1.155) arc (90:390:0.577) (2,1.155) arc (450:150:0.577) ;
      \end{scope}
      \begin{scope}[line width=1.5pt]
        \draw (-1,0) -- (3,0) ;
        \draw[fill=lightgray] (0,1.155) arc(90:30:0.577) arc(120:60:1)
        arc(150:90:0.577) -- (2,3) -- (0,3) -- cycle ;
        \draw[->] (2,3) -- (0.9,3) ;
        \draw[->] (0.5,0.866) arc (120:86:1) ;
      \end{scope}
      \draw (0,3.5) -- (0,0) node[below] {$0$}
      arc(180:0:1) node[below] {$1$} -- (2,3.5) ;
      \node at (1,2) {$\Delta_r$} ;
    \end{tikzpicture}
  \end{equation*}
  \caption{An ideal triangle, truncated by horocyclic arcs.}\label{Fig2d.1}
\end{figure}

\begin{ex}\label{Exp2d.1}
  Consider the ideal triangle~$\Delta$ of area~$\pi$ with corners~$0$, $1$
  and~$\infty\in\del_\infty\Hh$.
  By Remark~\ref{Rem2y.1} and Lemma~\ref{Lem:fixedgeo},
  the form~$\tilde\eta(\Aa)$ vanishes on its sides.
  All cusp angels equal~$1$.
  Let~$\Delta_r$ denote truncations of~$\Delta$ by horocyclic arcs centered
  at the corners that converge
  to the full triangle~$\Delta$ as~$r\to\infty$; see Figure~\ref{Fig2d.1}.
  Then by Theorem~\ref{Thm2d.1} and Proposition~\ref{Prop2y.1},
  we check that
  \begin{equation*}
    \lim_{r\to\infty}\int_{\del\Delta_r}\tilde\eta(\Aa)
    =\frac{\measuredangle_0(\infty,1)+\measuredangle_1(0,\infty)
      +\measuredangle_\infty(1,0)}{12}
    =\frac14=\frac{A_{\mathrm{hyp}}(\Delta)}{4\pi}
    =\int_\Delta d\tilde\eta(\Aa)\;.
  \end{equation*}
\end{ex}

The following Proposition is inspired by
a result~\cite[Theorem~5.11]{BuMa} of Bunke and Ma;
see also Liu in~\cite[Theorem~1.3]{Liu},
where the following equality is proved up to exact forms.
Moreover, Liu assumes that the fibrewise Dirac operator of the
fibration~$F\to\R$ is invertible~\cite[Ass~3.1]{Liu},
which is not the case here.

\begin{prop}\label{Prop2y.2}
  In the situation above,
  let~$\Aa'$ denote the superconnection associated with the fibrewise spin Dirac
  operator on the bundle~$\ft\to F$.
  Then
  \begin{equation*}
    \lim_{y\to\infty}\tilde\eta(\Aa)|_{\R+iy}
    =\int_{F/\R}\tilde\eta(\Aa')\;.
  \end{equation*}
\end{prop}

We postpone the proof to Section~\ref{A2}.

\subsection{Evaluation of the Eta Form
Integrals}
\label{2.ey}
With all preliminaries understood, we can now prove Theorem~\ref{Thm:B}.
We start by integrating the \etaform from Proposition~\ref{Prop2c.2}
along the geodesic rays~$\gamma_+$ and~$\gamma_-$.
These integrals exist by Theorem~\ref{Thm2x.1} and Proposition~\ref{Prop2c.2}.

\begin{thm}\label{Thm2e.1}
  Assume that~$\gll\ge 0$, $\glb>0$. Then we have
  \begin{multline}\label{2.e.8}
    \bar\nu(M_{+,\ar_+})+\bar\nu(M_{-,\ar_-})
    -\lim_{t\to 0}\bigl(\bar\nu(M_{+,t})+\bar\nu(M_{-,t})\bigr)\\
    =72\,\frac\rho\pi+24\,\biggl(-\frac \glr{k_-\glb}-\frac \gll{k_+\glb}
    +12\,\DS\Bigl(\frac{\gll-\eps_+^*\glb}{k_+},\glb\Bigr)\biggr)\;.
  \end{multline}
\end{thm}

It will follow from the proof below that the first three terms on the right
hand side stem from the area of the triangle spanned by~$\gamma_+$
and~$\gamma_-$.
The Dedekind sum comes from the polygon we get by omitting the finite corner.
The two regions are separated by the blue geodesic in Figure~\ref{fig:HypPoly}.

\begin{proof}
  By Propositions~\ref{Prop2c.2}, \ref{Prop2y.1}
  and the discussion at the beginning of subsection~\ref{2.ex},
  we have
  \begin{multline}\label{eq:EtaHyp}
    \bar\nu(M_{+,\ar_+})+\bar\nu(M_{-,\ar_-})
    -\lim_{t\to 0}\bigl(\bar\nu(M_{+,t})+\bar\nu(M_{-,t})\bigr)
    =288\int_{\gamma_+\cup\gamma_-}\tilde\eta(\Aa)
    =288\int_Pd\tilde\eta(\Aa)\\
    -24\measuredangle_{\frac{a_0}{b_0}}
    \biggl(\frac{\eps_++i\ar_+}{k_+},\frac{a_1}{b_1}\biggr)
    -24\measuredangle_{\frac{a_\ell}{b_\ell}}
    \biggl(\frac{a_{\ell-1}}{b_{\ell-1}},\frac{\eps_++i\ar_+}{k_+}\biggr)
    -24\sum_{j=0}^{\ell-2}\measuredangle_{\tfrac{a_{j+1}}{b_{j+1}}}
    \biggl(\frac{a_j}{b_j},\frac{a_{j+2}}{b_{j+2}}\biggr)\;.
  \end{multline}
  From~\eqref{2.d.5}, Lemma~\ref{Lem2v.1}~\ref{2y.1c}
  and the hyperbolic area formula, we get
  \begin{equation}\label{eq:HypArea}
    288\int_Pd\tilde\eta(\Aa)
    =\frac{72}\pi\,A_{\mathrm{hyp}}(P)
    =72\,\ell-144\frac\thet\pi
    =72\,(\ell-1)+72\,\frac\rho\pi\;.
  \end{equation}

  From Definition~\ref{Def2y.1} and equations~\eqref{eq:CuspAngles}, we get
  \begin{multline}\label{eq:CuspSum}
    \measuredangle_{\frac{a_0}{b_0}}
    \biggl(\frac{\eps_++i\ar_+}{k_+},\frac{a_1}{b_1}\biggr)
    +\measuredangle_{\frac{a_\ell}{b_\ell}}
    \biggl(\frac{a_{\ell-1}}{b_{\ell-1}},\frac{\eps_++i\ar_+}{k_+}\biggr)
    +\sum_{j=0}^{\ell-2}\measuredangle_{\tfrac{a_{j+1}}{b_{j+1}}}
    \biggl(\frac{a_j}{b_j},\frac{a_{j+2}}{b_{j+2}}\biggr)\\
    =\biggl(\frac{\glr}{k_-\glb}+\frac{b'_1}{b'_0}\biggr)
    +\biggl(\frac{\eps_+^*}{k_+}+c_1\biggr)
    +\sum_{j=2}^\ell c_j\;.
  \end{multline}
  This number can be interpreted along the lines of~\cite[\S V]{Zagier}.
  The product of the matrices on the right hand side of~\eqref{eq:ElMat}
  for~$j=0$ is given by
  \begin{equation*}
    A=
    \begin{pmatrix}
      c_1&-1\\1&0
    \end{pmatrix}\cdots
    \begin{pmatrix}
      c_\ell&-1\\1&0
    \end{pmatrix}=
    \begin{pmatrix}
      a'_0&-a'_1\\
      b'_0&-b'_1
    \end{pmatrix}\;.
  \end{equation*}
  Now it follows from~$a'_0=\frac{\gll-\eps_+^*\glb}{k_+}$, $b'_0=\glb$
  and~\cite[equations (6), (25) \& (26)]{Zagier} that
  \begin{multline}\label{eq:HypPolyResult}
    \biggl(\frac{\glr}{k_-\glb}+\frac{b'_1}{b'_0}\biggr)
    +\biggl(\frac{\eps_+^*}{k_+}+c_1\biggr)
    +\sum_{j=2}^\ell c_j\\
    \begin{aligned}
      &=\frac{\glr}{k_-\glb}+\frac{b'_1}{b'_0}+\frac{\eps_+^*}{k_+}
      +3(\ell-1)+N(A)\\
      &=\frac{\glr}{k_-\glb}+\frac{b'_1}{b'_0}+\frac{\eps_+^*}{k_+}
      +3(\ell-1)+\frac{\gll-\eps_+^*\glb}{k_+\glb}-\frac{b'_1}{b'_0}-12\,\DS(-b'_1,b'_0)\\
      &=3(\ell-1)+\frac \gll{k_+\glb}+\frac \glr{k_-\glb}-12\,\DS(-b'_1,\glb)\;,
    \end{aligned}
  \end{multline}
  where the Dedekind sum~$\DS(-b'_1,\glb)$ is defined in~\eqref{eq:DedekindSum},
  and~$N\colon SL(2,\Z)\to\Z$ is introduced in~\eqref{Ltransf},
  see also~\cite[(3)]{Zagier}.

  The Dedekind sum~$\DS(k,\glb)$ is odd and $\glb$-periodic in~$k$,
  and it does not change if~$k$ is replaced by its inverse modulo~$\glb$.
  Our claim~\eqref{2.e.8} follows from Remark~\ref{rem:b1x}~\ref{b1x.1}
  and~\eqref{eq:EtaHyp}--\eqref{eq:HypPolyResult}.
\end{proof}

\begin{proof}[Proof of Theorem~\ref{Thm:B}]
  We may assume that~$\glb>0$.
  If~$\gll\ge 0$,
  the theorem follows from Theorems~\ref{Thm2a.1}, \ref{Thm2x.1}
  and~\ref{Thm2e.1}.
  If~$\gll<0$,
  we additionally use Proposition~\ref{prop:isom}, combining~\eqref{symgrp.3}
  and~\eqref{symgrp.5} to reduce to the case~$\gll>0$, $\glb>0$.
\end{proof}

\begin{rmk}\label{rmk:2p1symm}
  The formula in Theorem~\ref{Thm:B} is indeed symmetric in the two
  halves of the twisted connected sum.
  Swapping the two halves amounts to
  exchanging $r_\pm$, $\eps_\pm$ and $k_\pm$; see Proposition~\ref{prop:isom},
  in particular~\eqref{symgrp.1}.
  By equation~\eqref{eq:invmodn},
  the number~$\frac{\gll-\eps_+^*\glb}{k_+}$
  is inverse to~$\frac{\glr-\eps_-^*\glb}{k_-}$ modulo~$\glb$,
  so the Dedekind sum above is the same in both cases.
\end{rmk}

\begin{rmk}\label{rmk:nu24}
  We can evaluate~$3n(P)-\ell(P)$ mod~$\Z$ in a slightly different way.
  By Remark~\ref{rem:b1x}~\ref{b1x.2},
  we have~$\frac{\glr}{k_-\glb}+\frac{b'_1}{b'_0}\equiv\frac{\eps_-^*}{k_-}$
  mod~$\Z$, so
  instead of~\eqref{eq:HypPolyResult}, we get
  \begin{equation*}
    \biggl(\frac{\glr}{k_-\glb}+\frac{b'_1}{b'_0}\biggr)
    +\biggl(\frac{\eps_+^*}{k_+}+c_1\biggr)
    +\sum_{j=2}^\ell c_j
    \equiv\frac{\eps_+^*}{k_+}+\frac{\eps_-^*}{k_-}\mod\Z\;.
  \end{equation*}
  Following the proofs of Theorems~\ref{Thm2e.1} and~\ref{Thm:B} above,
  we see that
  \begin{equation}\label{eq:nu24}
    \nu(M,g)\equiv
    D_{\gamma_+}(V_+)+D_{\gamma_-}(V_-)+3\,m_\rho(L;N_+,N_-)
      -24\biggl(\frac{\eps_+^*}{k_+}+\frac{\eps_-^*}{k_-}\biggr)
      \mod 24\Z\;.
  \end{equation}
  
  The terms~$\frac{\eps_\pm}{k_\pm}$ and~$D_{\gamma_\pm}(V_\pm)$
  depend on the $\Gamma_\pm$-action on~$\widetilde M_\pm$ only,
  and one can check that in all our examples
  \begin{equation}\label{eq:Depsrel}
    D_{\gamma_\pm}(V_\pm)-24\frac{\eps_\pm^*}{k_\pm}\in\Z\;,
  \end{equation}
  as one would expect from the formula above.
  In particular for~$k_\pm=5$,
  we only found examples where the $\Gamma_\pm$-action
  on~$V_\pm$ has isolated fixpoints,
  and~\eqref{eq:Depsrel} holds for all choices of~$\eps_\pm$.
\end{rmk}

\begin{ex}\label{ex:run-hyp}
  Let us illustrate the remark above using our standard example.
  We start with the $\Z/5$-block from Example~\ref{ex:five},
  whose fixpoint contribution was computed in Example~\ref{ex:run-fixpoint}.
  We check that modulo~24,
  \begin{equation*}
    \frac{24}{5\eps}-24\frac{\eps^*}5\equiv
    \begin{cases}
      0		&\text{if~$\eps\equiv\pm 1$, and}\\
      12	&\text{if~$\eps\equiv\pm 2$.}
    \end{cases}
  \end{equation*}

  The $\Z/3$-block from Example~\ref{ex:quadric} has no isolated fixpoints,
  so the contribution to~$\nu$ mod~$24$ is simply~$-24\frac{\eps^*}3=-8\eps$.
  Together with~$m_\rho(L;N_+,N_-)=-1$ from Example~\ref{ex:run-angles},
  we find that~$\nu(M)\equiv-11$ mod~$24$,
  which confirms our computations in Example~\ref{ex:run-zagier}.
\end{ex}

\begin{rmk}\label{Rem2f.1}
  One can check that we recover the formula for~$\bar\nu(M)$
  in~\cite{CGN} in the case where $k_+$,~$k_-\in\{1,2\}$.
  Involutive isomorphisms of Calabi-Yau manifolds cannot have isolated
  fixpoints; see Section~\ref{2.b},
  so the first two terms on the right side of~\eqref{eq:ThmB} vanish.

  We start with a rectangular twisted connected sum as in~\cite{Kovalev,CHNP2},
  so~$k_+=k_-=1$.
  Because~$\gll=\glr=0=\rho=A$, we have~$\bar\nu(M)=0$ by~\eqref{eq:ThmB}.

  Next, consider~\cite[Example~\ref{eta-ex:pi4rk1}]{CGN} with~$k_-=1$, $k_+=2$
  and gluing matrix~$\psmatrix{1&1\\1&-1}$, so~$\thet=\frac\pi4$.
  By~\eqref{eq:ThmB}, we get
  \begin{equation*}
    \bar\nu(M)=-3+24\,\biggl(-1-\frac12+12\,\DS(0,1)\biggr)=-39\;.
  \end{equation*}
  Then, swapping the roles of $V_+$ and $V_-$, we get entry~5
  in Table~\ref{table:matchings}. (Note that we obtain the same extended
  \nuinvt for all entries 1--18 in the table.)

  In~\cite[Example~\ref{eta-ex:pi6rk1}]{CGN},
  we considered a simply connected example with~$k_-=k_+=2$,
  gluing matrix $\psmatrix{1&1\\1&-3}$ and~$\thet=\frac\pi6$.
  We found
  \begin{equation*}
    \bar\nu(M)=-3+24\,\biggl(-\frac32-\frac12+12\,\DS(0,1)\biggr)=-51\;,
  \end{equation*}
  see entry~110 in Table~\ref{table:matchings}.
  The other examples in~\cite{CGN} involve building blocks of rank~$\ge2$.
\end{rmk}

\section{Examples}\label{sec:ex}
In this section, we generate examples of extra-twisted connected sums
and compute their \xnuinvts.
To do this we will define some building blocks whose polarising lattice
has rank~1. We will also describe the topology of those blocks, and compute
some parts of the topology of the resulting extra-twisted connected sums.

\subsection{The cohomology of an extra-twisted connected sum}\label{sec:coh}

We have previously explained how to compute the fundamental group of
an extra-twisted connected sum $M$ from the gluing matrix, in Proposition
\ref{prop:pi1}. We now compute some other basic topological features.
In particular we show that all our examples have $H_2(M) = 0$ (so those that
have $\pi_1 M = 0$ are in fact 2-connected), and give a formula~\eqref{eq:b3M}
for $b_3(M)$.

\begin{rmk}
The most important topological properties that we do not compute are the
torsion in $H^4(M)$, and the Pontryagin class $p_1(M) \in H^4(M)$.
In general, the torsion in $H^4(M)$ can have contributions both from
the action of $\Gamma_\pm$ on the two halves, as well as from the gluing.
In~\cite{xtcs}, attention was focussed on blocks with involution such that
the former contribution vanishes, and on certain matchings (namely ones
with gluing matrix $\gmatrix111{-1}$ or $\gmatrix111{-3}$) where the torsion in
$H^4(M)$ can then be determined in a simple way from the
configuration $N_+ + N_-$.
We do not attempt here to generalise those arguments.
\end{rmk}

\subsubsection*{Topology of \texorpdfstring{$V_\pm$}{V}}
Let us first recall from~\cite[\S5]{CHNP} some relevant facts about the
topology of the ACyl Calabi-Yau manifold $V := Z \setminus \K$ constructed from 
a building block $(Z, \K)$, and make some observations about the action on cohomology of an automorphism group $\Gamma$.
Let us assume that the kernel of the restriction map $H^2(Z) \to H^2(\K)$
is generated by $[\K]$.
Note that this implies that the action of $\Gamma$
on $H^2(Z)$ is trivial.

We identify~$\Sigma\subset Z$ with a standard K3 surface and denote
the image of~$H^2(Z)$ in~$L=H^2(\Sigma)$ by~$N$.
Then~$N\subset L$ is primitive, but typically not unimodular.
Let~$T=N^\perp\subset L$, and let~$N^*$ be the dual of~$N$,
such that we have a short exact sequence
\begin{equation*}
  0\longrightarrow T\longrightarrow L\longrightarrow N^*\longrightarrow 0\;.
\end{equation*}

By~\cite[Lemma~5.2]{CHNP}, $Z$ and~$V$ are simply connected.
Using excision and suspension,
we have
\begin{equation*}
  H^k(Z,V)\cong H^k(\Sigma\times D^2,\Sigma\times S^1)\cong H^{k-2}(\Sigma)\;.
\end{equation*}
The long exact sequence of the pair~$(Z,V)$ becomes
\begin{equation}\label{eq:ZVseq}
  \cdots\longrightarrow H^{k-2}(\Sigma)
  \stackrel{\iota_!}\longrightarrow H^k(Z)
  \stackrel{j^*}\longrightarrow H^k(V)
  \stackrel\delta\longrightarrow H^{k-1}(\Sigma)\longrightarrow\cdots\;.
\end{equation}

According to~\cite[Lemma~5.4]{CHNP}, we have~$H^1(V)=0= H^5(V)$,
and the exact sequence~\eqref{eq:ZVseq} gives rise
to split short exact sequences
\begin{align*}
  0\longrightarrow\Z\longrightarrow H^2(Z)
  &\longrightarrow H^2(V)\longrightarrow 0\;,\\
  0\longrightarrow H^3(Z)
  &\longrightarrow H^3(V)\longrightarrow T\longrightarrow 0\;,\\
  0\longrightarrow N^*\longrightarrow H^4(Z)
  &\longrightarrow H^4(V)\longrightarrow 0\;.
\end{align*}
It follows that~$H^\bullet(V)$ is torsion free
if~$H^\bullet(Z)$ is torsion free.
The inclusion~$\Z\into H^2(Z)$
maps~$1$ to the cohomology class~$\iota_!1$
induced by~$\iota\colon\Sigma\to Z$.
It is easy to see that the sequences above are $\Gamma$-equivariant.
In particular, $\Gamma$ acts trivially on $H^2(V)$, while the
$\Gamma$-invariant part of $H^3(V)$ is the direct sum of $T$ and
$H^3(Z)^\Gamma$.

The map~$\delta$ in~\eqref{eq:ZVseq} involves restriction
to~$\Sigma\times S^1_\lnn$ followed by integration over~$S^1_\lnn$.
Write
\begin{equation*}
  H^k(\Sigma\times S^1_\lnn)=H^k(\Sigma)\oplus H^{k-1}(\Sigma)\drn\;,
\end{equation*}
where $\drn \in H^1(S^1_\lnn)$ is the generator.
The restriction map~$\iota^*\colon H^\bullet(V)\to H^\bullet(\Sigma\times S^1_\lnn)$
is described in~\cite[Cor~5.5]{CHNP}.
We have in particular that $H^2(V)$ maps isomorphically to $N \subset H^2(\K)$,
while the image of $H^3(V)$ is $T\drn \subset H^2(\K)\drn$.

\subsubsection*{Topology of \texorpdfstring{$M_\pm$}{M+ / M-}}

We may regard~$M_\pm=(V_\pm\times S^1_{\lnx_\pm})/\Gamma_\pm$
as the mapping torus of a generator~$\gamma_0\in\Gamma_\pm\cong\Z/k_\pm$.
Generalising the discussion of blocks with involution from~\cite[\S2.2]{xtcs},
we can use excision and the Thom isomorphism to see that
\begin{equation*}
  H^\bullet(M_\pm,V_\pm)\cong H^{\bullet-1}(V_\pm)\;.
\end{equation*}
From the long exact sequence of the pair~$(M_\pm,V_\pm)$, we get
\begin{equation*}
  \cdots\longrightarrow H^{\ell-1}(V_\pm)\longrightarrow H^\ell(M_\pm)
  \stackrel{\iota_V^*}\longrightarrow H^\ell(V_\pm)
  \stackrel{\gamma_0^*-\id}\longrightarrow H^\ell(V_\pm)\longrightarrow\cdots\;,
\end{equation*}
where~$\iota_V\colon V_\pm\to M_\pm$ is the inclusion of~$V$ as fibre
of the obvious projection~$M_\pm\to S^1_{\lnx_\pm/k_\pm}$.

Since $H^1(V) = 0$ while $H^2(V)$ is $\Gamma$-invariant, it is immediate that
$H^2(M_\pm) \cong H^2(V_\pm)$. Since $H^3(V_\pm)^\Gamma$ is torsion-free,
we also have a splitting
$H^3(M_\pm) \cong H^2(V_\pm) \oplus H^3(V_\pm)^\Gamma$, and $H^3(M_\pm)$
is torsion-free too. While the splitting is not natural with $\Z$ coefficients,
it is natural with $\Q$ coefficients.

We can treat the cross-section~$\Sigma\times T^2$ similarly.
Now $H^2(\K \times T^2) = H^2(\K) \oplus H^2(T^2)$, and $H^2(M_\pm) \to H^2(\K \times T^2)$ maps isomorphically to $N_\pm$.

Meanwhile, the pull-back $H^1(T^2) \to H^1(S^1_\lnx \times S^1_\lnn)$ of the
quotient map is injective (with image of index~$k_\pm$).
We abuse notation slightly to use $\drn_\pm, \drx_\pm$ to denote not only
the generators of $H^1(S^1_{\lnn_\pm} \times S^1_{\lnx_\pm})$ obtained by
pulling back the generators of the two factors, but also their unique
pre-images in $H^1(T^2; \Q)$.
(In terms of de Rham cohomology, these classes are represented by the 1-forms
$\frac{1}{\lnn_+}du_+$ and $\frac{1}{\lnx_\pm}dv_\pm$.)
Switching to rational coefficients, we then have the splitting
$H^3(\K \times T^2; \Q) = H^2(\K; \Q) \drn_\pm \oplus H^2(\K; \Q)\drx_\pm$.
In the splitting
\[ H^3(M_\pm; \Q) = H^2(V_\pm)\drx_\pm \oplus H^3(V_\pm; \Q)^\Gamma , \]
the first term has image exactly
$N_\pm \drx_\pm$, the second term has image exactly $T_\pm \drn_\pm$,
and the kernel of $H^3(M_\pm; \Q) \to H^3(\K \times T^2; \Q)$ is the
$H^3(Z_\pm; \Q)^\Gamma$ component in $H^3(V_\pm; \Q)^\Gamma$.

\subsubsection*{The topology of \texorpdfstring{$M$}{M}}

Generalising the discussion from~\cite[\S7.1]{xtcs} of extra-twisted connected
sums that involve only involutions, we can now apply the Mayer-Vietoris
sequence to compute some basic features of the topology of an extra-twisted
connected sum.

\begin{prop}
Let $M$ be an extra-twisted connected sum of building blocks $(Z,\K)$ such
that $H^2(Z) \to H^2(\K)$ is generated by $[\K]$, with
configuration of
polarising lattices $N_+, N_- \into L$.
Let $\rho_\pm$ be the ranks of the polarising lattices
(so $\rho_\pm = b_2(Z_\pm) - 1$). If $\cos \thet \not= 0$ let $d_\thet$ be the
rank of $N_+^\thet \cong N_-^\thet$ defined in \S\ref{3.c}, otherwise
let $d_\thet = \rk N_+^\frac{\pi}{2} + \rk N_-^\frac{\pi}{2}$.
\begin{enumerate}
\item The free part of $H^2(M)$ is isomorphic to $N_+ \cap N_- \subset L$,
\item The torsion in $H^3(M)$ is isomorphic to the cotorsion of $N_+ + N_-$
in $L$.
\item $ b_3(M) = b_2(M) + 23 - \rho_+ - \rho_- + d_\thet +
b_3(Z_+)^{\Gamma_+} + b_3(Z_-)^{\Gamma_-}$.
\end{enumerate}
\end{prop}

The examples considered in this paper use configurations where $N_+$ is transverse to $N_-$, and $N_+ + N_-$ is embedded primitively in $L$.
Thus the proposition implies that our examples have $H_2(M) = 0$, and those
examples that are simply-connected are in fact 2-connected.
Moreover, all our examples have $\rho_+ = \rho_- = 1$, and hence
$\rk N_+^\thet = \rk N_-^\thet = 1$.
Thus if $\thet\notin\frac{\pi}{2}\,\Z$ we have
$d_\thet = 1$ and
\begin{equation}\label{eq:b3M}
  b_3(M)=22+b_3^\Gamma(Z_+)+b_3^\Gamma(Z_-)\;,
\end{equation}
while if $\thet\in\frac{\pi}{2}\,\Z\setminus\pi\Z$ then~$d_\thet=2$ and
\begin{equation}
  b_3(M)=23+b_3^\Gamma(Z_+)+b_3^\Gamma(Z_-)\; .
\end{equation}

\begin{proof}
We have a Mayer-Vietoris sequence
\begin{equation*}
  \cdots\longrightarrow H^k(M)
  \longrightarrow H^k(M_+)\oplus H^k(M_-)
  \longrightarrow H^k(\Sigma\times T^2)
  \longrightarrow H^{k+1}(M)\longrightarrow\cdots\;.
\end{equation*}
The image of $H^1(M_+) \oplus H^1(M_-)$ in $H^1(\K \times T^2)$ has finite
index, and indeed it is dual to the fundamental group computed in
Proposition~\ref{prop:pi1}. Since $H^2(M_\pm)$ maps isomorphically to $N_\pm \subset H^2(\K \times T^2)$, the image of $H^2(M)$ in $H^2(M_+) \oplus H^2(M_-)$
is isomorphic to $N_+ \cap N_- \subset L$ (as determined by the configuration).

For $H^3(M)$ we get a short exact sequence
\[ 0 \longrightarrow L/(N_+ + N_-) \oplus \Z \longrightarrow H^3(M) \longrightarrow
\ker\bigl(H^3(M_+) \oplus H^3(M_-) \to H^3(\K \times T^2)\bigr) \longrightarrow 0 \]
Since the last term is torsion-free, the sequence splits,
and the torsion of $H^3(M)$ equals the torsion of~$L/(N_+ + N_-)$.

Finally we want to determine $b_3(M)$.
The contribution from $L/(N_+ + N_-) \oplus \Z$
equals $23-\rho_+ - \rho_- + b_2(M)$.
The other term we describe as the sum of the kernels of
$H^3(M_\pm) \to H^3(\K \times T^2)$, which by the above have
rank $b_3(Z_\pm)^\Gamma$, and the intersection of the images in
$H^3(\K \times T^2)$.
Points in the intersection of the images are those that can be written as both
$n_+ \drx_+ + t_+ \drn_+$ and $n_- \drx_- + t_- \drn_-$,
with $n_\pm \in N_\pm$ and $t_\pm \in T_\pm$.
Now, the gluing identifies the tori in such a way that
\[ \lnx_+ \drx_+ = \cos \thet \, \lnx_- \drx_- + \sin \thet \, \lnn_- \drn_-, \qquad
 \lnn_+ \drn_+ = \sin \thet \, \lnx_- \drx_- - \cos \thet \, \lnn_- \drn_- \; . \]
If $\cos \thet \not= 0$ then $n_+$ and $n_-$ determine each other,
because the orthogonal projection of $\frac{1}{\lnn_-}n_-$ to $N_+$ must be
$\frac{\cos \thet}{\lnn_+} n_+$ and vice versa.
Thus, in the notation of \S\ref{3.c}, in fact $n_\pm \in N_\pm^\thet$, so the
intersection of the images is isomorphic to $N_+^\thet \cong N_-^\thet$.
On the other hand, if $\cos \thet = 0$ then we can simply take each $n_\pm$
freely in $N_\pm^\frac{\pi}{2}$ (\ie in the orthogonal complement to $N_\mp$ in
$N_\pm$).

Either way, the contribution to $b_3(M)$ from the intersection of the images
in $H^3(\K \times T^2)$ is what we denoted as $d_\thet$. Adding that to
the other contributions proves (iii).
\end{proof}

\subsection{Examples of building blocks}\label{sec:blocks}

We now give examples of building blocks. For simplicity, we restrict attention
to blocks whose polarising lattice $N$ has rank 1. 
We list the relevant data for the examples in Table~\ref{table:blocks}.

Each family of blocks~$Z$ is obtained by blowing up Fano 3-folds $Y$ of
Picard rank 1. We list the index $r$ of~$Y$, the anticanonical degree $-K_Y^3$,
the norm-square of the generator of the Picard lattice $N$ of $Y$
(which is isometric to the polarising lattice of~$Z$),
the third Betti number $b_3(Y)$, and the result of evaluating $c_2(Z)$
on the pull-back $H \in H^2(Z)$
of the generator $-\frac{1}{r}K_Y \in H^2(Y)$ (the latter number is needed
to compute the Pontryagin class of the extra-twisted connected sums built from
the block, although we do not do that in this paper).

Recall that the Picard lattice of a Fano 3-fold $Y$ is $H^2(Y)$ equipped with
the non-degenerate symmetric bilinear form $(A,B) \mapsto A.B.(-K_Y)$.
For rank 1 Fanos, the
norm-square of the generator $-\frac{1}{r}K_Y$ is thus simply computed as
$\frac{1}{r^2}(-K_Y)^3$.

\newcommand{\allblocks}{
\begin{table}[tb]
\vspace{-4mm}
\[
\begin{array}[t]{r l c  c c c c  c c c c } \toprule
& Y & r & -K_Y^3 & N & b_3(Y) & c_2(Z)H & Ex & k & b_3^\Gamma(Z) & \#\mathrm{fix} \\ \midrule
1&\PP^3 &  4 & 64 &  4 &   0 &  22 & \ref{ex:nonsymblocks} & 1 & 66  \\ 
2&    Q &  3 & 54 &  6 &   0 &  26 & \ref{ex:nonsymblocks} & 1 & 56 \\ 
3&  V_1 &  2 &  8 &  2 &  42 &  16 & \ref{ex:nonsymblocks} & 1 & 52\\ 
4&  V_2 &  2 & 16 &  4 &  20 &  20 & \ref{ex:nonsymblocks} & 1 & 38 \\
5&      &    &    &    &     &     & \ref{ex:invblocks} & 2 & 18 \\ 
6&  V_3 &  2 & 24 &  6 &  10 &  24 & \ref{ex:nonsymblocks} & 1 & 36 \\ 
7&  V_4 &  2 & 32 &  8 &   4 &  28 & \ref{ex:nonsymblocks} & 1 & 38 \\ 
8&  V_5 &  2 & 40 & 10 &   0 &  32 & \ref{ex:nonsymblocks} & 1 & 42 \\ 
9&      &  1 &  2 &  2 & 104 &  26 & \ref{ex:nonsymblocks} & 1 & 108 \\ 
10&     &    &    &    &     &     & \ref{ex:invblocks} & 2 & 46 \\
11&     &    &    &    &     &     & \ref{ex:six} & 3 &  24 & 2 \\
12&     &    &    &    &     &     & \ref{ex:five} & 5 &  8 & 1 \\
13&     &    &    &    &     &     & \ref{ex:six} & 6 &  4 & 2 \\
14&     &  1 &  4 &  4 &  60 &  28 & \ref{ex:nonsymblocks} & 1 & 66 \\ 
15&     &    &    &    &     &     & \ref{ex:invblocks} & 2 & 26 \\  
16&     &    &    &    &     &     & \ref{ex:three} & 3 & 12 & 1 \\ 
17&     &    &    &    &     &     & \ref{ex:p3} & 4 &  6 \\ 
18&     &  1 &  6 &  6 &  40 &  30 & \ref{ex:nonsymblocks} & 1 & 48 \\ 
19&     &    &    &    &     &     & \ref{ex:invblocks} & 2 & 18 \\ 
20&     &    &    &    &     &     & \ref{ex:quadric} & 3 &  8 \\ 
21&     &  1 &  8 &  8 &  28 &  32 & \ref{ex:nonsymblocks} & 1 & 38 \\ 
22&     &    &    &    &     &     & \ref{ex:invblocks} & 2 & 14 \\ 
23&     &  1 & 10 & 10 &  20 &  34 & \ref{ex:nonsymblocks} & 1 & 32 \\ 
24&     &    &    &    &     &     & \ref{ex:invblocks} & 2 & 12 \\ 
25&     &  1 & 12 & 12 &  14 &  36 & \ref{ex:nonsymblocks} & 1 & 28 \\ 
26&     &  1 & 14 & 14 &  10 &  38 & \ref{ex:nonsymblocks} & 1 & 26 \\ 
27&     &  1 & 16 & 16 &   6 &  40 & \ref{ex:nonsymblocks} & 1 & 24 \\ 
28&     &  1 & 18 & 18 &   4 &  42 & \ref{ex:nonsymblocks} & 1 & 24 \\ 
29&     &  1 & 22 & 22 &   0 &  46 & \ref{ex:nonsymblocks} & 1 & 24 \\ 
\bottomrule
\end{array} 
\]
\smallskip
\caption{Rank 1 building blocks}
\label{table:blocks}
\vspace{-4mm plus 20mm}
\end{table}}

\allblocks

\begin{ex}
\label{ex:nonsymblocks}
If we ignore the desire for automorphisms,
then we can simply take the list of rank 1 blocks from~\cite[Table 1]{CHNP}.
These are obtained by blowing up a Fano 3-fold $Y$ of Picard rank 1 along the
transverse intersection $C$ of two smooth anticanonical divisors.
As explained in~\cite[\S5.2]{CHNP}, the resulting building block $Z$ has
$b_3(Z) = b_3(Y) + b_1(C) = b_3(Y) + (-K_Y)^3 + 2$.
For the final piece of data we wish to include in Table~\ref{table:blocks},
\cite[(4.4)]{ex} gives $c_2(Z)H = \frac{24-K_Y^3}{r}$.
\end{ex}

All other examples we consider will in fact be subfamilies of the families
from Example~\ref{ex:nonsymblocks} that admit automorphisms.
For each suitable automorphism
that we have found on some elements of the family,
we list in Table~\ref{table:blocks} its order~$k$
and the rank~$b_3^\Gamma(Z)$ of the invariant part of $H^3(Z)$ (so the number
against $k = 1$ is $b_3(Z)$), and the number of
isolated fixed points (among all elements of $\Gamma$).
The formula~\eqref{eq:b3M} for the third Betti number of
an extra-twisted connected sum involves $b_3^\Gamma(Z)$, while
the computation of the \xnuinvt in Theorem~\ref{Thm2x.1}
relies on some further details about the fixed points that is not included in
the table, but only in the descriptions of the individual examples.

The pattern is that we consider special elements $Y$ of the given family of
Fano 3-folds that admit a group of automorphisms $\Gamma$, whose fixed set is
a union of a K3 divisor $\K$ and (possibly) some isolated fixed points.
After blowing up a curve $C \subset \K$ like in Example~\ref{ex:nonsymblocks},
one obtains a building block $Z$ with automorphisms whose fixed sets are the
proper transforms of the union of the fixed set of the corresponding
automorphism on $Y$ and a copy~$\wt C\subset V\subset Z$ of $C$;
it is a section of the exceptional set, which
is a trivial $\PP^1$-bundle over $C$.

In all but one of our examples (Example \ref{ex:six}),
the fixed set $Z^\gamma$ is the same for all
non-identity elements~$\gamma\in\Gamma$. Because the cohomology of $Z$ is
$\Gamma$-invariant except in degree 3 (so that
\eg $b_2(Z/\Gamma) = b_2(Z) = 2$) we can in those cases easily
compute $b_3^\Gamma(Z)$ from
\[ \chi(Z/\Gamma) = \frac{1}{k}\chi(Z) + \frac{k-1}{k}\chi(Z^\gamma) . \]

\begin{ex}
\label{ex:invblocks}
Blocks with involutions were already considered in~\cite{xtcs}.
In some sense, the simplest way to obtain examples is to start from a Fano
3-fold $X$ with even anti-canonical class $-K_X$ (\ie $\PP^3$ or a del Pezzo
3-fold), and consider a double cover $Y$ of $X$ branched over an anticanonical
divisor; see~\cite[Examples 3.24 \& 3.25]{xtcs}.
Blowing up the double cover $Y$ along a transverse intersection $C$ of the
ramification divisor and another anticanonical divisor yields a block $Z$
with involution.
\end{ex}

\pagebreak[3]
Similarly to Example~\ref{ex:invblocks}, we can take $X$ to be one of the two
Fano 3-folds with index $r > 2$.
Then the $r$-fold branched cover $Y$ of $X$ branched over an anticanonical
divisor can be blown up along the intersection of the ramification locus with
another anticanonical divisor of $Y$ to give a building block $Z$ with an
automorphism of order $r$.

\begin{ex}
\label{ex:quadric}
For $X = Q \subset \PP^4$ the quadric 3-fold (which has $r = 3$), $Y$ is
isomorphic to a complete intersection of a quadric and a cubic in $\PP^5$ of
the forms $X_1^2 + \cdots + X_5^2$ and $X_0^3 + F(X_1, \ldots, X_5)$,
and the branch switching automorphisms correspond to multiplying the homogeneous
coordinate $X_0$ by cube roots of unity.
The fixed set is the anticanonical divisor
$\{X_0 = 0\}$, which is smooth for a generic $F$.
Blowing up a transverse intersection $C$ with another anticanonical
divisor yields a block $Z$ with an automorphism group of order 3
(number~20 in Table~\ref{table:blocks});
these building blocks are then special elements of the family~18 in the table,
which was obtained in Example~\ref{ex:nonsymblocks}.

If we let $\tau \in \Gamma$ be the generator that multiplies
$X_0$ by $\zeta^{-1} = e^{-\frac{2\pi i}{3}}$,
then the fixed set~$Z^\tau\subset Z$ consists of the proper transform $\K$ of
the ramification locus and a section $\wt C$ of the exceptional set.
Clearly $\tau$ acts
on the normal bundle of $\K$ as multiplication by $\zeta^{-1}$.
\end{ex}

\begin{ex}
\label{ex:p3}
For $X = \PP^3$ (which has $r = 4$), $Y$ is isomorphic to a quartic in $\PP^4$
with defining equation of the form $X_0^4 + F(X_1, X_2, X_3, X_4)$,
giving entry~17 in Table~\ref{table:blocks}.
These are special elements of the family~14 of blocks obtained in Example
\ref{ex:nonsymblocks}.

Further, we can of course also consider this as a family of blocks with
involution; then we recover a subfamily of family~15,
which already came up in Example~\ref{ex:invblocks}.
\end{ex}

Before producing some examples with isolated fixpoints,
let us recall that we need to find a generator~$\tau$ of~$\Gamma$
that acts on the normal bundle~$\nu_\K$ by~$\zeta^{-1} = e^{-\frac{2\pi i}k}$.
Then by Remark~\ref{Rem2x.1},
the contribution of the fixpoint set to the extended \nuinvt
is given by~$D_{\tau^\eps}(Z)$ with~$\eps$ as in Section~\ref{2.w},
where for brevity,
we write~$D_\gamma(Z)$ instead of~$D_\gamma(V)$ as in Definition~\ref{Def2x.1}.

\begin{rmk}
While the action of $\Gamma$ on the normal bundle of the fixed curve $\wt C$ does
not affect the~\xnuinvt by Theorem~\ref{Thm2x.1}, %
let us mention that it is easy to describe in a uniform way in all our examples.

The exceptional divisor $E$ in $Z$ is biholomorphic to $C \times \PP^1$.
We can choose the identification so that $C \times \{(1:0)\}$ is the
intersection $E \cap \K$, while $C \times \{ (0:1) \}$ is the 1-dimensional
component $\wt C$ of the fixed set of $\Gamma$ in $Z$.
The action of $\tau \in \Gamma$ on $E$ is trivial on the $C$ factor, and can be
identified with $(Y_0 : Y_1) \mapsto (\zeta Y_0 : Y_1)$ on the $\PP^1$ factor,
for $\zeta$ such that $\tau$ acts on~$\nu_\K$ by $\zeta^{-1}$.

Then $\tau$ acts on the normal bundle of $\tilde C$ in $E$ (which is trivial)
as multiplication by $\zeta$.
If we write the normal bundle of $\wt C$ in $Z$ as a direct sum of this trivial
bundle and another line bundle, then (because $\wt C$ is contained in $V$ which
has a Calabi-Yau structure preserved by~$\Gamma$) the second summand must be
isomorphic to $T^*\wt C$, and $\tau$ must act on it as multiplication
by~$\zeta^{-1}$.
\end{rmk}

\begin{ex}
\label{ex:three}
Consider a smooth quartic $Y \subset \PP^4$ of the
form~$X_0^3 X_1 + F(X_1,X_2,X_3, X_4)$.
Multiplying $X_0$ by a primitive third root of unity,
say $\zeta^{-1} = e^{-\frac{2\pi i}3}$, defines an automorphism
$\tau$ of order~3. Its fixed set is the union of the K3 surface
$\K = \{ X_0 = 0\}$ and the isolated point \mbox{$(1:0:0:0:0)$}.

Blowing up $Y$ along the intersection $C$ (a quartic plane curve) of $\K$
with another anticanonical
divisor yields a building block $Z$ with automorphism group 
$\Gamma \cong \Z/3$, number~16 in Table~\ref{table:blocks}.
It is a different subfamily of family~14
than the one considered in Example~\ref{ex:p3}.

The block $Z$ has $\chi(Z) = -60$. The fixed set $Z^\tau$ of $\tau$ is the
union of the proper transform of~$\K$, a section $\wt C$ of the exceptional
divisor and the isolated fixed point, so its Euler characteristic is~21.
Thus $\chi(Z/\Gamma) = -\third 60 + \frac{2}{3}21 = -6$,
and hence $b_3^\Gamma(Z) = 12$.

Clearly, $\tau$ acts on the normal bundle of $\K$ as multiplication
by $\zeta^{-1}$.
Meanwhile, in the affine chart
$(z_1, \ldots, z_4) \mapsto (1: z_1 : \cdots : z_4)$, the action of $\tau$
is represented by multiplication with~$\zeta$.
Hence, the action of $\tau$ on the tangent space at~$(1:0:0:0:0)$
is diagonal with eigenvalue~$\zeta$.
We can now compute~$D_{\tau^\eps}(Z)=2\eps$ for~$\eps=\pm 1$.
\end{ex}

\begin{ex}
\label{ex:six}
Consider a smooth sextic $Y$ in the weighted projective space $\PP^4(1^4, 3)$
of the form $X_0^6 + F(X_1, X_2, X_3) + X_4^2$.
Multiplying $X_0$ by a primitive 6th root of unity,
say $\zeta^{-1} = e^{-\frac{\pi i}{3}}$, defines an automorphism
$\tau$ of order 6.
Its fixed set is the K3 surface $\K = \{ X_0 = 0\}$.
In addition, $\tau^2$ has two isolated fixed points,
at $x_\pm=(1 : 0 : 0: 0: \pm i)$, which are swapped by~$\tau$.

Blowing up $Y$ along the intersection $C$ of $\K$ with another anticanonical
divisor that is stable under~$\tau$ as a set
yields a building block $Z$ with automorphism group~$\Gamma \cong \Z/6$,
line~13 in Table~\ref{table:blocks}.
It can be considered as a more special subfamily of family~9
appearing in Example~\ref{ex:nonsymblocks}
or of family~10 of involution blocks from Example~\ref{ex:invblocks}.
But if we consider it as a block with automorphism group of order 3
(number~11 in the table),
then that is distinct from the previous examples.

Clearly, $\tau$ acts on the normal bundle of $\K$ as multiplication
by $\zeta^{-1}$.
Meanwhile, in the affine chart
$(z_1, \ldots, z_4) \mapsto (1: z_1 : \cdots : z_4)$, the action of $\tau$
is represented by
$(z_1, \ldots, z_4) \mapsto (\zeta z_1,  \zeta z_2, \zeta z_3, \zeta^3  z_4)$.
The isolated fixed points correspond to $(0,0,0,\pm i)$, and
have tangent space $z_4 = 0$.
Thus the action of $\tau^2$ on the tangent spaces is diagonal with
eigenvalue $\zeta^2$.
Again, we find~$D_{\tau^\eps}(Z)=2\eps$ for the automorphism group~$\Z/6$,
and~$D_{\tau^\eps}=4\eps$ if we restrict to the automorphism group~$\Z/3$.

We have $\chi(Z) = -102$, while the fixed set $Z^{\tau^2}$ of $\tau^2$ is the
union of the proper transform of~$\K$, a copy $\wt C$ of $C$ and the two
isolated fixed points, so has Euler characteristic 24.
In the case where we consider the automorphism group $\Gamma' \cong \Z/3$
generated by $\tau^2$, we thus find
$\chi(Z/\Gamma') = -\third 102 + \frac{2}{3}24 = -18$,
so $b_3^{\Gamma'}(Z) = 24$. 

In turn, we can consider $Z/\Gamma$ as a $\Z/2$ quotient of $Z/\Gamma'$
with fixed set of Euler characteristic $\chi(\K) + \chi(\wt C) = 22$.
Thus $\chi(Z/\Gamma) = -\half 18 + \half 22 = 2$, and $b_3^\Gamma(Z) = 4$.
\end{ex}

\begin{ex}
\label{ex:five}
Consider again a smooth sextic $Y$ in the weighted projective
space $\PP^4(1^4, 3)$, but now of the
form~$X_0^5 X_1 + F(X_1, \ldots, X_3) + X_4^2$.
Multiplying $X_0$ by a primitive 5th root of unity,
say $\zeta^{-1} = e^{-\frac{2\pi i}{5}}$, defines an automorphism
$\tau$ of order 5. Its fixed set is the union of the K3 surface
$\K = \{ X_0 = 0\}$ and the isolated point $(1:0:0:0:0)$.

Blowing up $Y$ along the intersection $C$ of $\K$ with another anticanonical
divisor yields a building block $Z$ with automorphism group 
$\Gamma \cong \Z/5$. It can be considered as another more special subfamily
(entry~12 in Table~\ref{table:blocks}) of
the family~9 that we already considered in Example~\ref{ex:six}.

Clearly, $\tau$ acts on the normal bundle of $\K$ as multiplication
by $\zeta^{-1}$.
Meanwhile, in the affine chart
$(z_1, \ldots, z_4) \mapsto (1: z_1 : \cdots : z_4)$, the action of $\tau$
is represented by
$(z_1, \ldots, z_4) \mapsto (\zeta z_1,  \zeta z_2, \zeta z_3, \zeta^3  z_4)$.
The tangent space at the fixed point is $z_1 = 0$,
so the eigenvalues of $\tau$ on the tangent space are $\zeta, \zeta$ and
$\zeta^3$.
In Example~\ref{ex:run-fixpoint},
we have computed~$D_{\tau^\eps}(Z)=-\frac{24}{5\eps}$ for~$\eps\in\{\pm1,\pm2\}$.

We have $\chi(Z) = -102$, while the fixed set $Z^\tau$ of $\tau$ is the union
of the proper transform of~$\K$, a copy $\wt C$ of $C$, and the 
isolated fixed point. Thus $Z^\tau$ has Euler characteristic 23, so
we find $\chi(Z/\Gamma) = -\frac{1}{5} 102 + \frac{4}{5}23 = -2$,
and $b_3^\Gamma(Z) = 8$. 
\end{ex}

In order to construct extra-twisted connected sums from our examples of blocks,
we need to note that they have a genericity property described in Definition
\ref{def:generic}.

\begin{prop}
\label{prop:generic}
Each family $\fbb$ of blocks above is $(N, \Amp)$-polarised, where $N$ is
the polarising lattice of the family, and $\Amp \subset N\Rlat$ is one of the
two open half-lines.
\end{prop}

\begin{proof}
For the families of blocks without automorphism in
Example~\ref{ex:nonsymblocks}, this is just an instance of
\cite[Proposition 6.9]{CHNP}, which is a consequence of the results of
Beauville~\cite{beauville04} on anticanonical divisors in Fano 3-folds.
The same argument applies to the families of blocks with involution that are
obtained in Examples~\ref{ex:invblocks}, \ref{ex:quadric} or~\ref{ex:p3}
from a cover of a Fano 3-fold $X$ branched over an anticanonical divisor $\K$,
since $\K$ can be any smooth anticanonical divisor in $X$.

In Example~\ref{ex:three}, the K3 divisor $\K$ is a hypersurface in $\PP^3$
defined by the quartic poly\-nomial~$F$. Clearly $F$ can be chosen to be any
smooth quartic, so a generic K3 surface with Picard lattice containing an ample
primitive class of norm-square 4 will appear this way. Indeed, in particular
any K3 surface with Picard lattice exactly $\gen{4}$ can embedded as a quartic
in~$\PP^3$ (see Saint-Donat~\cite[Theorem 6.1]{saintdonat74}).

Similarly, we see directly that any K3 that is a double cover of $\PP^2$
branched over a smooth sextic curve can appear as the K3 divisor in blocks of
the classes from Examples~\ref{ex:six} and~\ref{ex:five}, and a generic K3
surfaces whose Picard lattice contains an ample class of norm-square 2 can be
presented that way (see Reid~\cite[Theorem 3.8(d)]{reid97}).
\end{proof}

\subsection{Examples of matchings}\label{sec:match-ex}

We now study the matchings that can be produced from the blocks in the previous
section. Table~\ref{table:matchings} lists all extra-twisted connected sums
that can be made from the blocks in Table~\ref{table:blocks},
except those where both blocks have trivial automorphism group,
which were studied in~\cite{CHNP2,Kovalev}.
Note that some examples with~$k_\pm \leq 2$
were already considered in~\cite{CHNP2, CGN, xtcs},
in particular tables~4 and~5 in~\cite{xtcs} contain some of the
examples with~$k_+\le 2$, $k_-=2$ with additional information
on~$p_1(TM)$ and the torsion in~$H^4(M)$.
Table~\ref{table:matchings} contains~192 examples where~$k_-\ge 3$;
these are genuinely new.

We explained in \S\ref{subsec:combinatorics} one way to find all gluing
matrices for a given pair of orders $(k_+, k_-)$ of automorphism groups.
The gluing matrix determines the gluing angle $\thet$, and for each pair of
rank 1 blocks one can then determine whether there is a corresponding
configuration as explained in \S\ref{subsec:match}. However, we find it
convenient here to do these steps in the opposite order, and first enumerate
all possible configurations involving the blocks from Table~\ref{table:blocks}.

Given $k_+$ and $k_-$ and a configuration,
there may be several different gluing matrices $G$ that have the right gluing
angle and $\det G = -k_+k_-$, and several different choices of blocks
$Z_+, Z_-$ with the right polarising lattices and automorphism groups of
order $k_\pm$. Each such choice yields a family of extra-twisted connected sums
$M$ by application of Proposition~\ref{prop:matching}. The fundamental group
$\pi_1(M)$ depends only on the gluing matrix, while
$b_3(M)$ depends on the choices of $Z_+$ and $Z_-$.
The invariant $\bar \nu(M)$ depends on the gluing matrix together with data
for the isolated fixed points of the automorphisms on $Z_\pm$.
It turns out that for those pairs of configuration and gluing matrix where
there is more than one choice of $(Z_+, Z_-)$, there is never any ambiguity in
the isolated fixed point data, so in practice, $\bar \nu(M)$ only depends on the
gluing matrix.

We therefore organise Table~\ref{table:matchings} with the data about the
extra-twisted connected sums from blocks with polarising lattices of rank 1
as follows.
For each example %
we first list
the orders~$k_\pm$ of~$\Gamma_\pm$, the even lattice describing the
configuration of the K3 matching and $(\cos \thet)^2$ %
of the gluing angle $\thet$ as determined by~\eqref{3.c.3}.
Then follow the building blocks $Z_+$ and $Z_-$ (the
numbers referring to the entries in Table~\ref{table:blocks})
and the third Betti number of the extra-twisted connected sum, the gluing
matrix~$G=\psmatrix{\gll&p\\\glb&\sglr}$, and the parameters~$\eps_\pm$ (see Proposition~\ref{Prop2w.1}).
By Remark~\ref{rmk:geosym},
we always assume that~$\gll$, $\glb$, $p$, $\glr > 0$.
Moreover, if~$k_+=k_-$ we may swap the blocks if necessary
to make sure that~$\glr \geq \gll$. %
Finally, we list the value of the \xnuinvt.
Where there are several different choices of $Z_+$ with the same $k_+, k_-$
and configuration they are separated by commas, as are the corresponding values
of~$b_3(M)$, while the different choices of the gluing matrix $G$ 
(and the corresponding $\eps_+, \eps_-$ and~$\bar\nu(M)$) are listed on
separate rows.
The number at the very left is the running number of the first example in
the line, for example, the third line contains examples~5, 6 and~7.

\enlargethispage{0.4\baselineskip}

\begin{rmk}\label{rmk:table-cover}
  If a non simply-connected extra-twisted connected sum
  has a nontrivial covering constructed
  with the same groups~$\Gamma_+$, $\Gamma_-$,
  see Proposition~\ref{prop:cover},
  then it is listed a few lines above, \eg entry 20 is the
  universal cover of entry 22.
  If one needs to pass to a
  subgroup of at least one of the groups~$\Gamma_+$, $\Gamma_-$,
  then one should determine~$\tilde k_+$, $\tilde k_-$ and the gluing matrix
  by~\eqref{eq:cover-mat}
  and find the covering in a different section of the table,
  possibly with the roles of~$Z_+$ and~$Z_-$ swapped,
  except if~$\thet=\frac\pi 2$.
  In the latter case,
  the universal cover is an ordinary twisted sum of the type
  discussed in~\cite{CHNP2,Kovalev}, and therefore not listed here.
\end{rmk}

\newpage

\newcommand\matchtable{
  {%
\begin{longtable}[c]{r c c c c c c c c r r r}
      \hline\noalign{\smallskip}
      &$k_+$&$k_-$&$\!\!N_+{+}N_-\!\!$&$\cos^2\thet$&$Z_+$&$Z_-$&$b_3(M)$&$G$&$\eps_+$&$\eps_-$&$\bar\nu(M\rlap{)}$\\
      \hline\noalign{\smallskip}
      \endfirsthead
      \hline\noalign{\smallskip}
      &$k_+$&$k_-$&$\!\!N_+{+}N_-\!\!$&$\cos^2\thet$&$Z_+$&$Z_-$&$b_3(M)$&$G$&$\eps_+$&$\eps_-$&$\bar\nu(M\rlap{)}$\\
      \hline\noalign{\smallskip}
      \endhead
      \hline\noalign{\medskip}
      \caption{Examples of extra-twisted connected sums}
      \endfoot
      \hline\noalign{\medskip}
      \caption{Examples of extra-twisted connected sums}
      \label{table:matchings}
      \endlastfoot
      1&1&2&$\bigl(\begin{smallmatrix}2&2\\2&4\end{smallmatrix}\bigr)$&$1/2$&
    3, 9&\vtop{\hsize=1.5em\noindent\hfil\strut5\hfil\break\strut\hfil15\hfil}&
    \vtop{\hsize=6.5em\noindent\hfil\strut92, 148\hfil\break
        \strut\hfil100, 156\hfil}&
    $\bigl(\begin{smallmatrix}1&1\\1&-1\end{smallmatrix}\bigr)$&
    $ $&$1$&$-39$\\[\smallskipamount]
5& & &$\bigl(\begin{smallmatrix}4&2\\2&2\end{smallmatrix}\bigr)$&$1/2$&
    1, 4, 14&\vtop{\hsize=1.5em\noindent\hfil\strut10\hfil}&
    \vtop{\hsize=6.5em\noindent\hfil\strut134, 106, 134\hfil}&
    $\bigl(\begin{smallmatrix}1&1\\1&-1\end{smallmatrix}\bigr)$&
    $ $&$1$&$-39$\\[\smallskipamount]
8& & &$\bigl(\begin{smallmatrix}4&4\\4&8\end{smallmatrix}\bigr)$&$1/2$&
    1, 4, 14&\vtop{\hsize=1.5em\noindent\hfil\strut22\hfil}&
    \vtop{\hsize=6.5em\noindent\hfil\strut102, 74, 102\hfil}&
    $\bigl(\begin{smallmatrix}1&1\\1&-1\end{smallmatrix}\bigr)$&
    $ $&$1$&$-39$\\[\smallskipamount]
11& & &$\bigl(\begin{smallmatrix}8&4\\4&4\end{smallmatrix}\bigr)$&$1/2$&
    7, 21&\vtop{\hsize=1.5em\noindent\hfil\strut5\hfil\break\strut\hfil15\hfil}&
    \vtop{\hsize=6.5em\noindent\hfil\strut78, 78\hfil\break
        \strut\hfil86, 86\hfil}&
    $\bigl(\begin{smallmatrix}1&1\\1&-1\end{smallmatrix}\bigr)$&
    $ $&$1$&$-39$\\[\smallskipamount]
15& & &$\bigl(\begin{smallmatrix}12&6\\6&6\end{smallmatrix}\bigr)$&$1/2$&
    25&19&
    68&
    $\bigl(\begin{smallmatrix}1&1\\1&-1\end{smallmatrix}\bigr)$&
    $ $&$1$&$-39$\\[\smallskipamount]
16& & &$\bigl(\begin{smallmatrix}16&4\\4&2\end{smallmatrix}\bigr)$&$1/2$&
    27&10&
    92&
    $\bigl(\begin{smallmatrix}1&1\\1&-1\end{smallmatrix}\bigr)$&
    $ $&$1$&$-39$\\[\smallskipamount]
17& & &$\bigl(\begin{smallmatrix}16&8\\8&8\end{smallmatrix}\bigr)$&$1/2$&
    27&22&
    60&
    $\bigl(\begin{smallmatrix}1&1\\1&-1\end{smallmatrix}\bigr)$&
    $ $&$1$&$-39$\\
\end{longtable}
  }
}

  {%
    \begin{longtable}[c]{r c c c c c c c c r r r}
      \hline\noalign{\smallskip}
      &$k_+$&$k_-$&$\!\!N_+{+}N_-\!\!$&$\cos^2\thet$&$Z_+$&$Z_-$&$b_3(M)$&$G$&$\eps_+$&$\eps_-$&$\bar\nu(M\rlap{)}$\\
      \hline\noalign{\smallskip}
      \endfirsthead
      \hline\noalign{\smallskip}
      &$k_+$&$k_-$&$\!\!N_+{+}N_-\!\!$&$\cos^2\thet$&$Z_+$&$Z_-$&$b_3(M)$&$G$&$\eps_+$&$\eps_-$&$\bar\nu(M\rlap{)}$\\
      \hline\noalign{\smallskip}
      \endhead
      \hline\noalign{\medskip}
      \caption{Examples of extra-twisted connected sums}
      \endfoot
      \hline\noalign{\medskip}
      \caption{Examples of extra-twisted connected sums}
      \label{table:matchings}
      \endlastfoot
      1&1&2&$\bigl(\begin{smallmatrix}2&2\\2&4\end{smallmatrix}\bigr)$&$1/2$&
    3, 9&\vtop{\hsize=1.5em\noindent\hfil\strut5\hfil\break\strut\hfil15\hfil}&
    \vtop{\hsize=6.5em\noindent\hfil\strut92, 148\hfil\break
        \strut\hfil100, 156\hfil}&
    $\bigl(\begin{smallmatrix}1&1\\1&-1\end{smallmatrix}\bigr)$&
    $ $&$1$&$-39$\\[\smallskipamount]
5& & &$\bigl(\begin{smallmatrix}4&2\\2&2\end{smallmatrix}\bigr)$&$1/2$&
    1, 4, 14&\vtop{\hsize=1.5em\noindent\hfil\strut10\hfil}&
    \vtop{\hsize=6.5em\noindent\hfil\strut134, 106, 134\hfil}&
    $\bigl(\begin{smallmatrix}1&1\\1&-1\end{smallmatrix}\bigr)$&
    $ $&$1$&$-39$\\[\smallskipamount]
8& & &$\bigl(\begin{smallmatrix}4&4\\4&8\end{smallmatrix}\bigr)$&$1/2$&
    1, 4, 14&\vtop{\hsize=1.5em\noindent\hfil\strut22\hfil}&
    \vtop{\hsize=6.5em\noindent\hfil\strut102, 74, 102\hfil}&
    $\bigl(\begin{smallmatrix}1&1\\1&-1\end{smallmatrix}\bigr)$&
    $ $&$1$&$-39$\\[\smallskipamount]
11& & &$\bigl(\begin{smallmatrix}8&4\\4&4\end{smallmatrix}\bigr)$&$1/2$&
    7, 21&\vtop{\hsize=1.5em\noindent\hfil\strut5\hfil\break\strut\hfil15\hfil}&
    \vtop{\hsize=6.5em\noindent\hfil\strut78, 78\hfil\break
        \strut\hfil86, 86\hfil}&
    $\bigl(\begin{smallmatrix}1&1\\1&-1\end{smallmatrix}\bigr)$&
    $ $&$1$&$-39$\\[\smallskipamount]
15& & &$\bigl(\begin{smallmatrix}12&6\\6&6\end{smallmatrix}\bigr)$&$1/2$&
    25&19&
    68&
    $\bigl(\begin{smallmatrix}1&1\\1&-1\end{smallmatrix}\bigr)$&
    $ $&$1$&$-39$\\[\smallskipamount]
16& & &$\bigl(\begin{smallmatrix}16&4\\4&2\end{smallmatrix}\bigr)$&$1/2$&
    27&10&
    92&
    $\bigl(\begin{smallmatrix}1&1\\1&-1\end{smallmatrix}\bigr)$&
    $ $&$1$&$-39$\\[\smallskipamount]
17& & &$\bigl(\begin{smallmatrix}16&8\\8&8\end{smallmatrix}\bigr)$&$1/2$&
    27&22&
    60&
    $\bigl(\begin{smallmatrix}1&1\\1&-1\end{smallmatrix}\bigr)$&
    $ $&$1$&$-39$\\
\end{longtable}
  }

\section{Proofs of some intermediate results from local index theory}\label{sec:proofs}
For completeness,
we give short proofs of the claims~\eqref{2.b.2}, \eqref{eq:EtaProd}
and~\eqref{2.c.2}.
This section does not attempt to be self-contained.
Instead, we will state the analogue of statements in existing proofs,
and add explanations only where we deviate from those.

\subsection{Adiabatic limits of twisted products}

Let~$\Gamma\cong\Z/k$ be a finite group that acts effectively and isometrically
on an even-dimensional manifold~$V$ with boundary~$\del V$.
We assume that~$V$ has product geometry near~$\del V$.
We consider~$\of=V/\Gamma$ as an orbifold
with inertia orbifold~$\Lambda \of$.

Consider~$S^1\cong\R/\Z$ and let a generator~$\gamma_0\in\Gamma$
act by sending~$[v]\in\R/\xi\Z$ to~$\bigl[v+\frac\xi k\bigr]$.
Then we will consider the Seifert fibration
\begin{equation}\label{A1.1}
  p\colon M=\bigl(V\times S^1\bigr)/\Gamma\longrightarrow \of\;,
\end{equation}
where~$\Gamma$ acts diagonally on~$V\times S^1$.
We split~$TM=TW\oplus TS^1$ by abuse of notation
and consider a family of metrics
\begin{equation*}
  g^{TM}_\eps=\eps^{-2}\,g^{T\ofh}\oplus g^{TS^1}
\end{equation*}
for~$\eps>0$.
The example we have in mind is of course~$M_\pm$
with metric~$g^{TM}_\eps=\frac1{\eps\zeta_\pm}\,g_{\pm,\eps}$;
see paragraph~\ref{ssec:gell}.

By a {\em Dirac bundle\/} we mean a Hermitian vector bundle
with Hermitian connection and a compatible Clifford multiplication;
see~\cite[def~II.5.2]{LM}.
We assume that~$V$ is equipped with a fixed Dirac bundle~$E_V\to V$,
on which~$\Gamma$ acts, preserving its structure.
On~$M$, we consider the Dirac bundle
\begin{equation}\label{ass:spinbdl}
  E=p^*E_V/\Gamma\longrightarrow M\;.
\end{equation}
We let~$(e_1,\dots,e_m)$ denote a local orthonormal frame of~$TM$ for~$g^{TM}_1$
such that~$e_1$ is vertical and~$e_2$, \dots, $e_m$ are horizontal.
Clifford multiplication with~$e_i$ will be denoted~$c_i$.
Here, we assume that~$c_1$ acts as~$-i^{\frac{m+1}2}c_2\cdots c_m$ on~$p^*E_V$,
so that the Clifford volume element~$i^{\frac{m+1}2}c_1\cdots c_m$ acts as~$1$.
The examples we have in mind are the spinor bundle~$SV$
and the bundle of exterior forms~$\Lambda^\bullet T^*V$ on~$V$,
leading to the spinor bundle and the bundle of even forms on~$M$.

We consider a Dirac-type operator on~$E$
of the form
\begin{equation}\label{ass:dirac}
  D_{M,\eps}=D_{S^1}+\eps D_\ofh
\end{equation}
as in~\cite[(2.3)]{Gorbi},
where~$D_{S^1}=c_1\nabla^E_{e_1}$ is the fibrewise Dirac operator.
In the case of the odd signature operator on~$M$,
$D_\ofh=B_\ofh$ is the signature operator on the orbifold~$\of$.

In the case of the modified spin Dirac operator on~$M$,
we assume that~$SV$ admits a $\Gamma$-invariant spinor~$s$
such that~$\nabla^{SV}$ is supported away from~$\del V$.
If~$D'_{M,\eps}$ and~$D'_\ofh$ denote the geometric Dirac operators on~$M$
and on the orbifold~$\of$
with respect to the metrics above, equation~\eqref{Seq} becomes
\begin{equation*}
  D'_{M,\eps}(p^*s)=p^*(\eps D'_Ws)=\eps\,p^*(f\,s+h\,c_1s+r)\;,
\end{equation*}
with~$f$, $h$ and~$r$ as before.
We now consider the operator
\begin{equation}\label{DVmodEq}
  \begin{aligned}
    D_\ofh=D'_\ofh
    &-\<\punkt,s\>\,(f\,s+h\,c_1s+r)-\<\punkt,r\>\,s\\
    &-\<\punkt,c_1s\>\,(h\,s-f\,c_1s-c_1r)+\<\punkt,c_1r\>\,c_1s\;.
  \end{aligned}
\end{equation}
Then~\eqref{ass:dirac} and~\eqref{DVmodEq} are equivalent to~\eqref{DmodEq}.
We also conclude that~$D_\ofh-D'_\ofh$ is supported away from~$\del \of$
by Property~\ref{ssec:dirac}~\ref{2.g.9} and self-adjoint.

The situation here is simpler than in~\cite{Gorbi}
because~$D_W$ is independent of~$\eps$
and
\begin{equation}\label{ass:anticomm}
  D_{S^1}D_W+D_WD_{S^1}=0\;.
\end{equation}
In the case of the modified spin Dirac operator,
this follows from~\eqref{DVmodEq} because~$f$, $h$, $s$ and~$r$ all
have vanishing vertical derivative.

Because we are in a local product situation,
the space~$L^2(E)$ splits into eigenspaces of~$D_{S^1}^2$
which we may regard as spaces of $L^2$-sections of orbibundles over~$\of$.
These spaces are invariant under~$D_W$ by~\eqref{ass:anticomm}.
In particular, $H=\ker D_{S^1}\subset p_*E$ is isomorphic to the original~$E_V$,
and the connection~$\nabla^E$ induces a unitary
connection~$\nabla^{p_*E}=\nabla^H\oplus\nabla^{H^\perp}$.

To avoid a clash of notation later,
we write~$u$ for the inward normal coordinate on~$M$ and~$\of$ near their
respective boundary.
Then let~$e_2=\frac\del{\del u}$ be the inward normal unit vector
to~$\del M$ with respect to~$g_1$,
extended parallelly over the cylindrical neighbourhood~$u\in[0,1]$
of~$\del M$.
The boundary operator~$D_{\del M,\eps}$ splits in the same manner
as~$D_{M,\eps}$, and in that cylindrical neighbourhood of~$\del M$, we have
\begin{equation*}
  D_{M,\eps}=c_2\,\biggl(\eps\,\frac\del{\del u}+D^\del_{S^1}
  +\eps D_{\del \ofh}\biggr)\;,
\end{equation*}
where~$D^\del_{S^1}=-c_2D_{S^1}$ denotes the fibrewise boundary operator.
By~\eqref{ass:anticomm},
the operators~$D^\del_{S^1}$ and~$D_{\del\ofh}$ anticommute as well.
Both respect the splitting of~$(p|_{\del M})_*E$
into~$H|_{\del\of}$ and~$H^\perp|_{\del\of}$.

Let~$\Pi_{+,\eps}$ denote the spectral projection onto the subspace
of~$L^2(\del M;E)$ spanned by the eigenspinors
of~$D^\del_{S^1}+\eps D_{\del \ofh}$ with positive eigenvalues.
Then~$\Pi_{+,\eps}$ respects the splitting into~$H|_{\del \of}$
and its orthogonal complement,
so
\begin{equation*}
  \Pi_{+,\eps}=\Pi^H_+\oplus\Pi^\perp_{+,\eps}\;.
\end{equation*}
Moreover, $\ker(D^\del_{S^1}+\eps D_{\del \ofh})$
in the sections of~$H|_{\del\of}$ and equals the kernel %
of the restriction of~$D_{\del\ofh}$ to~$H|_{\del\of}$.
The relevant symplectic structure on~$\ker(D^\del_{S^1}+\eps D_{\del \ofh})$
is induced by~$c_2$.

We denote the restriction
of~$D_\ofh$ to~$H$ by~$D_{\ofh,1}$ in analogy with~\cite{Gorbi}.
It is a Dirac operator in the case of the odd signature operator,
and a Dirac operator modified by~\eqref{DVmodEq}
in the case of the modified spin Dirac operator.
Clifford multiplication with the global vertical tangent vector field~$e_1$
still acts on~$H$ and anticommutes with~$D_{\ofh,1}$.
Because~$c_1$ and~$c_2$ anticommute,
the Clifford multiplication~$c_1$ commutes with the boundary operator~$D_{\del \ofh}$,
so the projection~$\Pi^H_+$ commutes with~$c_1$ as well.
The Lagrangian subspaces~$L_D$ and~$L_B$ of~\eqref{LBeq} and~\eqref{LDeq}
are also invariant under~$c_1$.
We immediately conclude that
\begin{equation}\label{EtaDeff}
  \eta_{\mathrm{APS}}(D_{\ofh,1};L_D)=0\;.
\end{equation}

Recall that~$\eta(\Aa)\in\Omega^\bullet(\Lambda \of)$ denotes the orbifold
\etaform of the Bismut superconnection of the fibrewise spin Dirac
operator with respect to the fibrewise trivial spin structure.
We may regard the signature operator as a Dirac operator twisted
by the pullback of the spinor bundle on the base;
for this reason, the \etaform~$\eta(\Aa)$ occurs
in both formulas in the theorem below.

\begin{thm}[{Compare Dai~\cite[Theorem~1.1]{Dai},
  see also~\cite[Theorem~0.1]{Gorbi}}]\label{Thm5.1}
  With the assumptions and notations above, 
  \begin{align*}
    \lim_{\eps\to0}\eta\bigl(D_{M,\eps};L_D\bigr)
    &=\int_{\Lambda \of\setminus \of}
    \hat A_{\Lambda \of}\bigl(T\ofh,\nabla^{T\ofh}\bigr)
    \,2\eta_{\Lambda \of}(\Aa)\;,
    \\
    \lim_{\eps\to0}\eta\bigl(B_{M,\eps};L_B\bigr)
    &=\int_{\Lambda \of\setminus \of}
    \hat L_{\Lambda \of}\bigl(T\ofh,\nabla^{T\ofh}\bigr)
    \,2\eta_{\Lambda \of}(\Aa)\;.
  \end{align*}
\end{thm}

This result is not covered by~\cite{Dai} because the fibrewise operator
is allowed to have a kernel,
and because~$p\colon M\to V$ is a Seifert fibration.
It is not covered by~\cite{Gorbi} because
the base orbifold is allowed to have a boundary.
But of course, the Seifert fibration is locally a twisted product,
and hence the situation here is more specialised than in the two references
above.
A little extra complication comes from the construction~\eqref{DmodEq}
of the modified spin Dirac operator.
Our proof below relies crucially on the fact that the operators~$D_{M,\eps}$
and~$D_{\del M,\eps}$ both respect the splitting of the bundle~$p_*E\to B$
of fibrewise sections into the fibrewise harmonic spinors,
which form a bundle~$H\to B$ by assumption,
and its orthogonal complement~$H^\perp$.
We believe that with a little extra work, this proof extends
to totally geodesic Seifert fibrations.
Probably Dai's proof also extends to totally geodesic fibre bundles
because~\cite[Prop~5.2]{Dai} then holds for sections of~$H^\perp$.

\begin{proof}
  We follow the proof in~\cite{Gorbi} as far as possible.
  We will view~$D_\ofh$ as a differential operator on~$p_*E$
  and~$D_{S^1}$ as an endomorphism of~$p_*E$.
  We consider the restriction of~$D_{M,\eps}$ to~$H^\perp$.
  In accordance with~\cite[sect~2.c]{Gorbi},
  we denote it by~$D_{M,\eps,4}=D_{S^1}+\eps D_{\ofh,4}$,
  where~$D_{\ofh,4}$ describes the action of~$D_\ofh$ on sections of~$p_*E$.
  Note that in our setting, $D_{\ofh,2}=D_{\ofh,3}=0$.
  Let~$\<\punkt,\punkt\>_{M/\ofh}$ denote the fibrewise $L^2$-product,
  let~$\Div_\ofh$ denote the divergence of a vector field or a one-form on~$\of$,
  let~$\Delta^{H^\perp}$ denote the horizontal Laplacian on~$H^\perp\to \of$,
  and let~$c^\ofh$ denote Clifford multiplication by horizontal vectors.
  Then 
  \begin{align*}
    \Div_\ofh\bigl\<\nabla^{H^\perp}\sigma,\tau\bigr>_{M/\ofh}
    &=-\bigl\<\Delta^{H^\perp}\sigma,\tau\bigr>_{M/\ofh}
    +\bigl\<\nabla^{H^\perp}\sigma,\nabla^{H^\perp}\tau\bigr\>_{M/\ofh}\;,\\
    \Div_\ofh\bigl\<c^\ofh\sigma,\tau\bigr\>_{M/\ofh}
    &=\bigl\<D_\ofh\sigma,\tau\bigr\>_{M/\ofh}
    -\bigl\<\sigma,D_\ofh\tau\bigr\>_{M/\ofh}\;.
  \end{align*}
  Because~$D'_\ofh$ has been modified to~$D_\ofh$ by a self-adjoint operator
  of order~$0$ supported away from the boundary,
  and because~$\frac\del{\del r}$ is the inward normal direction,
  we conclude that
  \begin{multline}\label{Prop2.7Formel}
    \norm{\bigl(i-\eps^{-1}D_{M,\eps,4}\bigr)\sigma}_{L^2(\of;H^\perp)}^2
    -\bigl\|\nabla^{H^\perp}\sigma\bigr\|_{L^2(\of;H^\perp)}^2\\
    =\bigl\<\bigl(1+\eps^{-2}D_{S^1}^2+D_{\ofh,4}^2
      -\Delta^{H^\perp}\bigr)\sigma,
      \sigma\bigr\>_M
      -\<ic_2\sigma,\sigma\>_{\del M}
      -\eps^{-1}\<D_{\del M,\eps,4}\sigma,\sigma\>_{\del M}\;.
  \end{multline}

  Let~$H^1(\of,H^\perp;\Pi_{+,\eps}^\perp)$ denote the subspace of the first
  Sobolev space generated by sections that satisfy the APS boundary condition.
  If~$\sigma\in H^1(\of,H^\perp;\Pi^\perp_{+,\eps})$,
  then~$\<c_2\sigma,\sigma\>=0$ because~$c_2$ anticommutes
  with~$D_{\del M,\eps,4}$,
  and~$\<D_{\del M,\eps,4}\sigma,\sigma\>\le 0$, so
  \begin{equation}\label{Prop2.7Ineq}
    \norm{\bigl(i-\eps^{-1}D_{M,\eps,4}\bigr)\sigma}_{L^2(\of;H^\perp)}^2
    \ge\bigl\|\nabla^{H^\perp}\sigma\bigr\|_{L^2(\of;H^\perp)}^2
    +\bigl\<\bigl(1+\eps^{-2}D_{S^1}^2+D_{\ofh,4}^2-\Delta^{H^\perp}\bigr)
      \sigma,\sigma\bigr\>_M\;.
  \end{equation}

Let~$\lambda_B$ denote the smallest absolute value of
a nonzero eigenvalue of the effective horizontal operator~$D_{\ofh,1}$
with respect to the given boundary conditions,
and let~$0<c<\frac{\lambda_B}2$.
Let~$\Gamma=\Gamma_+\dotcup\Gamma_0\dotcup\Gamma_-$ denote a contour in~$\C$,
where~$\Gamma_\pm$ goes around~$\pm[\lambda_B,+\infty]$ with distance~$c$,
and~$\Gamma_0$ is a circle around~$0$ with radius~$c$.
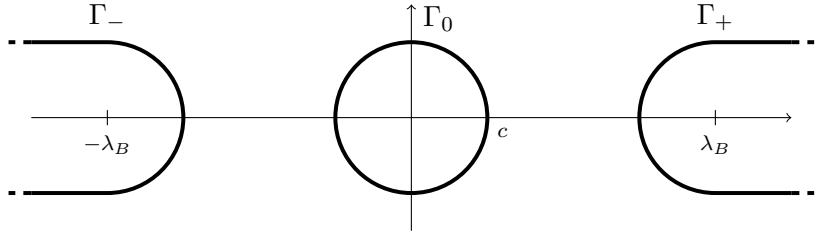
\begin{equation*}
\begin{tikzpicture}
  \draw[->] (-5,0) -- (5,0) ; %
  \draw[->] (0,-1.5) -- (0,1.5) ; %
  \draw (-4,0.1) -- (-4,-0.1) node[below] {$\scriptstyle-\lambda_B$} ;
  \draw (4,0.1) -- (4,-0.1) node[below] {$\scriptstyle\lambda_B$} ;
  \node[below right] at (1,0) {$\scriptstyle c$} ;
  \begin{scope}[line width=1.5pt]
    \draw (0,1) arc (90:-270:1cm) node[above right] {$\Gamma_0$} ;
    \draw (-5,1) -- (-4,1) node[above] {$\Gamma_-$}
	arc (90:-90:1cm) -- (-5,-1) ;
    \draw (5,1) -- (4,1) node[above] {$\Gamma_+$}
	arc (90:270:1cm) -- (5,-1) ;
    \begin{scope}[dashed]
      \draw (-5,1) -- (-5.3,1) ;
      \draw (-5,-1) -- (-5.3,-1) ;
      \draw (5,1) -- (5.3,1) ;
      \draw (5,-1) -- (5.3,-1) ;
    \end{scope}
  \end{scope}
\end{tikzpicture}
\end{equation*}

  Assume that~$\lambda$ is not in the spectrum of~$D_{M,\eps,4}$
  with APS boundary conditions given by~$\Pi^\perp_{+,\eps}$.
  Using parametrices on~$\del \of\times[0,\infty)$ and on the double of~$\of$,
  one can construct a resolvent
  \begin{equation*}
    R_\eps(\lambda)\colon L^2(\of;H^\perp)
  \to H^1(\of,H^\perp;\Pi^\perp_{+,\eps})
  \end{equation*}
  of~$D_{M,\eps,4}$.
  We define the family of Schatten norms of operators~$A$ acting
  on~$L^2(M;E)\cong L^2(\of;p_*E)$ by
  \begin{equation*}
    \norm A_p=\tr\Bigl((A^*A)^{\frac p2}\Bigr)
  \end{equation*}
  for~$1\le p<\infty$, and let~$A_\infty$ denote the operator norm.
  Because~$D_\ofh^{\prime 2}-\Delta^{p_*E}$ is a bundle endomorphism on~$E\to M$,
  we can use the inequality~\eqref{Prop2.7Ineq} above to prove the
  analogue of~\cite[Prop~2.7]{Gorbi}.
  In particular,
  there exists a constant~$\eps_0>0$ such that for all~$p>\dim M$,
  all~$\eps\in(0,\eps_0)$ and all~$\lambda\in\Gamma$, one has
  \begin{equation}\label{Prop2.7}
    \norm{R_\eps(\lambda)}=O(1,\eps\abs\lambda)\qquad\text{and}\qquad
    \norm{R_\eps(\lambda)}=O(\abs\lambda)\;.
  \end{equation}

  Let~$H^1(\of,H;\Pi^H_{+,\Lambda})$ denote the subspace of the first Sobolev
  space spanned by sections satisfying the Lagrangian APS boundary
  condition fixed above.
  Then we consider the resolvent
  \begin{equation*}
    (\lambda-D_{\ofh,1})^{-1}\colon L^2(\of;H)
    \longrightarrow H^1(\of,H;\Pi^H_{+,\Lambda})\;.
  \end{equation*}
  Obviously~$(\lambda-D_{M,\eps})^{-1}
  =(\lambda-D_{\ofh,1})^{-1}\oplus R_\eps(\lambda)$.
  Because~$D_{\ofh,1}$ is the effective horizontal operator,
  Proposition~2.8 in~\cite{Gorbi} reduces to
  \begin{equation}\label{Prop2.8}
    \norm{(\lambda-D_{\ofh,1})^{-1}}_\infty=O(1)\qquad\text{and}\qquad
    \norm{(\lambda-D_{\ofh,1})^{-1}}_p=O(\abs\lambda)
  \end{equation}
  for all~$\lambda\in\Gamma$,
  which can be proved in the same way as~\eqref{Prop2.7Formel}.
  As an analogue of~\cite[Prop~2.9]{Gorbi}, we get
  \begin{equation}\label{Prop2.9}
    \norm{(\lambda-\eps^{-1}D_{M,\eps})^{-1}-(\lambda-D_{\ofh,1})^{-1}}_\infty
    =O(\eps\abs\lambda)
  \end{equation}
  for all~$\lambda\in\Gamma$.

  Because~$D_{M,\eps}=\eps\,D_{\ofh,1}\oplus D_{M,\eps,4}$,
  the spectral projection~$P_\eps$ in~\cite[sect~2.f]{Gorbi}
  coincides with the spectral projection onto~$\ker(D_{\ofh,1})=\ker(D_{M,\eps})$
  independent of~$\eps$.
  Using~\eqref{EtaDeff}, (\ref{Prop2.7}--\ref{Prop2.9}),
  we can adapt the proof of~\cite[Prop~2.10]{Gorbi} to show that
  there exists a small~$\alpha>0$ such that
  \begin{multline}\label{Prop2.10}
    \lim_{\eps\to 0}\int_{\eps^{\alpha-2}}^\infty\frac1{\sqrt{\pi t}}
    \tr\Bigl(D_{M,\eps}e^{-tD_{M,\eps}^2}\Bigr)\,dt\\
    =\lim_{\eps\to 0}\int_{\eps^{\alpha-2}}^\infty\frac1{\sqrt{\pi t}}
    \tr\biggl((1-P_\eps)\Bigl(D_{M,\eps}e^{-tD_{M,\eps}^2}\Bigr)(1-P_\eps)\biggr)\,dt
    =\eta(D_{\ofh,1})=0\;.
  \end{multline}

  Note that the orbifold \etaform~$\eta_{\Lambda \of}(\Aa)$
  vanishes on the principal stratum~$\of\subset\Lambda \of$
  because the Seifert fibration~$M\to V$ is a twisted product
  and the fibrewise operator~$D_{S^1}$ has symmetric spectrum.
  The additional divergent terms in the heat asymptotics
  of the supertrace of~$e^{-tD_\ofh}$ caused by the non-geometric
  terms introduced in~\eqref{DmodEq} and~\eqref{DVmodEq}
  do not cause extra complications here because they are
  supported on the regular stratum (and away from the boundary).
  Because the singular stratum does not extend to the boundary~$\del \of$,
  the right hand side of the expression in the theorem vanishes
  near the boundary.

  We can now use finite propagation speed to obtain the analogue
  of~\cite[Prop~2.12]{Gorbi}, which says that
  \begin{equation*}
    \lim_{\eps\to 0}\int_0^{\eps^{\alpha-2}}\frac1{\sqrt{\pi t}}
    \tr\Bigl(D_{M,\eps}e^{-tD_{M,\eps}^2}\Bigr)\,dt
  \end{equation*}
  equals the right hand side of the expression in the theorem.
  Together with~\eqref{Prop2.10}, this finishes the proof.
\end{proof}

\subsection{A product formula for \etaforms}
\label{A3}
Recall the family~$\del\mathcal M_\pm=\Sigma_\pm\times E_\pm\to(0,\infty)$,
where~$\Sigma_\pm$ denotes a fixed K3 surface and~$E_\pm\to(0,\infty)$
denotes the family of
tori~$E_{\pm,a}=(S^1_{\xi_\pm}\times S^1_{a\zeta_\pm})/\Gamma_\pm$
for~$a\in(0,\infty)$.
We now consider the \etaforms~$\eta(\mathbb B)$ and~$\eta(\mathbb D)$
associated with the fibrewise signature and Dirac operators
on~$\del\mathcal M_\pm\to(0,\infty)$.

We start with the signature operator.
Because~$X_{\pm,a}=\Sigma\times E_{\pm,a}$, we may split
\begin{equation*}
  \Omega^\bullet(\mathcal M_\pm,(0,\infty))
  \cong\Omega^\bullet(\Sigma)\mathbin{\hat\otimes}
  \Omega^\bullet(E_\pm/(0,\infty))\;.
\end{equation*}
Let~$\mathbb B_E$ denote the superconnection associated to the fibrewise
signature operator for the family~$E_\pm\to(0,\infty)$,
equipped with the trivial spin structure.
Let~$B_\Sigma$ denote the signature operator on~$\Sigma$.
With respect to the splitting above, we have
\begin{equation*}
  \mathbb B_t
  =\sqrt t B_\Sigma\mathbin{\hat\otimes}\id
  +\id\mathbin{\hat\otimes}\mathbb B_{E,t}
  =\sqrt t B_\Sigma\mathbin{\hat\otimes}\id+\id\mathbin{\hat\otimes}\sqrt tB_E
  +\tilde\nabla^u\;.
\end{equation*}
Because~$B_\Sigma$ is independent of~$a$,
we have~$[\tilde\nabla^u,B_\Sigma]=0$.
Also~$\Sigma$ is even-dimensional, so the signature operator~$B_\Sigma$
has symmetric spectrum.
We conclude
\begin{equation}\label{eq:BProd}
  \begin{aligned}
    \str_{\Omega^\bullet(\del\mathcal M/(0,\infty))}
    \biggl(\frac{\del\mathbb B_t}{\del t}\,e^{-\mathbb B_t^2}\biggr)
    &=\str_{\Omega^\bullet(\Sigma)}
    \biggl(\frac1{\sqrt{4t}}\,B_\Sigma\,e^{-tB_\Sigma^2}\biggr)
    \cdot\str_{\Omega^\bullet(E/(0,\infty))}\bigl(e^{-\mathbb B_{E,t}^2}\bigr)\\
    &\qquad+\str_{\Omega^\bullet(\Sigma)}\Bigl(e^{-tB_\Sigma^2}\Bigr)
    \cdot\str_{\Omega^\bullet(E/(0,\infty))}
    \biggl(\frac{\del\mathbb B_{E,t}}{\del t}\,e^{-\mathbb B_{E,t}^2}\biggr)\\
    &=\ind(B_\Sigma)\cdot\frac12\str_{\Omega^\bullet(E/(0,\infty))}
    \Bigl(B_E\,\bigl[\tilde\nabla^u,B_E\bigr]
    \,e^{-tB_E^2}\Bigr)\;,
  \end{aligned}
\end{equation}
where we have used the McKean-Singer
formula for~$\ind(B_\Sigma)$
in the last step.

We now want to relate~$\mathbb B_{E,t}=\sqrt t\,B_E+\tilde\nabla^u$
to the superconnection~$\Aa$ of the fibrewise Dirac operator
on~$E_\pm\to(0,\infty)$ with respect to the trivial spin structure.
We regard the fibrewise spinor bundle as the pullback of
a bundle~$S=S^+\oplus S^-$ on the base~$(0,\infty)$.
The signature operator acts on the fibrewise spinor bundle twisted
by itself.
We therefore have
\begin{equation*}
  \Omega^\bullet(E_\pm/(0,\infty))\cong p_*p^*(S\otimes S)\;,
\end{equation*}
and the corresponding superconnection is now given
by~$\mathbb B_{E,t}=\Aa_t\otimes\id+\id\otimes\nabla^S$.
The two terms above supercommute, so we have
\begin{equation*}
  \str_{\Omega^\bullet(E/(0,\infty))}
  \biggl(\frac{\del\mathbb B_{E,t}}{\del t}
  \,e^{-\mathbb B_{E,t}^2}\biggr)
  =\str_{p_*p^*S}\biggl(\frac{\del\Aa_t}{\del t}\,e^{-\Aa_t^2}\biggr)
  \cdot\tr_S\Bigl(e^{-(\nabla^S)^2}\Bigr)\;.
\end{equation*}
Integrating over~$t$, we get
\begin{equation*}
  \tilde\eta(\mathbb B_E)=\tilde\eta(\Aa)\,\ch(\nabla^S)
\end{equation*}
In degree~$1$, this equals twice the spinorial \etaform
because~$\rk S=2$.
Equivalently, the reader is invited to compare Bismut and Cheeger's
results for the universal spinorial \etaform and the signature \etaform
of bundles of flat tori in~\cite[Thms~2.22, 2.25]{BChTorus}.
Together with~\eqref{eq:BProd},
we see that~$\tilde\eta(\mathbb B)=2\ind(B_\Sigma)\,\tilde\eta(\Aa)$,
which is exactly~\eqref{eq:EtaProdB}.

Similarly,
the fibrewise spinor bundle of~$X_{\pm,a}$
is the exterior tensor product of the spinor bundles of~$\Sigma$ and~$E_{\pm,a}$.
Proceeding as in~\eqref{eq:BProd},
we prove~\eqref{eq:EtaProdD}.

\subsection{Adiabatic limits of families of flat tori}\label{A2}
We consider a family of fibred manifolds~$E\to F\to\R$ as in Section~\ref{2.y},
diagram~\eqref{2.d.13}.
We will prove Proposition~\ref{Prop2y.2},
which is a special case of the adiabatic limit
formula for \etaforms of Bunke, Ma~\cite{BuMa} and Liu~\cite{Liu},
but as a strict equation of forms, not as an equation modulo exact forms.
To this end, we will simply compute both sides of the equation.
We believe that under suitable conditions, the adiabatic limit formula
for \etaforms holds in this strict sense for more general
iterated fibre bundles.

We fix~$y\in\R$; later we will consider the limit~$y\to\infty$.
For~$x\in\R$, we identify
\begin{equation*}
  E_x=\C/(\Z+(x+iy)\Z)\qquad\text{and}\qquad F_x=\R/y\Z\;.
\end{equation*}
The fibration~$E\to F$ is formed by taking the imaginary part.
The standard Euclidean metric on~$\C$ induces a fibrewise metric on~$E\to\R$.
The group~$S^1$ acts isometrically by translation in the real direction
in~$\C$.

On the total space of~$E$, we consider the fibrewise orthonormal
base induced by~$e_1=1$ and~$e_2=i\in\C$.
The connection~$\nabla^W$ in~\eqref{eq:NablaW} induces
a horizontal subspace~$T^HE\subset TE$ for the fibration~$E\to\R$.
It is spanned by the vector field~$e_3$ induced from the vector field
\begin{equation*}
  \C\times\R\longrightarrow\C\times\R\qquad\text{with}\qquad
  (u+iv,x)\longmapsto(v/y;1)\;,
\end{equation*}
which is invariant under the $x$-dependent action of~$\Z^2$ on~$\C$.
Obviously,
\begin{equation}\label{Eqn:Lie}
  [e_1,e_2]=[e_1,e_3]=0\qquad\text{and}\qquad[e_2,e_3]=\frac1y\,e_1\;.
\end{equation}
The connection~$\nabla^{T(E/\R)}$ on the vertical tangent bundle
is given by
\begin{equation*}
  \nabla^{T(E/\R)} e_1=\frac1{2y}e_2\,dx\qquad\text{and}\qquad
  \nabla^{T(E/\R)} e_2=-\frac1{2y}e_1\,dx\;,
\end{equation*}
see~\eqref{eq:Nabla0}.
Because~$y\in\R$ is constant, this connection is flat.

We identify the fibrewise spinor
bundles~$S(E/\R)=S^+(E/\R)\oplus S^-(E/\R)\to E$
with~$\underline\C\oplus\underline\C$.
If~$d$ denotes the trivial connection on the spinor bundle,
then~$\nabla^{T(E/\R)}$ induces the connection
\begin{equation*}
  \nabla^{S(E/\R)}=d+\frac1{4y}\,c_1c_2\,dx\;.
\end{equation*}
Let~$W=p_*S(E/\R)$ denote the infinite-dimensional vector bundle over~$\R$
whose fibres are the spaces of sections of~$S(E/\R)|_{E_x}$.
We identify sections of~${p_*S(E/\R)}$ with sections of~$S(E/\R)$.
Because the fibres of~$p$ have vanishing mean curvature,
the induced connection takes the form
\begin{equation*}
  \nabla^{p_*S(E/\R)}s=\nabla^{S(E/\R)}_{e_3}s\,dx\;.
\end{equation*}

Let~$D_x$ denote the fibrewise Dirac operator over~$x\in\R$.
Then the Bismut superconnection for the fibration~$E\to\R$ takes the form
\begin{equation*}
  \mathbb B_t
  =\sqrt t\,D_x+\nabla^{p_*S(E/\R)}\;.
\end{equation*}
Because
\begin{align*}
  [\nabla^{p_*S(E/\R)},D_x]
  &=-c_1[\nabla^0_{e_3},\nabla^0_{e_1}]\,dx-c_2[\nabla^0_{e_3},\nabla^0_{e_2}]\,dx
  -[c_1c_2,c_1]\nabla^0_{e_1}\,\frac{dx}{4y}
  -[c_1c_2,c_2]\nabla^0_{e_2}\,\frac{dx}{4y}\\
  &=c_2\nabla^0_{e_1}\,\frac{dx}y-c_2\nabla^0_{e_1}\,\frac{dx}{2y}
  +c_1\nabla^0_{e_2}\,\frac{dx}{2y}
  =c_2\nabla^S_{e_1}\,\frac{dx}{2y}
  +c_1\nabla^S_{e_2}\,\frac{dx}{2y}\;,
\end{align*}
by~\eqref{eq:EtaForm},
the \etaform for bundles with even-dimensional fibres is given by
\begin{align}
  \begin{split}\label{EtaER.1}
  \tilde\eta(\mathbb B)
  &=-\frac1{8y\pi i}
  \int_0^\infty\tr\Bigl(ic_1c_2
  \,\bigl(c_1\nabla^S_{e_1}+c_2\nabla^S_{e_2}\bigr)
  \bigl(c_1\nabla^S_{e_2}+c_2\nabla^S_{e_1}\bigr)\,dx\,e^{-tD_x^2}\Bigr)\,dt\\
  &=\frac1{8y\pi}
  \int_0^\infty\tr\Bigl(\bigl((\nabla^S_{e_1})^2-(\nabla^S_{e_2})^2\bigr)
  \,e^{-tD_x^2}\Bigr)\,dt\,dx\;.
  \end{split}
\end{align}

The space of vertical sections is spanned by sections
of the form~$\phy_{m,n}s_\pm$, where
\begin{equation*}
  \phy_{m,n}(u,v)=e^{2\pi i\bigl(m(u-\frac xy\,v)+n\frac vy\bigr)}
\end{equation*}
for~$m$, $n\in\Z$ and~$s_\pm$ is a fibrewise parallel section of~$S^\pm(E/\R)$.
The vertical Laplacian takes the form~$-\del_u^2$,
and its kernel is spanned by the functions~$\phy_{0,n}s^\pm$.
Because~$S(E/\R)$ has rank~$2$,
we can therefore rewrite the \etaform as
\begin{equation}\label{EtaER.2}
  \tilde\eta(\mathbb B)
  =\frac{dx}{4\pi y}\int_0^\infty\sum_{m,n\in\Z}4\pi^2
  \biggl(\Bigl(\frac{n-mx}y\Bigr)^2-m^2\biggr)
  e^{-4\pi^2t\bigl(\bigl(\frac{n-mx}y\bigr)^2+m^2\bigr)}\,dt\;.
\end{equation}
For fixed~$m$, the sum over~$n$ describes the spectrum of
a Dirac operator on a circle~$S^1_y$ of length~$y$ with coefficients
in a flat vector bundle.
Approximating the heat kernel on~$S^1_y$ by the Euclidean heat kernel gives
\begin{align*}
  \sum_n4\pi^2
  \biggl(\Bigl(\frac{n-mx}y\Bigr)^2-m^2\biggr)
  e^{-4\pi^2t\bigl(\frac{n-mx}y\bigr)^2}
  &=-\biggl(4\pi^2m^2+\frac\del{\del t}\biggr)
  \sum_{n\in\Z}e^{-4\pi^2t\bigl(\frac{n-mx}y\bigr)^2}\\
  &=-\biggl(4\pi^2m^2-\frac1{2t}\biggr)\,\frac y{\sqrt{4\pi t}}
  +O\Bigl((1+m^2)e^{-\frac{y^2-c}{4t}}\Bigr)
\end{align*}
for each small~$c>0$, uniformly in~$m$.
For~$\alpha>0$ small, we compute
\begin{multline}\label{EtaER.3}
  \frac{dx}{4\pi y}\int_0^{y^{2-\alpha}}\sum_{m,n\in\Z}4\pi^2
  \biggl(\Bigl(\frac{n-mx}y\Bigr)^2-m^2\biggr)
  e^{-4\pi^2t\bigl(\bigl(\frac{n-mx}y\bigr)^2+m^2\bigr)}\,dt\\
  =-\frac{dx}{4\pi}\int_0^{y^{2-\alpha}}\sum_{m\in\Z}
  \biggl(4\pi^2m^2-\frac1{2t}\biggr)\,\frac 1{\sqrt{4\pi t}}\,e^{-4\pi^2m^2t}
  \,dt
\end{multline}

For~$t\ge y^{2-\alpha}$,
we only need to study the contribution from~$\ker D^{E/F}$,
which can be written as
\begin{align*}
  \frac1{2\pi i}\int_{y^{2-a}}^\infty\str\biggl(P^{\ker D^{E/F}}
  \,\frac{\del\mathbb B_t}{\del t}\,e^{-\mathbb B_t^2}\biggr)\,dt
  &=\frac{dx}{2y\pi}\,\int_{y^{2-\alpha}}^\infty
  \sum_n\frac{4\pi^2n^2}{y^2}\,e^{-\frac{4\pi^2n^2t}{y^2}}\,dt\\
  &=\frac{dx}{4y\pi}\,\sum_ne^{-4\pi^2y^{-\alpha}n^2}\,dx\;.
\end{align*}
This sum converges to~$0$ as~$y\to\infty$ for~$\alpha<1$ because
\begin{equation*}
  \frac 1y\sum_ne^{-4\pi^2y^{-\alpha}n^2}
  \le\frac2y\sum_{n=0}^\infty e^{-4\pi^2y^{-\alpha}n}
  =\frac 2y\cdot\frac1{1-e^{-4\pi^2y^{-\alpha}}}
  =\frac 2{4\pi^2y^{1-\alpha}}+o\bigl(4\pi^2y^{-\alpha}\bigr)\;.
\end{equation*}
In general, one would expect here the \etaform of the effective
fibrewise operator on~$F\to\R$, acting on sections of~$\ker D^{E/F}$,
and some extra terms in the case that there are very small eigenvalues.
Because the kernel bundle is trivial here, it is not surprising that
this form vanishes in our situation.
Combining this with the computations above, we finally see that
\begin{equation}\label{EtaER.4}
  \lim_{y\to\infty}\tilde\eta(\mathbb B)
  =\frac{dx}{4\pi}\int_0^\infty\sum_{m\in\Z}
  \biggl(\frac1{2t}-4\pi^2m^2\biggr)\,e^{-4\pi^2m^2t}
  \,\frac{dt}{\sqrt{4\pi t}}
\end{equation}

We now consider the fibration~$E\to F$.
We choose the horizontal bundle spanned by the vectors~$e_2$ and~$e_3$
above.
We identify the spinor bundle~$S(E/F)\to E$ with~$\underline\C$.
From equation~\eqref{Eqn:Lie}, we get the superconnection~$\Aa_t$
for the family~$E\to F$ as
\begin{equation*}
  \Aa_t=\sqrt t\,c_1\nabla^S_{e_1}+\nabla^{p_*S(E/\R)}+\frac1{4y\sqrt t}\,c_1\,dv\,dx\;.
\end{equation*}
\enlargethispage{0.4\baselineskip}
Its curvature is given by
\begin{equation*}
  \Aa_t^2=-t(\nabla_{e_1}^0)^2+\frac1{2y}\,dv\,dx\;.
\end{equation*}
Assuming that the Clifford volume element~$ic_1$ acts as~$1$,
the \etaform of the bundle~$E\to F$ with odd-dimensional fibres
takes the form
\begin{align}
  \begin{split}\label{EtaEF.1}
  \tilde\eta(\Aa)
  &=(2\pi i)^{-\tfrac{N^F}2}
  \int_0^\infty\tr\biggl(\frac{\del\Aa_t}{\del t}\,e^{-\Aa_t^2}\biggr)\,
  \frac{dt}{\sqrt\pi}\\
  &=\int_0^\infty\tr\Biggl(c_1
  \biggl(\nabla_{e_1}^S-\frac1{8\pi iyt}\,dv\,dx\biggr)
  \biggl(1-\frac1{4\pi iy}\,dv\,dx\,\nabla^S_{e_1}\biggr)
  e^{t(\nabla_{e_1}^S)^2}\Biggr)
  \,\frac{dt}{\sqrt{4\pi t}}\\
  &=\int_0^\infty\tr\Biggl(\biggl(-i\nabla_{e_1}^S+\frac{dv\,dx}{8\pi yt}
  \bigl(1+2t(\nabla^S_{e_1})^2\bigr)\biggr)e^{t(\nabla_{e_1}^S)^2}\Biggr)
  \,\frac{dt}{\sqrt{4\pi t}}\\
  &=\int_0^\infty\tr\Biggl(\biggl(-i\nabla_{e_1}^S+\frac{dv\,dx}{8\pi yt}
  \Bigl(1+2t\,\frac\del{\del t}\Bigr)\biggr)e^{t(\nabla_{e_1}^S)^2}\Biggr)
  \,\frac{dt}{\sqrt{4\pi t}}\;.
  \end{split}
\end{align}

The space of vertical sections is spanned by sections of the form~$\phy_m$
for~$m\in\Z$, where
\begin{equation*}
  \phy_m(u)=e^{2\pi imu}\;.
\end{equation*}
We can now compute the integral of~$\tilde\eta(\Aa)$ over the fibres
of~$F\to\R$ as
\begin{align}
  \begin{split}\label{EtaEF.2}
  \int_{F/\R}\tilde A(\Aa)
  &=\frac{dx}{4\pi}\int_0^\infty
  \sum_{m\in\Z}\biggl(\frac1{2t}+\frac\del{\del t}\biggr)e^{-4\pi^2m^2t}
  \,\frac{dt}{\sqrt{4\pi t}}\\
  &=\frac{dx}{4\pi}\int_0^\infty\sum_{m\in\Z}
  \biggl(\frac1{2t}-4\pi^2m^2\biggr)\,e^{-4\pi^2m^2t}
  \,\frac{dt}{\sqrt{4\pi t}}\;.
  \end{split}
\end{align}

\begin{proof}[Proof of Proposition~\ref{Prop2y.2}]
  The Proposition follows by comparing~\eqref{EtaER.4}
  and~\eqref{EtaEF.2}.
\end{proof}

\pdfoutput=1

\section{On the values of the function
  \texorpdfstring{$F_{k,\eps}(s)$}{Fke(s)}}\label{A2v}

\def\be{\begin{equation}} \def\ee{\end{equation}} \def\bes{\begin{equation*}} \def\ees{\end{equation*}}
\def\={\;=\;}   \def\+{\,+\,} \def\mi{\,-\,} \def\iv{^{-1}}
\def\T{\Theta}  \def\z{\zeta_k} \def\l{\lambda}  \def\g{\gamma}  \def\vth{\vartheta}
\def\SL{SL_2(\Z)}  \def\sm#1#2#3#4{\bigl(\begin{smallmatrix}#1&#2\\#3&#4\end{smallmatrix}\bigr)}

In Section \ref{sec:nu} of this paper (Proposition~\ref{Prop2v.2}, Theorem~\ref{Thm2h.0}, Theorem~\ref{Thm2h.1}) it was
shown that the \nuinvts of extra twisted connected sums can be computed in terms of values of the analytic function
$F_{k,\eps}\colon(0,\infty)\to\R$ defined for each~$k\in\mathbb N$ and integer $\eps$ prime to~$k$ by
$$ F_{k,\eps}(s)\=\int_0^\infty\int_0^s\sum_{m\,\equiv\,\eps n\;(\text{mod}\,k)} mn\,e^{-t(m^2+n^2a^2)}\,da\,dt \qquad (s \in \R_{>0})$$
(Definition~\ref{Def2h.0}).
In this section we will give a closed formula for $F_{k,\eps}(s)$ in terms of the Dedekind 
\etafunc, show that it is equal to the arccosine (or arcsine, or arctangent) of a computable algebraic number 
whenever $s^2$ is rational, and show that the specific combinations of $F_{k,\eps}$-values occurring in 
Theorems~\ref{Thm:A} and~\ref{Thm:B} can be evaluated in terms of Dedekind sums.

\subsection{Evaluation of \texorpdfstring{$F_{k,\eps}(s)$}{F(k,epsilon)(s)} in terms of the Dedekind \etafunc}
For $\tau$ in the upper half-plane~$\Hh$ we denote by $\eta(\tau)$ and $\logeta(\tau)$ the Dedekind \etafunc and 
the principal branch (real on the positive imaginary axis) of its logarithm, given explicitly by
\begin{equation}
\label{eq:explicit_principal}
\eta(\tau)\=
e^{\pi i\tau/12}\,\prod_{n=1}^\infty\bigl(1- e^{2\pi in\tau}\bigr)\,, \qquad
    \logeta(\tau) \= \frac{\pi i\tau}{12}\,-\,\sum_{n=1}^\infty\frac{\sigma(n)}n\,e^{2\pi in\tau}\,,
\end{equation}
where $\sigma(n)$ denotes the sum of the positive divisors of~$n$. The fact that $\eta(\tau)^{24}$ is a modular form 
of weight~12 on $\,\SL$ implies that $\logeta$ satisfies the transformation equation
 \be\label{Ltransf}  \logeta\biggl(\frac{a\tau+b}{c\tau+d}\biggr) 
       \= \logeta(\tau) \+ \frac14\,\text{Log}\bigl(-(c\tau+d)^2\bigr)\+\frac{\pi i}{12}N(a,b,c,d)\,, \ee
for all~$\sm abcd\in\SL$, where $\,\text{Log}\,$ denotes the principal 
branch (real on the positive real axis) of the logarithm on~$\C\setminus(-\infty,0]$ and $N(a,b,c,d)$ is an integer
given by $N(a,b,c,d)=b/d$ ($=\pm b$) if~$c=0$ and by $N(a,b,c,d)=\frac{a+d}c-12\DS(d,c)$ if~$c\ne0$,
where the Dedekind sum~$\DS(d,c)$ is defined in~\eqref{eq:DedekindSum}.
Our first result is:

\begin{prop} \label{Prop1} 
The value of $F_{k,\eps}(s)$ for any $k\in\mathbb N$, integer $\eps$ prime to~$k$, and positive real number~$s$ is given by
  \be\label{New} F_{k,\eps}(s) \= 2\,\Im\logeta\Bigl(\frac{-\eps^*+is\iv}k\Bigr) \+ \frac{\pi\eps^*}{6k}\,, \ee
where $\eps^*\in\Z$ is any solution of $\eps\eps^*\equiv1\;\text{\rm(mod~$k$)}$. 
\end{prop} 
\begin{proof} 
We first rewrite the definition of $F_{k,\eps}$ as
 \be\label{Fint} s\,\frac d{ds}\,F_{k,\eps}(s) \= \frac\pi{k}\,\int_0^\infty\T_{k,\eps}(s,t)\,dt \qquad\text{and}\qquad F_{k,\eps}(0)=0\,, \ee
where $\T_{k,\eps}(s,t)$ is defined for $s,\,t>0$ by
 $$ \T_{k,\eps}(s,t) \= \sum_{m\,\equiv\,\eps n\;(\text{mod}\,k)}mn\,e^{-\pi t(m^2/s+n^2s)/k} \;. $$
This theta series satisfies the functional equations
 \be\label{FE}  \T_{k,\eps}(s,t) \= -\T_{k,-\eps}(s,t) \= \T_{k,\eps^*}(s\iv,t) \= t^{-3}\,\T_{k,\eps}(s,t\iv)\,, \ee
(where~$\eps^*\equiv\eps\iv\!\pmod k$ as above), as we see by changing the sign of~$m$, interchanging 
$m$ and~$n$, or applying the Poisson summation formula with respect to both $m$ and~$n$. If instead we
apply Poisson summation with respect to~$m$ only, we obtain the stronger identity
  \be\label{Th2} \T_{k,\eps}(s,t) \= \frac{(s/t)^{3/2}}{i\sqrt k}\,\sum_{m,n\in\Z} mn\,\z^{mn\eps}\,e^{-\pi s(m^2/t+n^2t)/k} 
      \qquad\bigl(\z:=e^{2\pi i/k}\bigr)\,,  \ee
which also makes it clear that the integral in~\eqref{Fint} converges, since it shows that
$\T_{k,\eps}(s,t)$ is exponentially small as $t$ tends to either~0 or~$\infty$.  Inserting~\eqref{Th2} into~\eqref{Fint}
and applying the elementary formula
  $$ \int_0^\infty e^{-c_1t-c_2/t}\,t^{-3/2}\,dt \=  %
          \= \sqrt{\frac\pi{c_2}}\;e^{-2\sqrt{c_1c_2}} \qquad(c_1,\,c_2>0)  $$
with $c_1=\pi sn^2/k$, $c_2=\pi sm^2/k$, we find
  $$ \frac{ik}{2\pi}\,F_{k,\eps}'(s) \= \sum_{m,n>0} n\,\bigl(\z^{\eps mn}-\z^{-\eps mn}\bigr)\,e^{-2\pi mns/k}
   \= \sum_{n=1}^\infty\sigma(n)\,\bigl(\z^{\eps n}-\z^{-\eps n}\bigr)\,e^{-2\pi ns/k} \,,$$
\enlargethispage{\baselineskip}
and this can be integrated immediately using the definition of~$\logeta$ to give the formula
\be\label{Old} F_{k,\eps}(s) \= 2\,\Im\logeta\Bigl(\frac{\eps+is}k\Bigr)  \+ c_{k,\eps} \ee
for some constant~$c_{k,\eps}$ depending only on $k$ and~$\eps$. We then use the modularity property~\eqref{Ltransf} 
and the fact that $\frac{\pm\eps+is}k=\g\bigl(\frac{\mp\eps^*+is^{-1}}k\bigr)$ with
$\g=\bigl(\begin{smallmatrix}\pm\eps&*\\k&\pm\eps^*\end{smallmatrix}\bigr)\in\SL$
to deduce~\eqref{New} from~\eqref{Old}, up to determining the value of the
constant $c_{k,\eps}$. That is determined immediately from the property
${F_{k,\eps}(0)=0}$,
because $\logeta(\tau)=\pi i\tau/12+\text o(1)$ for~$\Im(\tau)\to\infty\,$.
\end{proof}

Using the transformation law~\eqref{Ltransf} again,
we can evaluate the constant~$c_{k,\eps}$ in~\eqref{Old} to get
\begin{equation}
  F_{k,\eps}(s) \= 2\,\Im\logeta\Bigl(\frac{\eps+is}k\Bigr)
  \+2\pi \DS(\eps,k)-\frac{\pi\eps}{6k}\;.
\end{equation}
giving 
an alternative formula for the function $F_{k,\eps}(s)$.  
In some cases this same transformation law can be used to give a complete formula for~$F_{k,\eps}(s)$ in terms of
Dedekind sums.  This happens whenever one (and hence both) of the two $\SL$-equivalent numbers 
$\frac{\pm\eps+is}k$ and~$\frac{\mp\eps^*+is}k$ is~$\SL$-equivalent to its negative conjugate.
An easy calculation shows that the equation $\g\tau=-\overline\tau$ for $\g=\sm abcd\in\SL$ and $\tau$ in the upper 
half-plane holds if and only if~$a=d$ and~$|c\tau+a|=1$, which in our situation says that 
$s^2=\frac1{c^2}-\bigl(\frac\eps k+\frac ac\bigr)^2$ for some integers $a$~and~$c$ with $a^2\equiv1\;\text{(mod~$c$)}$.
In all such cases, the number $F_{k,\eps}(s)$ is the sum of a rational multiple of~$\pi$ and the arctangent of the 
square-root of a positive rational number. Concrete examples where this happens and where the Dedekind sum
occurring can be evaluated in closed form are the special values
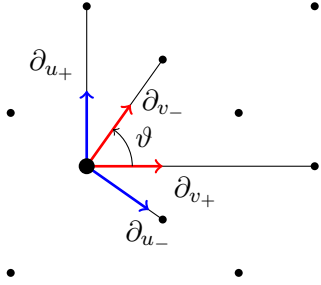
  $$ F_{k,1}\biggl(\frac1{\sqrt{k^2-1}}\biggr) \= -\,F_{k,1}\bigl(\sqrt{k^2-1}\bigr) \mi \frac{(k-1)(k-2)}{6k}\,\pi\;
    \=  \arctan\sqrt{\frac{k+1}{k-1}} \mi \frac{3k+2}{12k}\,\pi $$
for integers $k>1$ and 
  $$ F_{k,1}\biggl(\sqrt{\frac mn}\biggr) \= \arctan\sqrt{\frac mn} \mi \frac{km+2}{12k}\,\pi $$
for positive integers $m$ and~$n$ with $m+n=2k$.  We omit the details.

\subsection{Algebraic values}\label{app:2} Except in the cases just mentioned, there is in general no simple closed formula
for the values of $F_{k,\eps}(s)$.  However, if the square of the argument~$s$ is a rational number,
as is the case for all of the special values needed in this paper, one has the following general result.

\newcommand\FkeTable{
\begin{table}[tbp]
  \def\over{/}
  \[
  \begin{array}[htbp]{r r c c c r c} \toprule
    k&\eps&s&\DS(\eps,k)&b&\sigma&c \\ \midrule
      3&1&1&{{1}\over{18}}&0&\pm1&1\\
      4&1&1&{{1}\over{8}}&0&\pm1&1\\
      4&1&\sqrt{3}&{{1}\over{8}}&{{1}\over{12}}&\pm1&1\\
      5&1&1&{{1}\over{5}}&0&\pm1&1\\
      6&1&1&{{5}\over{18}}&0&\pm1&1\\
      6&1&\sqrt{3}&{{5}\over{18}}&{{1}\over{6}}&\pm1&1\\
      \midrule
      3&1&\sqrt{2}&{{1}\over{18}}&-{{1}\over{6}}&1&{{1}\over{3}}\\
      3&1&\sqrt{5}&{{1}\over{18}}&-{{1}\over{12}}&1&{{2}\over{3}}\\
      3&1&2\sqrt 2&{{1}\over{18}}&{{1}\over{4}}&-1&{{1}\over{3}}\\
      4&1&\sqrt{7}&{{1}\over{8}}&0&1&{{3}\over{4}}\\
      4&1&\sqrt{15}&{{1}\over{8}}&-{{1}\over{6}}&1&-{{1}\over{4}}\\
      4&1&\sqrt{5/3}&{{1}\over{8}}&-{{1}\over{6}}&1&{{1}\over{4}}\\
      5&1&2&{{1}\over{5}}&0&1&{{3}\over{5}}\\
      5&2&1&0&{{1}\over{10}}&-1&{{3}\over{5}}\\
      5&2&4&0&{{1}\over{10}}&-1&{{4}\over{5}}\\
      6&1&\sqrt{2}&{{5}\over{18}}&-{{1}\over{12}}&1&{{1}\over{3}}\\
      6&1&\sqrt{5}&{{5}\over{18}}&{{1}\over{12}}&1&{{2}\over{3}}\\
      6&1&\sqrt{11}&{{5}\over{18}}&{{1}\over{6}}&1&{{5}\over{6}}\\
      \midrule
      3&1&2&{{1}\over{18}}&{{1}\over{6}}&-1&\sqrt{3}-1\\
      4&1&\sqrt{2}&{{1}\over{8}}&-{{1}\over{8}}&1&\sqrt{2}-1\\
      4&1&\sqrt{5}&{{1}\over{8}}&-{{1}\over{4}}&1&\frac12\,\bigl(1-\sqrt{5}\bigr)\\
      4&1&3&{{1}\over{8}}&0&1&\sqrt{3}-1\\
      4&1&5&{{1}\over{8}}&0&1&3\,\sqrt{5}-6\\
      5&2&2&0&{{1}\over{10}}&-1&3\,\sqrt{5}-6\\
      6&1&\sqrt{7}&{{5}\over{18}}&{{2}\over{3}}&-1&\frac14\,\bigl(1-\sqrt{21}\bigr)\\
      \midrule
      3&1&\sqrt{3}&{{1}\over{18}}&-{{1}\over{6}}&1&\root3\of2-1\\
      4&1&3\sqrt 3&{{1}\over{8}}&-{{1}\over{12}}&1&\root3\of2-1\\
      \midrule
      3&1&2\,\sqrt{5}&{{1}\over{18}}&-{{1}\over{6}}&1&
      \frac13\,\bigl(1-\sqrt{5}+\sqrt{5\smash{(}\sqrt 5-\smash{1)/2}}
                    \bigr)\\
      3&1&4\sqrt 2&{{1}\over{18}}&-{{1}\over{12}}&1&
      \frac16\,\bigl(6-5\,\sqrt{2}+(4\sqrt 2+2)\sqrt{\sqrt{2}-1}\bigr)\\
      3&1&{{\sqrt{5}}\over{2}}&{{1}\over{18}}&0&1&
      \frac13\,\bigl(\sqrt{5}-1+\sqrt{5\smash{(}\sqrt 5-\smash{1)/2}}
      \bigr)\\
      4&1&3\,\sqrt{7}&{{1}\over{8}}&0&1&
      \frac1{16}\,\bigr(9+\sqrt{21}-\sqrt{26\,\sqrt{21}-114}\bigr)\\
      4&1&{{3}\over{\sqrt{7}}}&{{1}\over{8}}&0&1&
      \frac1{16}\,\bigr(9+\sqrt{21}+\sqrt{26\,\sqrt{21}-114}\bigr)\\
      \midrule
      3&1&5\,\sqrt{2}&{{1}\over{18}}&{{1}\over{6}}&-1&
      c=0.766\cdots,\;\,P(3c)=0\\
      3&1&{{5}\over{\sqrt{2}}}&{{1}\over{18}}&0&1&
      c=0.940\cdots,\;\,P(-3c)=0\\
      5&1&\sqrt{2}&{{1}\over{5}}&0&1&
      c=0.861\cdots,\;\,Q(c)=0\\
      5&2&\sqrt{2}&0&{{1}\over{10}}&-1&
      c=0.634\cdots,\;\,Q(-c)=0\\
    \bottomrule
  \end{array} 
  \]
  \smallskip
  \caption{Data needed to compute~$F_{k,\eps}$}
  \label{table:FkeTable}
\end{table}}
\FkeTable

\begin{prop} \label{Prop2}  
If $s>0$ is the square-root of a rational number, then the value of $F_{k,\eps}(s)$ for any $k$~and~$\eps$ is $i$ 
times the logarithm of a computable algebraic number.
\end{prop} 
\begin{proof} 
It is known from the theory of complex multiplication that the ratio of the values of the Dedekind \etafunc 
at any two arguments belonging to the same imaginary quadratic field is a computable algebraic number.  (More 
precisely, the value of $\eta(\tau)$ for $\tau$ belonging to any imaginary quadratic field is an algebraic 
multiple of a certain product of gamma-values, the so-called Chowla-Selberg number, that depends only on the field.
See \cite[Part~1, Section~6]{123} for further details.)
Since both $\frac{\eps^*+i/s}k$ and $\frac{-\eps^*+i/s}k$ belong to the imaginary quadratic
field~$\Q(is)$ when~$s^2$ is rational, this proposition is an immediate corollary of Proposition~\ref{Prop1}.
\end{proof}

We do not describe here the algorithm for computing special eta-values at CM
points, since it is standard in principle but is quite complicated.
We limit ourselves instead to giving the values of~$F_{k,\eps}(s)$
for the specific triples $(k,\eps,s)$ that are used in this paper.
These values are given by
\begin{align*}
  F_{k,\eps}(s)&=\pi\,\bigl(\DS(\eps,k)+b\bigr)+\tfrac\sigma 2\,\arccos(c)\;,\\
  F_{k,\eps^*}(1/s)&=\pi\,\bigl(\DS(\eps,k)-b\bigr)-\tfrac\sigma 2\,\arccos(c)\;.
\end{align*}
with $b\in\Q$, 
$\sigma\in\{\pm1\}$ and $c\in\overline\Q$ as in Table~\ref{table:FkeTable}.
The functions $P$ and $Q$
appearing in the last four lines of the table are the two sextic polynomials
\begin{align*}
  P(X)&=16\,X^6-416\,X^5+2440\,X^4+4880\,X^3-12615\,X^2-1826\,X-32159\;,\\
  Q(X)&=16\,X^6-32\,X^5+200\,X^4+560\,X^3+105\,X^2-402\,X-191\;.
\end{align*}

\subsection{Evaluation of \texorpdfstring{$A(k_+,\eps_+;k_-,\eps_-;G)$}{A(k+,e+;k-e-;G)} in terms of the Dedekind sums}\label{sec:evDedekind}
In this final subsection we place ourselves in the situation
of Theorem~\ref{Thm:B}. Specifically,
this means that we have two pairs of coprime numbers $(k_\pm,\eps_\pm)$
with $k_\pm$ positive and a $2\times2$
``gluing matrix" \hbox{$G=\sm \gll p\glb \sglr\in\text M_2(\Z)$} satisfying
conditions~\eqref{eq:det}--\eqref{eq:gcd}.
Combining Propositions~\ref{Prop2w.1} and~\ref{prop:dual},
we have~$\det G=-k_+k_-$ and
\be\label{glue} \gll-\eps_+^*\glb=Ak_+,\quad p+\eps_+^*\glr=Bk_+,\quad
p+\eps_-^*\gll=Ck_-,\quad \glr-\eps_-^*\glb=Dk_- \ee
for some integers $A,\,B,\,C,\,D$ with $(A,\glb)=(B,\glr)=(C,\gll)=(D,\glb)=1$.
We further assume that~$\glb>0$, $\gll\glb p\glr>0$ and
set $s_+=\sqrt{\glb\glr/\gll p}$, $s_-=\sqrt{\gll\glb/p\glr}$,
and $\rho=\pi-2\arg(\gll s_++i\glb)$.  Then the invariant we want 
to compute is the combination of $F_{k,\eps}$-values defined by
$$\mathcal F(k_+,\eps_+;k_-,\eps_-;G) \;:=\; \frac1\pi\,\biggl(F_{k_+,\eps_+}(s_+) \+ F_{k_-,\eps_-}(s_-) \+\frac\rho2\biggr)\,. $$

\begin{prop} \label{Prop3}  
The number $\mathcal F(k_+,\eps_+;k_-,\eps_-;G)$ is always rational and is given by
$$\mathcal F(k_+,\eps_+;k_-,\eps_-;G) \=
\frac 16\,\Bigl(\frac \gll{k_+\glb}+\frac \glr{k_-\glb}-12\,\DS(A,\glb)\Bigr)\;,$$
where $\DS(A,\glb)$ is the Dedekind sum as defined in~\eqref{eq:DedekindSum}.
\end{prop} 
\begin{proof} 
Set $\l=\frac{\eps_-^*A+B}{k_-}=\frac{C+\eps_+^*D}{k_+} = \frac{1-AD}{q}$,
which is an integer by~\eqref{eq:invmodn}.
The equations~\eqref{glue} can be rewritten as
\begin{equation*}
  \sm \sglr p{-\glb}{-\gll} \= \sm{k_-}{\eps_-^*}01 \, \g \,\sm1{\eps_+^*}0{k_+}
  \qquad\text{with}\qquad\g=\sm {-D}\l{-n}{-A}\in\SL\;.
\end{equation*}
It is easily checked that~$\g$
maps~$\tau_+=\frac{\eps_+^*+is_+^{-1}}{k_+}$
to~$\tau_-=\frac{-\eps_-^*+is_-^{-1}}{k_-}$.
From the
transformation law~\eqref{Ltransf} of~$\logeta$, we get
\begin{equation*}
  \logeta(\tau_-)-\logeta(\tau_+)
  =\frac14\,\Log\biggl(-\Bigl(\frac{\gll+i\glb s_+^{-1}}{k_+}\Bigr)^2\biggr)
  +\frac{\pi i}{12}\,\biggl(\frac{A+D}\glb-12\DS(A,\glb)\biggr)\;.
\end{equation*}
Because~$\logeta(-\bar z)=\overline{\logeta(z)}$, Proposition~\ref{Prop1} gives
\begin{align*}
  \mathcal F(k_+,\eps_+;k_-,\eps_-;G)
  &=\frac 2\pi\,\bigl(\Im\logeta(\tau_-)-\Im\logeta(\tau_+)\bigr)+\frac\rho{2\pi}
  +\frac{\eps_+^*}{6k_+}+\frac{\eps_-^*}{6k_-}\\
  &=\frac 16\biggl(\frac \gll{k_+\glb}+\frac \glr{k_-\glb}-12\,\DS(A,\glb)\biggr)\;.
  \qedhere
\end{align*}
\end{proof}

We observe that~\eqref{eq:thmA} and Proposition~\ref{Prop3}
give an alternative proof of Theorem~\ref{Thm:B}.

\enlargethispage{\baselineskip}

\begin{spacing}{0.92}

\end{spacing}

\end{document}